\documentclass[12pt]{article}

\usepackage{amsmath}
\usepackage{amssymb}
\usepackage{amsthm}
\usepackage{euler}
\usepackage[english]{babel}
\usepackage[margin=2.5cm]{geometry}

\makeatletter
\def\cleardoublepage{\clearpage\if@twoside \ifodd\c@page\else
    \thispagestyle{plain}\hbox{}\newpage\if@twocolumn\hbox{}\newpage\fi\fi\fi}

\def\ps@headings{\let\@mkboth\markboth
  \def\@oddfoot{}%
  \def\@evenfoot{}%
  \def\@evenhead{\small \sc\thepage\hfil\leftmark}
  \def\@oddhead{\small \sc \rightmark\hfil\thepage}
  \def\chaptermark##1{{
    \edef\@tempa{\ifnum \c@secnumdepth >\m@ne \@chapapp\ \thechapter. \fi}%
    \expandafter \markboth \expandafter{\@tempa ##1}{}}}%
  \def\schaptermark##1{\markboth {##1}{##1}}%
  \def\sectionmark##1{{
    \edef\@tempa{\ifnum \c@secnumdepth >\z@ \thesection. \fi}%
    \expandafter \markright \expandafter{\@tempa ##1}}}}

\def\thebibliography#1{\section*{References\@mkboth
 {References}{References}}\list
 {[\arabic{enumi}]}{\settowidth\labelwidth{[#1]}\leftmargin\labelwidth
 \advance\leftmargin\labelsep
 \usecounter{enumi}}
 \def\newblock{\hskip .11em plus .33em minus .07em}
 \sloppy\clubpenalty4000\widowpenalty4000
 \sfcode`\.=1000\relax}

\newif\if@restonecol
\def\theindex{\@restonecoltrue\if@twocolumn\@restonecolfalse\fi
\columnseprule \z@
\columnsep 35pt\twocolumn[\@makeschapterhead{Index}]
 \@mkboth{Index}{Index}\thispagestyle{plain}\parindent\z@
 \parskip\z@ plus .3pt\relax\let\item\@idxitem}
\def\@idxitem{\par\hangindent 40pt}

\def\endtheindex{\if@restonecol\onecolumn\else\clearpage\fi}

\def\footnoterule{\kern-3\p@ 
 \hrule width .4\columnwidth 
 \kern 2.6\p@} 
\@addtoreset{footnote}{chapter} 
\long\def\@makefntext#1{\parindent 1em\noindent 
 \hbox to 1.8em{\hss$^{\@thefnmark}$}#1}

\@addtoreset{equation}{section}

\renewcommand{\l@section}{\@dottedtocline{0}{1.5em}{2.3em}}
\renewcommand{\l@subsection}{\@dottedtocline{1}{3.8em}{3.2em}}
\renewcommand{\l@subsubsection}{\@dottedtocline{2}{7.0em}{4.1em}}

\makeatother

\usepackage{graphicx,float,wrapfig, caption}
\usepackage{pgf, tikz, tikz-cd} 
\graphicspath{ {pictures/}{tikzfigures/} } 
\usepackage{mathtools} 
\usetikzlibrary{arrows} 
\usepackage{extarrows}
\usepackage{standalone} 

\usepackage{enumitem}
\usepackage{upgreek}
\usepackage{lmodern}
\usepackage{array} 
\newcolumntype{C}[1]{>{\centering\arraybackslash}m{#1}}

\usepackage{hyperref}
\hypersetup{
    bookmarks=true,         
    unicode=false,          
    pdftoolbar=true,        
    pdfmenubar=true,        
    pdffitwindow=false,     
    pdfstartview={FitH},    
    pdftitle={Teichm\"uller curves in genus 2: Square-tiled surfaces and modular curves},    
    pdfauthor={Eduard Duryev},     
    pdfsubject={Subject},   
    pdfcreator={Creator},   
    pdfproducer={Producer}, 
    pdfkeywords={Translation surface} {Teichmuller curve} {Square-tiled surface}  {Modular curve} {Illumination} {Pagoda} {Elliptic covers} {Absolute period leaf} {Rel leaf} , 
    pdfnewwindow=true,      
    colorlinks=true,       
    linkcolor=ultramarine,          
    citecolor=ultramarine,        
    filecolor=magenta,      
    urlcolor=cyan           
}

\usepackage{thmtools}
\declaretheorem[numberwithin=section]{theorem}

\declaretheorem[sibling=theorem]{conjecture}
\declaretheorem[sibling=theorem]{proposition}
\declaretheorem[sibling=theorem]{lemma}
\declaretheorem[sibling=theorem]{corollary}

\declaretheorem[sibling=theorem, style=remark]{remark}
\newtheorem*{paritycor*}{Corollary of the parity conjecture}
\newtheorem*{theorem*}{Theorem}
\newtheorem*{question*}{Question}
\newtheorem*{conjecture*}{Conjecture}

\makeatletter
\newcommand{\rmnum}[1]{\romannumeral #1}
\newcommand{\Rmnum}[1]{\expandafter\@slowromancap\romannumeral #1@}
\makeatother

\usepackage{color}
\definecolor{adr}        {cmyk}{0.99,0.,0.,0.1}

\definecolor{ultramarine}{rgb}{0.07, 0.04, 0.56}

\newcommand{\N}{\mathbb N}
\newcommand{\Z}{\mathbb Z}
\newcommand{\Q}{\mathbb Q}
\newcommand{\R}{\mathbb R}
\newcommand{\C}{\mathbb C}
\newcommand{\Hyp}{\mathbb H}
\renewcommand{\P}{\mathbb P}

\newcommand{\M}{\mathcal M}

\newcommand{\A}{\mathcal A}
\newcommand{\OO}{\mathcal O}
\newcommand{\QQ}{\mathcal Q}
\newcommand{\CC}{\mathcal{C}}
\newcommand{\LL}{\mathcal{L}}
\newcommand{\EE}{\mathcal{E}}

\newcommand{\bfP}{\mathbf P}

\newcommand{\cM}{\overline{\mathcal M}}

\newcommand{\Jac}{\operatorname{Jac}}
\newcommand{\Pic}{\operatorname{Pic}}
\newcommand{\Aut}{\operatorname{Aut}}

\newcommand{\Stab}{\operatorname{Stab}}
\renewcommand{\mod}{\operatorname{mod}}

\newcommand{\Per}{\operatorname{Per}}
\newcommand{\RPer}{\operatorname{RPer}}
\newcommand{\Sym}{\operatorname{Sym}}
\newcommand{\Real}{\operatorname{Re}}
\newcommand{\Imag}{\operatorname{Im}}

\newcommand{\Isom}{\operatorname{Isom}}

\newcommand{\PSL}{\operatorname{PSL}}
\newcommand{\PSLZ}{\operatorname{PSL_2\Z}}
\newcommand{\SL}{\operatorname{SL}}
\newcommand{\SLZ}{\operatorname{SL_2\Z}}
\newcommand{\SLR}{\operatorname{SL_2\R}}
\newcommand{\GL}{\operatorname{GL^+_2\R}}

\renewcommand{\Gamma}{\Upgamma}
\DeclareMathOperator{\lcm}{lcm}

\newcommand\ZZ[1]{\Z/#1\Z}

\newcommand{\IWP}{\operatorname{IWP}}
\newcommand{\ST}{\operatorname{ST}}

\newcommand{\Mod}{\mathrm{Mod}}
\newcommand{\Aff}{\operatorname{Aff}}

\newcommand{\spans}{\operatorname{span}}

\usepackage{comment} 
\includecomment{comment1}
\includecomment{comment2}
\includecomment{comment3}
\includecomment{comment4}
\includecomment{comment6}
\includecomment{comment7}
\includecomment{comment8}
\includecomment{comment10}
\includecomment{comment11}
\includecomment{comment12}
\includecomment{commentB}

\title{Teichm\"uller curves in genus two: Square-tiled surfaces and modular curves}
\date{}
\author{Eduard Duryev}

\begin{document}

\maketitle

\vspace{-20pt}
\begin{abstract}
This work is a contribution to the classification of Teichm\"uller curves in the moduli space $\M_2$ of Riemann surfaces of genus 2. While the classification of primitive Teichm\"uller curves in $\M_2$ is complete, the classification of the imprimitive curves, which is related to branched torus covers and square-tiled surfaces, remains open. 

Conjecturally, the classification is completed as follows. Let $W_{d^2}[n] \subset \M_2$ be the 1-dimensional subvariety consisting of those $X \in \M_2$ that admit a primitive degree $d$ holomorphic map $\pi: X \to E$ to an elliptic curve $E$, branched over torsion points of order $n$. It is known that every imprimitive Teichm\"uller curve in $\M_2$ is a component of some $W_{d^2}[n]$. The {\em parity conjecture} states that (with minor exceptions) $W_{d^2}[n]$ has two components when $n$ is odd, and one when $n$ is even. In particular, the number of components of $W_{d^2}[n]$ does not depend on $d$.

In this work we establish the parity conjecture in the following three cases: (1) for all $n$ when $d=2,3,4,5$; (2) when $d$ and $n$ are prime and $n > (d^3-d)/4$; and (3) when $d$ is prime and $n > C_d$, where $C_d$ is an implicit constant that depends on $d$.

In the course of the proof we will see that the modular curve $X(d) = \overline{\Hyp \big/ \Gamma(d)}$ is itself a square-tiled surface equipped with a natural action of $\SLZ$. The parity conjecture is equivalent to the classification of the finite orbits of this action. It is also closely related to the following {\em illumination conjecture}: light sources at the cusps of the modular curve illuminate all of $X(d)$, except possibly some vertices of the square-tiling. Our results show that the illumination conjecture is true for $d \le 5$.
\end{abstract}

\newpage

\tableofcontents

\section{Introduction} \label{secintro}

This work is a contribution to the classification of Teichm\"uller curves in the moduli space $\M_2$ of Riemann surfaces of genus $2$. 

It is known (see \cite{McM06}) that a primitive Teichm\"uller curve in $\M_2$ is uniquely determined by two invariants: its discriminant $D$ and, when $D \equiv 1 \mod 8$, a spin invariant~$\epsilon \in \ZZ{2}$. Conjecturally, an imprimitive Teichm\"uller curve in $\M_2$ is uniquely determined by three integers: its degree $d$, its torsion $n$ and, when $n$ is odd, a spin invariant $\epsilon$. Our main result (Theorem~\ref{main}) establishes this conjecture in infinitely many cases.

In this introduction we will make this discussion more precise by introducing the {\em parity conjecture} (Conjecture~\ref{parity}). This conjecture can be expressed in several different ways:

\begin{enumerate}
\item[(\Rmnum{1})] in terms of algebraic curves $W_{d^2}[n] \subset \M_2$;

\item[(\Rmnum{2})] in terms of a natural square-tiling on the modular curve $X(d)$;

\item[(\Rmnum{3})] in terms of illumination of finite subsets $\A_{d^2}[n] \subset X(d)$ by the cusps of $X(d)$;

\item[(\Rmnum{4})] in terms of combinatorics of tilings of a surface of genus $2$ by squares; and

\item[(\Rmnum{5})] in terms of topological covers of a torus branched over a single point.
\end{enumerate}

Each of these perspectives will be discussed in turn below. While the parity conjecture and our main result are most easily stated using perspective (\Rmnum{1}), our proofs, to be sketched below, use perspectives (\Rmnum{2}) and (\Rmnum{3}). Perspective (\Rmnum{4}) relates the conjecture to combinatorics and translation surfaces, while perspective (\Rmnum{5}) allows to formulate it using only basic notions of topology.

We conclude with a survey of the previous research on the topic and pictures of natural square-tilings of the modular curves $X(d)$ for $d=2,3,4$ and $5$.


\noindent
{\bf \Rmnum{1}. Elliptic covers.} The first perspective gives the most succinct way of formulating the parity conjecture and our main result. 

Let $X$ be a Riemann surface of genus 2 and $E$ an elliptic curve. Elliptic cover is a ramified cover $\pi:X\to E$, where $X \in \M_2$ and $E$ is an elliptic curve. We call an elliptic cover {\em primitive} if the induced map $\pi_*: H_1(X,\Z) \to H_1(E,\Z)$ is a surjection. 

For each pair of integers $(d,n)$, where $d>1$ and $n\ge1$, consider the following~locus~in~$\M_2$:
$$
   W_{d^2}[n]  = \left\{ 
  X \in \M_2 \,\middle\vert\,
  \begin{aligned} 
  &\exists \mbox{ a primitive degree } d \mbox{ elliptic cover } \pi: X \to E,  \mbox{ with critical points}  \\ 
  &  x_1 \ne x_2 \in X \mbox{ such that } \pi(x_1) - \pi(x_2) \mbox{ has order } n\mbox{ in } \Jac(E)
  \end{aligned}
\right\}.
$$
Each $W_{d^2}[n]$ is a possibly reducible algebraic curve immersed in $\M_2$. It is known that the loci $W_{2^2}[1]$, $W_{3^2}[1]$ are empty, and $W_{4^2}[1]$, $W_{5^2}[1]$ are irreducible. The main~conjecture~states~that:

\begin{conjecture}[Parity conjecture] \label{parity}
Provided $(d,n) \ne (2,1), (3,1), (4,1)$ and $(5,1)$:\\ $W_{d^2}[n]$ is irreducible when $n$ is even, and consists of two irreducible components when~$n$~is~odd. 
\end{conjecture}

The main result of this work establishes the parity conjecture in infinitely many cases:

\begin{theorem} \label{main}
The parity conjecture holds for all $(d,n)$ such that:
\begin{enumerate}
\item[(\rmnum{1})]
$d=2,3,4,5$; or
\item[(\rmnum{2})]
$d$ and $n$ are prime and $n > (d^3-d)/4$; or
\item[(\rmnum{3})]
$d$ is prime and $n > C_d$, where $C_d$ is a constant that depends on $d$.
\end{enumerate}
\end{theorem}

Proof of Theorem~\ref{main} occupies \S\ref{secproof2}--\S\ref{secproof4}.

\noindent
{\bf Teichm\"uller curves.} The study of the parity conjecture is motivated by the following application. 
Let  $\M_g$ be the moduli space of Riemann surfaces of genus $g$ and define $\Omega\M_g \to \M_g$ to be the bundle of pairs $(X,\omega)$, where $\omega \ne 0$ is a holomorphic 1-form on a Riemann surface $X \in \M_g$. The {\em absolute periods} of $\omega$ will be denoted by $\Per(X,\omega) = \left\{ \int_\gamma \omega \mid \gamma \in H_1(X,\Z)\right\}$. There is a natural $\GL$ action on $\Omega\M_g$ that satisfies:
$$
\Per ( A \cdot (X,\omega) ) = A \cdot \Per (X,\omega).
$$
Let $\SL(X,\omega) \subset \SLR$ denote the stabilizer of $(X,\omega)$ under this action. If the stabilizer is a lattice in $\SLR$, then the image of the projection map $\GL \cdot (X,\omega) \to \M_g$ is an immersed algebraic curve $V \cong \Hyp \big/\SL(X,\omega) \to \M_g$. This immersion is an isometry with respect to the hyperbolic metric on $V$ and the Teichm\"uller metric on $\M_g$ and we refer to its image as the {\em Teichm\"uller curve} in $\M_g$ generated by $(X,\omega)$. 

\noindent
{\bf Classification in genus 2.}
Define a quadratic order $\OO_D \cong \Z[x] / (x^2 + b x + c)$, where $D = b^2 - 4c$.
For any $D\ge5$ with $D \equiv 0,1 \mod 4$, the {\em Weierstrass curve} is the following locus in $\M_2$:
\begin{flalign*}
& W_D =  \left\{ 
  X \in \M_2 \,\middle\vert\,
  \begin{aligned} 
   & \Jac(X) \mbox{ admits a real multiplication by $\OO_D$} \\
   & \mbox{with an eigenform with a double zero}
  \end{aligned}
\right\}.
\end{flalign*}
Every irreducible component of $W_D$ is a Teichm\"uller curve.
It is know that, when $D\ne 9$ and $D \equiv 1 \mod 8$, $W_{D}$ consists of two irreducible components $W^0_{D}$ and $W^1_{D}$ distinguished by the spin invariant $\epsilon$ (see \cite{McM05a}). It is also known that, when $n$ is odd, $W_{d^2}[n]$ has at least 2 components $W^0_{d^2}[n]$ and $W^1_{d^2}[n]$ distinguished by a slight generalization of the spin invariant $\epsilon$ (see \S\ref{secspin}).
As we will see in \S\ref{secbackground}, the parity conjecture suffices to complete the classification of Teichm\"uller curves in $\M_2$:

\begin{theorem} \label{paritycor}
The parity conjecture implies that the Teichm\"uller curves in $\M_2$ are given by:
\begin{itemize}
\item[(1)] $W_{D}$, where $D\ge 5$ and $D \equiv 0, 4$ or $5 \mod 8$ or $D=9$;

\item[(2)] $W_{D}^\epsilon$, where $D\ge 17$, $D \equiv 1 \mod 8$ and $\epsilon = 0$ or $1$;

\item[(3)] $W_{4^2}[1]$, $W_{5^2}[1]$ and $W_{d^2}[n]$, where $n$ is even;

\item[(4)] $W^\epsilon_{d^2}[n]$, where $d\cdot n >5$, $n$ is odd and $\epsilon = 0$ or $1$; and

\item[(5)]  the decagon curve generated by $\frac{dx}y$ on $y^2 = x^6 -x$.
\end{itemize}
\end{theorem}

The contribution of this work is to address the curves $(3)$ and $(4)$. They consist of imprimitive Teichm\"uller curves. We will discuss this in more details in \S\ref{secbackground}.

%
%
%


\noindent
{\bf \Rmnum{2}. Modular curves.} The second perspective relates the parity conjecture to a natural square-tiling of the modular curve:
$$X(d) = (\Hyp \cup \Q \cup \infty) \big/ \Gamma(d).$$

\noindent
{\bf Absolute period leaf $\A_{d^2}$.} 
Let $E_0 =  \C/\Z[i]$ be the square torus. The quadratic differential $dz^2$ on $\C$ descends to $E_0$ and the space $(E_0, |dz|^2)$ is isometric to a unit square with opposite sides identified. Let:
\begin{flalign*}
 \A^{\circ}_{d^2} = \left\{   (X,\omega) \in \Omega\M_2 \,\middle\vert\,
   \Per(X,\omega) = \Z[i] \mbox{ and } \displaystyle \int_X  |\omega|^2 = d  
\right\}.
\end{flalign*}
The {\em absolute period leaf} $\A_{d^2}$ is a smooth irreducible algebraic curve obtained as a completion of the locus $\A^\circ_{d^2} \subset \Omega \M_2$.

\noindent
{\bf Isomorphism with $X(d)$.} The modular curve $X(d)$ parametrizes elliptic curves $E$ with a choice of suitable basis for the $d$-torsion points $E[d]$. In \S\ref{secmodular} we will show that there exists a natural isomorphism $i: \A_{d^2} \xlongrightarrow{\sim} X(d)$, such that $\Jac(X)$ is isogenous to $E_0 \times E$, where $E = i(X,\omega)$. In particular, $X$ also admits a degree $d$ map to $E$. The isomorphism $i$ depends on the choice of an isomorphism $(\ZZ{d})^2 \cong E_0[d]$. We fix this choice once and for all and obtain an isomorphism that we denote by $\A_{d^2} \cong X(d)$.

\noindent
{\bf Square-tiling of $X(d)$.} Denote zeroes of $\omega$ by $z_1$ and $z_2$ and let:
$$
\rho = \int_{z_1}^{z_2} \omega
$$
be a (multivalued) holomorphic function on $\A^\circ_{d^2}$. The holomorphic quadratic differential $\tilde q = d\rho^2$ on $\A^\circ_{d^2}$ extends to a meromorphic quadratic differential $q$ on $\A_{d^2}$. For any $ (X,\omega) \in  \A^{\circ}_{d^2}$ there exists a primitive degree $d$ covering map $\pi: X \to E_0$ defined up to translation on $E_0$, such that $\pi^*(dz) = \omega$. The locus $\A^{\circ}_{d^2}$ is preserved by the action of $\SLZ \subset \SLR$ on $\Omega\M_2$ and the $\SLZ$ action on $\A^{\circ}_{d^2}$ extends to the action on $\A_{d^2} \cong X(d)$. Because $\int_{z_1}^{z_2} \omega = \int_{\pi(z_1)}^{\pi(x_2)} dz$, the metric space $(X(d), |q|)$ naturally decomposes as a union of unit squares compatible with this $\SLZ$ action (\S\ref{secabs}). We refer to this decomposition as the {\em square-tiling} of the modular curve $X(d)$.

The square-tilings of the modular curves $X(2), X(3), X(4)$ and $X(5)$ are illustrated in Figures~\ref{figX(2)intro},~\ref{figX(3)intro},~\ref{figX(4)intro} and~\ref{figX(5)intro}. 

In \S\ref{sectiling} we will explain how to generate the square-tilings of the modular curves in general and give some of their geometric properties.

\noindent
{\bf Reduction to $\SLZ$ action on $X(d)$.} Points of $\A_{d^2} \cong X(d)$ whose stabilizer is a lattice fall into one of the following finite subsets: 
\begin{flalign*}
\A_{d^2}[n]  = \left\{ (X,\omega) \in \A_{d^2} \, \middle\vert 
  \begin{aligned} 
  &  \mbox{integration of } \omega  \mbox{ defines } \pi: X \to E_0,  \\ 
  &  \mbox{whose critical points }  x_1 \ne x_2 \in X \mbox{ satisfy:} \\ 
  &  \pi(x_1) - \pi(x_2) \mbox{ has order } n\mbox{ in } \Jac(E_0)
  \end{aligned}
   \right\}.
\end{flalign*}
We will show that $\A_{d^2}[n]$ is the subset of primitive $n$-rational points of the squares in the tiling of $X(d)$ (\S\ref{secabs}). The action of $\SLZ$ preserves $\A_{d^2}[n] \subset X(d)$. The parity conjecture can be reformulated in terms of this action. In \S\ref{secbackground} we will show:

\begin{theorem} \label{thmintro1}
The number of irreducible components of $W_{d^2}[n]$ is equal to the number of $\SLZ$ orbits in $\A_{d^2}[n] \subset X(d)$.
\end{theorem}
We will use Theorem~\ref{thmintro1} and results of \S\ref{sectiling} to give a proof of the main result for $d=2$ (see \S\ref{secproof2}) and for all $(d,n)$, such that $d$ and $n$ are prime and $n > (d^3-d)/4$ (see \S\ref{secproofgen}). 

In \S\ref{secmodular} we will also see that the quadratic differential $(X(d),q)$ has no translation automorphisms and its $\GL$ orbit projection to $\M_g$, where $g$ is the genus of $X(d)$, is a point.


\noindent
{\bf \Rmnum{3}. Illumination.} The third perspective relates the parity conjecture to illumination on the modular curve $X(d)$. 

We say that a point $A \in X(d)$ {\em illuminates} a point $B \in X(d)$ if there is a geodesic segment in metric $|q|$ that connects $A$ to $B$ and does not pass through singularities of the metric. The illumination conjecture states that:

\begin{conjecture}[Illumination conjecture]
Light sources at the cusps of the modular curve illuminate all of $X(d)$ except possibly for some of the vertices of the square-tiling.
\end{conjecture}

One can easily verify that all of the $X(2)$, $X(3)$ and $X(4)$ are illuminated by their cusps (red points) by looking at Figures~\ref{figX(2)intro},~\ref{figX(3)intro} and~\ref{figX(4)intro}. However establishing the illumination conjecture for $X(5)$ (Figure~\ref{figX(5)intro}) requires more work (see \S\ref{secproof35}). In fact, we will show that there is a vertex of the square-tiling of $X(5)$ that is not illuminated by the cusps.

It turns out that the parity conjecture is strongly related to the illumination conjecture. In \S\ref{secillumination} we will show that the parity conjecture implies the illumination conjecture, using general results on illumination on translation surfaces (\cite{LMW16}) and the fact that the set of illuminated points is $\SLZ$ invariant. As for the converse, we will prove:

\begin{theorem} \label{thmillumintro}
Let $d$ be prime. Then, if the illumination conjecture holds for $d$, the parity conjecture holds for all $(d,n)$ with $n>1$.
\end{theorem}

We will use Theorem~\ref{thmillumintro} together with general results on illumination to prove the main result for all $(d,n)$, where $d$ is prime and $n> C_d$ (see \S\ref{secillumination}). We will establish the illumination conjecture for $d=3,4$ and $5$ and use Theorem~\ref{thmillumintro} to prove the main result for $d=3, 5$ (see~\S\ref{secproof35}). The proof for $d=4$ (see~\S\ref{secproof4}) is quite different in nature and will use the existence of the branched cover $X(4) \to X(3)$ that respects the square-tilings.


\noindent
{\bf \Rmnum{4}. Square-tiled surfaces.} The parity conjecture is also related to the ways of tiling a topological surface of genus 2 with squares, where only 4 or 8 corners of the squares come together at a vertex.

Let $\Sigma_2$ be a topological surface of genus 2. Preimages of the square under a suitable covering map $\pi: X \to E_0$ give a tiling of $\Sigma_2$ by $N = d \cdot n$ squares that we will call a {\em type $(d,n)$ square-tiling}. The $\SLZ$ action on 1-forms obtained by pulling back $dz$ via such covering maps gives an $\SLZ$ action on the square-tilings. In \S\ref{secbackground} we will show:
\begin{theorem}
The number of irreducible components of $W_{d^2}[n]$ is equal to the number of type $(d,n)$ square-tilings of $\Sigma_2$ up to the action of $\SLZ$.
\end{theorem}

The type $(2,2)$ square-tilings of $\Sigma_2$ and their $\SLZ$ orbits are illustrated in Figure~\ref{figtype22}. One can easily verify that exactly $8$ corners of the squares come together at each vertex of these tilings.

\begin{figure}[H]
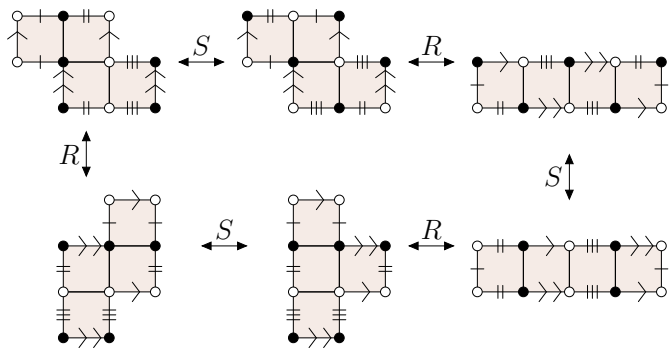

\captionsetup{singlelinecheck=off}
\centering
  \includestandalone[width=0.6\textwidth]{tikztype_2,2}
  \caption{The $\SLZ$ action on type $(2,2)$ square-tilings of a topological surface of genus $2$ presented by its generators $S = \protect\begin{pmatrix} 1 & 1 \\ 0 & 1 \protect\end{pmatrix}$ and $R = \protect\begin{pmatrix} 0 & -1 \\ 1 & 0 \protect\end{pmatrix}$.}
  \vspace{-5pt}
  \label{figtype22}
\end{figure}

The number of all reduced tilings (that do not admit a tiling by bigger squares) of $\Sigma_2$ by $N$ squares is given by (see \cite{EMS03} and \cite{KM17}):
$$
\displaystyle
 \frac{(N-2)(4N-3)}{12} \cdot \big\vert \PSL_2(\ZZ{N}) \big\vert + \sum_{\substack{d \ \mid \ N \\ d\ne N } } \frac{(d-1)d}{3N} \cdot \big\vert \PSL_2(\ZZ{d}) \big\vert
\cdot  \big\vert \SL_2(\ZZ{\tfrac Nd}) \big\vert. $$
The parity conjecture would imply that the formula is significantly simpler if one considers the tilings up to $\SLZ$ action:
\begin{equation*}
\left| \left\{
\begin{aligned}
\mbox{reduced tilings of } \\
\mbox{$\Sigma_2$ by $N$ squares}
\end{aligned}
\right\} \Big/ \SLZ \right| = 
 \sum_{\substack{n \ \mid \ N \\ n \mbox{ \footnotesize is odd} } } 2 \ \ +  \sum_{\substack{n \ \mid \ N \\ n \mbox{ \footnotesize is even} }} 1, \mbox{ for } N>5.
\end{equation*}

When $n$ is odd, let $t_{d,n,\epsilon}$ be the number of square-tiled surfaces of type $(d,n)$ and spin $\epsilon$. The formula for $t_{d,n,\epsilon}$ was proved for odd $d$ and conjectured for even $d$ in \cite{KM17}. In \S\ref{sectiling} we obtain this formula for any $d$ and $n>1$:
\begin{theorem}
For an arbitrary $d$ and $n>1$, the number of square-tiled surfaces of type $(d,n)$ and spin $\epsilon$ is:
$$
t_{d,n,0}  =
\displaystyle
\frac{d-1}{12n} \cdot \left| \PSL_2(\ZZ{d}) \right| \cdot  \left| \SL_2(\ZZ{n}) \right|,
$$
$$
t_{d,n,1} = 
\displaystyle
\frac{d-1}{4n} \cdot \left| \PSL_2(\ZZ{d}) \right| \cdot  \left| \SL_2(\ZZ{n}) \right|.
$$
\end{theorem}


\noindent
{\bf \Rmnum{5}. Topological torus covers.} The final perspective is related to the Hurwitz theory of branched covers of a torus. It gives a way to formulate the parity conjecture in purely topological terms.

Let $\Sigma_g$ be a closed, oriented topological surface of genus $g$ and $\pi: \Sigma_2 \to \Sigma_1$ a topological cover with two ramification points over a single branch point.
Two covers $\pi_1, \pi_2: \Sigma_2 \to \Sigma_{1}$ are {\em topologically equivalent} if there are orientation-preserving homeomorphisms $f_1: \Sigma_{1} \to \Sigma_{1}$ and $f_2: \Sigma_2 \to \Sigma_2$, such that $\pi_2 \circ f_2 = f_1 \circ \pi_1$. A cover $\pi: \Sigma_2 \to \Sigma_{1}$ is called {\em type $(d,n)$ cover} if it factors through:
$$
\Sigma_g \xrightarrow{d} \Sigma_{1} \xrightarrow{n} \Sigma_{1},
$$
where $\Sigma_{1} \xrightarrow{n} \Sigma_{1}$ is a cover of tori of degree $n$ and $\Sigma_g \xrightarrow{d} \Sigma_{1}$ is a primitive cover of degree $d$ branched over two distinct points unless $n=1$.
In \S\ref{secbackground} we will show:
\begin{theorem}
The number of irreducible components of $W_{d^2}[n]$ is equal to the number of topological classes of type $(d,n)$ covers $\pi: \Sigma_2 \to \Sigma_{1}$.
\end{theorem}

Questions about topological classes of branched covers have a long history dating back to Hurwitz. He used representation theory of symmetric groups to treat the topological classes of branched covers of the sphere.


\noindent
{\bf Previous results in genus 2.}
We now move to the references. 
Primitive Teichm\"uller curves in $\M_2$ were classified by McMullen in a series of works \cite{McM05a},  \cite{McM05b} and \cite{McM06}. In \cite{McM06} it was shown that every primitive Teichm\"uller curve in $\M_2$ is an irreducible component of $W_D$ or the decagon curve and in \cite{McM05a} it was shown that the Weierstrass curve $W_D$ has $1$ or $2$ components depending on the values of $D \mod 8$. The irreducible components of $W_D$ are primitive Teichm\"uller curves if and only if $D\ne d^2$. The components of $W_{d^2}$ for prime $d$ were also classified in \cite{HL06} using square-tiled surfaces.

The Euler characteristic of $W_D$ was computed by Bainbridge in \cite{Bai07}. In \cite{Muk14} Mukamel computed the number of elliptic points of $W_D$. The foliations of Hilbert modular surfaces $X_D$ for a general $D$ are discussed in \cite{McM07b}. The geometry and dynamics of the absolute period leaves $\A_D$ in the case $D\ne d^2$ were studied in \cite{McM14}. Our work extends these results to the case $D=d^2$. The major difference between these two cases is that $\A_D \cong \Hyp$, when $D \ne d^2$, and $\A_{d^2} \cong X(d)$. For more on real multiplication and Hilbert modular surfaces see \cite{McM03} and \cite{McM07a}. For another perspective see \cite{Cal04}.

\noindent
{\bf Previous work on the parity conjecture.}
The study of the square-tiling of $\A_{d^2}$ was initiated by Schmoll in \cite{Sch05}. The connection to the modular curves was established in \cite{Kani03}. The parity conjecture has been proved for $d=2$ and arbitrary $n$ in \cite{HWZ}, and investigated using a computer program by Delecroix and Leli\`evre. Another approach to the conjecture is presented in \cite{KM17}.

\noindent
{\bf Related research.}
Work \cite{EO01} relates the number of square-tiled surfaces to quasi-modular forms and volumes of moduli spaces. An algebro-geometric approach to Teichm\"uller curves generated by square-tiled surfaces is given in \cite{Mol05}, \cite{Che10} and \cite{KM17}.

The {\em cylinder coordinates} presented in \cite{EMS03} are used in this work to study the square-tiling of $\A_{d^2}$. For the most recent results on illumination on translation surfaces see \cite{HST08} and \cite{LMW16}.

\noindent
{\bf Background references.}
For expositions on $\GL$ action on $\Omega\M_g$ see \cite{MT02}, \cite{Zor06}, \cite{FM14} and \cite{Wri15}.
For a survey on square-tiled surfaces see \cite{Zm11}.

The first examples of primitive Teichm\"uller curves were given by \cite{Vee89} and came from the study of billiards in rational polygons. Further references on primitive Teichm\"uller curves in higher genera are \cite{Mol08}, \cite{MMW17}, \cite{Fil17}. 

The theory of topological classes of branched covers of the sphere starts with works \cite{Lur71} (1871), \cite{Cle73} (1873) and \cite{Hur91} (1891). More general results for generic covers of any topological closed surfaces are obtained in \cite{GK87} (1987). However the case of non-generic covers is widely unexplored. Some results on topological classes of non-generic covers of the sphere can be found in \cite{Pro88}.

\noindent
{\bf Pictures of square-tilings of $X(d)$.} We conclude by giving pictures of the square-tilings of all modular curves $X(d)$ of genus $0$: $X(2), X(3), X(4)$ and $X(5)$. 

The singularities of the flat metric $|q|$ are simple zeroes of $q$ (black points) and simple poles of $q$ (white and red points). We describe identifications of the edges of the squares in terms of horizontal and vertical intervals joining the singularities. The ones that are labeled with numbers and strokes are identified via parallel translations. The adjacent ones that are labeled with arrows are identified via rotations by $\pi$. 

The cusps of $X(d)$ are labeled with red points. In particular, one can easily verify that $X(2), X(3), X(4)$ and $X(5)$ have 3, 4, 6 and 12 cusps respectively.

\begin{comment1}
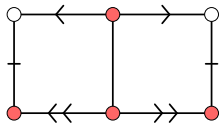
\begin{figure}[H] 
\centering
\definecolor{ffqqqq}{rgb}{1.,0.4,0.4}
\definecolor{ffffff}{rgb}{1.,1.,1.}
\begin{tikzpicture}[line cap=round,line join=round,>=triangle 45,x=1.0cm,y=1.0cm]
\clip(-12.5,1.) rectangle (12.5,5.);

\draw [line width=1.pt] (-2.,2.)-- (-2.,4.);
\draw [line width=1.pt] (-2.12,3.) -- (-1.88,3.);
\draw [line width=1.pt] (0.,2.)-- (-2.,2.);
\draw [line width=1.pt] (0.,2.)-- (0.,4.);
\draw [line width=1.pt] (0.,4.)-- (2.,4.);
\draw [line width=1.pt] (1.15,4.) -- (1.,3.82);
\draw [line width=1.pt] (1.15,4.) -- (1.,4.18);
\draw [line width=1.pt] (2.,4.)-- (2.,2.);
\draw [line width=1.pt] (2.12,3.) -- (1.88,3.);
\draw [line width=1.pt] (2.,2.)-- (0.,2.);
\draw [line width=1.pt] (0.,4.)-- (-2.,4.);
\draw [line width=1.pt] (-1.15,4.) -- (-1.,4.18);
\draw [line width=1.pt] (-1.15,4.) -- (-1.,3.82);

\draw [line width=1.pt] (-1.3,2.) -- (-1.15,2.18);
\draw [line width=1.pt] (-1.3,2.) -- (-1.15,2-0.18);
\draw [line width=1.pt] (-1.,2.) -- (-0.85,2.18);
\draw [line width=1.pt] (-1.,2.) -- (-0.85,2-0.18);

\draw [line width=1.pt] (1.3,2.) -- (1.15,2-0.18);
\draw [line width=1.pt] (1.3,2.) -- (1.15,2.18);
\draw [line width=1.pt] (1.,2.) -- (0.85,2-0.18);
\draw [line width=1.pt] (1.,2.) -- (0.85,2.18);

\begin{scriptsize}
\draw [fill=ffqqqq] (-2.,2.) circle (4pt);
\draw [fill=ffffff] (-2.,4.) circle (4pt);
\draw [fill=ffqqqq] (0.,4.) circle (4pt);
\draw [fill=ffqqqq] (0.,2.) circle (4pt);
\draw [fill=ffffff] (2.,4.) circle (4pt);
\draw [fill=ffqqqq] (2.,2.) circle (4pt);

\end{scriptsize}
\end{tikzpicture}
\caption{The square-tiling of the modular curve $X(2)\cong \A_{4}$.}
\label{figX(2)intro}
\end{figure}

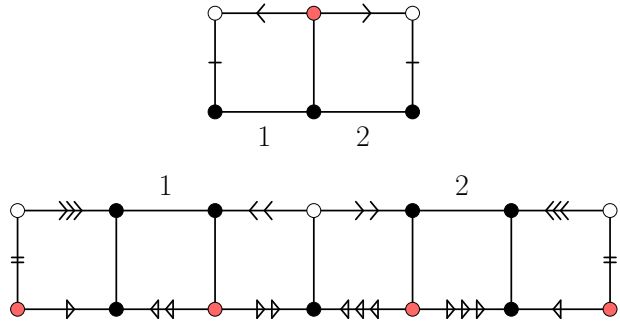
\begin{figure}[H] 
\centering
\definecolor{ffqqqq}{rgb}{1.,0.4,0.4}
\definecolor{ffffff}{rgb}{1.,1.,1.}
\begin{tikzpicture}[line cap=round,line join=round,>=triangle 45,x=1.0cm,y=1.0cm]
\clip(-12.5,-3.) rectangle (12.5,5.);

\draw [line width=1.pt] (-2.,2.)-- (-2.,4.);
\draw [line width=1.pt] (-2.12,3.) -- (-1.88,3.);
\draw [line width=1.pt] (0.,2.)-- (-2.,2.);
\draw [line width=1.pt] (0.,2.)-- (0.,4.);
\draw [line width=1.pt] (0.,4.)-- (2.,4.);
\draw [line width=1.pt] (2.,4.)-- (2.,2.);
\draw [line width=1.pt] (2.12,3.) -- (1.88,3.);
\draw [line width=1.pt] (2.,2.)-- (0.,2.);
\draw [line width=1.pt] (0.,4.)-- (-2.,4.);

\draw [line width=1.pt] (1.15,4.) -- (1.,3.82);
\draw [line width=1.pt] (1.15,4.) -- (1.,4.18);

\draw [line width=1.pt] (-1.15,4.) -- (-1.,4.18);
\draw [line width=1.pt] (-1.15,4.) -- (-1.,3.82);

\draw [line width=1.pt] (-6.,-2.)-- (-6.,0.);
\draw [line width=1.pt] (-6.12,-1.05) -- (-5.88,-1.05);
\draw [line width=1.pt] (-6.12,-0.95) -- (-5.88,-0.95);
\draw [line width=1.pt] (-6.,0.)-- (-4.,0.);
\draw [line width=1.pt] (-4.85,0.) -- (-5.,-0.18);
\draw [line width=1.pt] (-4.85,0.) -- (-5.,0.18);
\draw [line width=1.pt] (-4.,-2.)-- (-6.,-2.);
\draw [line width=1.pt] (-4.,-2.)-- (-4.,0.);
\draw [line width=1.pt] (-4.,0.)-- (-2.,0.);
\draw [line width=1.pt] (-2.,-2.)-- (-4.,-2.);
\draw [line width=1.pt] (-3.3,-2.) -- (-3.15,-1.82);
\draw [line width=1.pt] (-3.3,-2.) -- (-3.15,-2.18);
\draw [line width=1.pt] (-3.15,-1.82) -- (-3.15,-2.18);
\draw [line width=1.pt] (-3.,-2.) -- (-2.85,-1.82);
\draw [line width=1.pt] (-3.,-2.) -- (-2.85,-2.18);
\draw [line width=1.pt] (-2.85,-1.82) -- (-2.85,-2.18);
\draw [line width=1.pt] (-2.,-2.)-- (-2.,0.);
\draw [line width=1.pt] (0.,-2.)-- (-2.,-2.);
\draw [line width=1.pt] (0.,-2.)-- (0.,0.);
\draw [line width=1.pt] (0.,0.)-- (2.,0.);
\draw [line width=1.pt] (2.,0.)-- (2.,-2.);
\draw [line width=1.pt] (2.,-2.)-- (0.,-2.);

\draw [line width=1.pt] (0.85,-2.) -- (1.,-1.82);
\draw [line width=1.pt] (0.85,-2.) -- (1.,-2.18);
\draw [line width=1.pt] (1.,-1.82) -- (1.,-2.18);
\draw [line width=1.pt] (1.15,-2.) -- (1.3,-1.82);
\draw [line width=1.pt] (1.15,-2.) -- (1.3,-2.18);
\draw [line width=1.pt] (1.3,-1.82) -- (1.3,-2.18);
\draw [line width=1.pt] (0.55,-2.) -- (0.7,-1.82);
\draw [line width=1.pt] (0.55,-2.) -- (0.7,-2.18);
\draw [line width=1.pt] (0.7,-1.82) -- (0.7,-2.18);

\draw [line width=1.pt] (2.,0.)-- (4.,0.);
\draw [line width=1.pt] (4.,0.)-- (4.,-2.);
\draw [line width=1.pt] (4.,-2.)-- (2.,-2.);
\draw [line width=1.pt] (6.,0.)-- (6.,-2.);
\draw [line width=1.pt] (6.12,-0.95) -- (5.88,-0.95);
\draw [line width=1.pt] (6.12,-1.05) -- (5.88,-1.05);
\draw [line width=1.pt] (6.,-2.)-- (4.,-2.);

\draw [line width=1.pt] (4.85,-2.) -- (5.,-1.82);
\draw [line width=1.pt] (4.85,-2.) -- (5.,-2.18);
\draw [line width=1.pt] (5.,-1.82) -- (5.,-2.18);
\draw [line width=1.pt] (-4.85,-2.) -- (-5.,-2.18);
\draw [line width=1.pt] (-4.85,-2.) -- (-5.,-1.82);
\draw [line width=1.pt] (-5.,-2.18) -- (-5.,-1.82);

\draw [line width=1.pt] (-3.3+8,0.) -- (-3.15+8,-1.82+2);
\draw [line width=1.pt] (-3.3+8,0.) -- (-3.15+8,-0.18);
\draw [line width=1.pt] (-3.+8,0.) -- (-2.85+8,-1.82+2);
\draw [line width=1.pt] (-3.+8,0.) -- (-2.85+8,-0.18);

\draw [line width=1.pt] (-0.7-4,0.) -- (-0.85-4,-0.18);
\draw [line width=1.pt] (-0.7-4,-0.) -- (-0.85-4,-1.82+2);
\draw [line width=1.pt] (-1.-4,-0.) -- (-1.15-4,-0.18);
\draw [line width=1.pt] (-1.-4,-0.) -- (-1.15-4,-1.82+2);

\draw [line width=1.pt] (-0.7,-2.) -- (-0.85,-2.18);
\draw [line width=1.pt] (-0.7,-2.) -- (-0.85,-1.82);
\draw [line width=1.pt] (-0.85,-2.18) -- (-0.85,-1.82);
\draw [line width=1.pt] (-1.,-2.) -- (-1.15,-2.18);
\draw [line width=1.pt] (-1.,-2.) -- (-1.15,-1.82);
\draw [line width=1.pt] (-1.15,-2.18) -- (-1.15,-1.82);

\draw [line width=1.pt] (3.15,-2.) -- (3.,-2.18);
\draw [line width=1.pt] (3.15,-2.) -- (3.,-1.82);
\draw [line width=1.pt] (3.,-2.18) -- (3.,-1.82);
\draw [line width=1.pt] (2.85,-2.) -- (2.7,-2.18);
\draw [line width=1.pt] (2.85,-2.) -- (2.7,-1.82);
\draw [line width=1.pt] (2.7,-2.18) -- (2.7,-1.82);
\draw [line width=1.pt] (3.45,-2.) -- (3.3,-2.18);
\draw [line width=1.pt] (3.45,-2.) -- (3.3,-1.82);
\draw [line width=1.pt] (3.3,-2.18) -- (3.3,-1.82);

\draw [line width=1.pt] (0.,0.)-- (-2.,0.);

\draw [line width=1.pt] (1.3,0.) -- (1.15,-0.18);
\draw [line width=1.pt] (1.3,0.) -- (1.15,0.18);
\draw [line width=1.pt] (1.,0.) -- (0.85,-0.18);
\draw [line width=1.pt] (1.,0.) -- (0.85,0.18);

\draw [line width=1.pt] (-1.3,0.) -- (-1.15,0.18);
\draw [line width=1.pt] (-1.3,0.) -- (-1.15,-0.18);
\draw [line width=1.pt] (-1.,0.) -- (-0.85,0.18);
\draw [line width=1.pt] (-1.,0.) -- (-0.85,-0.18);

\draw [line width=1.pt] (6.,0.)-- (4.,0.);
\draw [line width=1.pt] (4.85,0.) -- (5.,0.18);
\draw [line width=1.pt] (4.85,0.) -- (5.,-0.18);


\draw (-1,1.5) node {\Large 1};
\draw (1,1.5) node {\Large 2};

\draw (-3,0.5) node {\Large 1};
\draw (3,0.5) node {\Large 2};


%
%
%

\begin{scriptsize}
\draw [fill=black] (-2.,2.) circle (4pt);
\draw [fill=ffffff] (-2.,4.) circle (4pt);
\draw [fill=ffqqqq] (0.,4.) circle (4pt);
\draw [fill=black] (0.,2.) circle (4pt);
\draw [fill=ffffff] (2.,4.) circle (4pt);
\draw [fill=black] (2.,2.) circle (4pt);
\draw [fill=ffqqqq] (-6.,-2.) circle (4pt);
\draw [fill=ffffff] (-6.,0.) circle (4pt);
\draw [fill=black] (-4.,0.) circle (4pt);
\draw [fill=black] (-4.,-2.) circle (4pt);
\draw [fill=black] (-2.,0.) circle (4pt);
\draw [fill=ffqqqq] (-2.,-2.) circle (4pt);
\draw [fill=ffffff] (0.,0.) circle (4pt);
\draw [fill=black] (0.,-2.) circle (4pt);
\draw [fill=black] (2.,0.) circle (4pt);
\draw [fill=ffqqqq] (2.,-2.) circle (4pt);
\draw [fill=black] (4.,0.) circle (4pt);
\draw [fill=black] (4.,-2.) circle (4pt);
\draw [fill=ffffff] (6.,0.) circle (4pt);
\draw [fill=ffqqqq] (6.,-2.) circle (4pt);
\end{scriptsize}
\end{tikzpicture}
\caption{The square-tiling of the modular curve $X(3) \cong \A_{9}$.}
\label{figX(3)intro}
\end{figure}

\begin{figure}[H]
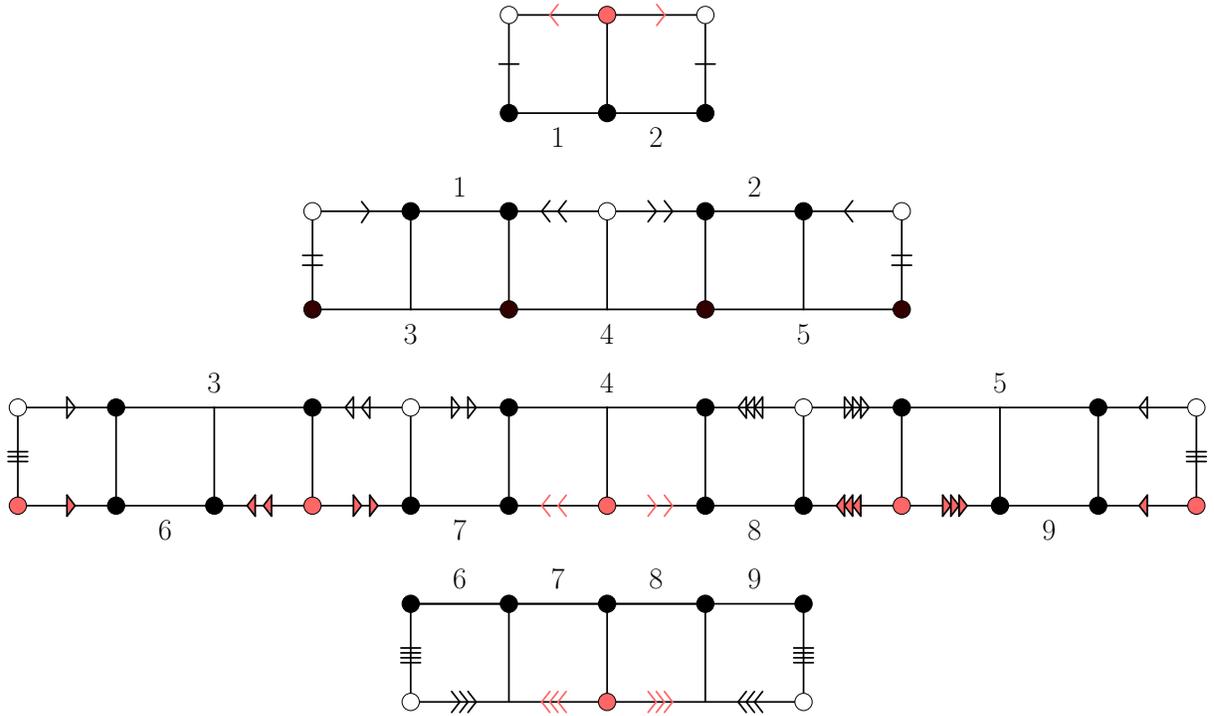

\centering
  \includestandalone[width=\textwidth]{tikzX4}
  \caption{The square-tiling of the modular curve $X(4) \cong \A_{16}$.}
  \vspace{-10pt}
  \label{figX(4)intro}
\end{figure}

Given the complexity of the square-tiling of $X(5)$, we prefer not to use arrows and strokes as labels, instead we describe the missing identifications as follows. The vertical edges of each horizontal rectangle are identified via parallel translation. The horizontal edges that are are labeled with letters and the adjacent horizontal line segments that start at a pole (white or red) and end at a singularity (black, white or red) are identified via rotations by $\pi$. 

\begin{figure}[H]
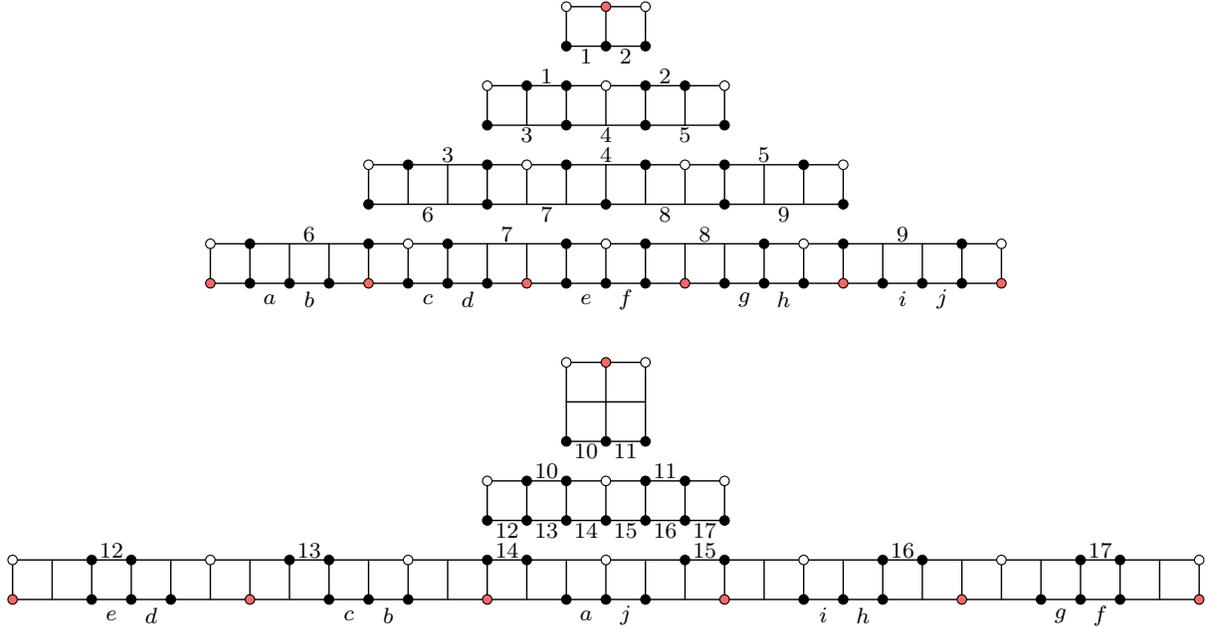
 
  \includestandalone[width=\textwidth]{tikzX5}
  \caption{The square-tiling of the modular curve $X(5) \cong \A_{25}$.}
  \label{figX(5)intro}
\end{figure}
\end{comment1}

\noindent
{\bf Appendices.} 
The remaining cases of the parity conjecture are open. In Appendix~\ref{sectioncounts} we review some counts of elliptic covers and give bounds on the number of irreducible components of $W_{d^2}[n]$ that works for all $d$ and $n$ (see Corollary~\ref{thmsimplebound}).

In Appendix~\ref{secpagoda} we will also give a simpler geometric description of the square-tilings of $X(d)$ for every prime $d$ and explain the pagoda structure of the modular curves that arise in this case. In particular, we will show (see Corollary~\ref{corgraph}):

\begin{theorem}
For every prime $d$, the modular curve $X(d)$ carries an embedded trivalent graph, well-defined up to the action of $\Aut(X(d))$, whose complement is a union of $\frac{d-1}2$ disks.
\end{theorem}


\noindent
{\bf Acknowledgments.} 
This is a version of my PhD thesis. I would like to thank my advisor Curtis T. McMullen for his guidance throughout my PhD starting from introducing me to the main problem of my thesis and inspiring my interest in it to engaging me in a thorough work over this text. I would also like to thank other people with whom I had many useful conversations and whose ideas contributed to my work: Matt Bainbridge, Dawei Chen, Vincent Delecroix, Philip Engel, Simion Filip, Pascal Hubert, Erwan Lanneau, Ronen Mukamel, Duc-Manh Nguyen, Anand Patel, Ananth Shankar, Peter Smillie, Alex Wright and Anton Zorich.


\section{Background and equivalence of conjectures} \label{secbackground}

In this section we give some background on Abelian differentials and establish the equivalence of different approaches to the parity conjecture mentioned in the introduction. In particular we will define Teichm\"uller curves, square-tiled surfaces, type $(d,n)$ covers, and type $(d,n)$ elliptic differentials. We will show:

\begin{theorem} \label{thmequivconj}
For any $(d,n)$, where $d>1$ and $n \ge 1$, the following finite sets are in bijection:
\begin{enumerate}
\item[(a)]
irreducible components of $W_{d^2}[n]$;

\item[(b)]
$\SLZ$ orbits in $\A_{d^2}[n]$;

\item[(c)]
$\SLZ$ orbits of type $(d,n)$ square-tiled surfaces;

\item[(d)]
topological classes of type $(d,n)$ covers of a torus; and

\item[(e)]
$\GL$ invariant loci of type $(d,n)$ elliptic differentials $\Omega W_{d^2}[n]$.
\end{enumerate}
\end{theorem}

\noindent
{\bf  Teichm\"uller curves.} Let $\M_g$ be the moduli space of Riemann surfaces of genus $g$. A {\em Teichm\"uller curve} is an isometric immersion $f:V \to \M_g$ with respect to the hyperbolic metric on an algebraic curve $V$ and the Teichm\"uller metric on $\M_g$.  Any pair $(X,q)$, where $q \ne 0$ is a holomorphic quadratic differential on a Riemann surface $X$, generates a complex geodesic $\tilde f: \Hyp \to \M_g$. When $\displaystyle \Stab(\tilde f) = \left\{ A \in \Isom(\Hyp) \cong \PSL_2(\R) \mid \tilde f ( A \tau ) = \tilde f ( \tau) \mbox{ for all } \tau \in \Hyp \right\}$ is a lattice in $\PSL_2(\R)$, the complex geodesic $\tilde f$ descends under the quotient by $\Stab(\tilde f)$ to a Teichm\"uller curve in $\M_g$.

We are going to focus on Tiechm\"uller curves generated by quadratic differentials of the form $(X,q)=(X,\omega^2)$, where $\omega$ is a holomorphic 1-form on $X$. Note however that the following discussion can be generalized for any quadratic differential.

\noindent 
{\bf Abelian differentials.} Let $X \in \M_g$ be a Riemann surface of genus $g$ and let $\Omega(X)$ denote the space of holomorphic 1-forms on $X$. Let $\omega \in \Omega(X)$ be a non-zero holomorphic 1-form on $X$. A pair $(X, \omega)$ is called an {\em Abelian differential}. The set of zeroes of $\omega$ is called the set of {\em conical points} or {\em singularities} and will be denoted by $Z(\omega)$. Any geodesic segment in the singular flat metric $\left| \omega \right|^2$ that starts and ends at singularities is called a {\em saddle connection}. 

For any Abelian differential $(X,\omega)$ the integration of $\omega$ produces a {\em translation structure}, an atlas of complex charts on $X \setminus Z(\omega)$ with parallel translations $z\mapsto z+c$ as transition functions. A neighborhood of a conical point possesses a singular flat structure obtained by pulling back flat metric on a disk via the map $z \mapsto z^k$. Abelian differentials are also called {\em translation surfaces}.  
The total space $\Omega \M_g$ of the bundle $\Omega \M_g \to \M_g$ is called the {\em moduli space of translation surfaces} or the {\em moduli space of Abelian differentials}. 

Any positive integer partition $\kappa =  (k_1, \ldots, k_n)$ of $2g-2$ defines a {\em stratum} of the moduli space of Abelian differentials:
$$\Omega \M_g (\kappa) = \left\{ (X, \omega)  \,\middle\vert\, X \in \M_g, \ Z(\omega) = k_1 p_1 + ... + k_n p_n, \mbox{ where } p_i  \mbox{ are distinct points of } X \right\}.
$$
For example, $\Omega \M_2$ consists of two strata $\Omega \M_2 (1,1) \textrm{ and } \Omega \M_2 (2)$.

The {\em period map} map $\Per_{(X,\omega)}: H_1(X, Z(\omega), \Z) \to \C$ is defined by $\Per_{(X,\omega)}(\gamma) = \int_{\gamma} \omega$. The period map is a local homeomorphism on each stratum $\Omega \M_g(\kappa)$. The {\em relative periods} of $(X,\omega)$ is the following subset of $\C$: 
$$\RPer(X,\omega) = \left\{ \int_\gamma \omega  \,\middle\vert\, \gamma \in H_1(X,(\omega), \Z)\right\}.$$

Abelian differentials can be presented as polygons in $\R^2 \cong \C$ with pairs of equal parallel sides identified by translations. The set $Z(\omega)$ is then contained in the set of vertices of the polygon. 
The sides of the polygon are saddle connections and, if viewed as vectors in $\C$, belong to $\RPer(X,\omega)$.  

\noindent
{\bf $\GL$ action.}
A natural $\GL$-action on $\R^2\cong \C$ induces an action on the relative periods and hence on $\Omega \M_g(\kappa) \subset \Omega\M_g$. By the results of \cite{EMM15} and \cite{Fil16} the projections of the $\GL$-orbit closures in $\Omega \M_g$ to $\M_g$ are algebraic subvarieties. In particular, the Teichm\"uller curve generated by a quadratic differential $(X, \omega^2)$ is a projection of a closed orbit $\GL \cdot (X,\omega)$ to $\M_g$.

\begin{figure}[H]    \centering
     \includegraphics[width=0.68\textwidth]{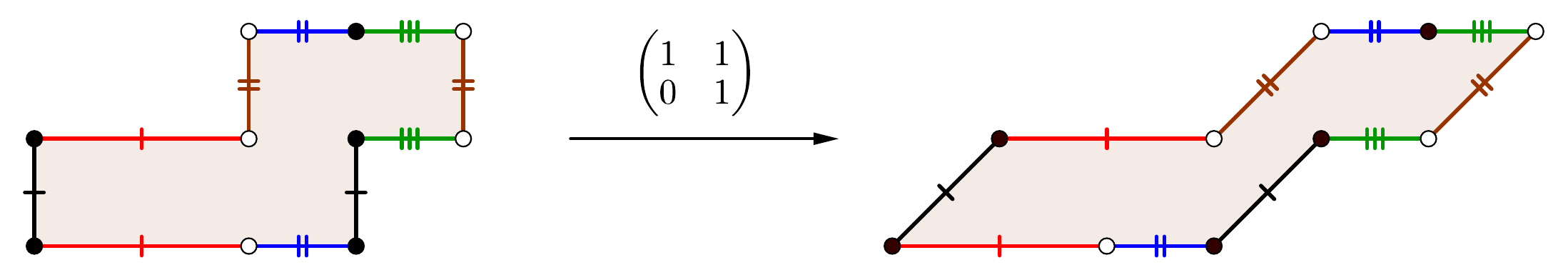}
    \caption{Action of a matrix on an Abelian differential.}
      \label{action}
\end{figure}

We will denote the stabilizer of $(X,\omega)$ under the action of $\GL$ by $\SL(X,\omega)$. The stabilizer $\SL(X,\omega)$ is a subgroup of $\SLR$, and for Teichm\"uller curves $\SL(X,\omega) \subset \SLR$ is a lattice.

\noindent
{\bf Teichm\"uller curves in $\M_2$.}
Here we review the results on the classification of Teichm\"uller curves in $\M_2$. First, we define the following loci in $\M_2$ (see \cite{McM06}):
\begin{itemize}
\item[] $\quad W_D \quad = \quad \left\{ 
  X \in \M_2 \,\middle\vert\,
  \begin{aligned} 
  & \Jac(X) \mbox{ admits a real multiplication by } \OO_D, \mbox{ and } \\ 
  & X \mbox{ carries an eigenform } \omega \mbox{ with a double zero } 
  \end{aligned}  
\right\}$,

\item[] $\ W_{d^2}[n] \  = \quad \left\{ 
  X \in \M_2 \,\middle\vert\,
  \begin{aligned} 
  &\exists \mbox{ a primitive degree } d \mbox{ elliptic cover } \pi: X \to E  \mbox{ for some} \\ 
  &  E\in\M_1, \mbox{ whose ramification points }  x_1 \ne x_2 \in X  \\ 
  &  \mbox{satisfy: }\pi(x_1)-\pi(x_2) \mbox{ has order } n\mbox{ in } \Jac(E)
  \end{aligned}
\right\}$.
\end{itemize}

It follows from the results of \cite{McM06} that every Teichm\"uller curve in $\M_2$ is an irreducible component of one of the following algebraic curves:
\begin{itemize}
\item[(1)] $W_{D}$, where $D\ge 5$ and $D \equiv 0, 4$ or $5 \mod 8$ or $D=9$;

\item[(2)] $W_{D}^\epsilon$, where $D\ge 17$, $D \equiv 1 \mod 8$ and $\epsilon = 0$ or $1$;

\item[(3)] $W_{4^2}[1]$, $W_{5^2}[1]$ and $W_{d^2}[n]$, where $n$ is even;

\item[(4)] $W^\epsilon_{d^2}[n]$, where $d\cdot n >5$, $n$ is odd and $\epsilon = 0$ or $1$ is the {\em spin invariant} (see \S\ref{secspin}); and

\item[(5)]  the decagon curve generated by $\frac{dx}y$ on $y^2 = x^6 -x$.
\end{itemize}

From \cite{McM05a} (Theorem~1.1) it is known that the algebraic curves from (1) and (2) are irreducible. The decagon curve (5) is a single Teichm\"uller curve. The irreducible components of (3) $W_{d^2}[n]$ and (4) $W^\epsilon_{d^2}[n]$ are unknown. The parity conjecture would imply that these algebraic curves are irreducible. Therefore the parity conjecture would complete the classification of Teichm\"uller curves in $\M_2$.

\noindent
{\bf Primitive and imprimitive Teichm\"uller curves.}
A 1-form $(X,\omega)$ is called {\em geometrically primitive} when it is not pulled back from a lower genus Riemann surface.  A Teichm\"uller curve is {\em primitive} if it is generated by a geometrically primitive form. 

The primitive Teichm\"uller curves in $\M_2$ are the decagon curve and the irreducible components of $W_D$, when $D \ne d^2$ (see \cite{McM06}). Hence the classification of primitive Teichm\"uller curves in $\M_2$ is complete. 

Every imprimitive Teichm\"uller curve in $\M_2$ is generated by a 1-form that is pulled back from an elliptic curve and therefore it is an irreducible component of $W_{d^2}$ or $W_{d^2}[n]$. In particular, $W_{d^2}$ consists of $X \in \M_2$ that admit a primitive degree $d$ elliptic cover with a single critical point. It follows from \cite{McM05a} (Theorem~1.1) that $W_{d^2}$ has two irreducible components when $d$ is odd, and one when $d$ is even.
The remaining open cases of the classification of imprimitive Teichm\"uller curves in $\M_2$ are the irreducible components of $W_{d^2}[n]$.

\begin{remark}
Note that our notation is slightly different from \cite{McM06}, where $D$ is assumed to be non-square. The Weierstrass curves $W_D$ are also denoted by $W_D[1]$ in \cite{McM06}. They arise from the 1-forms with a single zero, unlike $W_{d^2}[1]$ from above, which arise from the 1-forms with two simple zeroes.
\end{remark}

\noindent
{\bf Elliptic differentials.} An elliptic cover $\pi: X \to E$ is called {\em primitive} if the induced map $\pi_*:H_1(X,\Z) \to H_1(E,\Z)$ is surjective. For any $X \in W_{d^2}[n]$ there exists a primitive degree $d$ elliptic cover $\pi: X \to E$ defined up to translation automorphism of $E$, such that its ramification points $ x_1 \ne x_2 \in X$ satisfy $\pi(x_1)-\pi(x_2)$ has order $n$ in $\Jac(E)$. An Abelian differential $(X, \omega)$ is called a {\em type $(d,n)$ elliptic differential} if $X \in W_{d^2}[n]$ and $\omega = \pi^*(dz)$ for some holomorphic 1-form $dz$ on $E \in \M_1$.  The locus of type $(d,n)$ elliptic differentials in $\Omega\M_2$ will be denoted by:
\begin{flalign*} 
\indent \Omega W_{d^2}[n]  =  \left\{ 
  (X, \omega) \in \Omega \M_2 \,\middle\vert\,
   \begin{aligned} 
   X \in W_{d^2}[n] \mbox{ and } \omega = \pi^*(dz) \mbox{ for some } dz \in \Omega(E) \\
  \end{aligned}
\right\}.
\end{flalign*}
The locus $\Omega W_{d^2}[n]$ is a closed $\GL$ invariant two-dimensional complex subvariety of $\Omega \M_2$. The irreducible components of $W_{d^2}[n]$ are the projections of the topological connected components of $\Omega W_{d^2}[n] \subset \Omega\M_2$ to $\M_2$.

\noindent
{\bf Square-tiled surfaces.} Here we discuss a particular class of Abelian differentials in $\M_2$ that generate all the imprimitive Teichm\"uller curves in $\M_2$.

A {\em square-tiled surface} is an Abelian differential $(X,\omega)$ whose relative periods 
$\RPer(\omega)$
belong to $\Z[i]$. A square-tiled surface is called {\em reduced} if $\RPer(X,\omega) = \Z[i]$ and {\em primitive} if the absolute periods $\Per(X,\omega) = \Z[i]$. Note that a primitive square-tiled surface is necessarily reduced, but not the other way around (see Figure~\ref{figreduced}). 

\begin{comment2}
\begin{figure}[H]
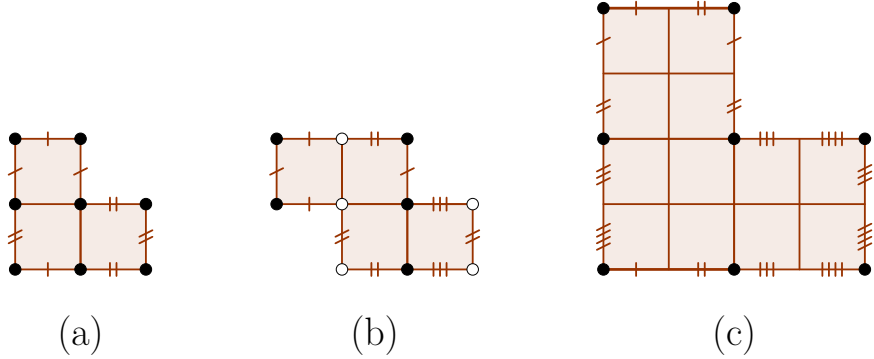
 
\centering
\includestandalone[width=0.8\textwidth]{reducedprimitive}
\caption{Examples of square-tiled surfaces that are: (a) primitive, (b) reduced but not primitive, and (c) not reduced.}
\label{figreduced}
\end{figure}
\end{comment2}

Recall that $E_0 = \C/\Z[i]$ is the square torus and 1-from $dz$ on $\C$ descends to $E_0$. Given a square-tiled surface $(X,\omega)$ there is a unique cover $\pi: X \to E_0$ branched over the origin such that $\omega = \pi^* dz$. A square-tiled surface $(X,\omega)$ is primitive if and only if the corresponding cover $\pi: X \to E_0$ is primitive, and a square-tiled surface is reduced if and only if $\pi: X \to E_0$ does not factor through another cover branched over a single point.

The metric $\left| dz \right|^2$ from $E_0$ pulls back to the metric $\left| \omega \right|^2$ that gives a tiling of $X$ by unit squares with matching sides. A square-tiled surface can be thought of as a number of unit squares in $\R^2$ with opposite sides identified by parallel translations. Note that square-tiled surfaces with two simple zeroes are in one-to-one correspondence with the isotopy classes of tilings of a topological surface of genus 2 with squares, discussed in the introduction. 

For any pair of integers $(d,n)$, where $d>1$ and $n\ge1$, define the set of type $(d,n)$ square-tiled surfaces $\ST(d,n)$ as:
\begin{flalign*}
\left\{ 
  (X,\omega) \in \Omega\M_2(1,1) \,\middle\vert\,
   \Per(\omega) \subset \RPer(\omega) = \Z[i] \mbox{ has index $n$ and } \displaystyle  \int_X \left| \omega \right|^2 = d\cdot n  
\right\}.
\end{flalign*}
We will also denote the set of primitive square-tiled surfaces in $\Omega\M_2(2)$ by $\ST(d,0)$. The group $\SL(E_0,dz) = \SLZ$ acts on the set of square-tiled surfaces and preserves $\ST(d,n)$. The $\SLZ$ action on the square-tiled surface of genus 2 made out of 4 tiles is illustrated in Figure~\ref{fig4tiled}. 

\begin{comment2}
\begin{figure}[H]
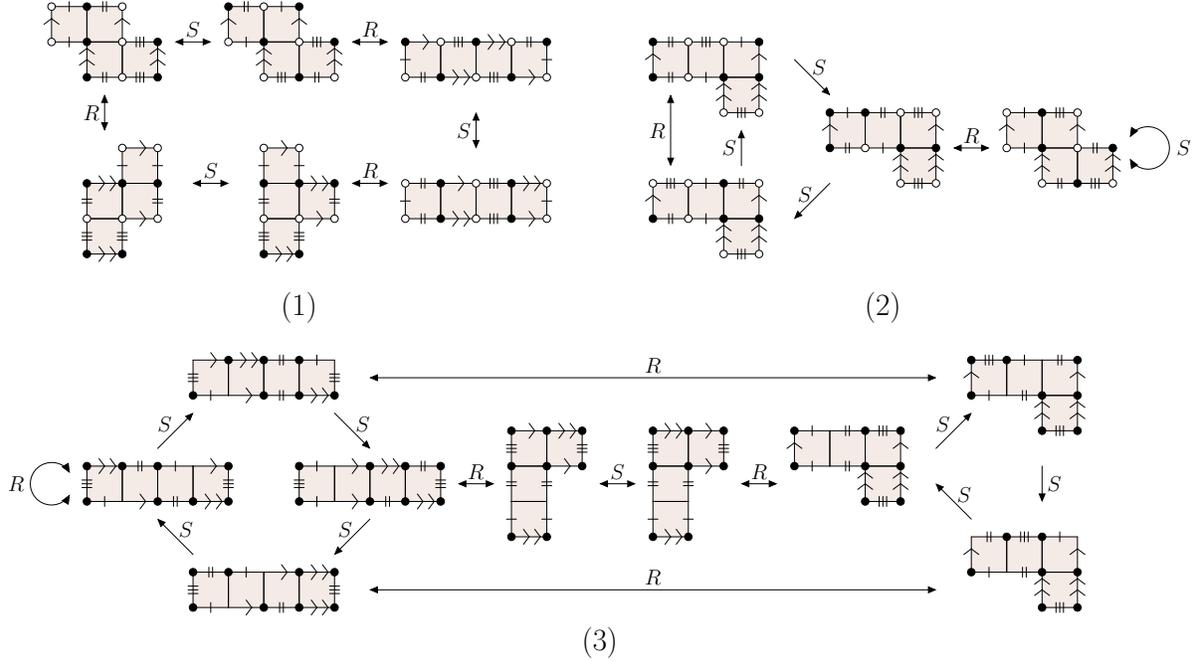

\centering
  \includestandalone[width=\textwidth]{tikz4tiled}
  \caption{The square-tiled surfaces of genus 2 made out of 4 tiles of (1) type $(2,2)$; (2) type $(4,1)$; and (3) type $(4,0)$ together with the $\SLZ$ action on them presented by its generators $S = \protect\begin{pmatrix} 1 & 1 \\ 0 & 1 \protect\end{pmatrix}$ and $R = \protect\begin{pmatrix} 0 & -1 \\ 1 & 0 \protect\end{pmatrix}$.}
  \vspace{-5pt}
  \label{fig4tiled}
\end{figure}
\end{comment2}

\noindent
{\bf Cover factorization.} A primitive elliptic cover is primitive if and only if it does not factor through another cover.
Given a reduced square-tiled surface $(X,\omega) \in \ST(d,n)$ with $n\ge1$ the integration of $\omega$ defines a unique degree $d\cdot n$ covering map to the square torus branched over the origin. This cover uniquely factors through a primitive degree $d$ elliptic cover and an isogeny of elliptic curves of degree $n$:
\begin{equation} \label{sqtfactors}
(X,\omega)  \xrightarrow{d} (E',\eta)  \xrightarrow{n}  (E_0, dz).
\end{equation}
We will call $d$ the {\em degree} and $n$ the {\em torsion} of a reduced square-tiled surface. We summarize the discussion above in the following proposition:

\begin{proposition}
Any reduced square-tiled surface $(X,\omega) \in \Omega\M(1,1)$ belongs to $\ST(d,n)$ for some $d>1$ and $n\ge1$.
\end{proposition}

\begin{proof}
Let $(X,\omega) \in \Omega\M(1,1)$ be any reduced square-tiled surface. The integration of $\omega$ defines a unique covering map $\pi: X \to E_0$ branched over the origin. This cover uniquely factors through a primitive degree $d$ elliptic cover and an isogeny of elliptic curves of degree $n$, for some $d$ and $n$:
\begin{equation}
(X,\omega)  \xrightarrow{d} (E' ,\eta) \xrightarrow{n} (E_0,dz).
\end{equation}
Since $(X,\omega)$ is reduced, $\RPer(X,\omega) = \Z[i]$. The degree $d$ cover is primitive, therefore $\Per(X,\omega) = \Per(E',\eta)$. The degree $n$ map is an isogeny of elliptic curves, hence $\Per(E',\eta)$ is an index $n$ sublattice of $\Per(E_0,dz) = \Z[i]$. This implies that $(X,\omega) \in \ST(d,n)$.
\end{proof}

In particular this implies that any primitive square-tiled surfaces $(X,\omega) \in \Omega\M(1,1)$ belongs to $\ST(d,1)$ for some $d>1$.

\noindent
{\bf The set $\A_{d^2}[n]$.} Next we discuss a subset of points of the absolute period leaf $\A_{d^2}$ that will be define in \S\ref{secabs} and used for the proof of the main result.

 Let $(X,\omega) \in \Omega\M_2$ be an Abelian differential pulled back from the square torus $(E_0 = \C/\Z[i],dz)$ via a primitive elliptic cover $\pi: X \xrightarrow{d} E_0$ of degree $d$. 
Define the following finite subset of such Abelian differentials:
\begin{flalign*}
\A_{d^2}[n]  = \left\{ (X,\omega) \in \Omega\M_2 \, \middle\vert 
  \begin{aligned} 
  &  \exists \mbox{ a primitive cover } \pi: X \xrightarrow{d} E_0 \mbox{ such that } \\ 
  &  \omega = \pi^* dz \mbox{ and} \mbox{the critical points }  x_1 \ne x_2 \in X \mbox{ satisfy:}   \\ 
  &  \pi(x_1) - \pi(x_2)  \mbox{ has order } n\mbox{ in } \Jac(E_0)
  \end{aligned}
   \right\}.
\end{flalign*}
The subgroup $\SLZ \subset \GL$ preserves the set $\A_{d^2}[n]$. 

\noindent
{\bf Topological branched covers.} 
Let $\pi: \Sigma_g \to \Sigma_h$ be a topological branched covering between two closed connected surfaces of genera $g$ and $h$. Similarly to the case of elliptic covers, such cover is called {\em primitive}, if the induced map $\pi_*:H_1(\Sigma_g,\Z) \to H_1(\Sigma_h,\Z)$ is surjective. Two primitive covers $\pi_1, \pi_2: \Sigma_g \to \Sigma_h$ are {\em topologically equivalent} if there are orientation-preserving homeomorphisms $f_1: \Sigma_h \to \Sigma_h$ and $f_2: \Sigma_g \to \Sigma_g$, such that the following diagram commutes:
\begin{equation} \label{equiv}
\begin{tikzcd}
\Sigma_g \arrow{d}{\pi_1} \arrow{r}{f_2}
& \Sigma_g \arrow{d}{\pi_2} \\
\Sigma_h \arrow{r}{f_1}
& \Sigma_h.
\end{tikzcd}
\end{equation}

A cover is called {\em generic} (or {\em simply branched}) if every fiber contains at least $d-1$ points, where $d$ is the degree of the cover. Since the works of L\"uroth (\cite{Lur71}, 1871), Clebsch (\cite{Cle73}, 1873) and 
Hurwitz (\cite{Hur91}, 1891) it was known that classes of topological equivalence of generic primitive covers of the 2-sphere are classified by their degree. Later Gabai and Kazez (\cite{GK87}, 1987) showed that this is true if one replaces sphere with any closed surfaces. However the case of non-generic covers is widely unexplored. Some results on topological classes of non-generic covers of the sphere can be found in \cite{Pro88}.

This work is related to the case of covers of the 2-torus branched over a single point with ramification profile $(2,2,1,\ldots,1)$. One can verify by Riemann-Hurwitz formula that such covers exist in any degree and the genus of the covering space is always 2. A cover $\pi: \Sigma_2 \to \Sigma_{1}$ branched over a single point is called {\em reduced} if does not factor through another cover branched over a single point.
Every reduced cover $\pi: \Sigma_2 \to \Sigma_{1}$ uniquely factors as:
$$
\Sigma_2 \xrightarrow{d} \Sigma_{1} \xrightarrow{n} \Sigma_{1},
$$
where $\Sigma_2 \xrightarrow{d} \Sigma_{1}$ is a primitive cover of degree $d$. In this case a cover $\pi: \Sigma_2 \to \Sigma_{1}$ is called a {\em type $(d,n)$ cover}.

We are now ready to give a proof of Theorem~\ref{thmequivconj}.

\begin{proof}[Proof of Theorem~\ref{thmequivconj}]
(a) $\iff$ (e): The $\GL$ orbits of type $(d,n)$ elliptic differentials are the loci $\Omega W_{d^2}[n]$. The projections of topological components of $\Omega W_{d^2}[n]$ are the Teichm\"uller curves in $W_{d^2}[n]$.

(b) $\iff$ (e): Let $(X, \omega) \in \A_{d^2}[n]$. The integration of $\omega$ defines a cover $\pi: X \to E_0$ branched over $z_1, z_2 \in E_0$ satisfying $z_1 - z_2$ has order $n$ in $\Jac(E_0)$. Therefore $(X,\omega)$ is a type $(d,n)$ elliptic differential and generates a $\GL\cdot (X,\omega)$ from (b). On the other hand every $\GL$ orbit of a type $(d,n)$ elliptic differential contains an element of $\A_{d^2}[n]$ since $\GL$ acts surjectively on $\Omega\M_1$. Since $\SL(X,\omega) \subset \SLZ$ two Abelian differentials $(X, \omega), (X', \omega') \in \A_{d^2}[n]$ are in the same $\GL$ orbit if and only if they differ by an element of $\SLZ$.

(c) $\iff$ (e): Let $(X,\omega)$ be a type $(d,n)$ square-tiled surface. Note that the cover $(X,\omega)  \xrightarrow{d} (E',\eta)$ from (\ref{sqtfactors}) is branched over two points $z_1$ and $z_2 \in E'$ satisfying $z_1-z_2$ has order $n$ in $\Jac(E')$, since $z_1$ and $z_2$ are sent to a single point via $E' \to E$. Therefore $(X,\omega)$ is a type $(d,n)$ elliptic differential and generates a $\GL$ orbit from (e). This $\GL$ orbit does not depend on a representative of $\SLZ \cdot (X,\omega)$ since $\SLZ \subset \GL$.

This association is surjective. Indeed, consider a type $(d,n)$ elliptic differential $(X,\omega)$ admitting a translation cover $(X,\omega) \to (E, \eta)$ for some $E \in \M_1$ with ramification points $x_1 \ne x_2 \in X$ satisfying $\pi(x_1) - \pi(x_2)$ has order $n$ in $\Jac(E)$. Set $x_1$ to be an origin of $E$, then points $x_1$ and $x_2$ generate a subgroup $\ZZ{n}$ in $E$. Quotienting by this subgroup we obtain:
$$
(X, \omega) \xrightarrow{d} (E, \eta) \xrightarrow{n} (E', \eta'),
$$
where $X \xrightarrow{\pi'} E'$ is a cover branched over a single point. Choose a matrix $A \in \GL$ that sends $(E',\eta')$ to $(E_0, dz)$. Then $(X', \omega') = A \cdot (X,\omega)$ belongs to $\ST(d,n)$ since it admits a factorization of covers as in (\ref{sqtfactors}):
$$
(X',\omega') = A \cdot (X,\omega) \xrightarrow{d} A \cdot (E,\eta)  \xrightarrow{n} A \cdot (E', \eta') = (E_0, dz),
$$
where degree $d$ cover is primitive.

This association is also injective. Indeed, if two Abelian differentials $(X,\omega), (X',\omega') \in \ST(d,n)$ satisfy $(X',\omega') = A \cdot (X,\omega)$ for some $A \in \GL$, then $A$ induces a bijection $\RPer(\omega) = \Z[i] \to \RPer(\omega') = \Z[i]$ and hence $A \in \SLZ$. Thus $(X,\omega)$ and $(X',\omega')$ belong to the same $\SLZ$ orbit.

(d) $\iff$ (e): Start with any cover $\pi: \Sigma_2 \to \Sigma_1$ from (d). Mark the branch point on $\Sigma_1$ and endorse $\Sigma_{1,1}$ with complex structure $E_0$ and a holomorphic 1-form $dz$. Pulling it back to $\Sigma_2$ via $\pi: \Sigma_2 \to \Sigma_1$ one obtains $(X,\omega) \in \ST(d,n)$. Note that action of $A \in \SLZ$  yields a commutative diagram:
\begin{equation*}
\begin{tikzcd}
(X,\omega) \arrow{d}{\pi_1} \arrow{r}{f'_A}
& A \cdot (X,\omega) \arrow{d}{\pi_2} \\
(E_0,dz) \arrow{r}{f_A}
& (E_0,dz).
\end{tikzcd}
\end{equation*}
Forgetting an extra data of complex structures and 1-forms gives a diagram (\ref{equiv}). The flat metric on $E_0$ is defined by $\left| dz\right|^2$. The $\SLZ$ acts on $(E_0,|dz|^2)$ by affine automorphisms $f_A$ in this metric.  Forgetting the complex structure and the 1-form one obtains the action of $\Mod_{1,1} \cong \SLZ$ on $\Sigma_{1,1}$. Therefore two reduced covers are of the same type if and only if the corresponding elements of $\ST(d,n)$ are in the same $\SLZ$ orbit. This implies the bijection.
\end{proof}


\section{The spin invariant} \label{secspin}

In this section we introduce the spin invariant $\epsilon$, that distinguishes the components of $\Omega W_{d^2}[n]$ when $n$ is odd. Spin is valued in $\ZZ2$ and depends on the number of integer Weierstrass points. It gives a lower bound on the number of $\SLZ$ orbits in $\ST(d,n)$. The goal of the future sections is to show that degree, torsion and spin give the full set of invariants in infinitely many cases.

Spin invariant is a generalization of the integer Weierstrass points invariant presented in \cite{HL06} 4.2 in the case of square-tiled surfaces $(X,\omega) \in \Omega \M_2 (2)$. In that case it is defined as the number of Weierstrass points of $X$ that get mapped to the branch point under the covering map $\pi$. We will show:

\begin{theorem} \label{thmspin}
Let $(X,\omega) \in \Omega W_{d^2}[n]$ and let the cover $\pi: X \to E$, obtained by integration of $\omega$, be branched over $z_1, z_2 \in E$. A Weierstrass point $W_i$ is called {\em integer} if $\pi(W_i) - z_1 = z_2 - \pi(W_i) \in \Jac(E)[n]$. Then:
\begin{enumerate}

\item[(i)]  the number of integer Weierstrass points is:
\begin{equation} \label{spinformula}
\IWP(X,\omega) = \begin{cases}
d\mod 2 \mbox{ or } (d\mod 2 )+ 2,  & \mbox{when }  n \mbox{ is odd} \\

0, & \mbox{when }  n  \mbox{ is even};
\end{cases}
\end{equation}

\item[(ii)] $\IWP$ is a locally constant function on $\Omega W_{d^2}[n]$, that is globally constant when $n$ is even, and takes two values when $n$ is odd.

\end{enumerate} 
\end{theorem}

For odd $n$ we will define the {\em spin invariant} of $(X,\omega) \in \Omega W_{d^2}[n]$ as:
$$
\epsilon(X,\omega) = \begin{cases}
0, & \mbox{ when } \IWP(X,\omega) = 3 \mbox{ or } 0  \\

1, & \mbox{ when } \IWP(X,\omega) = 1 \mbox{ or } 2.
\end{cases}
$$
The corresponding components of $\Omega W_{d^2}[n]$ will be denoted by $\Omega W^0_{d^2}[n]$ and $\Omega W^1_{d^2}[n]$.

\noindent 
{\bf Weierstrass points.} 
Every genus 2 algebraic curve $X$ is hyperelliptic, i.e.\ admits a degree $2$ map to $\P^1$ ramified at 6 points $W_1, W_2, W_3, W_4, W_5, W_6$ on $X$, called {\em Weierstrass points}. The following lemma will be used in the proof of Theorem~\ref{thmspin}:

\begin{lemma} \label{lemmaWP}
Let $\pi: X \to E$ be a cover branched over $z_1, z_2 \in E$, for some $X \in \M_2$ and  $E\in \M_1$. Then:
\begin{enumerate}
\item[(i)] $\left\{ \pi(W_1), \pi(W_2), \pi(W_3), \pi(W_4), \pi(W_5), \pi(W_6) \right\} $ is a subset of four points $\left\{ P_0, P_1, P_2, P_3\right\} \subset E$, such that $P_i - P_j \in \Jac(E)[2]$ for any $i,j$; and

\item[(ii)] $z_1 - P_i = P_i - z_2$ in $\Jac(E)$, or equivalently the branch points are symmetric with respect to $P_i$.
\end{enumerate}
\end{lemma}

\begin{proof}
(i) A Weierstrass point $W_i$ is characterized by the fact that it admits a holomorphic 1-form on $X$ vanishing to the order two at $W_i$. Then $2 W_i = K \in \Pic(X)$, where $K$ is a canonical divisor and $\Pic(X)$ is a group of divisors up to linear equivalence. The Jacobian $\Jac(X) \subset \Pic(X)$ is a subgroup of degree $0$ divisors on $X$. The covering map $\pi$ induces a map $\pi_*: \Jac(X) \to \Jac(E)$. Suppose $\pi$ maps $W_i$ to $P_i$ and $W_j$ to $P_j$. Then:
$$
2 W_i - 2 W_j = K - K = 0 \in \Jac(X) \implies 2(P_i-P_j) = \pi_*(2 W_i - 2 W_j) = 0 \in \Jac(E),
$$ 
which implies $P_i - P_j \in \Jac(E)[2]$.

(ii) Let $x_1 \ne x_2$ be ramification points of $\pi$. Then $x_1 + x_2 = K$, since for any non-zero holomorphic 1-form $\eta$ on $E$, 1-form $\omega = \pi^*(\eta)$ has simple zeroes at $x_1$ and $x_2$. Therefore:
$$z_1 + z_2 = \pi_*( x_1 + x_2 ) = \pi_*(K) = \pi_*(2 W_i) = 2 P_i \mbox{ in } \Pic(E),$$
which implies $z_1 - P_i = P_i - z_2$. 
\end{proof}

\noindent
{\bf Weierstrass profile.} Let $N_i = \left| \left\{ W_j \ \middle\vert \  \pi(W_j)=P_i \right\} \right|$, then $(N_0, N_1, N_2, N_3)$ is a partition of $6$ called a {\em Weierstrass profile} of $\pi: X  \to E$.

\begin{proposition} \label{WPprofile}
Let $X \in \M_2$, $E\in \M_1$ and $\pi: X \to E $ be a primitive branched cover of degree $d$.
The Weierstrass profile of $\pi$ is:
\begin{equation} \label{WPprofileformula}
(N_0, N_1, N_2, N_3) = 
\begin{cases}
(3,1,1,1), &  \mbox{ when } d \mbox{ is odd} \\
(0,2,2,2), & \mbox{ when } d \mbox{ is even}.
\end{cases}
\end{equation}
\end{proposition}

Compare to Lemma 2.2 in \cite{FK09}. The algebro-geometric proof of Proposition~\ref{WPprofile} can be found in Section 1 and 5 of \cite{Kuhn88}. 

\begin{proof}[Proof of Theorem~\ref{thmspin}]
Let us first show that when $n$ is even, none of the $P_i$ satisfy $P_i - z_1 \in \Jac(E)[n]$. Suppose some $P_i$ does, then $n (P_i - z_1) =0$. From Lemma~\ref{lemmaWP} (ii) we know that $z_2 - P_i = P_i - z_1$, which implies $2(P_i - z_1) = z_2 - z_1$. Now $n (P_i - z_1) = \frac{n}2 2 (P_i - z_1) = \frac{n}2 (z_2 - z_1) = 0 \in \Jac(E)$, which contradicts with the fact that $z_2 - z_1$ has order $n$ in $\Jac(E)$ (see the definition of $\Omega W_{d^2}[n]$).

For odd $n$ we will first show that the set $\{ n(P_i - z_1) \mid i = 0,1,2,3 \}$ is equal to $\Jac(E)[2]$. Clearly $n(P_i - z_1) \in \Jac(E)[2]$ since $2 n(P_i - z_1) = n (z_2 - z_1) = 0$. It remains to show that all $n (P_i - p)$ are different. Suppose $n (P_i - p) = n (P_j - p)$, then $n (P_i - P_j) = 0$, which implies that $P_i=P_j$, because $n$ is odd and $2 (P_i - P_j)=0$ from Lemma~\ref{lemmaWP} (i). Therefore $\{ n(P_i - z_1) \} = \Jac(E)[2]$. Then there is a unique point $P_i$, for which $P_i - z_1 \in \Jac(E)[n]$, and the value of the Weierstrass profile on that point is the value of $\IWP$. Together with Proposition~\ref{WPprofile} it finishes the proof of (i). 

For (ii) note that $\IWP(X,\omega)$ is a continuous and discreet function, hence it is an invariant of any topological connected component of $\Omega W_{d^2}[n]$. In the end of this section we show that both values of the invariant are achieved, when $n$ is odd, by giving examples of the corresponding $(X,\omega)$. This ends the proof.
\end{proof}

\noindent
{\bf Normalized cover.} Proposition~\ref{WPprofile} implies that for any primitive genus 2 cover $\pi: X \to E$ there is a unique point $P_0 \in E$ with a distinguished value in the Weierstrass profile. A choice of the origin of $E$ fixes an isomorphism $E \cong \C /\Z[\tau]$ for some $\tau \in \Hyp$. A primitive genus 2 cover $\pi: X \to \C /\Z[\tau]$ is called {\em normalized} if $P_0$ is the origin. Under this choice we have:
$$
P_i \in E[2] \mbox{ and } z_1 = - z_2,
$$
where $z_1$ and $z_2 \in E \cong \C /\Z[\tau]$ are the branch points of $\pi$. This convention gives a convenient way of computing the spin invariant:

\begin{proposition} \label{spincompute}
For odd $n$, let $(X,\omega) \in W_{d^2}[n]$ and let the cover $\pi: X \to E$, obtained by integration of $\omega$, be a normalized cover branched over $\pm z$. Then:
\begin{equation}
\epsilon(X,\omega) = \begin{cases}
0, & \mbox{ when } nz =0  \\

1, & \mbox{ otherwise}.
\end{cases}
\end{equation}
\end{proposition}

\begin{proof}
Assume $\epsilon(X,\omega) = 0$ and note that it happens if and only if the image of integer Weierstrass points is $P_0$. By definition $W_i$ is integer if $z - \pi(W_i) =  z - P_0 = z \in \Jac(E)[n]$, hence $nz=0$. Clearly if  $\epsilon(X,\omega) = 1$, then $\pi(W_i) \in E[2]^*$ and $z - \pi(W_i) \notin \Jac(E)[n]$.
\end{proof}

\noindent
{\bf Examples.} We conclude by providing examples for each of the values of the spin invariant $\epsilon$, when $n$ is odd. In Figure~\ref{figspinex} one can see Abelian differentials from $\A_{d^2}[n]$ ($d$ is odd in figure $(1)$ and even in figure $(2)$) together with covering maps to the square torus. The horizontal sides of these surfaces are identified to the opposite ones directly below or above and the vertical sides -- to the opposite ones directly on the left or right. Note that the punctured lines split each surface into four rectangles. The hyperelliptic involution is given by rotating by $\pi$ of each of these rectangles around its centers. The Weierstrass points and their images are labeled with crosses. 

The surface in figure $(1)$ has odd number of squares and its $\epsilon = k \mod 2$. Indeed, note that the red cross is the distinguished point $P_0$ with $3$ Weierstrass points in its fiber. Then $P_0 - z = \frac k{2n} \in \C/\Z[i]$, where $z$ is one of the branch points. Therefore $n(P_0 - z) = 0 \iff k \equiv 0\mod 2$.

The surface in figure $(2)$ has even number of squares and its $\epsilon = k+1 \mod 2$. Indeed, note that the black cross is the distinguished point $P_0$ with no Weierstrass points in its fiber. Then $P_0 - z = \frac {n-k}{2n} \in \C/\Z[i]$, where $z$ is one of the branch points. Therefore $n(P_0 - z) = 0 \iff k \equiv 1\mod 2$.

\begin{comment3}
\begin{figure}[H]
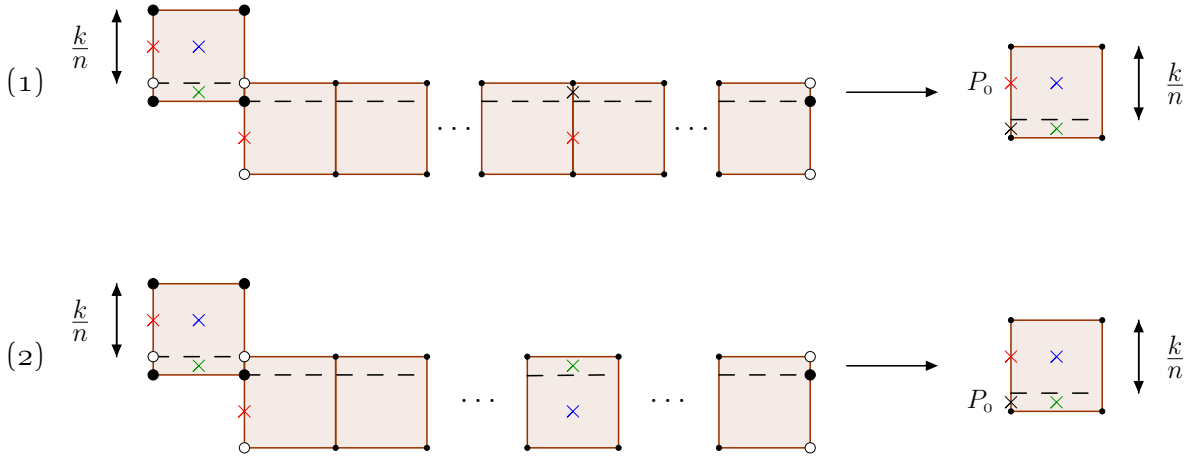
 
\centering
\includestandalone[width=\textwidth]{tikzspinex}
\vspace{-10pt}
\caption{Examples of $(X,\omega) \in W_{d^2}[n]$, when $n$ is odd, with $(1)$ $\epsilon(X,\omega) = k \mod 2$, when $d$ is odd, and $(2)$ $\epsilon(X,\omega) = k \mod 2$, when $d$ is even.}
\label{figspinex}
\end{figure}
\end{comment3}

%
%
%
%
%


\section{The absolute period leaf} \label{secabs}

In \S\ref{secbackground} we showed that the number of irreducible components of $W_{d^2}[n]$ is equal to the number of $\SLZ$ orbits of points $\A_{d^2}[n]$ on the absolute period leaf $\A_{d^2}$. In this section we will present:
\begin{itemize}
\item the absolute period leaf $\A_{d^2}$;
\item the relative period map $\rho: \A_{d^2} \to \P^1$;
\item the discriminant map $\delta: \A_{d^2} \to \P^1$;
\item the meromorphic quadratic differential $q$ on $\A_{d^2}$ that gives a square-tiling of $\A_{d^2}$; and
\item a natural $\SLZ$ action on $\A_{d^2}$ respects the square-tiling.
\end{itemize} 

We summarize this in the following theorem:

\begin{theorem} \label{thmquaddiff}
For any $d>1$, the absolute period leaf $\A_{d^2}$ admits a meromorphic quadratic differential $q$ with the following properties:
\begin{enumerate}

\item[(i)] the flat metric $\left|q\right|$ defines a square-tiling of $\A_{d^2}$; and

\item[(ii)] there is a natural $\SLZ$ action on $\A_{d^2}$ compatible with this square-tiling. This action is an extension of the natural $\SLZ \subset \GL$ action on $\A^\circ_{d^2} \subset \Omega\M_2$.
\end{enumerate}
\end{theorem}

We will show relation between the vertices of the squares, the zeros and poles of $q$ and the boundary points of $\A_{d^2}$. We will give a geometric description of the corresponding Abelian differentials. In particular, we will define square-tiled surfaces with separating and non-separating nodes and will show:

\begin{theorem} \label{thmvertices}
For any $d>1$, the set of the singularities of $(\A_{d^2},q)$ is a subset of the vertices of the square-tiling of $\A_{d^2}$. The vertices are regular points of $q$, simple zeroes of $q$ or simple poles of $q$. They correspond to:
\begin{enumerate}
\item[(i)] (regular points) primitive $d$-square-tiled surfaces in $\Omega\M_2(1,1)$;

\item[(ii)] (simple zeroes) primitive $d$-square-tiled surfaces in $\Omega\M_2(2)$; and

\item[(iii)] (simple poles) genus $2$ $d$-square-tiled surfaces with separating and non-separating nodes.
\end{enumerate}

The set of the boundary points of $\A^\circ_{d^2}$ is the set of simple poles of $q$.
\end{theorem}

We then will define the set $\bfP[n]$ of primitive $n$-torsion points of the pillowcase $\bfP$ and show that $\A_{d^2}[n]$ is the set of primitive $n$-rational points of the squares of $\A_{d^2}$:

\begin{theorem} \label{thmlocating}
For any $d>1$ and $n>1$, $\A_{d^2}[n]$ is the preimage of $\bfP[n]$ under the discriminant map $\delta$.
\end{theorem}

When $n>1$ is odd, $\A_{d^2}[n]$ consists of two $\SLZ$ invariant subsets of $\A^{\epsilon}_{d^2}[n]$ distinguished by the spin $\epsilon$. We will interpret these subsets in terms of the square-tiling on $\A_{d^2}$ and use it to show:
\begin{theorem} \label{thmformula}
For any $d>1$ and odd $n>1$, let $t_{d,n,\epsilon}$ be the number of square-tiled surfaces of type $(d,n)$ and spin $\epsilon$. Then:
$$
t_{d,n,\epsilon}  =
\displaystyle (2 \epsilon +1 ) \cdot
\frac{d-1}{12n} \cdot \left| \PSL_2(\ZZ{d}) \right| \cdot  \left| \SL_2(\ZZ{n}) \right|.
$$
\end{theorem}

\noindent
{\bf Absolute period leaf.} Recall that:
\begin{flalign*}
 \A^{\circ}_{d^2} = \left\{   (X,\omega) \in \Omega\M_2 \,\middle\vert\,
   \Per(\omega) = \Z[i] \mbox{ and } \displaystyle \int_X  |\omega|^2 = d  
\right\}.
\end{flalign*}
This set is also referred to as a {\em rel leaf} or a {\em modular fiber}. The locus $\A^{\circ}_{d^2} \subset \Omega\M_2$ is a complex 1-dimensional subvariety that parametrizes a family of Abelian differentials with varying relative periods but fixed absolute periods. We define its completion following \cite{McM14}.

Let $\cM_g$ be the Deligne-Mumford compactification of the moduli space $\M_g$ by stable curves. A {\em stable form} on $X \in \cM_g$ is a non-zero meromorphic 1-form on $X$, which is holomorphic on the smooth locus and has at worst simple poles with opposite residues at the nodes of $X$. Then $\Omega\cM_g$ will denote the moduli space of stable forms of genus $g$. The {\em absolute period leaf} $\A_{d^2}$ is the completion $\A^{\circ}_{d^2}$ in $\Omega\cM_2$.

Let $P(\omega)$ denote the set of simple poles of a stable form $\omega$.  Then the periods of a stable form are defined as $\Per(X,\omega) = \left\{ \int_\gamma \omega \mid \gamma \in H_1(X \setminus P(\omega),\Z)\right\}$. Then one obtains:
\begin{proposition} \label{prop:completion}
The completion of $\A^{\circ}_{d^2}$ in $\Omega\cM_2$ is:
\begin{flalign*}
 \A_{d^2} = \left\{   (X,\omega) \in \Omega\cM_2 \,\middle\vert\,
   \Per(X,\omega) = \Z[i] \mbox{ and } \displaystyle \int_X  |\omega|^2 = d  
\right\}.
\end{flalign*}
\end{proposition}

\begin{proof}
Let $(X,\omega)$ be a stable form in $\A_{d^2}$. It has no poles, since $\int_X  |\omega|^2 = d$ is finite. Hence its periods are $\Per(X,\omega) = \left\{ \int_\gamma \omega \mid \gamma \in H_1(X,\Z)\right\}$. Every curve $\gamma \subset X$ can be obtained as a limit of curves $\gamma_t \subset X_t$, where $(X_t,\omega_t) \in \A^{\circ}_{d^2}$, therefore $\Per(X,\omega) = \Z[i]$. 

On the other hand for any stable form $(X,\omega)$ with $\Per(X,\omega) = \Z[i]$ and $\int_X  |\omega|^2 = d$, the neighborhood of the node of $X$ can be replaced with a neighborhood of two saddle connections $s_t$ and $s'_t$ between two simple zeroes, such that $t = \int_{s_t} \omega_t =  - \int_{s'_t} \omega_t$. The resulting Abelian differentials belong to $\A^{\circ}_{d^2}$ and converge to $(X,\omega)$ as $t\to 0$. Therefore $(X,\omega) \in \A_{d^2}$ and this finishes the proof.
\end{proof}


\noindent
{\bf Relative period map $\rho$.} For any $(X,\omega) \in \A_{d^2}$ the integration of $\omega$ defines a degree $d$ map to $\C / \Per(X,\omega) = E_0$. If $X$ is smooth this map is a branched cover, and if $X$ is singular the induced map from the normalization of $X$ to $E_0$ is a covering map. In the first case, let $z_1$ and $z_2$ denote the critical values of this map, and in the latter case let $z_1=z_2$ denote the image of the node of $X$ under this map.

Let $\iota: \C \to \C$ be an involution given by $\iota(z)=-z$. Define the {\em relative period map} $\rho: \A_{d^2} \to  \iota \backslash\C/\Z[i] \cong \P^1$ by setting:
$$\rho(X,\omega) = \pm \int_{z_1}^{z_2} \omega.$$ 
This is a well-defined continuous map. Indeed, it depends neither on the choice of the path between $z_1$ and $z_2$ nor on the ordering of the critical values, since it is valued in the quotient of $\C$ by $\Z[i]=\Per(X,\omega)$ and $\iota$. The continuity follows from that fact that as $(X_t,\omega_t) \A_{d^2}$ approaches the boundary there is a relative period $\int_{z_1(t)}^{z_2(t)} \omega_t \to 0$ (see proof of Proposition~\ref{prop:completion}).

The quadratic differential $dz^2$ on $\C$ descends to $\iota \backslash\C/\Z[i]$. Define the holomorphic quadratic differential $q = \rho^*dz^2$ on $\A_{d^2}$. As we will see later the singular flat metric given by $|q|$ defines the square-tiling of $\A_{d^2}$. But first we give another perspective on $\A_{d^2}$ and $q$ using the discriminant map to the pillowcase.

\noindent
{\bf Pillowcase.} Now define:
$$\bfP = \iota \backslash \C / 2\Z[i] \cong \P^1.$$ 
The quadratic differential $dz^2$ on $\C$ descends to $\bfP$, and $\bfP$ equipped with this quadratic differential is called the {\em pillowcase}. In the singular flat metric $\left| dz\right|^2$ the pillowcase is isometric to two unit squares put together and sewn along the edges (see Figure~\ref{figpillow}). This quadratic differential has four simple poles at points that we denote by $Q_i$ and refer to them as the vertices of the pillowcase.

\begin{comment4}
\begin{figure}[H] 
\centering
\definecolor{ffzzqq}{rgb}{1.,0.6,0.}
\definecolor{qqffqq}{rgb}{0.,1.,0.}
\definecolor{ttqqqq}{rgb}{0.2,0.,0.}
\definecolor{qqccqq}{rgb}{0.,0.8,0.}
\definecolor{ttzzqq}{rgb}{0.2,0.6,0.}
\definecolor{ffffff}{rgb}{1.,1.,1.}
\definecolor{ffqqqq}{rgb}{1.,0.,0.}
\definecolor{uuuuuu}{rgb}{0.26666666666666666,0.26666666666666666,0.26666666666666666}
\definecolor{zzttqq}{rgb}{0.6,0.2,0.}
\definecolor{qqqqff}{rgb}{0.,0.,1.}
\begin{tikzpicture}[line cap=round,line join=round,>=triangle 45,x=1.0cm,y=1.0cm]
\clip(-12.5,-39.) rectangle (-2.5,-35.);

\fill[line width=0.8pt,color=qqqqff,fill=qqqqff,fill opacity=0.08] (-6.,-36.) -- (-4.,-36.) -- (-4.,-38.) -- (-6.,-38.) -- cycle;
\fill[line width=0.8pt,color=qqqqff,fill=qqqqff,fill opacity=0.08] (-12.,-36.) -- (-10.,-36.) -- (-10.,-38.) -- (-12.,-38.) -- cycle;
\fill[line width=0.8pt,color=qqccqq,fill=qqccqq,fill opacity=0.08] (-10.,-36.) -- (-8.,-36.) -- (-8.,-38.) -- (-10.,-38.) -- cycle;

\draw [line width=1pt] (-4.,-36.)-- (-4.,-38.);
\draw [line width=1pt] (-6.,-38.)-- (-6.,-36.);
\draw [line width=1pt] (-4.,-36.)-- (-6.,-36.);
\draw [line width=1pt] (-4.,-38.)-- (-6.,-38.);
\draw [shift={(-5.,-35.)},line width=1pt]  plot[domain=4.2487413713838835:5.176036589385496,variable=\t]({1.*2.23606797749979*cos(\t r)+0.*2.23606797749979*sin(\t r)},{0.*2.23606797749979*cos(\t r)+1.*2.23606797749979*sin(\t r)});
\draw [shift={(-5.,-40.)},line width=1pt,dash pattern=on 2pt off 2pt]  plot[domain=1.2490457723982544:1.892546881191539,variable=\t]({1.*3.1622776601683795*cos(\t r)+0.*3.1622776601683795*sin(\t r)},{0.*3.1622776601683795*cos(\t r)+1.*3.1622776601683795*sin(\t r)});
\draw [line width=1pt] (-12.,-36.)-- (-10.,-36.);

\draw [line width=1pt] (-8.889443886982805,-36.) -- (-9.,-36.13266733562064);
\draw [line width=1pt] (-8.889443886982805,-36.) -- (-9.,-35.86733266437936);

\draw [line width=1pt] (-11.110556113017195,-36.) -- (-11.,-35.86733266437936);
\draw [line width=1pt] (-11.110556113017195,-36.) -- (-11.,-36.13266733562064);

\draw [line width=0.8pt] (-10.,-36.)-- (-10.,-38.);
\draw [line width=1pt] (-12.,-38.)-- (-12.,-36.);
\draw [line width=1pt] (-10.,-38.)-- (-10.,-36.);
\draw [line width=1pt] (-12.088444890413754,-37.) -- (-11.911555109586244,-37.);
\draw [line width=1pt] (-8.,-36.)-- (-8.,-38.);
\draw [line width=1pt] (-7.911555109586244,-37.) -- (-8.088444890413756,-37.);
\draw [line width=1pt] (-12.,-38.)-- (-10.,-38.);
\draw [line width=1pt] (-8.778887773965604,-38.) -- (-8.889443886982798,-38.13266733562064);
\draw [line width=1pt] (-8.778887773965604,-38.) -- (-8.889443886982798,-37.86733266437936);
\draw [line width=1pt] (-9.,-38.) -- (-9.110556113017187,-38.13266733562064);
\draw [line width=1pt] (-9.,-38.) -- (-9.110556113017187,-37.86733266437936);
\draw [line width=1pt] (-8.,-36.)-- (-10.,-36.);
\draw [line width=1pt] (-8.,-38.)-- (-10.,-38.);
\draw [line width=1pt] (-11.22111222603439,-38.) -- (-11.110556113017195,-37.86733266437936);
\draw [line width=1pt] (-11.22111222603439,-38.) -- (-11.110556113017195,-38.13266733562064);
\draw [line width=1pt] (-11.,-38.) -- (-10.889443886982804,-37.86733266437936);
\draw [line width=1pt] (-11.,-38.) -- (-10.889443886982804,-38.13266733562064);
\draw (-12.3,-35.2) node[anchor=north west] {$Q_0$};
\draw (-10.3,-35.2) node[anchor=north west] {$Q_1$};
\draw (-8.3,-35.2) node[anchor=north west] {$Q_0$};

\draw (-6.3,-35.2) node[anchor=north west] {$Q_0$};
\draw (-4.3,-35.2) node[anchor=north west] {$Q_1$};

\draw (-7.5,-36.8) node[anchor=north west] {\Huge $=$};

\draw (-12.3,-38.2) node[anchor=north west] {$Q_2$};
\draw (-10.3,-38.2) node[anchor=north west] {$Q_3$};
\draw (-8.3,-38.2) node[anchor=north west] {$Q_2$};

\draw (-6.3,-38.2) node[anchor=north west] {$Q_2$};
\draw (-4.3,-38.2) node[anchor=north west] {$Q_3$};

\begin{scriptsize}
\draw [fill=black] (-6.,-36.) circle (2pt);
\draw [fill=black] (-6.,-38.) circle (2pt);
\draw [fill=black] (-4.,-36.) circle (2pt);
\draw [fill=black] (-4.,-38.) circle (2pt);
\draw [fill=black] (-12.,-36.) circle (2pt);
\draw [fill=black] (-12.,-38.) circle (2pt);
\draw [fill=black] (-10.,-36.) circle (2pt);
\draw [fill=black] (-8.,-36.) circle (2pt);
\draw [fill=black] (-8.,-38.) circle (2pt);
\draw [fill=black] (-10.,-38.) circle (2pt);
\end{scriptsize}
\end{tikzpicture}
\vspace{-10pt}
\caption{The pillowcase $\bfP$. Folding the left picture along the middle vertical line and zipping up the edges one obtains a shape on the right that reminds of a pillowcase.}
    \label{figpillow}
\end{figure}
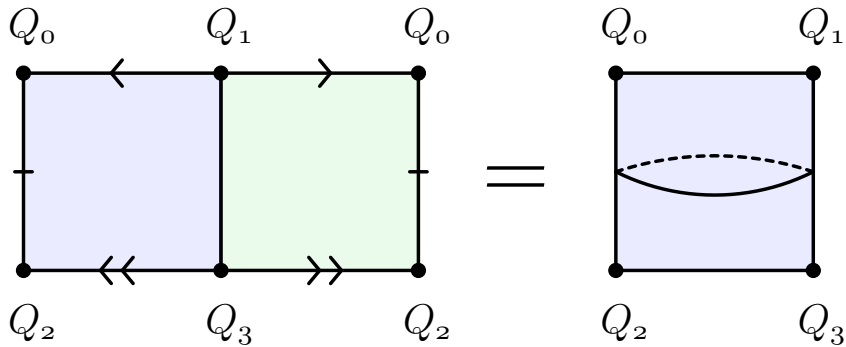
\end{comment4}

The involution $\iota: \C \to \C$ descends to $\iota_{E_0}: E_0 \to E_0$ that fixes the 2-torsion points of $E_0$. There is a map $\eta: E_0 \to \bfP$ given by:
$$
z \mod \Z[i] \mapsto \pm 2z \mod 2\Z[i],
$$
and the following diagram commutes:
\begin{equation} \label{eta}
\begin{tikzcd}
\C \arrow{d}{\pi_{E_0}} \arrow{r}{\times 2}
& \C \arrow{d}{\pi_{\bfP}} \\
E_0 = \C/\Z[i] \arrow{r}{\eta}
&  \iota \backslash \C / 2\Z[i] = \bfP.
\end{tikzcd}
\end{equation}

The map $\eta: E_0 \to \bfP$ has degree $2$, it is ramified at the 2-torsion points of $E_0$ and branched over the vertices of the pillowcase $Q_i$'s. Note that $\eta^*dz^2 = 4 dz^2$ and the area of the pillowcase is $2$. 

\noindent
{\bf Discriminant map $\delta$.}  For any $(X,\omega) \in \A_{d^2}$ the integration of $\omega$ gives a map $\pi: X \to E_0 = \C/\Z[i]$ well-defined up to translation. Recall from \S\ref{secspin} that if $(X,\omega)$ is smooth, requiring a distinguished image of the Weierstrass points $P_0$ to be the origin in $E_0$ one obtains a unique normalized cover $\pi$ branched over $z$ and $-z$. For a stable form $(X,\omega)$ on the boundary of $\Omega\cM_2$, the map is {\em normalized} when the node is mapped to the origin $z=0$ on $E_0$. For convenience we will keep calling such map for a singular curve $X$ a {\em normalized cover} and $z=0$ will be called a {\em branch point}. Define a {\em discriminant map} $\delta: \A_{d^2} \to \bfP \cong \P^1$ as a unique holomorhpic map such that for any $(X,\omega) \in \A_{d^2}$:
\begin{align*}
\delta(X,\omega) = \eta(z) = \pm 2z \in \bfP.
\end{align*}
Note that although we cannot distinguish between the branch points $z$ and $-z \in E_0 \cong \C/\Z[i]$, the map $\delta$ is well-defined since $\eta(z) = \eta(-z)$. Compare this definition to the one from Section 3 of \cite{Kani06}.

The discriminant map $\delta: \A_{d^2} \to \bfP$ is a local homeomorphism as long as the branch points $z$ and $-z$ stay distinct. These points collide exactly when $z \in E_0[2]$. In particular $\delta$ is branched over the vertices of the pillowcase and maps the boundary $\A_{d^2} \setminus \A^\circ_{d^2}$ to the vertices of the pillowcase.

\noindent
{\bf Square-tiling of $\A_{d^2}$.} We summarize the discussion above and prove Theorem~\ref{thmquaddiff}. 
Let $(X,\omega) \in \A_{d^2}$ and let $\pi: X \to E_0$ be a normalized cover branched over $\pm z$. The maps $\rho$ and $\delta$ are given by:
$$
\rho(X,\omega) = \pm 2z \in  \iota \backslash\C/\Z[i] \mbox{ and }
\delta(X,\omega) = \pm 2z \in \iota \backslash\C/2\Z[i].
$$
These maps are related by the following diagram:
\begin{equation} \label{diag:rhodelta}
\begin{tikzcd}
\A_{d^2} \arrow{d}{\delta} \arrow{r}{Id}
& \A_{d^2} \arrow{d}{\rho} \\
\iota \backslash\C/2\Z[i] \arrow{r}{\sigma}
&  \iota \backslash \C / \Z[i] ,
\end{tikzcd}
\end{equation}
where $\sigma: \pm z \mod 2\Z[i] \mapsto \pm z \mod \Z[i]$. Note that $\sigma$ is a degree 4 rational map and $dz^2$ on $\iota \backslash \C / \Z[i]$ pulls back to $dz^2$ on $\iota \backslash \C / 2\Z[i]$ via $\sigma$, and therefore: 
$$q =  \delta^* (dz^2) = \rho^* (dz^2).$$

\begin{proof}[Proof of Theorem~\ref{thmquaddiff}]
(i) The quadratic differential $q$ is a pullback of $dz^2$ from $\bfP$ via $\delta: \A_{d^2} \to \bfP$. The metric $|dz^2|$ on the pillowcase $\bfP$ defines a tiling by two unit squares. The pull back of this metric to $\A_{d^2}$ is $|q|$ and gives a square-tiling of $\A_{d^2}$. Note that $(\A_{d^2}, q)$ is not a translation surface and the identifications of the edges of the squares can be rotations by $\pi$.

(ii) Since $\SLZ \subset \GL$ action preserves $\Z[i]$, it acts on $\A_{d^2}$ by homeomorphisms. The local coordinate on $\A_{d^2}$ defined by $q$ is given by a branch point $\pm z$ on $E_0$ and $\SLZ$ action on it is affine. Hence $\SLZ$ acts by affine automorphisms on $(\A_{d^2}, q)$ and respects the square-tiling.
\end{proof}

Note that although $\rho$ and $\delta$ give the same quadratic differentials, they give different square-tilings.
The tiling defined by the pullback of $|dz^2|$ on $\iota \backslash \C / \Z[i]$ via $\rho$ consists of the squares whose sides have lengths $1/2$ and the tiling defined by the pullback of $|dz^2|$ on $\iota \backslash \C / 2\Z[i]$ via $\delta$ consists of the unit squares. The former tiling has 4 times more squares and can be obtained from the latter by subdivision of each unit square into 4 squares whose sides have lengths $1/2$. We prefer to use the latter square-tiling given by the discriminant map $\delta$ in order to avoid having extra vertices of the tiling.

\noindent
{\bf Vertices of the square-tiling and boundary points of $\A_{d^2}$.}
The preimages of the vertices $Q_i$ of the pillowcase $\bfP$ under the discriminant map $\delta$ are the vertices of the square-tiling of $\A_{d^2}$. In this subsection we give a geometric interpretation of the vertices of $\A_{d^2}$, in particular the boundary points $\A_{d^2} \setminus \A^\circ_{d^2}$.

Consider a small linear segment $s: (0,\varepsilon) \to \bfP$ of angle $\theta$ with a limit $s(t) \to Q_i$ as $t \to 0$ for some vertex $Q_i$. Choose its lift $\tilde s: (0,\varepsilon) \to \A_{d^2}$. Degeneration of the family $\pi_t: X_t \to E_0$ along this segment corresponds to colliding the branch points $z_t$ and $-z_t$ along a linear segment in the direction $\theta$ on $E_0$, i.e.\ in such a way that $2z_t = \pm t(\cos\theta + i \sin\theta)$. There are two scenarios of what can happen to $(X_t, \omega_t)$ as $t \to 0$: (1) two conical singularities stay different, or (2) two conical singularities collide.

In the case (1), one obtains an Abelian differential $(X_0, \omega_0)$ that is a primitive square-tiled surface in $\Omega\M_2(1,1)$. The subset of such square-tiled surfaces in $\A_{d^2}$ will be denoted by $\A_{d^2}[1]  \subset \A_{d^2}$.

In the case (2) there is a subset of saddle connections joining two singularities of $(X_t, \omega_t)$ that are being contracted, while the absolute periods stay fixed. Next we analyze the possibilities for contracting saddle connections joining distinct singularities. 

There are at most two such saddle connections in a fixed irrational direction $\theta$, since the conical singularities have total angle $4 \pi$ each. The union of two saddle connections on $(X, \omega)$ cannot be a contractible curve, otherwise they would bound a flat disk with a boundary consisting of two parallel saddle connections. Therefore there are exactly three possibilities for the set of contracting saddle connections to be (see Figure~\ref{figsc}):
\begin{enumerate}
\item[(2a)] a single saddle connection; or

\item[(2b)] union of two saddle connections that is a separating closed curve on $X$; or

\item[(2c)] union of two saddle connections that is a non-separating closed curve on $X$.
\end{enumerate}

\begin{figure}[H]
    \centering
    \includegraphics[width=0.8\textwidth]{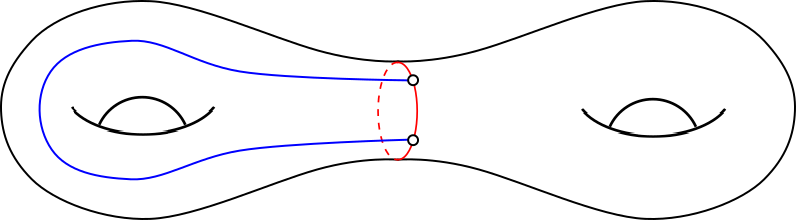}
    \caption{Three types of contracting curves: (2a) a single red saddle connection; (2b) a union of two red saddle connections (a separating closed curve); and (2c) a union of red and blue saddle connections (a non-separating closed curve).}
    \label{figsc}
\end{figure}

Contracting them we obtain respectively:
\begin{enumerate}
\item[(2a)] a primitive degree $d$ cover $\pi_0: X_0 \to E_0$ branched over the origin with a single ramification of order 3; or

\item[(2b)] two elliptic curves $E_1$ and $E_2$ joined at a node $p$, together with a pair of unbranched covers $\pi_1: E_1 \to E_0$ and $\pi_2: E_2 \to E_0$, satisfying:
\begin{itemize}
\item $\deg(\pi_1)+\deg(\pi_2)=d$;
\item $\pi_1(p)=\pi_2(p)$; and
\item $\pi_1$ and $\pi_2$ do not simultaneously factor through a non-trivial cover $\pi': E' \to E_0$; or
\end{itemize}

\item[(2c)] an elliptic curve $E$ with two points identified points $x_1, x_2 \in E$, together with an unbranched cover $\pi: E \to E_0$, satisfying:
\begin{itemize}
\item $\deg(\pi)=d$;
\item $\pi(x_1)=\pi(x_2)$; and
\item $\pi$ does not factor through a cover $\pi': E \to E'$, such that $\pi'(x_1)=\pi'(x_2)$.
\end{itemize}
\end{enumerate}

Note that the covers in the case (2a) correspond to primitive $d$-square-tiled surface in $\Omega\M_2(2)$. The subset of such square-tiled surfaces in $\A_{d^2}$ will be denoted by $\A_{d^2}[0]  \subset \A_{d^2}$.

In the case (2b) the pullback of $\left|dz^2\right|$ from $E_0$ gives a square-tiling of a nodal curve with a separating node. We will call it a {\em square-tiled surface with a separating node}. Similarly, in the case (2c) the pullback of $\left|dz^2\right|$ from $E_0$ gives a square-tiling of a nodal curve with a non-separating node. We will call it a {\em square-tiled surface with a non-separating node}.
The subsets of $\A_{d^2}$ corresponding to the square-tiled surfaces with a separating node and a non-separating node  will be denoted by $P_{s}(d) \subset \A_{d^2}$ and $P_{ns}(d) \subset \A_{d^2}$ respectively.

\begin{proposition} \label{cornerfibers}
For any $d>1$ and any vertex $z$ of the square-tiling of $\A_{d^2}$, one of the following holds:
\begin{enumerate}

\item[(i)] $z \in \A_{d^2}[1]$ and the local degree of $\delta$ at $z$ is 2; or

\item[(ii)] $z \in \A_{d^2}[0]$ and the local degree of $\delta$ at $z$ is 3; or

\item[(iii)] $z \in P_{s}(d) \cup P_{ns}(d)$ and $\delta$ is a local homeomorphism at $z$.
\end{enumerate}
\end{proposition}

\begin{proof}
$(i)$ The vertex $z$ is a square-tiled surface $(X, \omega)$ with two simple zeroes. There are only two ways to deform $(X, \omega)$ keeping its absolute periods in $\Z[i]$ and making relative periods to belong to $\pm t + \Z[i]$ for some $t \in \R$ (see Figure~\ref{figsplit1}). Thus locally $\delta$ has degree $2$ at $z \in \A_{d^2}[1]$.

\begin{comment4}
\begin{figure}[H]
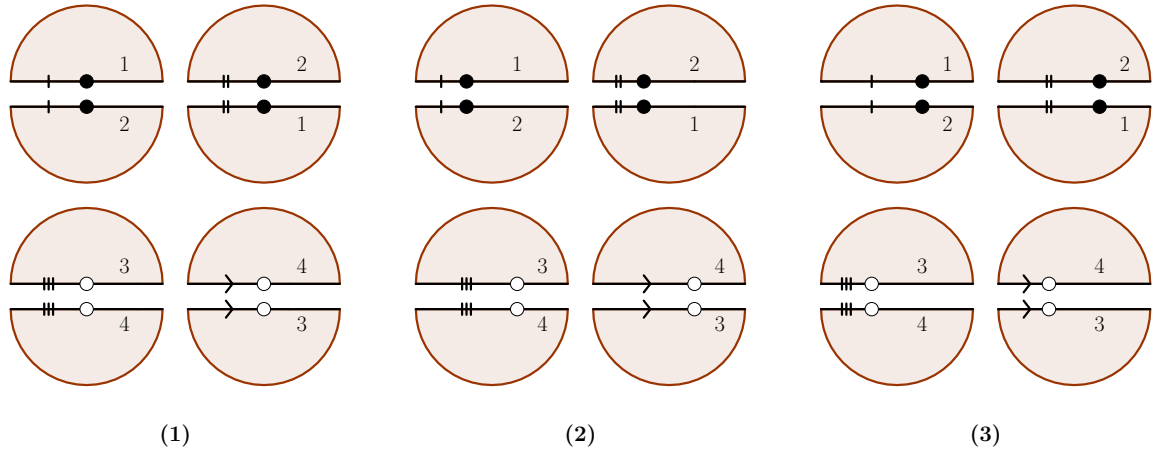
 
\centering
\includestandalone[width=\textwidth]{tikzsplit1}
\vspace{-10pt}
\caption{{\bf (1)} Neighborhoods of simple zeroes of an Abelian differential; {\bf (2)-(3)} two ways to move simple zeroes in a horizontal direction.}
\label{figsplit1}
\end{figure}
\end{comment4}

$(ii)$ The vertex $z$ is a square-tiled surface $(X, \omega)$ with a double zero. The double zero of $(X,\omega)$ has three prongs with slope $0$. There are three ways to split this double zero into two simple zeroes (see Figure~\ref{figsplit2}). Thus $\delta$ locally has degree $3$ at $\A_{d^2}[0]$.

\begin{comment4}
\begin{figure}[H]
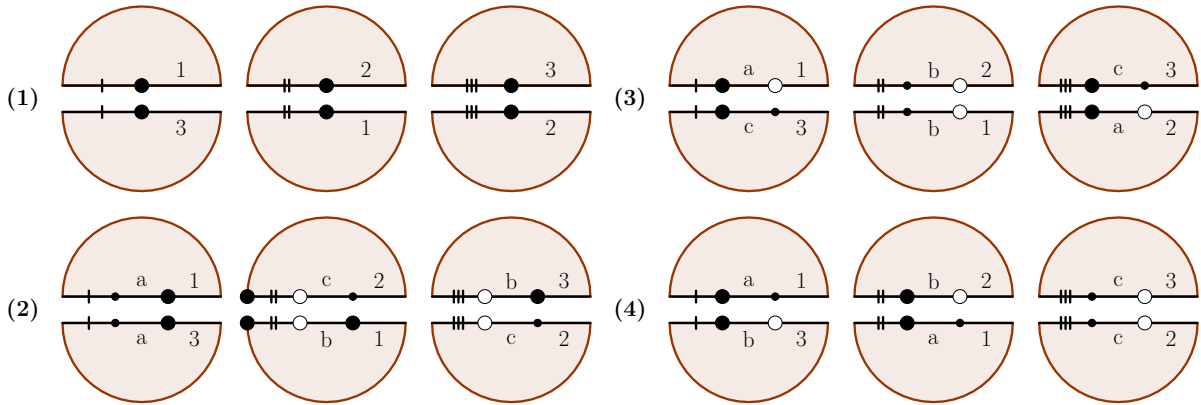
 
\centering
\includestandalone[width=\textwidth]{tikzsplit2}
\vspace{-10pt}
\caption{{\bf (1)} A neighborhood of a double zero of an Abelian differential; and {\bf (2)-(4)} three ways to split the double zero creating slits in horizontal direction and identify them to obtain a genus 2 surface.}
\label{figsplit2}
\end{figure}
\end{comment4}

$(iii)$ A neighborhood of the node $O$ of a square-tiled surface with a node is isomorphic to two disks glued at a point. There is a single way of producing slits at the node $O$ and identifying them to obtain a smooth genus 2 surface (see Figure~\ref{figsplit0}). In other words, for a square-tiled surface with a separating node there is a unique way to produce a horizontal slit of a given length on each of the tori. For square-tiled surface with a non-separating node there is a unique way to create a pair of horizontal slits and then identify them to obtain a genus 2 surface. Thus $\delta$ is a local homeomorphism at $P_{s}(d) \cup P_{ns}(d)$.
\begin{comment4}
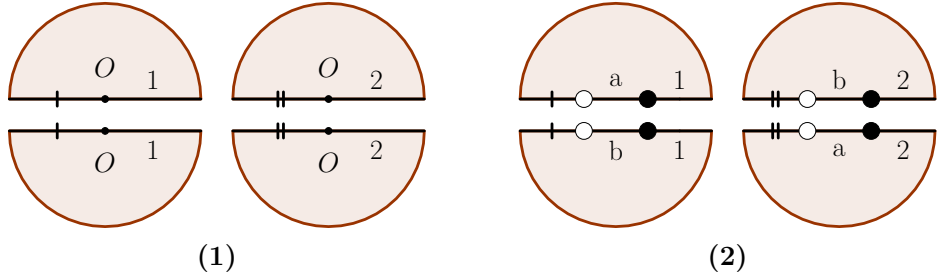
\begin{figure}[H] 
\centering
\definecolor{ffffff}{rgb}{1.,1.,1.}
\definecolor{zzttqq}{rgb}{0.6,0.2,0.}
\begin{tikzpicture}[line cap=round,line join=round,>=triangle 45,x=1.0cm,y=1.0cm]
\clip(-20.,-8.) rectangle (11.,3.);
\draw [shift={(-16.,-1.)},line width=2.5pt,color=zzttqq,fill=zzttqq,fill opacity=0.10000000149011612]  (0,0) --  plot[domain=0.:3.141592653589793,variable=\t]({1.*3.*cos(\t r)+0.*3.*sin(\t r)},{0.*3.*cos(\t r)+1.*3.*sin(\t r)}) -- cycle ;
\draw [shift={(-16.,-2.)},line width=2.5pt,color=zzttqq,fill=zzttqq,fill opacity=0.10000000149011612]  (0,0) --  plot[domain=-3.141592653589793:0.,variable=\t]({1.*3.*cos(\t r)+0.*3.*sin(\t r)},{0.*3.*cos(\t r)+1.*3.*sin(\t r)}) -- cycle ;
\draw [shift={(-9.,-1.)},line width=2.5pt,color=zzttqq,fill=zzttqq,fill opacity=0.10000000149011612]  (0,0) --  plot[domain=0.:3.141592653589793,variable=\t]({1.*3.*cos(\t r)+0.*3.*sin(\t r)},{0.*3.*cos(\t r)+1.*3.*sin(\t r)}) -- cycle ;
\draw [shift={(-9.,-2.)},line width=2.5pt,color=zzttqq,fill=zzttqq,fill opacity=0.10000000149011612]  (0,0) --  plot[domain=-3.141592653589793:0.,variable=\t]({1.*3.*cos(\t r)+0.*3.*sin(\t r)},{0.*3.*cos(\t r)+1.*3.*sin(\t r)}) -- cycle ;
\draw [shift={(0.,-1.)},line width=2.5pt,color=zzttqq,fill=zzttqq,fill opacity=0.10000000149011612]  (0,0) --  plot[domain=0.:3.141592653589793,variable=\t]({1.*3.*cos(\t r)+0.*3.*sin(\t r)},{0.*3.*cos(\t r)+1.*3.*sin(\t r)}) -- cycle ;
\draw [shift={(0.,-2.)},line width=2.5pt,color=zzttqq,fill=zzttqq,fill opacity=0.10000000149011612]  (0,0) --  plot[domain=-3.141592653589793:0.,variable=\t]({1.*3.*cos(\t r)+0.*3.*sin(\t r)},{0.*3.*cos(\t r)+1.*3.*sin(\t r)}) -- cycle ;
\draw [shift={(7.,-1.)},line width=2.5pt,color=zzttqq,fill=zzttqq,fill opacity=0.10000000149011612]  (0,0) --  plot[domain=0.:3.141592653589793,variable=\t]({1.*3.*cos(\t r)+0.*3.*sin(\t r)},{0.*3.*cos(\t r)+1.*3.*sin(\t r)}) -- cycle ;
\draw [shift={(7.,-2.)},line width=2.5pt,color=zzttqq,fill=zzttqq,fill opacity=0.10000000149011612]  (0,0) --  plot[domain=-3.141592653589793:0.,variable=\t]({1.*3.*cos(\t r)+0.*3.*sin(\t r)},{0.*3.*cos(\t r)+1.*3.*sin(\t r)}) -- cycle ;

\draw [line width=2.5pt] (-19.,-1.)-- (-16.,-1.);
\draw [line width=2.5pt] (-17.5,-0.7765305269141868) -- (-17.5,-1.2234694730858153);
\draw [line width=2.5pt] (-16.,-1.)-- (-13.,-1.);
\draw [line width=2.5pt] (-19.,-2.)-- (-16.,-2.);
\draw [line width=2.5pt] (-17.5,-1.776530526914185) -- (-17.5,-2.2234694730858138);
\draw [line width=2.5pt] (-16.,-2.)-- (-13.,-2.);
\draw [line width=2.5pt] (-12.,-1.)-- (-9.,-1.);
\draw [line width=2.5pt] (-10.586904808022327,-0.7765305269141868) -- (-10.586904808022327,-1.2234694730858153);
\draw [line width=2.5pt] (-10.413095191977673,-0.7765305269141868) -- (-10.413095191977673,-1.2234694730858153);
\draw [line width=2.5pt] (-9.,-1.)-- (-6.,-1.);
\draw [line width=2.5pt] (-12.,-2.)-- (-9.,-2.);
\draw [line width=2.5pt] (-10.586904808022327,-1.776530526914185) -- (-10.586904808022327,-2.2234694730858138);
\draw [line width=2.5pt] (-10.413095191977673,-1.776530526914185) -- (-10.413095191977673,-2.2234694730858138);
\draw [line width=2.5pt] (-9.,-2.)-- (-6.,-2.);
\draw [line width=2.5pt] (-3.,-1.)-- (-1.,-1.);
\draw [line width=2.5pt] (-2.,-0.7765305269141868) -- (-2.,-1.2234694730858153);
\draw [line width=2.5pt] (-1.,-1.)-- (2.,-1.);
\draw [line width=2.5pt] (2.,-1.)-- (3.,-1.);
\draw [line width=2.5pt] (-3.,-2.)-- (-1.,-2.);
\draw [line width=2.5pt] (-2.,-1.776530526914185) -- (-2.,-2.2234694730858138);
\draw [line width=2.5pt] (-1.,-2.)-- (2.,-2.);
\draw [line width=2.5pt] (2.,-2.)-- (3.,-2.);
\draw [line width=2.5pt] (4.,-1.)-- (6.,-1.);
\draw [line width=2.5pt] (4.913095191977672,-0.7765305269141868) -- (4.913095191977672,-1.2234694730858153);
\draw [line width=2.5pt] (5.086904808022326,-0.7765305269141868) -- (5.086904808022326,-1.2234694730858153);
\draw [line width=2.5pt] (4.,-2.)-- (6.,-2.);
\draw [line width=2.5pt] (4.913095191977672,-1.776530526914185) -- (4.913095191977672,-2.2234694730858138);
\draw [line width=2.5pt] (5.086904808022326,-1.776530526914185) -- (5.086904808022326,-2.2234694730858138);
\draw [line width=2.5pt] (6.,-2.)-- (8.,-2.);
\draw [line width=2.5pt] (8.,-2.)-- (10.,-2.);
\draw [line width=2.5pt] (6.,-1.)-- (8.,-1.);
\draw [line width=2.5pt] (8.,-1.)-- (10.,-1.);


\draw (-16,-0) node {\Huge $O$};
\draw (-16,-3) node {\Huge $O$};

\draw (-9,-0) node {\Huge $O$};
\draw (-9,-3) node {\Huge $O$};


\draw (-14.5,-0.4) node {\Huge 1};
\draw (-14.5,-2.6) node {\Huge 1};

\draw (-7.5,-0.4) node {\Huge 2};
\draw (-7.5,-2.6) node {\Huge 2};


\draw (2,-0.4) node {\Huge 1};
\draw (2,-2.6) node {\Huge 1};

\draw (0,-0.4) node {\Huge a};
\draw (0,-2.6) node {\Huge b};

\draw (9,-0.4) node {\Huge 2};
\draw (9,-2.6) node {\Huge 2};

\draw (7,-0.4) node {\Huge b};
\draw (7,-2.6) node {\Huge a};


\draw (-12.5,-6) node {\Huge \bf (1)};
\draw (3.5,-6) node {\Huge \bf (2)};

\begin{scriptsize}
\draw [fill=black] (-16.,-1.) circle (3pt);
\draw [fill=black] (-16.,-2.) circle (3pt);
\draw [fill=black] (-9.,-1.) circle (3pt);
\draw [fill=black] (-9.,-2.) circle (3pt);
\draw [fill=ffffff] (-1.,-1.) circle (7.5pt);
\draw [fill=ffffff] (-1.,-2.) circle (7.5pt);
\draw [fill=ffffff] (6.,-2.) circle (7.5pt);
\draw [fill=ffffff] (6.,-1.) circle (7.5pt);
\draw [fill=black] (1.,-1.) circle (7.5pt);
\draw [fill=black] (1.,-2.) circle (7.5pt);
\draw [fill=black] (8.,-1.) circle (7.5pt);
\draw [fill=black] (8.,-2.) circle (7.5pt);
\end{scriptsize}
\end{tikzpicture}
\vspace{-10pt}
\caption{{\bf (1)} A neighborhood of a node; and {\bf (2)} a unique way to create slits in horizontal direction at the node and identify them to obtain a genus $2$ surface. \qed}
\label{figsplit0}
\end{figure}
\let\qed\relax
\end{comment4}
\end{proof}

\begin{proof}[Proof of Theorem~\ref{thmvertices}]
Because $dz^2$ has simple poles at $Q_i$'s and $\delta$ is a local homeomorphism away from $Q_i$'s, the proof follows from Proposition~\ref{cornerfibers}.
\end{proof}

\noindent
{\bf The subset $\A_{d^2}[n]$.} Define the {\em set of primitive $n$-rational points} of $\C$ as follows:
$$\frac1n \Z[i]^* = \left\{ \frac an + i \frac bn  \in \C \ \middle\vert \ \gcd(a,b,n)=1\right\}.$$
Note that $\frac1n \Z[i]^*$ is invariant under $\Z[i]$ translations and under $\iota: z\mapsto -z$. Then the set of primitive $n$-torsions of $E_0$ is:
$$E_0[n]^*= \frac1n \Z[i]^* / \Z[i].$$
Analogously we define the {\em set of primitive $n$-torsions} of $\bfP= \iota \backslash \C / 2 \Z[i]$ as:
$$\bfP[n] = \iota \backslash \frac1n \Z[i]^* / 2\Z[i] \subset \bfP.$$

\begin{proof}[Proof of Theorem~\ref{thmlocating}]
Let $(X,\omega) \in \A_{d^2}$ and let $\pi: X \to E_0$ be a corresponding normalized cover branched over $z_1=z$ and $z_2=-z$. For $n>1$ we can rewrite the definition of $\A_{d^2}[n]$ as:
$$
\A_{d^2}[n]  = \left\{ (X,\omega) \in \A_{d^2} \, \middle\vert z_1-z_2 \in E_0[n]^*  \right\}.
$$

Recall that $\delta (X,\omega)= \eta(z)$, where $\eta: E_0 \to \bfP$ is given by $z \mod \Z[i] \to \pm 2z \mod 2\Z[i]$ (see diagram~(\ref{eta})). Then: 
$$\delta(X,\omega) = \eta( z ) = \pm 2 z  \in \bfP[n] \iff 2z \in \frac1n \Z[i]^* \iff$$
\[
\iff z_1 - z_2 = 2z = \frac{a+ib}n, \ \gcd(a,b,n)=1 \iff z_1 - z_2 \in E_0[n]^*.
\pushQED{\qed} 
\qedhere
\popQED
\]
\let\qed\relax
\end{proof}

\noindent
{\bf Spin invariant and $\A_{d^2}[n]$.} We conclude by interpreting the spin invariant in terms of the square-tiling of $\A_{d^2}$ and proving Theorem~\ref{thmformula}.

For odd $n$ the spin invariant distinguishes two $\SLZ$ invariant subsets of $\A_{d^2}[n]$:
\begin{equation*}
\A_{d^2}^0[n] =  \left\{ (X, \omega) \in  \A_{d^2}[n]  \ \middle\vert \
\epsilon(X,\omega) = 0\right\},
\end{equation*}
\begin{equation*}
\A_{d^2}^1[n] =  \left\{ (X, \omega) \in  \A_{d^2}[n] \ \middle\vert \
\epsilon(X,\omega) = 1\right\}.
\end{equation*}
Define the following subsets:
$$
\bfP[n]^{0} = \left\{ \frac{a+ib}n \in \bfP[n]  \ \middle\vert  \ a \equiv b \equiv 0 \mod 2 \right\}, \
\bfP[n]^{1} = \bfP[n] \setminus \bfP[n]^{0}.
$$

\begin{proposition} \label{propIWPvalues}
For any $d>1$ and odd $n>1$ we have:
$$
\A^0_{d^2}[n] = \delta^{-1}\left(\bfP[n]^0\right) \mbox{ and }
\A^1_{d^2}[n] = \delta^{-1}\left(\bfP[n]^1\right).
$$
\end{proposition}

\begin{proof}
Recall from Proposition~\ref{spincompute} that $\epsilon(X,\omega) = 0$ if and only if $nz = 0$, where $z=\displaystyle\frac{a+ib}{2n}$ and $-z$ are the branch points for some $a,b\in\Z$ with $\gcd(a,b,n)=1$. Clearly $nz = 0$ if and only if $a \equiv b \equiv 0 \mod 2$ and then:
$$
\delta(\pi) = \eta(z) = 2z = \frac{a+ib}{n} \in \bfP[n]^0.
$$
Therefore $\A^0_{d^2}[n] = \delta^{-1}\left(\bfP[n]^0\right)$ and, similarly, $\A^1_{d^2}[n] = \delta^{-1}\left(\bfP[n]^1\right)$. 
\end{proof}

\begin{proposition} \label{ratioorbits}
$\bigm| \A_{d^2}[n]^{1} \bigm| \ = 3\  \cdot \bigm| \A_{d^2}[n]^{0} \bigm|$.
\end{proposition}

\begin{proof}
By Proposition~\ref{propIWPvalues} it suffices to show that $\left| \bfP[n]^{1} \right| \ = 3\  \cdot \left| \bfP[n]^{0} \right|$, since $\delta$ is unramified at $\A_{d^2}[n]$. We start with the case, when $n$ is prime. Note that when $1 \le l \le n-1$ and $l$ is odd, all the points $\left(\frac kn,\frac ln\right) \in \bfP[n]$ belong to $\bfP[n]^{1}$, and when $1 \le l \le n-1$ and $l$ is even, half of the points $\left(\frac kn,\frac ln\right) \in \bfP[n]$ belong to $\bfP[n]^{1}$ and half to $\bfP[n]^{0}$. Thus when $1 \le l \le n-1$, there are $3$ times more points in $\bfP[n]^{1}$ than in $\bfP[n]^{0}$. It remains to consider the points on the horizontal edges of the pillowcase, i.e. $l=0$ and $l=n$. When $l=0$ there are $\frac {n-1}2$ points in $\bfP[n]^{1}$ and $\frac {n-1}2$ points in $\bfP[n]^{0}$, whereas when $l=n$ there are $n-1$ points in $\bfP[n]^{1}$ and none in $\bfP[n]^{0}$, and the ratio of the points in two invariant subsets is again $3$. That finishes the proof when $n$ is prime.

When $n$ is not prime $\left(\frac kn,\frac ln\right) \in \bfP[n]$ satisfies $\gcd(k,l,n)=1$. Therefore we have to throw away the subsets $D_r = \left\{ \left(\frac kn,\frac ln\right) \in \bfP[n] \bigm| \gcd(k,l,n)=r, \mbox{ where }r | n \right\}$. Following the inclusion-exclusion procedure one can observe that the ratio of $\left| \bfP[n]^{1} \right|$ and $\left|\bfP[n]^0 \right|$ remains $3$.
\end{proof}

\begin{proof}[Proof of Theorem~\ref{thmformula}]
Proposition~\ref{ratioorbits} implies that $t_{d,n,\epsilon} =\left| \A_{d^2}^{\epsilon}[n] \right|= \frac{2 \epsilon + 1}4 \left| \A_{d^2}[n]\right|$. On the other hand from Theorem~\ref{thmlocating} we know that:
$$
\left| \A_{d^2}[n] \right| = \deg\delta \cdot \left| \bfP[n] \right|,
$$
and:
$$
 \left| \bfP[n] \right| = 2 \cdot \left| E_0[n]^*\right| = 2 \left| \SLZ : \Gamma_1(n) \right| = \frac 2n \cdot \left| \SL_2(\ZZ{n}) \right| = 2 n^2 \prod_{q | n} (1 - \frac1{q^2}).
$$
The degree of $\delta$ is computed in \cite{EMS03} (Remark after Lemma 4.9) or in \cite{Kani06} (equation (31) and Corollary 30):
$$
\deg\delta = \frac{(d-1)}6  \cdot \left| \PSL_2(\ZZ{d}) \right|.
$$
Therefore we obtain:
\[
t_{d,n,\epsilon}  =
\displaystyle
(2 \epsilon + 1) \cdot \frac{d-1}{12n} \cdot \left| \PSL_2(\ZZ{d}) \right| \cdot \left| \SL_2(\ZZ{n}) \right|.
\pushQED{\qed} 
\qedhere
\popQED
\]
\let\qed\relax
\end{proof}


\section{Modular curves} \label{secmodular}

In this section we will show that the absolute period leaf $\A_{d^2}$ is isomorphic to the modular curve $X(d)$, and the isomorphism is canonical up to the automorphisms of $X(d)$. Therefore every modular curve $X(d)$ comes equipped with a quadratic differential $q$ that defines a square-tiling on it. In particular we will discuss the relation between the cusps of $Y(d) \subset X(d)$ and the poles of $q$. We then will compare symmetries of $X(d)$ as an algebraic curve and as a flat surface with metric given by $|q|$. We will also show that $(X(d), q)$ provide examples of two interesting phenomena: they are flat surfaces with Veech group $\PSLZ$ and no translation automorphisms; and they are quadratic differentials whose $\GL$ orbit projects to a point in $\M_g$.

We begging by reviewing the following result:

\begin{theorem}[\cite{Kuhn88}, \cite{Kani03}] \label{thmmodular}
For any choice of isomorphism $f_0:  (\ZZ{d})^2 \cong E_0[d]$ respecting the Weil pairing there is a natural isomorphism $i_{f_0}: \A_{d^2} \to X(d)$, such that for any $(X,\omega) \in \A^\circ_{d^2}$:
$$
i_{f_0}(X,\omega)=(E,f), \mbox{ where } \Jac(X) \sim E_0 \times E.
$$
\end{theorem}

\noindent
{\bf Isomorphism with the modular curve.} Recall that the modular curve is a Riemann surface $Y(d) = \Hyp \big/ \Gamma(d)$, where:
$$
 \Gamma(d) = \left\{
\begin{pmatrix}
a & b \\
c & d
\end{pmatrix} \in \SLZ \ \middle\vert \ a, d \equiv 1  \mbox{ and } b, c \equiv 0 \mod d
\right\}.
$$
The modular curve $Y(d)$ parametrizes equivalence classes of pairs $(E,f)$, where $E \in \M_1$ and $f:  (\ZZ{d})^2 \cong E[d] $ is an isomorphism that respects the Weil pairing. Its compactification is $X(d) = (\Hyp \cup \Q \cup \infty) \big/ \Gamma(d)$. The group of automorphisms $\Aut(X(d))$ of $X(d)$ is isomorphic to $\PSL(2,\ZZ{d}) = \PSL(2,\Z)/(\Gamma(d)/\pm \mathrm{Id})$. For any $g \in \PSL(2,\ZZ{d})$ the composition of $g$ with $f: (\ZZ{d})^2 \cong E[d]$ defines an automorphism of $X(d)$. In particular, Theorem~\ref{thmmodular} implies:

\begin{corollary} \label{corcanonical}
The isomorphism between $\A_{d^2}$ and $X(d)$ is canonical up to composition with elements of $\Aut(X(d))$.
\end{corollary}

We say that two abelian varieties $A$ and $B$ are {\em isogenous} is there exists a finite subgroup $\Gamma \subset A$ such that $A \big/ \Gamma \cong B$.  We now sketch a proof of Theorem~\ref{thmmodular} (see \cite{Kuhn88} and \cite{Kani03} for the details):

\begin{proof}[Proof of Theorem~\ref{thmmodular}]
Let $(X,\omega) \in \A^\circ_{d^2}$. The integration of $\omega$ defines a unique degree $d$ normalized cover $\pi: X \to E_0 = \C/\Z[i]$. It induces two maps of Jacobians: $\pi_*: \Jac(X) \to \Jac(E_0) \cong E_0$ by push forward of degree $0$ divisors, and $\pi^*: \Jac(E_0) \cong E_0 \to \Jac(X)$ by their pull back. Note that the isomorphism $\Jac(E_0) \cong E_0$ is well-defined, since $E_0$ comes with a choice of origin. The kernel of $\pi_*$ is a 1-dimensional subvariety of $\Jac(X)$. It is connected, since $\pi$ is a primitive cover. Therefore $\ker(\pi_*) = E$ for some $E \in \M_1$, and we obtain two exact sequences of maps:
\begin{align*}
& 0 \to E \xrightarrow{\phi}  \Jac(X) \xrightarrow{\pi_*} E_0 \to 0, \\
& 0 \to E_0 \xrightarrow{\pi^*}  \Jac(X) \xrightarrow{\phi^*} E \to 0.
\end{align*}
 Note that the maps $\phi$ and $\phi^*$ are induced by a degree $d$ branched cover $\pi': X \to E$ obtained in the following way. Let $u: X \to \Jac(X)$ be the Abel-Jacobi map, then $\pi' = \phi^* \circ u: X \to E$.

Define the map
$
\pi^*-\phi: E_0 \times E \to \Jac(X)
$, by
$(\pi^*-\phi)(x,y) = \pi^*(x) - \phi(y)$. It defines another exact sequence of maps:
$$
0 \to K \to E_0 \times E \to \Jac(X) \to 0,
$$
where $K = \ker(\pi^*-\phi)$. Then $K = \pi^*(E_0) \cap \phi(E) =  \pi^*(E_0) \cap \ker(\pi_*) = \pi^* (\ker ( \pi_* \circ \pi^*)) =\pi^*( E_0[d])$ and for the same reason $K = \phi(E[d])$. Since $\pi^*$ and $\phi$ are injective, this gives an isomorphism $\Psi: E_0[d] \cong E[d]$ that reverses the Weil pairing and $K = \Gamma_\Psi$ is a graph of this isomorphism in $E_0 \times E$. Then $f = \Psi \circ f_0$ and $i_{f_0}(X,\omega)=(E,f)$, which clearly satisfies $ \Jac(X) \sim E_0 \times E$.

This construction can be inverted on the open subset of points $(E,f) \in X(d)$ for which the abelian variety $E_0 \times E \big/ \Gamma_{f \circ f_{0}^{-1}}$ is a Jacobian of some Riemann surface $X \in \M_2$.

This birational morphism extends to an isomorphism $i_{f_0}: \A_{d^2} \to X(d)$ that only depends on the choice of  $f_0: E_0[d] \cong (\ZZ{d})^2 $.
\end{proof}

\noindent
{\bf Cusps of $Y(d)$ and poles of $q$.} We now show that some of the simple poles of $q$ on $\A_{d^2}$ are cusps of $Y(d) \subset X(d)$.

\begin{theorem} \label{thmcusps}
The set of cusps of the modular curve $X(d) \cong \A_{d^2}$ is a subset of simple poles of $q$ corresponding to stable curves with a non-separating node.
\end{theorem}

\begin{proof}
Let $\mathfrak A_2$ be the moduli space of principally polarized Abelian varieties of dimension $2$. The Jacobian map $j: \M_2 \to \mathfrak A_2$ is given by $j(X) = \Jac(X)$. It embeds $\A^\circ_{d^2}$ into $\mathfrak A_2$. The closure of its image $\overline{j(\A^\circ_{d^2})}$ is an algebraic curve isomorphic to a non-compactified modular curve:
$$
Y(d) = \Hyp \big/  \Gamma(d).
$$
The boundary $\overline{j(\A^\circ_{d^2})} \setminus j(\A^\circ_{d^2})$ consists of products of elliptic curves $E_0 \times E$ for some $E \in \M_1$. Recall from Theorem~\ref{thmvertices} that the boundary points of $\A^{\circ}_{d^2} \subset \A_{d^2}$ are poles of $q$ and they have two types: the ones supported on stable curves with a separating node and the ones supported on stable curves with a non-separating node. The first type has compact Jacobian and corresponds to the boundary locus of $j(\A^\circ_{d^2})$ in $\mathfrak A_2$. The second type has non-compact Jacobian and corresponds to the cusps of $Y(d)$. 
\end{proof}

The subset $P_{ns}(d) \subset \A_{d^2}$ corresponding to the square-tiled surfaces with a non-separating node will be referred to as the set of {\em cusp poles} of $q$, and the subset $P_{s}(d) \subset \A_{d^2}$ corresponding to the square-tiled surfaces with a separating node will be referred to as the set of {\em non-cusp poles} of $q$.

\noindent
{\bf Symmetries of $(X(d),q)$.} 
Let $\M_{g,n}$ be the moduli space of Riemann surfaces of genus $g$ with $n$ marked points. Then let $\QQ\M_{g,n} \to \M_{g,n}$ define the bundle of pairs $(X,\eta)$, where $\eta\ne 0$ is a meromorphic quadratic differential on $X \in \M_g$ with $n$ simple poles located at $n$ marked points. Similarly to the case of $\Omega\M_g$, there is an action of $\GL / \pm Id$ on $\QQ\M_{g,n}$. The stabilizer of $(X,\eta)$ under this action is denoted by $\PSL(X,\eta) \subset \PSL(2, \R)$. Let $\Aff^+(X,\eta)$ denote the group of affine automorphisms in the metric $|\eta|$. 

\begin{theorem} For any $d>1$ we have:
$\Aff^+(X(d),q) \cong \PSL(X(d),q) \cong \PSLZ$ and $\Aut(X(d)) \cap \Aff^+(X(d),q) \cong \ZZ{2}$.
\end{theorem}

In particular there are no automorphisms of $X(d)$ that preserve $q$, and the only affine automorphism of $(X(d),q)$ that acts by an automorphism of $X(d)$ has order $2$. For example, the modular curve $X(2)$ has a group of symmetries of a regular tetrahedron, however the only symmetry that persists on the level of the square-tiling of $X(2)$ (see Figure~\ref{figX(2)intro}) is given by rotating each square by $\pm \pi/2$ and switching their places.

\begin{proof}
In \S\ref{seclighteave} we will show that $\A_{d^2}$ contains a unique embedded open horizontal cylinder of circumference $2$ and height $1$ and one of its boundaries contains a single cusp. It follows that this cylinder, and hence all of $\A_{d^2}$, must be fixed by any holomorphic automorphism that preserves $q$. Together with Theorem~\ref{thmquaddiff} it implies that $\Aff^+(X(d),q) \cong \PSL(X(d),q) \cong \PSLZ$

The only non-trivial elements of $\PSL(X(d),q)$ that act holomorphically on $\C$ are rotations by $\pm \pi/2$, therefore $\Aut(X(d)) \cap \Aff^+(X(d),q) \cong \ZZ{2}$.
\end{proof}

\noindent
{\bf Teichm\"uller point.} Unlike for holomorphic 1-forms, for meromorphic quadratic differentials the projection of an orbit $\GL \cdot (X,\eta) \subset \QQ\M_{g,n}$ to $\M_g$ can be a single point. 

For any $E \cong \C/\Z[\tau] \in \M_1$, where $\tau \in \Hyp$, one can define an absolute period leaf $\A_{d^2}(E)$ as:
\begin{flalign*}
 \A_{d^2}(E) = \left\{   (X,\omega) \in \Omega\cM_2 \,\middle\vert\,
   \Per(\omega) = \Z[\tau] \mbox{ and } \displaystyle \int_X  |\omega|^2 = d \cdot \Imag\tau
\right\}.
\end{flalign*}
Defining the discriminant map in the same way as for $E_0$ one obtains a quadratic differentials $q_E$ on $\A_{d^2}(E)$, which gives a tiling of $\A_{d^2}(E)$ by parallelograms of the shape $\langle 1, \tau \rangle$. Clearly:
$$
\begin{pmatrix}
1 & \Real\tau \\
0 & \Imag\tau
\end{pmatrix} \cdot (\A_{d^2} , q) = (\A_{d^2}(E), q_E).
$$
Note also that Theorem~\ref{thmmodular} does not rely on any special properties of $E_0$ and works as well for any choice of $E \in \M_1$ and an isomorphism of torsion $f_E:  (\ZZ{d})^2 \cong E[d]$ that respects the Weil pairing. This implies $\A_{d^2}(E) \cong X(d)$ and therefore:

\begin{theorem}
The projection of the orbit $\GL \cdot (X(d),q) \subset \QQ\M_{g,n}$ to $\M_g$, where $g$ is the genus of $X(d)$,  is a point.
\end{theorem}


%
%


\section{The square-tilings of the modular curves} \label{sectiling}

In this section we describe the procedure that can be used to construct the square-tiling of the modular curve $X(d)$ defined in the previous sections for any $d>1$. The squares of the tiling of $X(d)$ form maximal horizontal strips of various heights and widths. We call such strip a {\em horizontal cylinder} of $X(d)$, since its vertical edges are identified. We will enumerate horizontal cylinders of $X(d)$ and find their dimensions:

\begin{theorem} \label{thmdecomposition}
For any $d>1$ the square-tiling of the modular curve $X(d)$ naturally decomposes into a union of horizontal cylinders, that consists of squares and for which the following conditions hold:

(i) (enumeration) The set of horizontal cylinders is in bijection with the set of unordered pairs of triples $\{ (w_1, s_1, T_1), (w_2, s_2, T_2 ) \}  \in \Sym^2\N^3,$ satisfying the following conditions:

\begin{itemize}
\item (area) $s_1 w_1 + s_2 w_2 = d$,

\item (twist) $0 \le T_1,T_2 < \gcd(w_1, w_2) $,

\item (primitivity) $\gcd(s_1, s_2) = 1$ and $\gcd(T_1s_2 - T_2s_1, w_1, w_2)=1$;
\end{itemize}

(ii) (dimensions) The height of the cylinder $\CC =  \{ (w_1, s_1, T_1), (w_2, s_2, T_2) \}$ is \newline
$H_{\CC} = \min(s_1,s_2),$
its circumference is
$W_{\CC} = \lcm(w_1, w_2, w_1 + w_2)$.
\end{theorem}

We will use the decomposition into horizontal cylinders to define {\em cylinder coordinates} and {\em Euclidean coordinates} on $X(d)$. We will conclude by discussing the square-tilings of $X(d)$ for $d=2,3,4,5$ presented in the introduction.

\noindent
{\bf Horizontal foliation.}
According to Theorem~\ref{thmmodular} there is an isomorphism $X(d) \cong \A_{d^2}$. We will carry out the description of the square-tiling in terms of $\A_{d^2}$ and meromorphic differential $q$ that defines the square-tiling of $\A_{d^2}$ (see \S\ref{secabs} for definitions). The kernel of the harmonic 1-form $\Imag( \pm \sqrt q)$ defines a singular foliation of $\A_{d^2}$ that we call the {\em horizontal foliation} of $\A_{d^2}$. Every non-singular leaf of the horizontal foliation is closed. Any maximal open connected subset of $\A_{d^2}$ that is a union of non-singular closed leaves of the horizontal foliation is called a {\em horizontal cylinder}. Every non-singular closed leaf belongs to some horizontal cylinder. Therefore removing the singular leaves from $\A_{d^2}$ we obtain a finite disjoint union of horizontal cylinders. We say that $\A_{d^2}$ naturally decomposes into horizontal cylinders. We begin by presenting the combinatorial approach to the study of primitive genus $2$ covers of the square torus introduced in \cite{EMS03} and by defining {\em cylinder and Euclidean coordinates} of $(X,\omega) \in \A_{d^2}$.

\noindent
{\bf Cylinder decomposition of $(X,\omega)$.} 
An Abelian differential $(X, \omega) \in \A_{d^2}$ with $Z(\omega) = \{ x_1, x_2 \}$ is called {\em generic}, if none of the relative periods $\int_{x_1}^{x_2}\omega$ are purely real. In terms of the horizontal foliation defined by $\Imag \omega$ this simply means that no horizontal leaf contains both singularies $x_1$ and $x_2$.
Therefore singular horizontal leaves of a generic $(X,\omega)$ start and end at the same singularity. Then a generic $(X,\omega) \in \A_{d^2}$ naturally decomposes into a disjoint union of horizontal cylinders $C_j$, whose boundaries are formed by singular leaves $\gamma_i$. The total angle around each of the two singularities of $(X, \omega)$ is $4\pi$, hence they have four prongs in horizontal directions and the union of loops $\gamma_i$ on $X$ is topologically a pair of figure eights. The only possibility for this cylinder decomposition is $(X, \omega) = C_1 \cup C_2 \cup C_3$, where the circumference of one of the cylinders is the sum of the circumferences of the other two (see Figure~\ref{fig3cyl}). We review a proof of the following result:

\begin{proposition}[\cite{EMS03}] \label{encoding}
Let $d>1$ be any integer.
Consider an Abelian differential $(X,\omega)$ obtained as three horizontal cylinders of circumferences $w_1, w_2, w_3 \in \R$, heights $h_1, h_2, h_3 \in \R$ with boundaries identified by the twists $t_1, t_2, t_3 \in \R$, where $t_i \in [0,w_i)$, as in Figure~\ref{fig3cyl}. Then $(X,\omega)$ belongs to $\A_{d^2}$ and generic if and only if the following conditions hold:
\begin{itemize}
\item (circumference) $w_1+w_2=w_3$;

\item (area) $w_1(h_1+h_3)+w_2(h_2+h_3)=d$;

\item (generic) $h_1, h_2, h_3 > 0$;

\item (integral periods) $w_1, w_2, h_1+h_3, h_2+h_3, t_1-t_3, t_2-t_3 \in \Z$; and

\item (primitivity) $\gcd(h_1+h_3,h_2+h_3)=1=\gcd( (t_1-t_3) (h_2+h_3) - (t_2-t_3) (h_1+h_3), w_1, w_2) ) = 1$.
\end{itemize}
\end{proposition}

\begin{comment6}
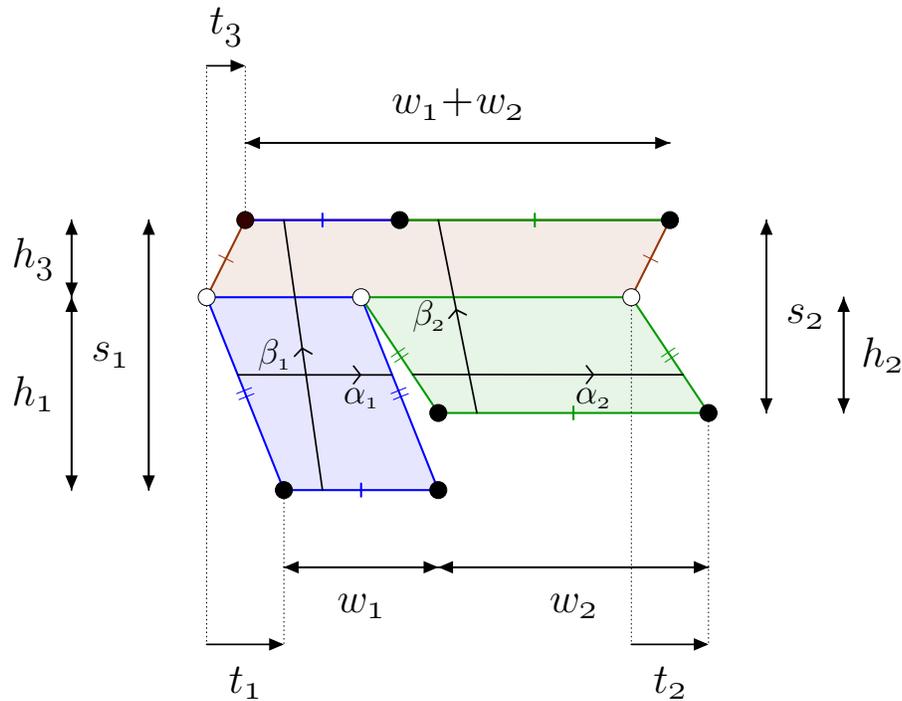
\begin{figure}[H] 
\centering
  \definecolor{zzttqq}{rgb}{0.6,0.2,0.}
\definecolor{ttqqqq}{rgb}{0.2,0.,0.}
\definecolor{qqzzqq}{rgb}{0.,0.6,0.}
\definecolor{qqqqff}{rgb}{0.,0.,1.}
\definecolor{ffffff}{rgb}{1.,1.,1.}
\begin{tikzpicture}[line cap=round,line join=round,>=triangle 45,x=0.4cm,y=0.4cm]
\clip(-8.,-56.) rectangle (42.,-18.);
\fill[line width=1pt,color=qqqqff,fill=qqqqff,fill opacity=0.10000000149011612] (12.,-34.) -- (16.,-44.) -- (8.,-44.) -- (4.,-34.) -- cycle;
\fill[line width=1pt,color=qqzzqq,fill=qqzzqq,fill opacity=0.10000000149011612] (12.,-34.) -- (16.,-40.) -- (30.,-40.) -- (26.,-34.) -- cycle;
\fill[line width=1pt,color=zzttqq,fill=zzttqq,fill opacity=0.10000000149011612] (6.,-30.) -- (4.,-34.) -- (26.,-34.) -- (28.,-30.) -- cycle;
\draw [->,line width=1pt] (4.,-52.) -- (8.,-52.);
\draw [->,line width=1pt] (26.,-52.) -- (30.,-52.);
\draw [->,line width=1pt] (8.,-48.) -- (16.,-48.);
\draw [->,line width=1pt] (16.,-48.) -- (8.,-48.);
\draw [->,line width=1pt] (16.,-48.) -- (30.,-48.);
\draw [->,line width=1pt] (30.,-48.) -- (16.,-48.);

\draw [line width=0.5pt,color=qqqqff] (14.325090844330536,-38.69038967206617) -- (13.551065024495935,-39.);
\draw [line width=1/2pt,color=qqqqff] (14.448934975504065,-39.) -- (13.674909155669464,-39.30961032793383);
\draw [line width=1/2pt,color=qqqqff] (5.674909155669464,-39.30961032793383) -- (6.44893497550407,-39.);
\draw [line width=1/2pt,color=qqqqff] (5.551065024495931,-39.) -- (6.325090844330536,-38.69038967206617);
\draw [line width=1/2pt,color=qqzzqq] (14.25433460322516,-36.630058758945246) -- (13.56069477624746,-37.09248531026369);
\draw [line width=1/2pt,color=qqzzqq] (14.43930522375254,-36.90751468973632) -- (13.74566539677484,-37.36994124105476);
\draw [line width=1/2pt,color=qqzzqq] (27.74566539677484,-37.36994124105476) -- (28.439305223752537,-36.90751468973632);
\draw [line width=1/2pt,color=qqzzqq] (27.56069477624746,-37.09248531026369) -- (28.25433460322516,-36.630058758945246);
\draw [line width=1/2pt,color=zzttqq] (5.372820204638625,-32.1864101023193) -- (4.627179795361374,-31.8135898976807);
\draw [line width=1/2pt,color=zzttqq] (26.62717979536137,-31.8135898976807) -- (27.37282020463862,-32.1864101023193);

\draw [line width=1.2pt,color=qqqqff] (12.,-34.)-- (16.,-44.);
\draw [line width=1.2pt,color=qqqqff] (16.,-44.)-- (8.,-44.);
\draw [line width=1.2pt,color=qqqqff] (12.,-44.338856944588855) -- (12.,-43.66114305541113);
\draw [line width=1.2pt,color=qqqqff] (8.,-44.)-- (4.,-34.);
\draw [line width=1.2pt,color=qqqqff] (4.,-34.)-- (12.,-34.);
\draw [line width=1.2pt,color=qqzzqq] (12.,-34.)-- (16.,-40.);
\draw [line width=1.2pt,color=qqzzqq] (16.,-40.)-- (30.,-40.);
\draw [line width=1.2pt,color=qqzzqq] (23.,-39.66114305541113) -- (23.,-40.338856944588855);
\draw [line width=1.2pt,color=qqzzqq] (30.,-40.)-- (26.,-34.);
\draw [line width=1.2pt,color=qqzzqq] (26.,-34.)-- (12.,-34.);
\draw [line width=1.2pt,color=zzttqq] (6.,-30.)-- (4.,-34.);
\draw [line width=1.2pt,color=zzttqq] (26.,-34.)-- (28.,-30.);
\draw [line width=1.2pt,color=zzttqq] (28.,-30.)-- (6.,-30.);
\draw [line width=1.2pt,color=qqqqff] (6.,-30.)-- (14.,-30.);
\draw [line width=1.2pt,color=qqqqff] (10.,-29.661143055411138) -- (10.,-30.33885694458886);
\draw [line width=1.2pt,color=qqzzqq] (14.,-30.)-- (28.,-30.);
\draw [line width=1.2pt,color=qqzzqq] (21.,-29.661143055411138) -- (21.,-30.33885694458886);
\draw [line width=0.5pt,dotted] (26.,-34.)-- (26.,-52.);
\draw [line width=0.5pt,dotted] (4.,-34.)-- (4.,-22.);
\draw [line width=0.5pt,dotted] (6.,-30.)-- (6.,-22.);
\draw [line width=0.5pt,dotted] (30.,-40.)-- (30.,-52.);
\draw [->,line width=1pt] (4.,-22.) -- (6.,-22.);
\draw [line width=0.5pt,dotted] (4.,-34.)-- (4.,-52.);
\draw [line width=0.5pt,dotted] (8.,-44.)-- (8.,-52.);
\draw [->,line width=1pt] (6.,-26.) -- (28.,-26.);
\draw [->,line width=1pt] (28.,-26.) -- (6.,-26.);
\draw [line width=1pt] (10.,-44.)-- (8.,-30.);
\draw [line width=1pt] (8.949682310892102,-36.64777617624472) -- (9.452859201971087,-36.93530582828985);
\draw [line width=1pt] (8.949682310892102,-36.64777617624472) -- (8.547140798028916,-37.064694171710144);
\draw [line width=1pt] (18.,-40.)-- (16.,-30.);
\draw [line width=1pt] (16.930221919945623,-34.651109599728116) -- (17.448573371778142,-34.910285325644374);
\draw [line width=1pt] (16.930221919945623,-34.651109599728116) -- (16.551426628221854,-35.08971467435562);
\draw [line width=1pt] (5.609640230186187,-38.02410057546547)-- (13.611221145192355,-38.028052862980886);
\draw [line width=1pt] (11.966230436104452,-38.02625246235676) -- (11.610204732231802,-38.48353353861411);
\draw [line width=1pt] (11.966230436104452,-38.02625246235676) -- (11.610656643146736,-37.56861989983224);
\draw [line width=1pt] (14.670503071814585,-38.005754607721876)-- (28.679127367076514,-38.01869105061477);
\draw [line width=1pt] (24.03061485955387,-38.01255139688693) -- (23.67439277523591,-38.469679509307575);
\draw [line width=1pt] (24.03061485955387,-38.01255139688693) -- (23.675237663655196,-37.554766149029064);

\draw [->,line width=1pt] (33.,-40.) -- (33.,-30.);
\draw [->,line width=1pt] (33.,-30.) -- (33.,-40.);
\draw [->,line width=1pt] (1.,-30.) -- (1.,-44.);
\draw [->,line width=1pt] (1.,-44.) -- (1.,-30.);

\draw [->,line width=1pt] (37.,-34.) -- (37.,-40.);
\draw [->,line width=1pt] (37.,-40.) -- (37.,-34.);

\draw [->,line width=1pt] (-3.,-34.) -- (-3.,-44.);
\draw [->,line width=1pt] (-3.,-44.) -- (-3.,-34.);

\draw [->,line width=1pt] (-3.,-30.) -- (-3.,-34.);
\draw [->,line width=1pt] (-3.,-34.) -- (-3.,-30.);

\draw (6,-54) node  {\Huge $t_1$};
\draw (28,-54) node  {\Huge $t_2$};
\draw (5,-20) node  {\Huge $t_3$};

\draw (-1,-37) node  {\Huge $s_1$};
\draw (35,-35) node  {\Huge $s_2$};

\draw (-5,-39) node  {\Huge $h_1$};
\draw (-5,-32) node  {\Huge $h_3$};
\draw (39,-37) node  {\Huge $h_2$};

\draw (12,-50) node  {\Huge $w_1$};
\draw (23,-50) node  {\Huge $w_2$};
\draw (17,-24) node  {\Huge $w_1\!+\!w_2$};

\draw (12,-39.25) node  {\LARGE $\alpha_1$};
\draw (24,-39.25) node  {\LARGE $\alpha_2$};

\draw (7.5,-37) node  {\LARGE $\beta_1$};
\draw (15.5,-35.) node  {\LARGE $\beta_2$};

\begin{scriptsize}
\draw [fill=ffffff] (4.,-34.) circle (5pt);
\draw [fill=ffffff] (12.,-34.) circle (5pt);
\draw [fill=black] (16.,-44.) circle (5pt);
\draw [fill=black] (14.,-30.) circle (5pt);
\draw [fill=black] (28.,-30.) circle (5pt);
\draw [fill=black] (8.,-44.) circle (5pt);
\draw [fill=ffffff] (26.,-34.) circle (5pt);
\draw [fill=black] (16.,-40.) circle (5pt);
\draw [fill=black] (30.,-40.) circle (5pt);
\draw [fill=ttqqqq] (6.,-30.) circle (5pt);
\end{scriptsize}
\end{tikzpicture}
  \caption{The 3-cylinder decomposition of $(X,\omega)$.}
  \label{fig3cyl}
\end{figure}
\end{comment6}

\begin{proof}[Proof of Proposition~\ref{encoding}]
The circumference condition comes from the discussion above. Since $(X,\omega)$ is a degree $d$ cover of an area $1$ torus, its area is $d = w_1h_1 + w_2h_2+w_3h_3 = w_1(h_1+h_3)+w_2(h_2+h_3)$. The heights are positive if and only if the singularities are not on the same horizontal leaf. 

Primitivity is equivalent to checking that the absolute periods of $(X, \omega)$ generate $\Z[i]$. Choose a symplectic basis $\alpha_1, \beta_1, \alpha_2, \beta_2 \in H_1(X,\Z)$ (see Figure~\ref{fig3cyl}). The integration of $\omega$ along these cycles gives respectively:
$$
w_1, (t_3-t_1)+i(h_1+h_3), w_2,  (t_3-t_2)+i(h_2+h_3).
$$
For $(X,\omega)$ to be in $\A_{d^2}$ the real and imaginary parts of the absolute periods must be integers, which justifies the integral periods condition. For the primitivity condition, first note that the imaginary parts $h_1+h_3$ and $h_2+h_3$ must span $1$ over $\Z$, which is equivalent to $\gcd(h_1+h_3,h_2+h_3)=1$. Secondly, note that $\R \cap \spans_{\Z}(\int_{\beta_1}\omega,\int_{\beta_2}\omega) = ((t_1-t_3) (h_2+h_3) - (t_2-t_3) (h_1+h_3))\cdot \Z $ and, since the real parts must span $1$, we have:
\[
\gcd( (t_1-t_3) (h_2+h_3) - (t_2-t_3) (h_1+h_3), w_1, w_2) ) = 1.
\pushQED{\qed} 
\qedhere
\popQED
\]
\let\qed\relax
\end{proof} 

\noindent
{\bf Cylinder coordinates on $\A_{d^2}$.} We say that a vector $(w_1, s_1, w_2, s_2, t_1, t_2, t_3, h) \in \N^4 \times \R^4$ is {\em admissible} if the numbers $w_1, w_2, h_1 = s_1 - h, h_2 = s_2 -h, h_3 = h, t_1, t_2$ and $t_3$ satisfy the conditions of Proposition~\ref{encoding}. We define the {\em cylinder coordinates} of a generic $(X,\omega) \in  \A_{d^2}$ to be an equivalence class of admissible vectors:
\begin{equation} \label{eq:cylcoord}
(w_1, s_1, w_2, s_2, t_1, t_2, t_3, h) \sim (w_2, s_2, w_1, s_1, t_2, t_1, t_3, h) \in \N^4 \times \R^4.
\end{equation}
We impose an equivalence relation on vectors, because there is no consistent way to order the narrower cylinders. 

Since generic $(X,\omega)$ are dense in $\A_{d^2}$, the cylinder coordinates extend to all of $\A_{d^2}$ by continuity, however this extension is not globally injective.

\noindent
{\bf Horizontal cylinders of $\A_{d^2}$.} Let $Cyl(\A_{d^2})$ denote the set of horizontal cylinders of $\A_{d^2}$.
We give a proof of Theorem~\ref{thmdecomposition}, that enumerates the horizontal cylinders and gives their dimensions, below.

\begin{comment6}
\begin{figure}[H] 
\centering
  \definecolor{zzttqq}{rgb}{0.6,0.2,0.}
\definecolor{ttqqqq}{rgb}{0.2,0.,0.}
\definecolor{qqzzqq}{rgb}{0.,0.6,0.}
\definecolor{qqqqff}{rgb}{0.,0.,1.}
\definecolor{ffffff}{rgb}{1.,1.,1.}
\begin{tikzpicture}[line cap=round,line join=round,>=triangle 45,x=0.4cm,y=0.4cm]
\clip(0.,-58.) rectangle (72.,-16.);
\fill[line width=1pt,color=qqqqff,fill=qqqqff,fill opacity=0.10000000149011612] (12.,-34.) -- (16.,-44.) -- (8.,-44.) -- (4.,-34.) -- cycle;
\fill[line width=1pt,color=qqzzqq,fill=qqzzqq,fill opacity=0.10000000149011612] (12.,-34.) -- (16.,-40.) -- (30.,-40.) -- (26.,-34.) -- cycle;
\fill[line width=1pt,color=zzttqq,fill=zzttqq,fill opacity=0.10000000149011612] (6.,-30.) -- (4.,-34.) -- (26.,-34.) -- (28.,-30.) -- cycle;
\fill[line width=1pt,color=qqqqff,fill=qqqqff,fill opacity=0.10000000149011612] (46.,-34.) -- (54.,-44.) -- (46.,-44.) -- (38.,-34.) -- cycle;
\fill[line width=1pt,color=qqzzqq,fill=qqzzqq,fill opacity=0.10000000149011612] (46.,-34.) -- (54.,-40.) -- (68.,-40.) -- (60.,-34.) -- cycle;
\fill[line width=1pt,color=zzttqq,fill=zzttqq,fill opacity=0.10000000149011612] (44.,-30.) -- (38.,-34.) -- (60.,-34.) -- (66.,-30.) -- cycle;
\draw [->,line width=1pt] (4.,-52.) -- (8.,-52.);
\draw [->,line width=1pt] (26.,-52.) -- (30.,-52.);
\draw [line width=1.2pt,color=qqqqff] (12.,-34.)-- (16.,-44.);
\draw [line width=1.2pt,color=qqqqff] (16.,-44.)-- (8.,-44.);
\draw [line width=1.2pt,color=qqqqff] (12.,-44.41682566047866) -- (12.,-43.583174339521335);
\draw [line width=1.2pt,color=qqqqff] (8.,-44.)-- (4.,-34.);

\draw [line width=1.2pt,color=qqqqff] (14.325090844330536,-38.69038967206617) -- (13.551065024495935,-39.);
\draw [line width=1.2pt,color=qqqqff] (14.448934975504065,-39.) -- (13.674909155669464,-39.30961032793383);
\draw [line width=1.2pt,color=qqqqff] (5.674909155669464,-39.30961032793383) -- (6.44893497550407,-39.);
\draw [line width=1.2pt,color=qqqqff] (5.551065024495931,-39.) -- (6.325090844330536,-38.69038967206617);
\draw [line width=1.2pt,color=qqzzqq] (14.25433460322516,-36.630058758945246) -- (13.56069477624746,-37.09248531026369);
\draw [line width=1.2pt,color=qqzzqq] (14.43930522375254,-36.90751468973632) -- (13.74566539677484,-37.36994124105476);
\draw [line width=1.2pt,color=qqzzqq] (27.74566539677484,-37.36994124105476) -- (28.439305223752537,-36.90751468973632);
\draw [line width=1.2pt,color=qqzzqq] (27.56069477624746,-37.09248531026369) -- (28.25433460322516,-36.630058758945246);
\draw [line width=1.2pt,color=zzttqq] (5.372820204638625,-32.1864101023193) -- (4.627179795361374,-31.8135898976807);
\draw [line width=1.2pt,color=zzttqq] (26.62717979536137,-31.8135898976807) -- (27.37282020463862,-32.1864101023193);

\draw [line width=1.2pt,color=qqqqff] (4.,-34.)-- (12.,-34.);
\draw [line width=1.2pt,color=qqzzqq] (12.,-34.)-- (16.,-40.);
\draw [line width=1.2pt,color=qqzzqq] (16.,-40.)-- (30.,-40.);
\draw [line width=1.2pt,color=qqzzqq] (23.,-39.58317433952134) -- (23.,-40.416825660478665);
\draw [line width=1.2pt,color=qqzzqq] (30.,-40.)-- (26.,-34.);
\draw [line width=1.2pt,color=qqzzqq] (26.,-34.)-- (12.,-34.);
\draw [line width=1.2pt,color=zzttqq] (6.,-30.)-- (4.,-34.);
\draw [line width=1.2pt,color=zzttqq] (26.,-34.)-- (28.,-30.);

\draw [line width=1.2pt,color=zzttqq] (28.,-30.)-- (6.,-30.);
\draw [line width=1.2pt,dash pattern=on 17pt off 17pt] (14.,-30.)-- (12.,-34.);
\draw [line width=1.2pt,color=qqqqff] (6.,-30.)-- (14.,-30.);
\draw [line width=1.2pt,color=qqqqff] (10.,-29.58317433952134) -- (10.,-30.41682566047866);
\draw [line width=1.2pt,color=qqzzqq] (14.,-30.)-- (28.,-30.);
\draw [line width=1.2pt,color=qqzzqq] (21.,-29.58317433952134) -- (21.,-30.41682566047866);
\draw [line width=0.5pt,dotted] (26.,-34.)-- (26.,-52.);
\draw [line width=0.5pt,dotted] (4.,-34.)-- (4.,-22.);
\draw [line width=0.5pt,dotted] (6.,-30.)-- (6.,-22.);
\draw [line width=0.5pt,dotted] (30.,-40.)-- (30.,-52.);
\draw [->,line width=1pt] (4.,-22.) -- (6.,-22.);
\draw [line width=0.5pt,dotted] (4.,-34.)-- (4.,-52.);
\draw [line width=0.5pt,dotted] (8.,-44.)-- (8.,-52.);
\draw [->,line width=1pt] (38.,-52.) -- (46.,-52.);
\draw [->,line width=1pt] (60.,-52.) -- (68.,-52.);
\draw [line width=1.2pt,color=qqqqff] (46.,-34.)-- (54.,-44.);
\draw [line width=1.2pt,color=qqqqff] (50.221330586925426,-38.60941661130807) -- (49.570358272438845,-39.13019446289731);
\draw [line width=1.2pt,color=qqqqff] (50.42964172756113,-38.86980553710269) -- (49.77866941307455,-39.39058338869193);
\draw [line width=1.2pt,color=qqqqff] (54.,-44.)-- (46.,-44.);
\draw [line width=1.2pt,color=qqqqff] (50.,-44.41682566047866) -- (50.,-43.583174339521335);
\draw [line width=1.2pt,color=qqqqff] (46.,-44.)-- (38.,-34.);
\draw [line width=1.2pt,color=qqqqff] (41.77866941307457,-39.39058338869193) -- (42.42964172756115,-38.86980553710269);
\draw [line width=1.2pt,color=qqqqff] (41.57035827243886,-39.13019446289731) -- (42.22133058692544,-38.60941661130807);
\draw [line width=1.2pt,color=qqqqff] (38.,-34.)-- (46.,-34.);
\draw [line width=1.2pt,color=qqzzqq] (46.,-34.)-- (54.,-40.);
\draw [line width=1.2pt,color=qqzzqq] (50.11671118493402,-36.566501313102194) -- (49.61652039235961,-37.233422369868045);
\draw [line width=1.2pt,color=qqzzqq] (50.38347960764037,-36.76657763013196) -- (49.88328881506596,-37.43349868689781);
\draw [line width=1.2pt,color=qqzzqq] (54.,-40.)-- (68.,-40.);
\draw [line width=1.2pt,color=qqzzqq] (61.,-39.58317433952134) -- (61.,-40.416825660478665);
\draw [line width=1.2pt,color=qqzzqq] (68.,-40.)-- (60.,-34.);
\draw [line width=1.2pt,color=qqzzqq] (63.88328881506597,-37.43349868689781) -- (64.38347960764038,-36.76657763013196);
\draw [line width=1.2pt,color=qqzzqq] (63.61652039235962,-37.233422369868045) -- (64.11671118493402,-36.566501313102194);
\draw [line width=1.2pt,color=qqzzqq] (60.,-34.)-- (46.,-34.);
\draw [line width=1.2pt,color=zzttqq] (44.,-30.)-- (38.,-34.);
\draw [line width=1.2pt,color=zzttqq] (41.23121327565923,-32.34681991348884) -- (40.768786724340764,-31.653180086511167);
\draw [line width=1.2pt,color=zzttqq] (60.,-34.)-- (66.,-30.);
\draw [line width=1.2pt,color=zzttqq] (62.76878672434076,-31.653180086511167) -- (63.23121327565923,-32.34681991348884);
\draw [line width=1.2pt,color=zzttqq] (66.,-30.)-- (44.,-30.);
\draw [line width=1.2pt,dash pattern=on 17pt off 17pt] (52.,-30.)-- (46.,-34.);
\draw [line width=1.2pt,color=qqqqff] (44.,-30.)-- (52.,-30.);
\draw [line width=1.2pt,color=qqqqff] (48.,-29.58317433952134) -- (48.,-30.41682566047866);
\draw [line width=1.2pt,color=qqzzqq] (52.,-30.)-- (66.,-30.);
\draw [line width=1.2pt,color=qqzzqq] (59.,-29.58317433952134) -- (59.,-30.41682566047866);
\draw [line width=0.5pt,dotted] (60.,-34.)-- (60.,-52.);
\draw [line width=0.5pt,dotted] (38.,-34.)-- (38.,-22.);
\draw [line width=0.5pt,dotted] (44.,-30.)-- (44.,-22.);
\draw [line width=0.5pt,dotted] (68.,-40.)-- (68.,-52.);
\draw [->,line width=1pt] (38.,-22.) -- (44.,-22.);
\draw [line width=0.5pt,dotted] (38.,-34.)-- (38.,-52.);
\draw [line width=0.5pt,dotted] (46.,-44.)-- (46.,-52.);

\draw (6,-54) node  {\Huge $t_1$};
\draw (28,-54) node  {\Huge $t_2$};
\draw (5,-20) node  {\Huge $t_3$};

\draw (42,-54) node  {\Huge $t_1+t$};
\draw (64,-54) node  {\Huge $t_2+t$};
\draw (41,-20) node  {\Huge $t_3+t$};

\begin{scriptsize}
\draw [fill=ffffff] (4.,-34.)  circle (5pt);
\draw [fill=ffffff] (12.,-34.)  circle (5pt);
\draw [fill=black] (16.,-44.)  circle (5pt);
\draw [fill=black] (14.,-30.)  circle (5pt);
\draw [fill=black] (28.,-30.)  circle (5pt);
\draw [fill=black] (8.,-44.)  circle (5pt);
\draw [fill=ffffff] (26.,-34.)  circle (5pt);
\draw [fill=black] (16.,-40.)  circle (5pt);
\draw [fill=black] (30.,-40.)  circle (5pt);
\draw [fill=ttqqqq] (6.,-30.)  circle (5pt);
\draw [fill=ffffff] (38.,-34.)  circle (5pt);
\draw [fill=ffffff] (46.,-34.)  circle (5pt);
\draw [fill=black] (54.,-44.)  circle (5pt);
\draw [fill=black] (52.,-30.)  circle (5pt);
\draw [fill=black] (66.,-30.)  circle (5pt);
\draw [fill=black] (46.,-44.)  circle (5pt);
\draw [fill=ffffff] (60.,-34.)  circle (5pt);
\draw [fill=black] (54.,-40.)  circle (5pt);
\draw [fill=black] (68.,-40.)  circle (5pt);
\draw [fill=ttqqqq] (44.,-30.)  circle (5pt);
\end{scriptsize}
\end{tikzpicture}
  \caption{Varying relative periods of $(X, \omega)$ by a horizontal vector $t\in\C$.}
    \label{fig3cylvar}
\end{figure}
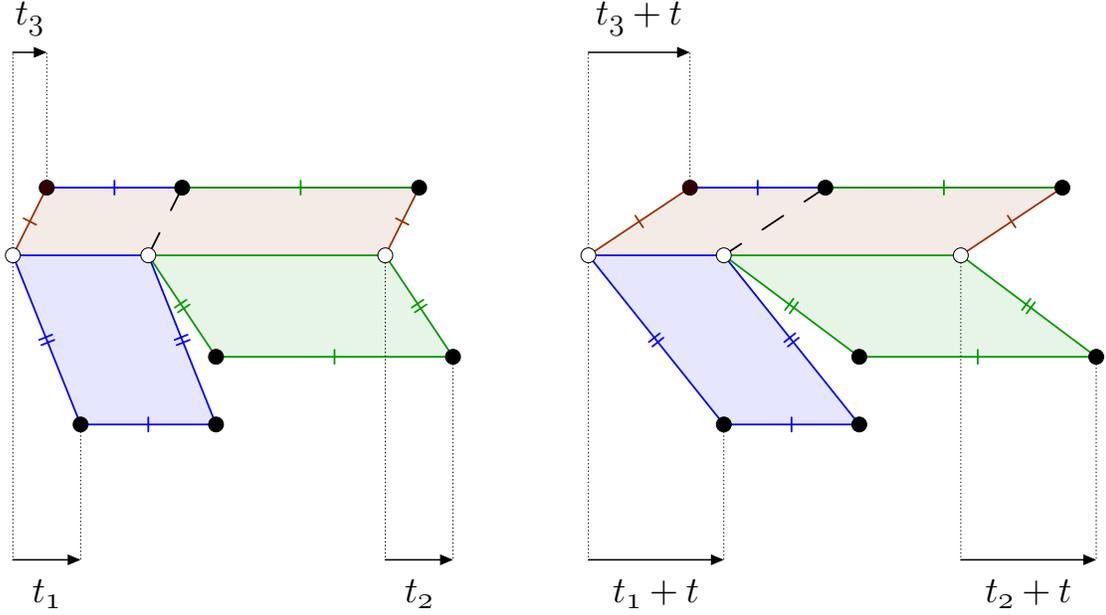
\end{comment6}

\begin{proof}[Proof of Theorem~\ref{thmdecomposition}]
We first give a proof for prime $d$, in particular we have $\gcd(w_1,w_2)=1$. Generic covers correspond to points in the interior of the cylinders and $(i)$ follows from Proposition~\ref{encoding} by setting $s_i = h_i + h_3$. It is clear that $h_3$ can vary between $0$ and $\min(s_1, s_2)$ while leaving a cover generic, hence $H_{\mathcal C} = \min(s_1,s_2)$. A pair of nearby points of $\A_{d^2}$ that differ by a small horizontal vector $t \in \C$ corresponds to a pair of Abelian differentials whose relative periods differ by $t$. Varying the relative periods of $(X,\omega)$ by $t \in \R$ changes the twists $(t_1, t_2, t_3)$ to $(t_1 + t, t_2 + t , t_3 + t) \in \R / w_1\R \times \R / w_2\R \times \R / (w_1+w_2)\R$ and leaves all other parameters fixed (see Figure~\ref{fig3cylvar}). The points on the vertical edges of the squares in the tiling of $\A_{d^2}$ have integral twist. Varying the relative periods of such points by $t=1$ moves us to the next square in the horizontal direction. The element $(1,1,1) \in \ZZ{w_1} \times \ZZ{w_2} \times \ZZ{(w_1+w_2)}$ generates the whole group and has order $w_1 w_2 (w_1 + w_2)$, which shows $(ii)$ for prime $d$.

When $d$ is not prime, the area and primitivity conditions still follow from Proposition~\ref{encoding}. The only difference from the case of prime $d$ is that the element $(1,1,1) \in \ZZ{w_1} \times \ZZ{w_2} \times \ZZ{(w_1+w_2)}$ does not generate the whole group, when $\gcd(w_1,w_2) \ne 1$. The order of $(1,1,1)$ is $\lcm(w_1,w_2,w_1+w_2)$, therefore $W_{\CC} = \lcm(w_1, w_2, w_1 + w_2)$. The element $(1,1,1)$ acts $\ZZ{w_1} \times \ZZ{w_2} \times \ZZ{(w_1+w_2)}$ and every equivalence class under this action has a unique representative $(T_1,T_2,0)$, such that $0 \le T_1,T_2 < \gcd(w_1, w_2)$, which justifies the twist condition and completes the enumeration of horizontal cylinders of $\A_{d^2}$.
\end{proof}

\begin{remark} \label{remarkordering}
If $d$ is prime, then we have $\gcd(w_1,w_2)=1$. Therefore the set $Cyl(\A_{d^2})$ is in bijection with the set of unordered pairs $\{ (w_1, s_1), (w_2, s_2) \} \in \Sym^2\N^2$, satisfying:
\begin{itemize}
\item $s_1 w_1 + s_2 w_2 = d$; and
\item $\gcd(s_1, s_2) = 1$.
\end{itemize}
The integers $w_1,s_1,w_2,s_2$ can be ordered in such a way that $w_1 < w_2$ if $s_1=s_2=1$, and $s_1<s_2$ otherwise. We will denote the corresponding horizontal cylinder by an ordered 4-tuple $(w_1,s_1,w_2,s_2)$.
\end{remark}

\noindent
{\bf Euclidean coordinates on $\A_{d^2}$.}
Note that the interiors $\CC^\circ$ of the cylinders $\CC \in Cyl(\A_{d^2})$ consist of generic $(X,\omega) \in \A_{d^2}$. Let $(X,\omega) \in \CC^\circ$ have cylinder coordinates $(w_1,s_1,w_2,s_2,t_1,t_2,t_3,h)$. Then its {\em Euclidean coordinates} are $(x,y) \in \R / W_{\CC} \R \times (0, H_{\CC})$ such that:
\begin{equation} \label{eq:euclcoord}
t_1 = x \% w_1, t_2 = x \% w_2, t_3 = x \% (w_1+w_2), y = h,
\end{equation}
where $a\%b = b \cdot \left\{ \frac ab \right\}$, or the distance from $a$ to the largest integer multiple of $b$ that does not exceed $a$. 

Note that these coordinates give isometry between the union of the interiors of all cylinders of $\A_{d^2}$ and: 
$$\bigsqcup_{\CC \in Cyl(\A_{d^2})}\R / W_{\CC} \R \times (0, H_{\CC})$$
in the flat metric $|q|$. For each cylinder $\CC$ its Euclidean coordinates also extend to its boundary by continuity, however this extension is not globally injective on $\CC$.

\noindent
{\bf Construction of the square-tiling of $\A_{d^2}$.} We now describe an algorithm that can be used to construct the square-tiling of $\A_{d^2}$ for any $d>1$. In Appendix~\ref{secpagoda}, for any prime $d$, we give a more uniform and efficient way of constructing the square-tiling of $\A_{d^2}$ that reveals the {\em pagoda structure} of $\A_{d^2} \cong X(d)$.

The center of each square in the square-tiling of $\A_{d^2}$ corresponds to a generic Abelian differential $(X,\omega) \in \A_{d^2}$ with cylinder coordinates $(w_1, s_1, w_2, s_2, t_1, t_2, t_3, h)$, such that $t_i - \frac12, h- \frac12 \in \Z$ for all $i=1,2,3$. We draw a square for each equivalence class of such an admissible vector. It remains to describe the identifications of the side of these squares.

In the discussion above we have obtained that the square-tiling of $\A_{d^2}$ is given by horizontal cylinders $\CC \in Cyl(\A_{d^2})$ with widths $W_{\CC}$ and heights $H_{\CC}$ (see Theorem~\ref{thmdecomposition}). We first interpret that result in terms of the identification between the side of the squares.

The vertical sides are identified in the following way. The right side of the square with the center at $(w_1, s_1, w_2, s_2, t_1, t_2, t_3, h)$ is identified to the left side of the square with the center at $(w_1, s_1, w_2, s_2, (t_1+1)\%w_1, (t_2+1)\%w_2, (t_3+1)\%(w_1+w_2), h)$ by a parallel translation.

The horizontal sides within a single horizontal cylinder of $\A_{d^2}$ are identified as follows. The bottom side of the square with the center at $(w_1, s_1, w_2, s_2, t_1, t_2, t_3, h)$ with $1/2 \le h < \min(s_1,s_2) - 1/2$  is identified to the top side of the square with the center at $(w_1, s_1, w_2, s_2, t_1, t_2, t_3, h+1)$ by a parallel translation.

It remains to understand the identifications of the boundaries of the horizontal cylinders of $\A_{d^2}$. In other words we need to describe the identifications among (1) the bottom sides of the squares with the centers at $(w_1, s_1, w_2, s_2, t_1, t_2, t_3, \min(s_1,s_2)-1/2)$ and (2) top sides of the squares with the centers at $(w_1, s_1, w_2, s_2, t_1, t_2, t_3, 1/2)$. For this we find limits of the Abelian differentials $(X_h,\omega_h)$ given by $(w_1, s_1, w_2, s_2, t_1, t_2, t_3, h)$ as $h \to \min(s_1,s_2)$ in the case (1), and $h \to 0$ in the case (2).
Informally we have to look at the 3-cylinder decompositions (see Figure~\ref{fig3cyl}) and vertically zip down the white singularity in the case (1), and zip it up in the case (2). For an example of zipping see Figure~\ref{zipupdown}. We then obtain a collection of non-generic Abelian differentials in $\A_{d^2}$ that correspond to the centers of the edges on the boundaries of horizontal cylinders of $\A_{d^2}$. This collection splits into pairs of equal Abelian differentials. Each pair determines the edges that must be identified. If the Abelian differentials in a pair were obtained by zipping in different directions, the corresponding edges are identified by a parallel translation. If the Abelian differentials in a pair were obtained by zipping in the same direction, the corresponding edges are identified by a rotation by $\pi$.

\begin{comment6}
\begin{figure}[H]
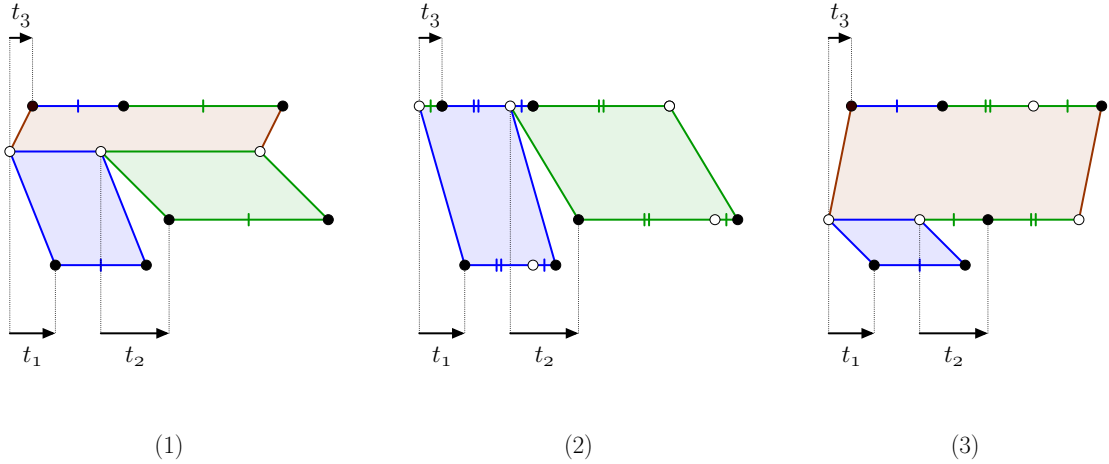
 
\centering
  \includestandalone[width=\textwidth]{tikztobdry1}
  \caption{(1) Zipping up and (3) zipping down a singularity in (2) the 3-cylinder decomposition of a generic $(X,\omega) \in \A_{d^2}$.}
  \label{zipupdown}
\end{figure}
\end{comment6}

\noindent
{\bf Examples.} One can verify that following these instructions one obtains the square-tilings of $\A_{4}$, $\A_{9}$, $\A_{16}$ and $\A_{25}$ as in Figures~\ref{figX(2)intro},~\ref{figX(3)intro},~\ref{figX(4)intro} and~\ref{figX(5)intro}. 

\begin{remark} \label{orientation}
Note that the horizontal cylinders on all our pictures are oriented in such a way that Euclidean coordinate $x$ increases left-to-right and Euclidean coordinate $y$ increases top-to-bottom.
\end{remark}

We describe the square-tiling of $\A_{4}$ in detail in \S\ref{secproof2}. For the association between the vertices of the squares of $\A_9$ (see Figure~\ref{figX(3)}) and the corresponding square-tiled surfaces see Figure~\ref{figX(3)corners}. From this point on the identifications of the unlabeled vertical sides will always be determined by parallel translations by horizontal vectors. The picture of $\A_{9}$ can also be found in \cite{Sch05}.

\begin{comment6}
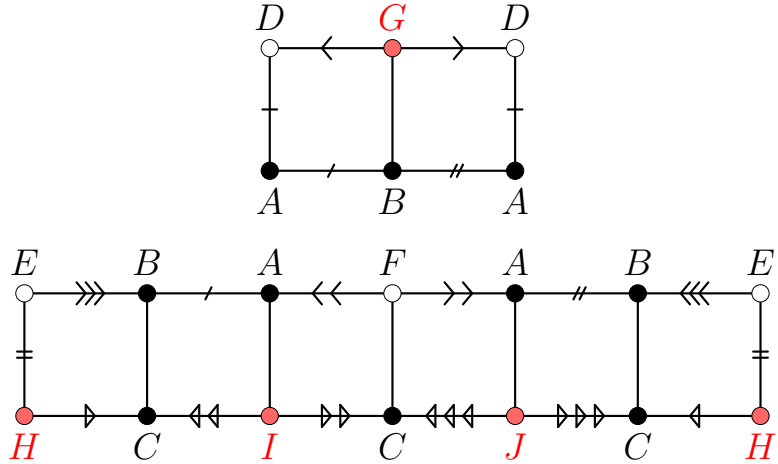
\begin{figure}[H] 
\centering
\definecolor{ffqqqq}{rgb}{1.,0.4,0.4}
\definecolor{ffffff}{rgb}{1.,1.,1.}
\begin{tikzpicture}[line cap=round,line join=round,>=triangle 45,x=1.0cm,y=1.0cm]
\clip(-8,-3.) rectangle (8,5.);

\draw [line width=1.pt] (-0.95,2.1) -- (-1.05,1.9);

\draw [line width=1.pt] (-0.95-2,2.1-2) -- (-1.05-2,1.9-2);

\draw [line width=1.pt] (1.15,2.1) -- (1.05,1.9);
\draw [line width=1.pt] (1.05,2.1) -- (1.05-0.1,1.9);

\draw [line width=1.pt] (1.15+2,2.1-2) -- (1.05+2,1.9-2);
\draw [line width=1.pt] (1.05+2,2.1-2) -- (1.05-0.1+2,1.9-2);

\draw [line width=1.pt] (-2.,2.)-- (-2.,4.);
\draw [line width=1.pt] (-2.12,3.) -- (-1.88,3.);
\draw [line width=1.pt] (0.,2.)-- (-2.,2.);
\draw [line width=1.pt] (0.,2.)-- (0.,4.);
\draw [line width=1.pt] (0.,4.)-- (2.,4.);
\draw [line width=1.pt] (2.,4.)-- (2.,2.);
\draw [line width=1.pt] (2.12,3.) -- (1.88,3.);
\draw [line width=1.pt] (2.,2.)-- (0.,2.);
\draw [line width=1.pt] (0.,4.)-- (-2.,4.);

\draw [line width=1.pt] (1.15,4.) -- (1.,3.82);
\draw [line width=1.pt] (1.15,4.) -- (1.,4.18);

\draw [line width=1.pt] (-1.15,4.) -- (-1.,4.18);
\draw [line width=1.pt] (-1.15,4.) -- (-1.,3.82);

\draw [line width=1.pt] (-6.,-2.)-- (-6.,0.);
\draw [line width=1.pt] (-6.12,-1.05) -- (-5.88,-1.05);
\draw [line width=1.pt] (-6.12,-0.95) -- (-5.88,-0.95);
\draw [line width=1.pt] (-6.,0.)-- (-4.,0.);
\draw [line width=1.pt] (-4.85,0.) -- (-5.,-0.18);
\draw [line width=1.pt] (-4.85,0.) -- (-5.,0.18);
\draw [line width=1.pt] (-4.,-2.)-- (-6.,-2.);
\draw [line width=1.pt] (-4.,-2.)-- (-4.,0.);
\draw [line width=1.pt] (-4.,0.)-- (-2.,0.);
\draw [line width=1.pt] (-2.,-2.)-- (-4.,-2.);
\draw [line width=1.pt] (-3.3,-2.) -- (-3.15,-1.82);
\draw [line width=1.pt] (-3.3,-2.) -- (-3.15,-2.18);
\draw [line width=1.pt] (-3.15,-1.82) -- (-3.15,-2.18);
\draw [line width=1.pt] (-3.,-2.) -- (-2.85,-1.82);
\draw [line width=1.pt] (-3.,-2.) -- (-2.85,-2.18);
\draw [line width=1.pt] (-2.85,-1.82) -- (-2.85,-2.18);
\draw [line width=1.pt] (-2.,-2.)-- (-2.,0.);
\draw [line width=1.pt] (0.,-2.)-- (-2.,-2.);
\draw [line width=1.pt] (0.,-2.)-- (0.,0.);
\draw [line width=1.pt] (0.,0.)-- (2.,0.);
\draw [line width=1.pt] (2.,0.)-- (2.,-2.);
\draw [line width=1.pt] (2.,-2.)-- (0.,-2.);

\draw [line width=1.pt] (-3.3+8,0.) -- (-3.15+8,-1.82+2);
\draw [line width=1.pt] (-3.3+8,0.) -- (-3.15+8,-0.18);
\draw [line width=1.pt] (-3.+8,0.) -- (-2.85+8,-1.82+2);
\draw [line width=1.pt] (-3.+8,0.) -- (-2.85+8,-0.18);

\draw [line width=1.pt] (-0.7-4,0.) -- (-0.85-4,-0.18);
\draw [line width=1.pt] (-0.7-4,-0.) -- (-0.85-4,-1.82+2);
\draw [line width=1.pt] (-1.-4,-0.) -- (-1.15-4,-0.18);
\draw [line width=1.pt] (-1.-4,-0.) -- (-1.15-4,-1.82+2);

\draw [line width=1.pt] (0.85,-2.) -- (1.,-1.82);
\draw [line width=1.pt] (0.85,-2.) -- (1.,-2.18);
\draw [line width=1.pt] (1.,-1.82) -- (1.,-2.18);
\draw [line width=1.pt] (1.15,-2.) -- (1.3,-1.82);
\draw [line width=1.pt] (1.15,-2.) -- (1.3,-2.18);
\draw [line width=1.pt] (1.3,-1.82) -- (1.3,-2.18);
\draw [line width=1.pt] (0.55,-2.) -- (0.7,-1.82);
\draw [line width=1.pt] (0.55,-2.) -- (0.7,-2.18);
\draw [line width=1.pt] (0.7,-1.82) -- (0.7,-2.18);

\draw [line width=1.pt] (2.,0.)-- (4.,0.);
\draw [line width=1.pt] (4.,0.)-- (4.,-2.);
\draw [line width=1.pt] (4.,-2.)-- (2.,-2.);
\draw [line width=1.pt] (6.,0.)-- (6.,-2.);
\draw [line width=1.pt] (6.12,-0.95) -- (5.88,-0.95);
\draw [line width=1.pt] (6.12,-1.05) -- (5.88,-1.05);
\draw [line width=1.pt] (6.,-2.)-- (4.,-2.);

\draw [line width=1.pt] (4.85,-2.) -- (5.,-1.82);
\draw [line width=1.pt] (4.85,-2.) -- (5.,-2.18);
\draw [line width=1.pt] (5.,-1.82) -- (5.,-2.18);
\draw [line width=1.pt] (-4.85,-2.) -- (-5.,-2.18);
\draw [line width=1.pt] (-4.85,-2.) -- (-5.,-1.82);
\draw [line width=1.pt] (-5.,-2.18) -- (-5.,-1.82);

\draw [line width=1.pt] (-0.7,-2.) -- (-0.85,-2.18);
\draw [line width=1.pt] (-0.7,-2.) -- (-0.85,-1.82);
\draw [line width=1.pt] (-0.85,-2.18) -- (-0.85,-1.82);
\draw [line width=1.pt] (-1.,-2.) -- (-1.15,-2.18);
\draw [line width=1.pt] (-1.,-2.) -- (-1.15,-1.82);
\draw [line width=1.pt] (-1.15,-2.18) -- (-1.15,-1.82);

\draw [line width=1.pt] (3.15,-2.) -- (3.,-2.18);
\draw [line width=1.pt] (3.15,-2.) -- (3.,-1.82);
\draw [line width=1.pt] (3.,-2.18) -- (3.,-1.82);
\draw [line width=1.pt] (2.85,-2.) -- (2.7,-2.18);
\draw [line width=1.pt] (2.85,-2.) -- (2.7,-1.82);
\draw [line width=1.pt] (2.7,-2.18) -- (2.7,-1.82);
\draw [line width=1.pt] (3.45,-2.) -- (3.3,-2.18);
\draw [line width=1.pt] (3.45,-2.) -- (3.3,-1.82);
\draw [line width=1.pt] (3.3,-2.18) -- (3.3,-1.82);

\draw [line width=1.pt] (0.,0.)-- (-2.,0.);

\draw [line width=1.pt] (1.3,0.) -- (1.15,-0.18);
\draw [line width=1.pt] (1.3,0.) -- (1.15,0.18);
\draw [line width=1.pt] (1.,0.) -- (0.85,-0.18);
\draw [line width=1.pt] (1.,0.) -- (0.85,0.18);

\draw [line width=1.pt] (-1.3,0.) -- (-1.15,0.18);
\draw [line width=1.pt] (-1.3,0.) -- (-1.15,-0.18);
\draw [line width=1.pt] (-1.,0.) -- (-0.85,0.18);
\draw [line width=1.pt] (-1.,0.) -- (-0.85,-0.18);

\draw [line width=1.pt] (6.,0.)-- (4.,0.);
\draw [line width=1.pt] (4.85,0.) -- (5.,0.18);
\draw [line width=1.pt] (4.85,0.) -- (5.,-0.18);


%


%
%
%

\draw (-2,4.5) node { \Large $D$};
\draw (-0.,4.5) node {\color{red}  \Large $G$};
\draw (2,4.5) node { \Large $D$};

\draw (-2,1.5) node { \Large $A$};
\draw (-0,1.5) node { \Large $B$};
\draw (2,1.5) node { \Large $A$};

\draw (-6,0.5) node { \Large $E$};
\draw (-4,0.5) node { \Large $B$};
\draw (-2,0.5) node { \Large $A$};
\draw (-0,0.5) node { \Large $F$};
\draw (2,0.5) node { \Large $A$};
\draw (4,0.5) node { \Large $B$};
\draw (6,0.5) node { \Large $E$};

\draw (-6,-2.5) node {\color{red}  \Large $H$};
\draw (-4,-2.5) node { \Large $C$};
\draw (-2,-2.5) node {\color{red}  \Large $I$};
\draw (-0,-2.5) node { \Large $C$};
\draw (2,-2.5) node {\color{red}  \Large $J$};
\draw (4,-2.5) node { \Large $C$};
\draw (6,-2.5) node {\color{red}  \Large $H$};

\begin{scriptsize}
\draw [fill=black] (-2.,2.) circle (4pt);
\draw [fill=ffffff] (-2.,4.) circle (4pt);
\draw [fill=ffqqqq] (0.,4.) circle (4pt);
\draw [fill=black] (0.,2.) circle (4pt);
\draw [fill=ffffff] (2.,4.) circle (4pt);
\draw [fill=black] (2.,2.) circle (4pt);
\draw [fill=ffqqqq] (-6.,-2.) circle (4pt);
\draw [fill=ffffff] (-6.,0.) circle (4pt);
\draw [fill=black] (-4.,0.) circle (4pt);
\draw [fill=black] (-4.,-2.) circle (4pt);
\draw [fill=black] (-2.,0.) circle (4pt);
\draw [fill=ffqqqq] (-2.,-2.) circle (4pt);
\draw [fill=ffffff] (0.,0.) circle (4pt);
\draw [fill=black] (0.,-2.) circle (4pt);
\draw [fill=black] (2.,0.) circle (4pt);
\draw [fill=ffqqqq] (2.,-2.) circle (4pt);
\draw [fill=black] (4.,0.) circle (4pt);
\draw [fill=black] (4.,-2.) circle (4pt);
\draw [fill=ffffff] (6.,0.) circle (4pt);
\draw [fill=ffqqqq] (6.,-2.) circle (4pt);
\end{scriptsize}
\end{tikzpicture}
\caption{The square-tiling of $X(3) \cong \A_{9}$. The sides identified by rotation by $\pi$ are labeled with arrows, the sides labeled with numbers are identified by parallel translations. The points $A, B, C$ are the zeroes of $q$, the points $D,E,F$ are the non-cusp poles of $q$ and the points $G,H,I,J$ are the cusp poles of $q$.
}
\label{figX(3)}
\end{figure}
\end{comment6}

\begin{comment6}
\begin{figure}[H]
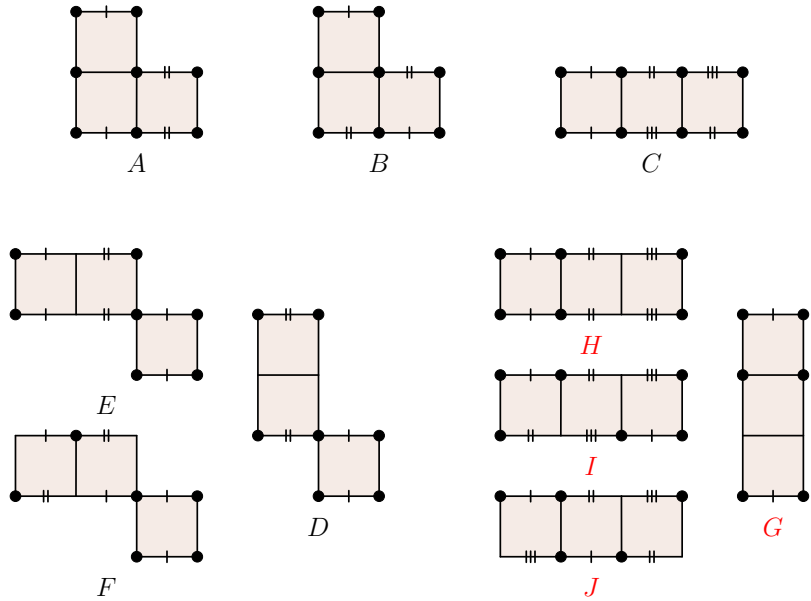
 
\centering
\includestandalone[width=0.7\textwidth]{tikzX3corners}
\vspace{-10pt}
\caption{The square-tilings corresponding to the vertices of $\A_9$: $A,B,C$ are square-tiled surfaces in $\Omega\M_2(2)$; $D,E,F$ are square-tiled surfaces with separating nodes; $G,H,I,J$ are square-tiled surfaces with non-separating nodes. The vertical sides are identified by horizontal parallel translations.}
\label{figX(3)corners}
\end{figure}
\end{comment6}


\section{Lighthouses and eaves} \label{seclighteave}

In this section we describe a class of horizontal cylinders of $\A_{d^2}$, called {\em lighthouses}, for each $d>1$, and another class of horizontal cylinders of $\A_{d^2}$, called {\em eaves}, for each prime $d$. We give some properties of the $\SLZ$ action on these cylinders, that will be used in the proof of the main result. We will show:

\begin{theorem} \label{thmrotation} 
For any prime $d$ and $1 \le k \le (d-1)/2$, the matrix
$R = \begin{pmatrix} 0 & -1 \\ 1 & 0 \end{pmatrix} \in \SLZ$ acts on the lighthouse $\LL_k \subset \A_{d^2}$ as follows:
\begin{itemize}
\item the left $k$ squares of the lighthouse $\LL_k$ are rotated by $-\pi/2$ and sent to the right $k$ squares of the eave $\EE_k \subset \A_{d^2}$; and

\item the right $k$ squares of the lighthouse $\LL_k \subset \A_{d^2}$ are rotated by $\pi/2$ and sent to the left $k$ squares of the eave $\EE_k \subset \A_{d^2}$.
\end{itemize}
\end{theorem}

\begin{theorem} \label{thmunipotent}
For any prime $d$ and $1 \le k \le (d-1)/2$, the matrix
$S = \begin{pmatrix} 1 & -1 \\ 0 & 1 \end{pmatrix} \in \SLZ$ acts on the eave $\EE_k \subset \A_{d^2}$ as follows. Let $\left(x,y\right)$ be the Euclidean coordinates of a point in $\EE_k$, then:
\begin{equation*}
S : \left( x, y \right) \mapsto \left( x + y + T_k \cdot d  , y \right) \in \EE_k,
\end{equation*}
where $0 \le T_k < k(d-k)$ is uniquely determined by $T_k\cdot d \equiv -1 \mod k(d-k)$.
\end{theorem}

\noindent
{\bf Lighthouses.} The cylinder $\CC = \{(w_1,s_1,T_1),(w_2,s_2,T_2)\}$ is called a {\em lighthouse} if $w_1 = w_2 = 1$. Let $\phi(d)$ be an Euler's totient function, which returns the number of integers from $1$ to $d$ that are coprime with $d$. For every $d>1$, the absolute period leaf $\A_{d^2}$ has exactly $\phi(d)/2$ lighthouses. They will be denoted by:
$$
\LL_k = \{(1, k, 0), (1, d-k, 0)\},
$$
where $1 \le k \le (d-1)/2$ and $\gcd(k,d)=1$. In this case $T_1=T_2=0$ and the pair of triples is ordered, hence we can write $\LL_k = (1, k, 1, d-k)$ instead.

The lighthouse $\LL_k = (1, k, 1, d-k)$ is a horizontal cylinder of height $k$ and circumference $2$ (see Figure~\ref{figlighthouse}). Its top boundary consists of a cusp pole, non-cusp pole and two edges between them, which are identified by a rotation by $\pi$. Its bottom boundary consists of two zeroes and two edges between them.

\begin{comment7}
\begin{figure}[H] 
\centering
\definecolor{ffqqqq}{rgb}{1.,0.4,0.4}
\definecolor{ffffff}{rgb}{1.,1.,1.}
\begin{tikzpicture}[line cap=round,line join=round,>=triangle 45,x=0.52cm,y=0.52cm]
\clip(-2,1.5) rectangle (6.5,9.5);
\draw [line width=0.8pt] (1.,8.)-- (1.,3.);
\draw [line width=0.8pt] (1.,3.)-- (3.,3.);
\draw [line width=0.8pt] (3.,3.)-- (3.,8.);
\draw [line width=0.8pt] (3.,8.)-- (1.,8.);

\draw[<->] (-1,3) -- (-1,8) ;
\draw (-1.5, 5.5) node {$k$};

\draw [line width=0.6pt] (3.15-0.5,8.) -- (3.-0.5,7.82);
\draw [line width=0.6pt] (3.15-0.5,8.) -- (3.-0.5,8.18);

\draw [line width=0.6pt] (-1.15+2.5,8.) -- (-1.+2.5,8.18);
\draw [line width=0.6pt] (-1.15+2.5,8.) -- (-1.+2.5,7.82);
%

\draw [line width=0.8pt] (2.,8.)-- (2.,3.);
\draw [line width=0.8pt] (1.,7.)-- (3.,7.);
\draw [line width=0.8pt] (1.,6.)-- (3.,6.);
\draw [line width=0.8pt] (1.,5.)-- (3.,5.);
\draw [line width=0.8pt] (1.,4.)-- (3.,4.);

\draw (1, 8.7) node {$A$};
\draw (2, 8.7) node {$\color{red} B$};
\draw (3, 8.7) node {$A$};
\draw (1, 2.3) node {$C$};
\draw (2, 2.3) node {$D$};
\draw (3, 2.3) node {$C$};
\begin{scriptsize}
\draw [fill=ffffff] (1.,8.) circle (2.5pt);
\draw [fill=black] (1.,3.) circle (2.5pt);
\draw [fill=black] (3.,3.) circle (2.5pt);
\draw [fill=ffffff] (3.,8.) circle (2.5pt);
\draw [fill=ffqqqq] (2.,8.) circle (2.5pt);
\draw [fill=black] (2.,3.) circle (2.5pt);
\end{scriptsize}
\end{tikzpicture}
\caption{The lighthouse $\LL_k = (1, k, 1, d-k) \subset \A_{d^2}$.
Red points are the cusp poles of $q$, white points are the non-cusp poles of $q$, black points are the zeroes of $q$. In this specific example $k=5$ and $d=11$.
}
\label{figlighthouse}
\end{figure}
\end{comment7}

To see that, note that the vertices $A, B, C, D$ on Figure~\ref{figlighthouse} correspond to the square-tilings on Figure~\ref{figlighthousevert}. Since $A$ and $B$ are simple poles, the total angle around each of them has to be $\pi$. This forces the adjacent edges to be ``folded'',~i.e. identified by $\pi$ rotation. In general, this applies to any horizontal line segment between a pole and any other singularity.

\begin{comment7}
\begin{figure}[H] 
\centering
\includestandalone[width=0.8\textwidth]{tikzlighthousevert}
\caption{Square-tilings corresponding to the vertices on the boundaries of the lighthouse $\LL_k \subset \A_{d^2}$. In this specific example $k=5$ and $d=11$.
}
\label{figlighthousevert}
\end{figure}
\end{comment7}

\noindent
{\bf Eaves.} 
Let $d$ be any prime number. Recall from Remark~\ref{remarkordering} that in this case every horizontal cylinder $\CC \in Cyl(\A_{d^2})$ corresponds to a vector $(w_1,s_1,w_2,s_2)$, where $s_1 < s_2$ if $w_1=w_2=1$, and $w_1<w_2$ otherwise. The cylinder $\CC = (w_1,s_1,w_2,s_2)$ is called an {\em eave} if $s_1 = s_2 = 1$. For every prime $d$, the absolute period leaf $\A_{d^2}$ has exactly $\phi(d)/2 = (d-1)/2$ eaves. They will be denoted by:
$$
\EE_k = (k,1,d-k,1),
$$
where $1 \le k \le (d-1)/2$. 

The eave $\EE_k = (k,1,d-k,1)$ is a horizontal cylinder of height 1 and circumference $k(d-k)d$ (see Figure~\ref{figeave}). Its bottom boundary has $d$ cusp poles at every point with the Euclidean coordinates $(i \cdot k(d-k),1)$, where $0 \le i < d$. Its top boundary has $k(d-k)$ non-cusp poles at every point with the Euclidean coordinates $(j\cdot d,0)$, where $0 \le j < k(d-k)$.

\begin{figure}[H] 
\centering
\definecolor{ffqqqq}{rgb}{1.,0.4,0.4}
\definecolor{ffffff}{rgb}{1.,1.,1.}
\begin{tikzpicture}[line cap=round,line join=round,>=triangle 45,x=0.52cm,y=0.52cm]
\clip(-5.4,-4) rectangle (25.6,3);
\draw [line width=0.8pt] (-5.,0.)-- (25.,0.);
\draw [line width=0.8pt] (25.,0.)-- (25.,-1.);
\draw [line width=0.8pt] (25.,-1.)-- (-5.,-1.);
\draw [line width=0.8pt] (-5.,-1.)-- (-5.,0.);
\draw [line width=0.8pt] (-4.,0.)-- (-4.,-1.);
\draw [line width=0.8pt] (-3.,0.)-- (-3.,-1.);
\draw [line width=0.8pt] (-2.,0.)-- (-2.,-1.);
\draw [line width=0.8pt] (-1.,0.)-- (-1.,-1.);
\draw [line width=0.8pt] (0.,0.)-- (0.,-1.);
\draw [line width=0.8pt] (1.,0.)-- (1.,-1.);
\draw [line width=0.8pt] (2.,0.)-- (2.,-1.);
\draw [line width=0.8pt] (3.,0.)-- (3.,-1.);
\draw [line width=0.8pt] (4.,0.)-- (4.,-1.);
\draw [line width=0.8pt] (5.,0.)-- (5.,-1.);
\draw [line width=0.8pt] (6.,0.)-- (6.,-1.);
\draw [line width=0.8pt] (7.,0.)-- (7.,-1.);
\draw [line width=0.8pt] (8.,0.)-- (8.,-1.);
\draw [line width=0.8pt] (9.,0.)-- (9.,-1.);
\draw [line width=0.8pt] (10.,0.)-- (10.,-1.);
\draw [line width=0.8pt] (11.,0.)-- (11.,-1.);
\draw [line width=0.8pt] (12.,0.)-- (12.,-1.);
\draw [line width=0.8pt] (13.,0.)-- (13.,-1.);
\draw [line width=0.8pt] (14.,0.)-- (14.,-1.);
\draw [line width=0.8pt] (15.,0.)-- (15.,-1.);
\draw [line width=0.8pt] (16.,0.)-- (16.,-1.);
\draw [line width=0.8pt] (17.,0.)-- (17.,-1.);
\draw [line width=0.8pt] (18.,0.)-- (18.,-1.);
\draw [line width=0.8pt] (19.,0.)-- (19.,-1.);
\draw [line width=0.8pt] (20.,0.)-- (20.,-1.);
\draw [line width=0.8pt] (21.,0.)-- (21.,-1.);
\draw [line width=0.8pt] (22.,0.)-- (22.,-1.);
\draw [line width=0.8pt] (23.,0.)-- (23.,-1.);
\draw [line width=0.8pt] (24.,0.)-- (24.,-1.);

\draw (-5, 0.75) node {$B_1$};
\draw (0, 0.75) node {$B_2$};
\draw (5, 0.75) node {$B_3$};
\draw (10, 0.75) node {$B_4$};
\draw (15, 0.75) node {$B_5$};
\draw (20, 0.75) node {$B_6$};
\draw (25, 0.75) node {$B_1$};

\draw (-5, -1.75) node {\color{red} $A_1$};
\draw (1, -1.75) node {\color{red}$A_2$};
\draw (7, -1.75) node {\color{red}$A_3$};
\draw (13, -1.75) node {\color{red}$A_4$};
\draw (19, -1.75) node {\color{red}$A_5$};
\draw (25, -1.75) node {\color{red}$A_1$};


\draw[<->] (-5,0.75+1) -- (0,0.75+1);
\draw (-2.5,1.5+1) node  {$d$};

\draw[<->] (-5,-1.75-1) -- (1,-1.75-1);
\draw (-2,-2.5-1) node  {$k(d-k)$};

\draw[<->] (23,1.75) -- (25,1.75);
\draw (24,2.5) node  {$k$};
%
\begin{scriptsize}
\draw [fill=ffffff] (-5.,0.) circle (2.5pt);
\draw [fill=ffffff] (25.,0.) circle (2.5pt);
\draw [fill=ffqqqq] (25.,-1.) circle (2.5pt);
\draw [fill=ffqqqq] (-5.,-1.) circle (2.5pt);
\draw [fill=black] (-3.,0.) circle (2.5pt);
\draw [fill=black] (-3.,-1.) circle (2.5pt);
\draw [fill=black] (-2.,0.) circle (2.5pt);
\draw [fill=black] (-2.,-1.) circle (2.5pt);
\draw [fill=black] (-1.,-1.) circle (2.5pt);
\draw [fill=ffffff] (0.,0.) circle (2.5pt);
\draw [fill=ffqqqq] (1.,-1.) circle (2.5pt);
\draw [fill=black] (2.,0.) circle (2.5pt);
\draw [fill=black] (3.,0.) circle (2.5pt);
\draw [fill=black] (3.,-1.) circle (2.5pt);
\draw [fill=black] (4.,-1.) circle (2.5pt);
\draw [fill=ffffff] (5.,0.) circle (2.5pt);
\draw [fill=black] (5.,-1.) circle (2.5pt);
\draw [fill=black] (7.,0.) circle (2.5pt);
\draw [fill=ffqqqq] (7.,-1.) circle (2.5pt);
\draw [fill=black] (8.,0.) circle (2.5pt);
\draw [fill=black] (9.,-1.) circle (2.5pt);
\draw [fill=ffffff] (10.,0.) circle (2.5pt);
\draw [fill=black] (10.,-1.) circle (2.5pt);
\draw [fill=black] (11.,-1.) circle (2.5pt);
\draw [fill=black] (12.,0.) circle (2.5pt);
\draw [fill=black] (13.,0.) circle (2.5pt);
\draw [fill=ffqqqq] (13.,-1.) circle (2.5pt);
\draw [fill=ffffff] (15.,0.) circle (2.5pt);
\draw [fill=black] (15.,-1.) circle (2.5pt);
\draw [fill=black] (16.,-1.) circle (2.5pt);
\draw [fill=black] (17.,0.) circle (2.5pt);
\draw [fill=black] (17.,-1.) circle (2.5pt);
\draw [fill=black] (18.,0.) circle (2.5pt);
\draw [fill=ffqqqq] (19.,-1.) circle (2.5pt);
\draw [fill=ffffff] (20.,0.) circle (2.5pt);
\draw [fill=black] (21.,-1.) circle (2.5pt);
\draw [fill=black] (22.,0.) circle (2.5pt);
\draw [fill=black] (22.,-1.) circle (2.5pt);
\draw [fill=black] (23.,0.) circle (2.5pt);
\draw [fill=black] (23.,-1.) circle (2.5pt);
\end{scriptsize}
\end{tikzpicture}
\caption{The eave $\EE_k = (k, 1, d-k, 1) \subset \A_{d^2}$.
Red points are the cusp poles of $q$, white points are the non-cusp poles of $q$, black points are the zeroes of $q$. In this specific example $k=2$ and $d=5$.
}
\label{figeave}
\end{figure}

To see that, note that the vertices $A_i, B_j$ on Figure~\ref{figeave} correspond to the square-tilings on Figure~\ref{figeavevert}. Note that the twists depend on $i$ and $j$, however we do not specify them on the picture.

\begin{comment7}
\begin{figure}[H] 
\centering
\definecolor{zzttqq}{rgb}{0.6,0.2,0.}
\definecolor{cqcqcq}{rgb}{0.7529411764705882,0.7529411764705882,0.7529411764705882}
\begin{tikzpicture}[line cap=round,line join=round,>=triangle 45,x=1.0cm,y=1.0cm]

\clip(1.5,1) rectangle (17,7);

\draw (4.5,2) node {\LARGE  \color{red}$A_i$};
\draw (13.5,2) node {\LARGE $B_j$};

\draw[<->] (2,6) -- (5,6);
\draw (3.5,6.5) node {\Large $d-k$};
\draw[<->] (5,6) -- (7,6);
\draw (6,6.5) node {\Large $k$};

\draw[<->] (11,6) -- (14,6);
\draw (12.5,6.5) node {\Large $d-k$};
\draw[<->] (14,6) -- (16,6);
\draw (15,6.5) node {\Large $k$};

\fill[line width=0.8pt,color=zzttqq,fill=zzttqq,fill opacity=0.10000000149011612] (2.,3.) -- (2.,4.) -- (3.,4.) -- (3.,3.) -- cycle;
\fill[line width=0.8pt,color=zzttqq,fill=zzttqq,fill opacity=0.10000000149011612] (3.,3.) -- (3.,4.) -- (4.,4.) -- (4.,3.) -- cycle;
\fill[line width=0.8pt,color=zzttqq,fill=zzttqq,fill opacity=0.10000000149011612] (4.,3.) -- (4.,4.) -- (5.,4.) -- (5.,3.) -- cycle;
\fill[line width=0.8pt,color=zzttqq,fill=zzttqq,fill opacity=0.10000000149011612] (5.,3.) -- (5.,4.) -- (6.,4.) -- (6.,3.) -- cycle;
\fill[line width=0.8pt,color=zzttqq,fill=zzttqq,fill opacity=0.10000000149011612] (6.,3.) -- (6.,4.) -- (7.,4.) -- (7.,3.) -- cycle;
\fill[line width=0.8pt,color=zzttqq,fill=zzttqq,fill opacity=0.10000000149011612] (11.,5.) -- (11.,4.) -- (12.,4.) -- (12.,5.) -- cycle;
\fill[line width=0.8pt,color=zzttqq,fill=zzttqq,fill opacity=0.10000000149011612] (12.,4.) -- (13.,4.) -- (13.,5.) -- (12.,5.) -- cycle;
\fill[line width=0.8pt,color=zzttqq,fill=zzttqq,fill opacity=0.10000000149011612] (13.,4.) -- (14.,4.) -- (14.,5.) -- (13.,5.) -- cycle;
\fill[line width=0.8pt,color=zzttqq,fill=zzttqq,fill opacity=0.10000000149011612] (14.,3.) -- (15.,3.) -- (15.,4.) -- (14.,4.) -- cycle;
\fill[line width=0.8pt,color=zzttqq,fill=zzttqq,fill opacity=0.10000000149011612] (15.,3.) -- (16.,3.) -- (16.,4.) -- (15.,4.) -- cycle;
\draw [line width=0.8pt] (2.,3.)-- (2.,4.);
\draw [line width=0.8pt] (2.,4.)-- (3.,4.);
\draw [line width=0.8pt] (2.4128969262353372,4.108878842205828) -- (2.4128969262353372,3.8911211577941707);
\draw [line width=0.8pt] (2.5,4.108878842205828) -- (2.5,3.8911211577941707);
\draw [line width=0.8pt] (2.5871030737646628,4.108878842205828) -- (2.5871030737646628,3.8911211577941707);
\draw [line width=0.8pt] (3.,4.)-- (3.,3.);
\draw [line width=0.8pt] (3.,3.)-- (2.,3.);
\draw [line width=0.8pt] (2.5,2.8911211577941716) -- (2.5,3.1088788422058298);
\draw [line width=0.8pt] (3.,3.)-- (3.,4.);
\draw [line width=0.8pt] (3.,4.)-- (4.,4.);

\draw [line width=0.8pt] (3.63,3.8911211577941707) -- (3.63,4.108878842205828);
\draw [line width=0.8pt] (3.543551536882333,3.8911211577941707) -- (3.543551536882333,4.108878842205828);
\draw [line width=0.8pt] (3.45644846311767,3.8911211577941707) -- (3.45644846311767,4.108878842205828);
\draw [line width=0.8pt] (3.36,3.8911211577941707) -- (3.36,4.108878842205828);

\draw [line width=0.8pt] (4.,4.)-- (4.,3.);
\draw [line width=0.8pt] (4.,3.)-- (3.,3.);
\draw [line width=0.8pt] (3.5435515368823327,2.8911211577941716) -- (3.5435515368823327,3.1088788422058298);
\draw [line width=0.8pt] (3.4564484631176695,2.8911211577941716) -- (3.4564484631176695,3.1088788422058298);
\draw [line width=0.8pt] (4.,3.)-- (4.,4.);
\draw [line width=0.8pt] (4.,4.)-- (5.,4.);
\draw [line width=0.8pt] (4.412896926235338,4.108878842205828) -- (4.412896926235338,3.8911211577941707);
\draw [line width=0.8pt] (4.5,4.108878842205828) -- (4.5,3.8911211577941707);
\draw [line width=0.8pt] (4.587103073764664,4.108878842205828) -- (4.587103073764664,3.8911211577941707);
\draw [line width=0.8pt] (4.324,4.1088788422058298) -- (4.324,3.8911211577941716);
\draw [line width=0.8pt] (4.67,4.1088788422058298) -- (4.67,3.8911211577941716);

\draw [line width=0.8pt] (5.,4.)-- (5.,3.);
\draw [line width=0.8pt] (5.,3.)-- (4.,3.);

\draw [line width=0.8pt] (4.587103073764664,2.8911211577941716) -- (4.587103073764664,3.1088788422058298);
\draw [line width=0.8pt] (4.5,2.8911211577941716) -- (4.5,3.1088788422058298);
\draw [line width=0.8pt] (4.412896926235338,2.8911211577941716) -- (4.412896926235338,3.1088788422058298);

\draw [line width=0.8pt] (5.,3.)-- (5.,4.);
\draw [line width=0.8pt] (5.,4.)-- (6.,4.);
\draw [line width=0.8pt] (5.5,4.108878842205828) -- (5.5,3.8911211577941707);
\draw [line width=0.8pt] (6.,4.)-- (6.,3.);
\draw [line width=0.8pt] (6.,3.)-- (5.,3.);

\draw [line width=0.8pt] (5.63,2.8911211577941707) -- (5.63,3.108878842205828);
\draw [line width=0.8pt] (5.543551536882333,2.8911211577941707) -- (5.543551536882333,3.108878842205828);
\draw [line width=0.8pt] (5.45644846311767,2.8911211577941707) -- (5.45644846311767,3.108878842205828);
\draw [line width=0.8pt] (5.36,2.8911211577941707) -- (5.36,3.108878842205828);

\draw [line width=0.8pt] (6.,3.)-- (6.,4.);
\draw [line width=0.8pt] (6.,4.)-- (7.,4.);
\draw [line width=0.8pt] (6.4564484631176695,4.108878842205828) -- (6.4564484631176695,3.8911211577941707);
\draw [line width=0.8pt] (6.543551536882332,4.108878842205828) -- (6.543551536882332,3.8911211577941707);
\draw [line width=0.8pt] (7.,4.)-- (7.,3.);
\draw [line width=0.8pt] (7.,3.)-- (6.,3.);

\draw [line width=0.8pt] (6.587103073764664,2.8911211577941716) -- (6.587103073764664,3.1088788422058298);
\draw [line width=0.8pt] (6.5,2.8911211577941716) -- (6.5,3.1088788422058298);
\draw [line width=0.8pt] (6.412896926235338,2.8911211577941716) -- (6.412896926235338,3.1088788422058298);
\draw [line width=0.8pt] (6.324,3.1088788422058298) -- (6.324,2.8911211577941716);
\draw [line width=0.8pt] (6.67,3.1088788422058298) -- (6.67,2.8911211577941716);

\draw [line width=0.8pt] (11.,5.)-- (11.,4.);
\draw [line width=0.8pt] (11.,4.)-- (12.,4.);
\draw [line width=0.8pt] (11.5,4.108878842205828) -- (11.5,3.8911211577941707);
\draw [line width=0.8pt] (12.,4.)-- (12.,5.);
\draw [line width=0.8pt] (12.,5.)-- (11.,5.);
\draw [line width=0.8pt] (11.543551536882333,4.891121157794172) -- (11.543551536882333,5.10887884220583);
\draw [line width=0.8pt] (11.45644846311767,4.891121157794172) -- (11.45644846311767,5.10887884220583);
\draw [line width=0.8pt] (12.,4.)-- (13.,4.);
\draw [line width=0.8pt] (12.45644846311767,4.108878842205828) -- (12.45644846311767,3.8911211577941707);
\draw [line width=0.8pt] (12.543551536882333,4.108878842205828) -- (12.543551536882333,3.8911211577941707);
\draw [line width=0.8pt] (13.,4.)-- (13.,5.);
\draw [line width=0.8pt] (13.,5.)-- (12.,5.);
\draw [line width=0.8pt] (12.587103073764665,4.891121157794172) -- (12.587103073764665,5.10887884220583);
\draw [line width=0.8pt] (12.5,4.891121157794172) -- (12.5,5.10887884220583);
\draw [line width=0.8pt] (12.412896926235339,4.891121157794172) -- (12.412896926235339,5.10887884220583);
\draw [line width=0.8pt] (12.,5.)-- (12.,4.);
\draw [line width=0.8pt] (13.,4.)-- (14.,4.);

\draw [line width=0.8pt] (13.412896926235339,4.108878842205828) -- (13.412896926235339,3.8911211577941707);
\draw [line width=0.8pt] (13.5,4.108878842205828) -- (13.5,3.8911211577941707);
\draw [line width=0.8pt] (13.587103073764665,4.108878842205828) -- (13.587103073764665,3.8911211577941707);

\draw [line width=0.8pt] (14.,4.)-- (14.,5.);
\draw [line width=0.8pt] (14.,5.)-- (13.,5.);
\draw [line width=0.8pt] (13.5,4.891121157794172) -- (13.5,5.10887884220583);
\draw [line width=0.8pt] (13.,5.)-- (13.,4.);
\draw [line width=0.8pt] (14.,3.)-- (15.,3.);

\draw [line width=0.8pt] (14.324,3.1088788422058298) -- (14.324,2.8911211577941716);
\draw [line width=0.8pt] (14.412896926235339,3.1088788422058298) -- (14.412896926235339,2.8911211577941716);
\draw [line width=0.8pt] (14.5,3.1088788422058298) -- (14.5,2.8911211577941716);
\draw [line width=0.8pt] (14.587103073764665,3.1088788422058298) -- (14.587103073764665,2.8911211577941716);
\draw [line width=0.8pt] (14.67,3.1088788422058298) -- (14.67,2.8911211577941716);

\draw [line width=0.8pt] (15.,3.)-- (15.,4.);
\draw [line width=0.8pt] (15.,4.)-- (14.,4.);

\draw [line width=0.8pt] (14.63,3.8911211577941707) -- (14.63,4.108878842205828);
\draw [line width=0.8pt] (14.543551536882333,3.8911211577941707) -- (14.543551536882333,4.108878842205828);
\draw [line width=0.8pt] (14.45644846311767,3.8911211577941707) -- (14.45644846311767,4.108878842205828);
\draw [line width=0.8pt] (14.36,3.8911211577941707) -- (14.36,4.108878842205828);

\draw [line width=0.8pt] (14.,4.)-- (14.,3.);
\draw [line width=0.8pt] (15.,3.)-- (16.,3.);

\draw [line width=0.8pt] (15.45644846311767,3.1088788422058298) -- (15.45644846311767,2.8911211577941716);
\draw [line width=0.8pt] (15.543551536882333,3.1088788422058298) -- (15.543551536882333,2.8911211577941716);
\draw [line width=0.8pt] (15.63,3.108878842205828) -- (15.63,2.8911211577941707);
\draw [line width=0.8pt] (15.36,3.108878842205828) -- (15.36,2.8911211577941707);

\draw [line width=0.8pt] (16.,3.)-- (16.,4.);
\draw [line width=0.8pt] (16.,4.)-- (15.,4.);

\draw [line width=0.8pt] (15.67,3.8911211577941707) -- (15.67,4.108878842205828);
\draw [line width=0.8pt] (15.587103073764665,3.8911211577941707) -- (15.587103073764665,4.108878842205828);
\draw [line width=0.8pt] (15.5,3.8911211577941707) -- (15.5,4.108878842205828);
\draw [line width=0.8pt] (15.412896926235339,3.8911211577941707) -- (15.412896926235339,4.108878842205828);
\draw [line width=0.8pt] (15.324,3.8911211577941707) -- (15.324,4.108878842205828);

\draw [line width=0.8pt] (15.,4.)-- (15.,3.);
\begin{scriptsize}
\draw [fill=black] (2.,3.) circle (3pt);
\draw [fill=black] (2.,4.) circle (3pt);
\draw [fill=black] (4.,3.) circle (3pt);
\draw [fill=black] (5.,4.) circle (3pt);
\draw [fill=black] (7.,4.) circle (3pt);
\draw [fill=black] (7.,3.) circle (3pt);
\draw [fill=black] (11.,4.) circle (3pt);
\draw [fill=black] (13.,5.) circle (3pt);
\draw [fill=black] (14.,4.) circle (3pt);
\draw [fill=black] (15.,3.) circle (3pt);
\draw [fill=black] (16.,4.) circle (3pt);
\end{scriptsize}
\end{tikzpicture}
\caption{Square-tilings corresponding to the vertices on the boundaries of the eave $\EE_k \subset \A_{d^2}$. In this specific example $k=2$ and $d=5$.
}
\label{figeavevert}
\end{figure}
\end{comment7}

Note that the data of the eave boundaries described above is not complete. For example, it is missing the positions of the zeroes of $q$. However this data will suffice to give the proof of the main result. For any prime $d$, we will give a full and detailed description of all boundaries of the horizontal cylinders of $\A_{d^2}$ (including the ones that are neither lighthouses, nor eaves) and their identifications in Appendix~\ref{secpagoda}.

\noindent
{\bf Action of $\SLZ$.} The group $\SLZ$ is generated by two matrices:
\begin{align*}
S = \begin{pmatrix} 1 & -1 \\ 0 & 1 \\ \end{pmatrix}
\mbox{ and }
R = \begin{pmatrix} 0 & -1 \\ 1 & 0 \\ \end{pmatrix}.
\end{align*}

The essential part of the proofs of the main result is analyzing how $R$ and $S$ act on the subsets $\A_{d^2}[n]$. Below we describe properties of the action of $R$ and $S$ on lighthouses and eaves that will be sufficient to understand the $\SLZ$ orbits in $\A_{d^2}[n]$.

\noindent
{\bf Rotation of lighthouses.} We now show how lighthouses and eaves are related by $\pi/2$ rotation. 

\begin{proof}[Proof of Theorem~\ref{thmrotation}]
The proof follows from the structures of eaves and lighthouses (see Figure~\ref{figlighthouse} and~\ref{figeave}) and the following observation: the cusp pole with cylinder coordinates $(1,k,1,d-k,0,0,1) \in \LL_k \subset \A_{d^2}$ is sent to the cusp pole with cylinder coordinates $(k,1,d-k,1,0,0,0) \in \EE_k \subset \A_{d^2}$ via rotation by $\pi/2$ (see Figure~\ref{figlighttoeave}).
\end{proof}

\begin{comment7}
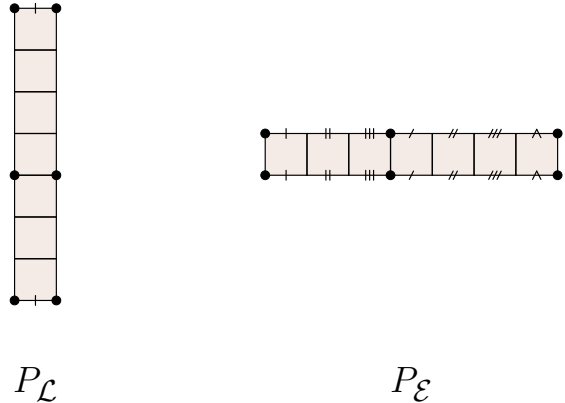
\begin{figure}[H] 
\centering
\definecolor{ffqqff}{rgb}{0.,0.,0.}
\definecolor{zzttqq}{rgb}{0.6,0.2,0.}
\definecolor{cqcqcq}{rgb}{0.7529411764705882,0.7529411764705882,0.7529411764705882}
\begin{tikzpicture}[line cap=round,line join=round,>=triangle 45,x=1.0cm,y=1.0cm]
\clip(2.,-6) rectangle (22.5,5);

\draw (6.5,-5) node {\Huge $P_{\mathcal L}$};
\draw (15.5,-5) node {\Huge $P_{\mathcal E}$};

\fill[line width=0.8pt,color=zzttqq,fill=zzttqq,fill opacity=0.10000000149011612] (6.,-1.) -- (6.,0.) -- (7.,0.) -- (7.,-1.) -- cycle;
\fill[line width=0.8pt,color=zzttqq,fill=zzttqq,fill opacity=0.10000000149011612] (6.,-2.) -- (6.,-1.) -- (7.,-1.) -- (7.,-2.) -- cycle;
\fill[line width=0.8pt,color=zzttqq,fill=zzttqq,fill opacity=0.10000000149011612] (6.,-3.) -- (6.,-2.) -- (7.,-2.) -- (7.,-3.) -- cycle;
\fill[line width=0.8pt,color=zzttqq,fill=zzttqq,fill opacity=0.10000000149011612] (6.,3.) -- (6.,4.) -- (7.,4.) -- (7.,3.) -- cycle;
\fill[line width=0.8pt,color=zzttqq,fill=zzttqq,fill opacity=0.10000000149011612] (6.,2.) -- (6.,3.) -- (7.,3.) -- (7.,2.) -- cycle;
\fill[line width=0.8pt,color=zzttqq,fill=zzttqq,fill opacity=0.10000000149011612] (6.,1.) -- (6.,2.) -- (7.,2.) -- (7.,1.) -- cycle;
\fill[line width=0.8pt,color=zzttqq,fill=zzttqq,fill opacity=0.10000000149011612] (6.,0.) -- (6.,1.) -- (7.,1.) -- (7.,0.) -- cycle;
\fill[line width=0.8pt,color=zzttqq,fill=zzttqq,fill opacity=0.10000000149011612] (14.,1.) -- (15.,1.) -- (15.,0.) -- (14.,0.) -- cycle;
\fill[line width=0.8pt,color=zzttqq,fill=zzttqq,fill opacity=0.10000000149011612] (13.,1.) -- (14.,1.) -- (14.,0.) -- (13.,0.) -- cycle;
\fill[line width=0.8pt,color=zzttqq,fill=zzttqq,fill opacity=0.10000000149011612] (12.,1.) -- (13.,1.) -- (13.,0.) -- (12.,0.) -- cycle;
\fill[line width=0.8pt,color=zzttqq,fill=zzttqq,fill opacity=0.10000000149011612] (18.,1.) -- (19.,1.) -- (19.,0.) -- (18.,0.) -- cycle;
\fill[line width=0.8pt,color=zzttqq,fill=zzttqq,fill opacity=0.10000000149011612] (17.,1.) -- (18.,1.) -- (18.,0.) -- (17.,0.) -- cycle;
\fill[line width=0.8pt,color=zzttqq,fill=zzttqq,fill opacity=0.10000000149011612] (16.,1.) -- (17.,1.) -- (17.,0.) -- (16.,0.) -- cycle;
\fill[line width=0.8pt,color=zzttqq,fill=zzttqq,fill opacity=0.10000000149011612] (15.,1.) -- (16.,1.) -- (16.,0.) -- (15.,0.) -- cycle;
\draw [line width=0.8pt] (6.,-1.)-- (6.,0.);
\draw [line width=0.8pt] (6.,0.)-- (7.,0.);
\draw [line width=0.8pt] (7.,0.)-- (7.,-1.);
\draw [line width=0.8pt] (7.,-1.)-- (6.,-1.);
\draw [line width=0.8pt] (6.,-2.)-- (6.,-1.);
\draw [line width=0.8pt] (6.,-1.)-- (7.,-1.);
\draw [line width=0.8pt] (7.,-1.)-- (7.,-2.);
\draw [line width=0.8pt] (7.,-2.)-- (6.,-2.);
\draw [line width=0.8pt] (6.,-3.)-- (6.,-2.);
\draw [line width=0.8pt] (6.,-2.)-- (7.,-2.);
\draw [line width=0.8pt] (7.,-2.)-- (7.,-3.);
\draw [line width=0.8pt] (7.,-3.)-- (6.,-3.);
\draw [line width=0.8pt] (6.5,-3.117918050941307) -- (6.5,-2.8820819490586933);
\draw [line width=0.8pt] (6.,3.)-- (6.,4.);
\draw [line width=0.8pt] (6.,4.)-- (7.,4.);
\draw [line width=0.8pt] (6.5,4.117918050941307) -- (6.5,3.882081949058693);
\draw [line width=0.8pt] (7.,4.)-- (7.,3.);
\draw [line width=0.8pt] (7.,3.)-- (6.,3.);
\draw [line width=0.8pt] (6.,2.)-- (6.,3.);
\draw [line width=0.8pt] (6.,3.)-- (7.,3.);
\draw [line width=0.8pt] (7.,3.)-- (7.,2.);
\draw [line width=0.8pt] (7.,2.)-- (6.,2.);
\draw [line width=0.8pt] (6.,1.)-- (6.,2.);
\draw [line width=0.8pt] (6.,2.)-- (7.,2.);
\draw [line width=0.8pt] (7.,2.)-- (7.,1.);
\draw [line width=0.8pt] (7.,1.)-- (6.,1.);
\draw [line width=0.8pt] (6.,0.)-- (6.,1.);
\draw [line width=0.8pt] (6.,1.)-- (7.,1.);
\draw [line width=0.8pt] (7.,1.)-- (7.,0.);
\draw [line width=0.8pt] (7.,0.)-- (6.,0.);
\draw [line width=0.8pt] (14.,1.)-- (15.,1.);
\draw [line width=0.8pt] (14.405665559246957,1.1179180509413065) -- (14.405665559246957,0.8820819490586927);
\draw [line width=0.8pt] (14.5,1.1179180509413065) -- (14.5,0.8820819490586927);
\draw [line width=0.8pt] (14.594334440753046,1.1179180509413065) -- (14.594334440753046,0.8820819490586927);
\draw [line width=0.8pt] (15.,1.)-- (15.,0.);
\draw [line width=0.8pt] (15.,0.)-- (14.,0.);
\draw [line width=0.8pt] (14.594334440753046,-0.11791805094130689) -- (14.594334440753046,0.11791805094130689);
\draw [line width=0.8pt] (14.5,-0.11791805094130689) -- (14.5,0.11791805094130689);
\draw [line width=0.8pt] (14.405665559246957,-0.11791805094130689) -- (14.405665559246957,0.11791805094130689);
\draw [line width=0.8pt] (14.,0.)-- (14.,1.);
\draw [line width=0.8pt] (13.,1.)-- (14.,1.);
\draw [line width=0.8pt] (13.452832779623478,1.1179180509413065) -- (13.452832779623478,0.8820819490586927);
\draw [line width=0.8pt] (13.547167220376522,1.1179180509413065) -- (13.547167220376522,0.8820819490586927);
\draw [line width=0.8pt] (14.,1.)-- (14.,0.);
\draw [line width=0.8pt] (14.,0.)-- (13.,0.);
\draw [line width=0.8pt] (13.547167220376522,-0.11791805094130689) -- (13.547167220376522,0.11791805094130689);
\draw [line width=0.8pt] (13.452832779623478,-0.11791805094130689) -- (13.452832779623478,0.11791805094130689);
\draw [line width=0.8pt] (13.,0.)-- (13.,1.);
\draw [line width=0.8pt] (12.,1.)-- (13.,1.);
\draw [line width=0.8pt] (12.5,1.1179180509413065) -- (12.5,0.8820819490586927);
\draw [line width=0.8pt] (13.,1.)-- (13.,0.);
\draw [line width=0.8pt] (13.,0.)-- (12.,0.);
\draw [line width=0.8pt] (12.5,-0.11791805094130689) -- (12.5,0.11791805094130689);
\draw [line width=0.8pt] (12.,0.)-- (12.,1.);
\draw [line width=0.8pt,color=ffqqff] (18.,1.)-- (19.,1.);
\draw [line width=0.8pt,color=ffqqff] (18.5,1.1) -- (18.4,0.9);
\draw [line width=0.8pt,color=ffqqff] (18.5,1.1) -- (18.6,0.9);
\draw [line width=0.8pt] (19.,1.)-- (19.,0.);
\draw [line width=0.8pt,color=ffqqff] (19.,0.)-- (18.,0.);
\draw [line width=0.8pt,color=ffqqff] (18.5,0.1) -- (18.4,0.9-1);
\draw [line width=0.8pt,color=ffqqff] (18.5,0.1) -- (18.6,0.9-1);
\draw [line width=0.8pt] (18.,0.)-- (18.,1.);
\draw [line width=0.8pt,color=ffqqff] (17.,1.)-- (18.,1.);
\draw [line width=0.8pt,color=ffqqff] (17.65,1.1) -- (17.55,0.9);
\draw [line width=0.8pt,color=ffqqff] (17.55,1.1) -- (17.45,0.9);
\draw [line width=0.8pt,color=ffqqff] (17.45,1.1) -- (17.35,0.9);
\draw [line width=0.8pt] (18.,1.)-- (18.,0.);
\draw [line width=0.8pt,color=ffqqff] (18.,0.)-- (17.,0.);
\draw [line width=0.8pt,color=ffqqff] (17.65,0.1) -- (17.55,-0.1);
\draw [line width=0.8pt,color=ffqqff] (17.55,0.1) -- (17.45,-0.1);
\draw [line width=0.8pt,color=ffqqff] (17.45,0.1) -- (17.35,-0.1);
\draw [line width=0.8pt] (17.,0.)-- (17.,1.);
\draw [line width=0.8pt,color=ffqqff] (16.,1.)-- (17.,1.);
\draw [line width=0.8pt,color=ffqqff] (16.5,1.1) -- (16.4,0.9);
\draw [line width=0.8pt,color=ffqqff] (16.6,1.1) -- (16.5,0.9);
\draw [line width=0.8pt] (17.,1.)-- (17.,0.);
\draw [line width=0.8pt,color=ffqqff] (17.,0.)-- (16.,0.);
\draw [line width=0.8pt,color=ffqqff] (16.5,0.1) -- (16.4,0.9-1);
\draw [line width=0.8pt,color=ffqqff] (16.6,0.1) -- (16.5,0.9-1);
\draw [line width=0.8pt] (16.,0.)-- (16.,1.);
\draw [line width=0.8pt,color=ffqqff] (15.,1.)-- (16.,1.);
\draw [line width=0.8pt,color=ffqqff] (15.55,1.1) -- (15.45,0.9);
\draw [line width=0.8pt] (16.,1.)-- (16.,0.);
\draw [line width=0.8pt,color=ffqqff] (16.,0.)-- (15.,0.);
\draw [line width=0.8pt,color=ffqqff] (15.45,-0.1) -- (15.55,0.1);
\draw [line width=0.8pt] (15.,0.)-- (15.,1.);
\begin{scriptsize}
\draw [fill=black] (6.,0.) circle (3pt);
\draw [fill=black] (7.,0.) circle (3pt);
\draw [fill=black] (6.,4.) circle (3pt);
\draw [fill=black] (7.,4.) circle (3pt);
\draw [fill=black] (6.,-3.) circle (3pt);
\draw [fill=black] (7.,-3.) circle (3pt);
\draw [fill=black] (15.,1.) circle (3pt);
\draw [fill=black] (15.,0.) circle (3pt);
\draw [fill=black] (19.,1.) circle (3pt);
\draw [fill=black] (19.,0.) circle (3pt);
\draw [fill=black] (12.,1.) circle (3pt);
\draw [fill=black] (12.,0.) circle (3pt);
\end{scriptsize}
\end{tikzpicture}
\caption{The rotation sends the cusp pole $P_\LL$ in the lighthouse $\LL_k \subset \A_{d^2}$ to the cusp pole $P_\EE$ in the eave 
$\EE_k \subset \A_{d^2}$. In this specific example $k=3$ and $d=7$.
}
\label{figlighttoeave}
\end{figure}
\end{comment7}

\noindent
{\bf Unipotent action on eaves.} We denote the unipotent subgroup of $\SLZ$ by:
$$
U = \left\{
\begin{pmatrix}
1 & n \\
0 & 1 
\end{pmatrix} \bigg\vert \ n \in \Z
\right\} \subset \SLZ.
$$
It is generated by the matrix $S= \begin{pmatrix} 1& -1 \\ 0&1 \end{pmatrix}$.

The point of the eave $\EE_k$ with the Euclidean coordinates $(0,0)$ is the non-cusp pole of $\EE_k$ with cylinder coordinates $(k,1,d-k,1,0,0,0)$, which is the leftmost white point on Figure~\ref{figeave}. 
Recall from Remark~\ref{orientation} that in Euclidean coordinates $(x,y)$: $x$ is a standard horizontal axes and $y$ is a vertical axes pointing downwards. 

\begin{proof}[Proof of Theorem~\ref{thmunipotent}]
First note that $S$ preserves horizontal cylinders, and in particular $\EE_k$. Let $(0,y) \in \EE_k \subset \A_{d^2}$ be a point on the leftmost vertical edge $v$ of $\EE_k$. This edge connects a cusp $A_1$ and a non-cusp pole $B_1$ in $\EE_k$ (see Figure~\ref{figeave}). Since $\SLZ$ preserves the sets of cusps and non-cusp poles of $\A_{d^2}$, the matrix $S$ has to send $v$ to an interval of slope $-1$ between a cusp $S(A_1)$ and a non-cusp pole $S(B_1)$ in $\EE_k$. Then from the structure of the eave (see Figure~\ref{figeave}), $S(B_1)$ has flat coordinates  $(T_k \cdot d,0)$ for some $0\le T_k < k(d-k)$. Because of the slope $-1$ condition, the flat coordinates of $S(A_1)$ have to be $(T_k\cdot d +1,1)$. Therefore we obtain $T_k\cdot d +1 \equiv 0 \mod k(d-k)$, which is equivalent to $T_k\cdot d \equiv -1 \mod k(d-k)$.

It remains to notice that $S$ shifts points on the same horizontal line by the same amount and therefore we obtain:
\[
S : \left( x, y \right) \mapsto \left( x + y + T_k \cdot d, y \right).
\pushQED{\qed} 
\qedhere
\popQED
\]
\let\qed\relax
\end{proof}


\section{Proof for $d=2$} \label{secproof2}

In this section we present a square-tiling of $(\A_{4}, q)$ and give proof of the parity conjecture in the case when $d=2$. We will show:

\begin{proposition} \label{propproof2}
The discriminant map $\delta: \A_{4} \to \bfP$ is an isomorphism. For any integer $n>1$, the action of 
$\SLZ$ on:
$$
\A_{4}[n] = \left\{ \left( \frac kn,\frac ln \right) \in \A_{4} \cong \bfP \ \middle\vert \ \gcd(k,l,n)=1, 0\le k\le 2n, 0\le l\le n \right\}
$$
is transitive, when $n$ is even, and has two orbits:
\begin{align*}
\A_{4}[n]^{0} &= \left\{ \left( \frac kn,\frac ln \right) \in \A_{4}[n] \ \middle\vert \ k\equiv l\equiv 0 \mod 2 \right\}, \\
\A_{4}[n]^{1}  &= \left\{ \left( \frac kn,\frac ln \right) \in \A_{4}[n] \ \middle\vert \ k \mbox{ or } l\equiv 1 \mod 2 \right\},
\end{align*}
when $n$ is odd.
\end{proposition}

\noindent
{\bf $(\A_{4},q)$ is a pillowcase.} Given any two points $z_1, z_2$ on the square torus $E_0$, there are exactly four degree $2$ covers of genus $2$ (necessarily primitive, because the degree is prime) branched over $z_1$ and $z_2$. To show this we analyze an example using monodromies of covers. 

Let $\pi: X \to E_0$ be a degree $d$ cover branched over $z_1$ and $z_2$. Choose a point $x_0 \in E_0 \setminus \{ z_1,z_2 \}$ and fix a labeling of the fiber over $x_0$ by numbers from $1$ through $d$. Lifting a loop representative of an element of the fundamental group of $E_0$ gives a permutation of points in the fiber and hence an element of $S_d$. The corresponding representation $\rho: \pi_1(E_0 \setminus \{ z_1,z_2 \}, x_0) \to S_d$ is called a {\em monodromy} of the cover $\pi$.

For example, let $z_1=0$, $z_2=\frac in$. Let $h$ be a horizontal and $v$ a vertical loops on $E_0$ with endpoints at the center $x_0 = \frac12 + \frac i2$, $\gamma_1$ and $\gamma_2$ be small loops around $z_1$ and $z_2$ with endpoints at $x_0$. Fix a labeling of the fiber over $x_0$ by numbers 1 and 2 and let $s_h, s_v, s_1, s_2 \in S_2$ be the images of $h,v,\gamma_1,\gamma_2$ under the monodromy representation. Because the cover is simply branched over $z_1$ and $z_2$ we have $s_1=s_2=(12)\in S_2$. Any choice of $s_h, s_v \in S_2$ produces a desired cover, and there are four of them (see Figure~\ref{X(2)genfiber}).

There are also four images of the covers ramified over $z_1=0$, $z_2=\frac in$ under $\delta$ on the pillowcase: $(0, \frac1n)$, $(0, \frac{n-1}n)$, $(1, \frac1n)$ and $(1, \frac{n-1}n) \in \bfP = \iota \backslash \C / 2\Z[i]$. Therefore $\deg(\delta)=1$ and $(\A_{4},q) \cong (\bfP,dz^2)$. 

\begin{comment8}
\begin{figure}[H]
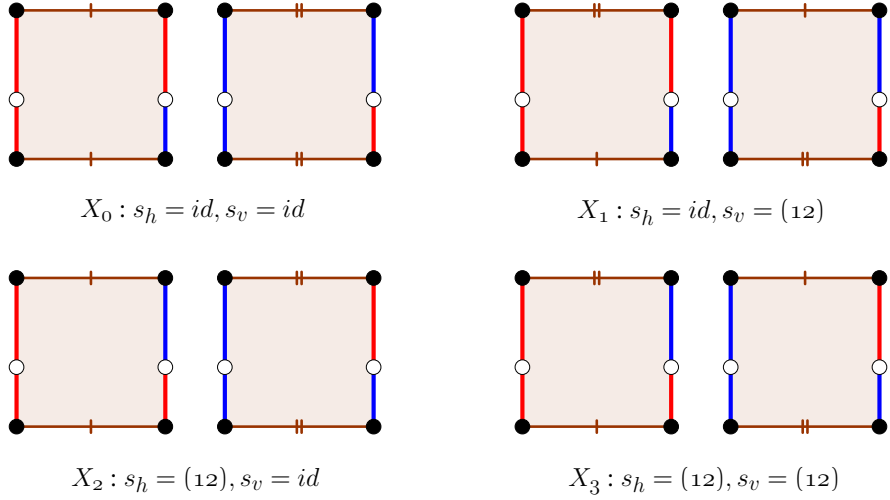
 
\centering
\includestandalone[width=0.8\textwidth]{tikzdeg2covers}
\vspace{-10pt}
\caption{Four degree $2$ covers of a genus $2$ surface over $E_0$ branched over two given points and their monodromies.}
\label{X(2)genfiber}
\end{figure}
\end{comment8}

Compare this result to Theorem~\ref{thmdecomposition}, according to which $(\A_{4}, q)$ has only one cylinder given by $(w_1,s_1,w_2,s_2) = (1,1,1,1)$ with height $\min(s_1,s_2)=1$ and circumference is $w_1 w_2 (w_1 + w_2)=2$.


\noindent
{\bf Singular covers.}
Denote the vertices of $\A_{4}$ by $Q_0, Q_1, Q_2, Q_3$. It is easy to exhaust all singular covers: one square-tiled surface with a separating node and three with non-separating nodes (see Figure~\ref{X(2)singular}).
\begin{figure}[H] 
    \centering
    \includegraphics[width=0.8\textwidth]{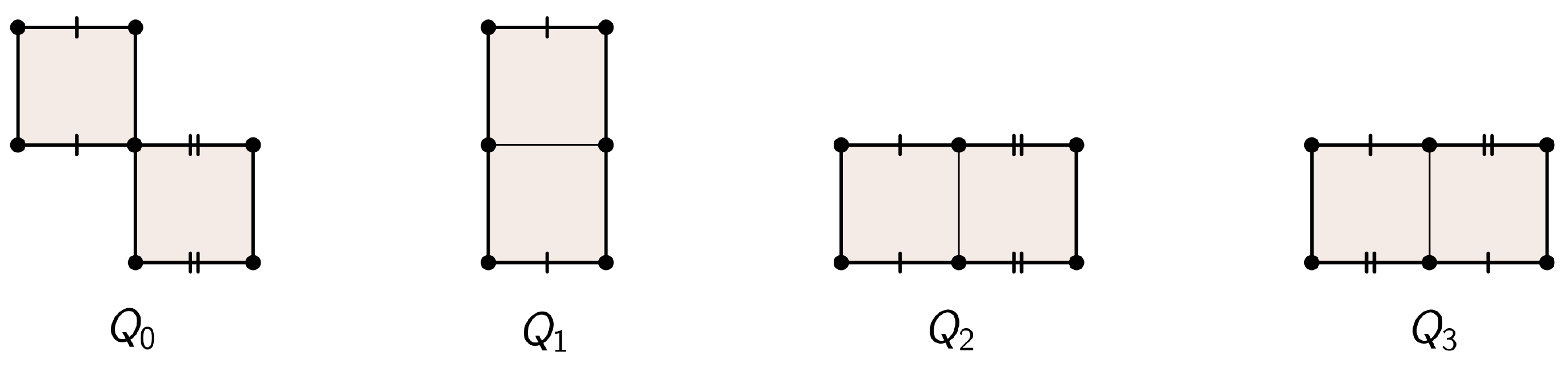}
    \caption{The square-tiled surface $Q_0$ with a separating node and the square-tiled surfaces $Q_1, Q_2, Q_3$ with non-separating nodes.}
     \label{X(2)singular}
\end{figure}

\begin{wrapfigure}{r}{0.29\textwidth}
    \begin{center}
    \includegraphics[width=0.15\textwidth]{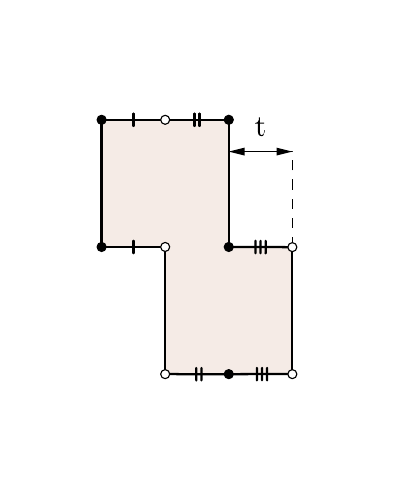}
    \caption{Sliding.}
     \label{sliding}
    \end{center}
\end{wrapfigure}
By horizontally varying the relative period $t$ in the example in Figure~\ref{sliding}, one obtains $Q_1$ as $t \to 0$ and $Q_0$ as $t \to 1$. This produces a family of points $Q_t, 0 \le t \le 1$ on $\A_4$ that lie on the horizontal edge joining $Q_0$ and $Q_1$. The same idea applies for the vertical edges and as a result it becomes clear how $Q_0, Q_1, Q_2, Q_3$ are positioned on the pillowcase (see Figure~\ref{X(2)}): $Q_0$ is on the same horizontal line with $Q_1$ and strictly above $Q_2$, which leaves only one possibility for the position of $Q_3$. We label the unique non-cusp pole (corresponding to $Q_0$) with a white point. The three cusps poles of $X(2)$ (corresponding to $Q_1, Q_2, Q_3$) are labeled with red points. The four covers from Figure~\ref{X(2)genfiber} are represented by the points labeled $X_0, X_1, X_2, X_3$ in Figure~\ref{X(2)}.

\begin{comment8}
\begin{figure}[H] 
\centering
\definecolor{xdxdff}{rgb}{0.49019607843137253,0.49019607843137253,1.}
\definecolor{ffqqqq}{rgb}{1.,0.4,0.4}
\definecolor{ffffff}{rgb}{1.,1.,1.}
\begin{tikzpicture}[line cap=round,line join=round,>=triangle 45,x=1.0cm,y=1.0cm]
\clip(1.,-5.) rectangle (15.,4.);
\draw [line width=1.5pt] (3.,2.)-- (3.,-3.);
\draw [line width=1.5pt] (3.151164154601988,-0.5) -- (2.8488358453980123,-0.5);
\draw [line width=1.5pt] (3.,-3.)-- (8.,-3.);
\draw [line width=1.5pt] (8.,-3.)-- (8.,2.);
\draw [line width=1.5pt] (8.,2.)-- (3.,2.);
\draw [line width=1.5pt] (5.323641819631013,2.) -- (5.5,2.226746231902983);
\draw [line width=1.5pt] (5.323641819631013,2.) -- (5.5,1.7732537680970182);
\draw [line width=1.5pt] (8.,-3.)-- (13.,-3.);
\draw [line width=1.5pt] (10.852716360737972,-3.) -- (10.676358180368986,-3.226746231902982);
\draw [line width=1.5pt] (10.852716360737972,-3.) -- (10.676358180368986,-2.7732537680970175);
\draw [line width=1.5pt] (10.5,-3.) -- (10.323641819631014,-3.226746231902982);
\draw [line width=1.5pt] (10.5,-3.) -- (10.323641819631014,-2.7732537680970175);
\draw [line width=1.5pt] (13.,-3.)-- (13.,2.);
\draw [line width=1.5pt] (12.848835845398014,-0.5) -- (13.15116415460199,-0.5);
\draw [line width=1.5pt] (13.,2.)-- (8.,2.);
\draw [line width=1.5pt] (8.,2.)-- (8.,-3.);
\draw [line width=1.5pt] (8.,-3.)-- (3.,-3.);
\draw [line width=1.5pt] (5.147283639262027,-3.) -- (5.323641819631013,-2.7732537680970175);
\draw [line width=1.5pt] (5.147283639262027,-3.) -- (5.323641819631013,-3.226746231902982);
\draw [line width=1.5pt] (5.5,-3.) -- (5.676358180368985,-2.7732537680970175);
\draw [line width=1.5pt] (5.5,-3.) -- (5.676358180368985,-3.226746231902982);
\draw [line width=1.5pt] (8.,2.)-- (13.,2.);
\draw [line width=1.5pt] (10.676358180368986,2.) -- (10.5,1.7732537680970182);
\draw [line width=1.5pt] (10.676358180368986,2.) -- (10.5,2.226746231902983);
\draw (2.7,3.) node[anchor=north west] {\Large $Q_0$};
\draw (7.7,3.) node[anchor=north west] {\Large $Q_1$};
\draw (2.7,-3.3) node[anchor=north west] {\Large $Q_2$};
\draw (7.7,-3.3) node[anchor=north west] {\Large $Q_3$};
\draw (12.7,3.) node[anchor=north west] {\Large $Q_0$};
\draw (12.7,-3.3) node[anchor=north west] {\Large $Q_2$};

\draw (2.,1.3) node[anchor=north west] {\Large $X_0$};
\draw (7.,1.3) node[anchor=north west] {\Large $X_1$};
\draw (2.,-1.7) node[anchor=north west] {\Large $X_2$};
\draw (7.,-1.7) node[anchor=north west] {\Large $X_3$};
\draw (12.,1.3) node[anchor=north west] {\Large $X_0$};
\draw (12.,-1.7) node[anchor=north west] {\Large $X_2$};

\begin{scriptsize}
\draw [fill=ffffff] (3.,2.) circle (5pt);
\draw [fill=ffqqqq] (3.,-3.) circle (5pt);
\draw [fill=ffqqqq] (8.,-3.) circle (5pt);
\draw [fill=ffqqqq] (8.,2.) circle (5pt);
\draw [fill=ffqqqq] (13.,-3.) circle (5pt);
\draw [fill=ffffff] (13.,2.) circle (5pt);
\draw [fill=xdxdff] (3.,-2.) circle (3pt);
\draw [fill=xdxdff] (3.,1.) circle (3pt);
\draw [fill=xdxdff] (8.,1.) circle (3pt);
\draw [fill=xdxdff] (8.,-2.) circle (3pt);
\draw [fill=xdxdff] (13.,-2.) circle (3pt);
\draw [fill=xdxdff] (13.,1.) circle (3pt);
\end{scriptsize}
\end{tikzpicture}
\vspace{-10pt}
\caption{The square-tiling of $X(2)\cong \A_{4}$.}
\label{X(2)}
\end{figure}
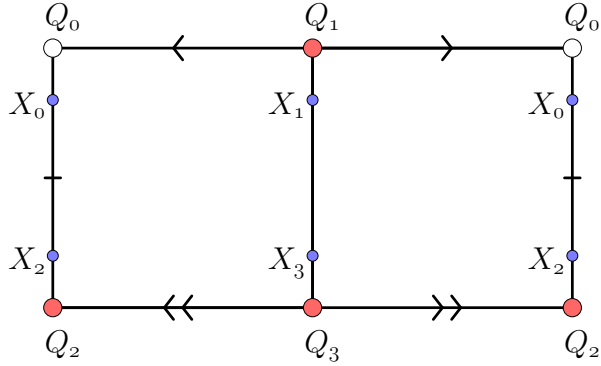
\end{comment8}

The only cylinder of $(\A_{4},q)$ is at the same time a lighthouse and an eave. Indeed, it is a lighthouse, since $w_1=w_2=1$ and its top boundary consists of a cusp, a non-cusp pole and two edges identified by rotation. It is an eave, since $s_1=s_2=1$ and both vertices of the bottom are cusp poles. This is the only case when lighthouse and eave coincide.

\noindent
{\bf Locating $\A_{4}[n]$.}
The results of Theorem~\ref{thmlocating} and Proposition~\ref{propIWPvalues} are illustrated in Figure~\ref{X(2)orbits} for $n=5$. The parity conjecture states that all green points belong to one $\SLZ$ orbit and all blue points belong to another. Showing that will finish the proof of Proposition~\ref{propproof2}.

\begin{comment8}
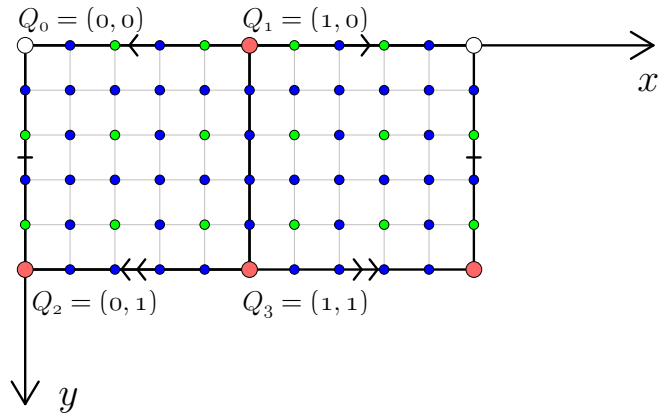
\begin{figure}[H] 
\centering
\definecolor{qqffqq}{rgb}{0.,1.,0.}
\definecolor{qqqqff}{rgb}{0.,0.,1.}
\definecolor{ffqqqq}{rgb}{1.,0.4,0.4}
\definecolor{ffffff}{rgb}{1.,1.,1.}
\definecolor{cqcqcq}{rgb}{0.7529411764705882,0.7529411764705882,0.7529411764705882}
\begin{tikzpicture}[line cap=round,line join=round,>=triangle 45,x=1.0cm,y=1.0cm]
\draw [color=cqcqcq,, xstep=1.0cm,ystep=1.0cm] (3.,-3.) grid (13.,2.);
\clip(2.,-7.) rectangle (18.,4.);

\draw (2.7,3.) node[anchor=north west] {\Large $Q_0 = (0,0)$};
\draw (7.7,3.) node[anchor=north west] {\Large $Q_1 = (1,0)$};
\draw (3,-3.3) node[anchor=north west] {\Large $Q_2 = (0,1)$};
\draw (7.7,-3.3) node[anchor=north west] {\Large $Q_3 = (1,1)$};

\draw [line width=1.5pt] (3.,2.)-- (17.,2.);
\draw [line width=1.5pt] (3.,2.)-- (3.,-6.);

\draw [line width=1.5pt] (17,2.) -- (16.5,1.7);
\draw [line width=1.5pt] (17,2.) -- (16.5,2.3);
\draw (16.5,1.5) node[anchor=north west] {\Huge $x$};

\draw [line width=1.5pt] (3,-6.) -- (3.3,-5.5);
\draw [line width=1.5pt] (3,-6.) -- (2.7,-5.5);
\draw (3.6,-5.5) node[anchor=north west] {\Huge $y$};

\draw [line width=1.5pt] (3.,2.)-- (3.,-3.);
\draw [line width=1.5pt] (3.151164154601988,-0.5) -- (2.8488358453980123,-0.5);
\draw [line width=1.5pt] (3.,-3.)-- (8.,-3.);
\draw [line width=1.5pt] (8.,-3.)-- (8.,2.);
\draw [line width=1.5pt] (8.,2.)-- (3.,2.);
\draw [line width=1.5pt] (5.323641819631013,2.) -- (5.5,2.226746231902983);
\draw [line width=1.5pt] (5.323641819631013,2.) -- (5.5,1.7732537680970182);
\draw [line width=1.5pt] (8.,-3.)-- (13.,-3.);
\draw [line width=1.5pt] (10.852716360737972,-3.) -- (10.676358180368986,-3.226746231902982);
\draw [line width=1.5pt] (10.852716360737972,-3.) -- (10.676358180368986,-2.7732537680970175);
\draw [line width=1.5pt] (10.5,-3.) -- (10.323641819631014,-3.226746231902982);
\draw [line width=1.5pt] (10.5,-3.) -- (10.323641819631014,-2.7732537680970175);
\draw [line width=1.5pt] (13.,-3.)-- (13.,2.);
\draw [line width=1.5pt] (12.848835845398014,-0.5) -- (13.15116415460199,-0.5);
\draw [line width=1.5pt] (13.,2.)-- (8.,2.);
\draw [line width=1.5pt] (8.,2.)-- (8.,-3.);
\draw [line width=1.5pt] (8.,-3.)-- (3.,-3.);
\draw [line width=1.5pt] (5.147283639262027,-3.) -- (5.323641819631013,-2.7732537680970175);
\draw [line width=1.5pt] (5.147283639262027,-3.) -- (5.323641819631013,-3.226746231902982);
\draw [line width=1.5pt] (5.5,-3.) -- (5.676358180368985,-2.7732537680970175);
\draw [line width=1.5pt] (5.5,-3.) -- (5.676358180368985,-3.226746231902982);
\draw [line width=1.5pt] (8.,2.)-- (13.,2.);
\draw [line width=1.5pt] (10.676358180368986,2.) -- (10.5,1.7732537680970182);
\draw [line width=1.5pt] (10.676358180368986,2.) -- (10.5,2.226746231902983);
\begin{scriptsize}
\draw [fill=ffffff] (3.,2.) circle (5pt);
\draw [fill=ffqqqq] (3.,-3.) circle (5pt);
\draw [fill=ffqqqq] (8.,-3.) circle (5pt);
\draw [fill=ffqqqq] (8.,2.) circle (5pt);
\draw [fill=ffqqqq] (13.,-3.) circle (5pt);
\draw [fill=ffffff] (13.,2.) circle (5pt);
\draw [fill=qqqqff] (4.,1.) circle (3pt);
\draw [fill=qqqqff] (5.,1.) circle (3pt);
\draw [fill=qqqqff] (6.,1.) circle (3pt);
\draw [fill=qqqqff] (7.,1.) circle (3pt);
\draw [fill=qqqqff] (9.,1.) circle (3pt);
\draw [fill=qqqqff] (10.,1.) circle (3pt);
\draw [fill=qqqqff] (11.,1.) circle (3pt);
\draw [fill=qqqqff] (12.,1.) circle (3pt);
\draw [fill=qqqqff] (3.,1.) circle (3pt);
\draw [fill=qqqqff] (8.,1.) circle (3pt);
\draw [fill=qqqqff] (13.,1.) circle (3pt);
\draw [fill=qqqqff] (4.,2.) circle (3pt);
\draw [fill=qqqqff] (4.,0.) circle (3pt);
\draw [fill=qqqqff] (4.,-1.) circle (3pt);
\draw [fill=qqqqff] (4.,-2.) circle (3pt);
\draw [fill=qqqqff] (4.,-3.) circle (3pt);
\draw [fill=qqqqff] (6.,2.) circle (3pt);
\draw [fill=qqqqff] (10.,2.) circle (3pt);
\draw [fill=qqqqff] (12.,2.) circle (3pt);
\draw [fill=qqqqff] (12.,0.) circle (3pt);
\draw [fill=qqqqff] (12.,-2.) circle (3pt);
\draw [fill=qqqqff] (12.,-3.) circle (3pt);
\draw [fill=qqqqff] (6.,0.) circle (3pt);
\draw [fill=qqqqff] (8.,0.) circle (3pt);
\draw [fill=qqqqff] (10.,0.) circle (3pt);
\draw [fill=qqqqff] (3.,-1.) circle (3pt);
\draw [fill=qqqqff] (5.,-1.) circle (3pt);
\draw [fill=qqqqff] (6.,-1.) circle (3pt);
\draw [fill=qqqqff] (7.,-1.) circle (3pt);
\draw [fill=qqqqff] (8.,-1.) circle (3pt);
\draw [fill=qqqqff] (9.,-1.) circle (3pt);
\draw [fill=qqqqff] (10.,-1.) circle (3pt);
\draw [fill=qqqqff] (11.,-1.) circle (3pt);
\draw [fill=qqqqff] (12.,-1.) circle (3pt);
\draw [fill=qqqqff] (13.,-1.) circle (3pt);
\draw [fill=qqqqff] (6.,-2.) circle (3pt);
\draw [fill=qqqqff] (8.,-2.) circle (3pt);
\draw [fill=qqqqff] (10.,-2.) circle (3pt);
\draw [fill=qqqqff] (10.,-3.) circle (3pt);
\draw [fill=qqqqff] (6.,-3.) circle (3pt);
\draw [fill=qqffqq] (5.,2.) circle (3pt);
\draw [fill=qqffqq] (7.,2.) circle (3pt);
\draw [fill=qqffqq] (9.,2.) circle (3pt);
\draw [fill=qqffqq] (11.,2.) circle (3pt);
\draw [fill=qqffqq] (11.,0.) circle (3pt);
\draw [fill=qqffqq] (13.,0.) circle (3pt);
\draw [fill=qqffqq] (9.,0.) circle (3pt);
\draw [fill=qqffqq] (7.,0.) circle (3pt);
\draw [fill=qqffqq] (5.,0.) circle (3pt);
\draw [fill=qqffqq] (3.,0.) circle (3pt);
\draw [fill=qqffqq] (3.,-2.) circle (3pt);
\draw [fill=qqffqq] (5.,-2.) circle (3pt);
\draw [fill=qqffqq] (7.,-2.) circle (3pt);
\draw [fill=qqffqq] (9.,-2.) circle (3pt);
\draw [fill=qqffqq] (11.,-2.) circle (3pt);
\draw [fill=qqffqq] (13.,-2.) circle (3pt);
\draw [fill=qqqqff] (11.,-3.) circle (3pt);
\draw [fill=qqqqff] (9.,-3.) circle (3pt);
\draw [fill=qqqqff] (7.,-3.) circle (3pt);
\draw [fill=qqqqff] (5.,-3.) circle (3pt);
\end{scriptsize}
\end{tikzpicture}
\vspace{-10pt}
\caption{The set $\A_{4}[5]$ consists two $\SLZ$ orbits: $\A^0_{4}[5] $ (green) and $\A^1_{4}[5]$ (blue).}
\label{X(2)orbits}
\end{figure}
\end{comment8}

\noindent
{\bf $\SLZ$ action.} We will now describe the action of $\SLZ$ on $\A_{4}[n]$ and give a proof of Proposition~\ref{propproof2}. Recall that the group $\SLZ$ is generated by two matrices:
\begin{align*}
S = \begin{pmatrix} 1 & -1 \\ 0 & 1 \\ \end{pmatrix}
\mbox{ and }
R = \begin{pmatrix} 0 & -1 \\ 1 & 0 \\ \end{pmatrix}.
\end{align*}

The vertex $Q_0$ is stabilized by $\SLZ$, since it is a unique non-cusp pole. It makes sense to set $Q_0$ to be the origin. Recall that we define Euclidean coordinates so that $x$-axes points to the right and $y$-axes points downward (see Figure~\ref{X(2)orbits}).

The rotation $R$ acts by permuting the squares, while rotating the left one by $-\pi/2$ and the right one by $\pi/2$. The shear $S$ acts simply by shearing the whole picture followed by a suitable cut-and-paste, which can be written in Euclidean coordinates as:
\begin{equation} \label{unipotent2}
\displaystyle S :\left(\frac an, \frac bn\right) \mapsto \left(\frac an + \frac bn ,\frac bn\right).
\end{equation}
Compare this to Theorem~\ref{thmunipotent}.

\begin{proof}[Proof of Proposition~\ref{propproof2}]
Since $\SLZ$ acts transitively on $E_0[n]^*$, it suffices to show that the points 
$$
X_0 = \left(0, \frac 1n\right), X_1 = \left(1, \frac 1n\right), X_2 = \left(0, \frac{n-1}n\right), X_3 = \left(1, \frac{n-1}n\right)
$$ 
are connected by the elements of $\SLZ$. For any $n>1$ we obtain:
$$
S^{n}:X_0 = \left(0, \frac 1n\right) \mapsto \left(1, \frac 1n\right) = X_1,
$$
$$
X_1 = \left(1, \frac1n\right) \xmapsto{R} \left(\frac{2n-1}n,1\right)=\left(\frac1n, 1\right) \xmapsto{S} \left(\frac{n+1}n,1\right) \xmapsto{R} \left(1, \frac{n-1}n\right) = X_3.
$$
For any even $n$ we have:
$$
S^{n}:X_2 = \left(0, \frac{n-1}n\right) \mapsto  \left(n-1, \frac{n-1}n\right) = \left(1, \frac{n-1}n\right) = X_3.
$$ 
It follows that for any even $n$ there a single $\SLZ$ orbit that coincides with $\A_{4}[n]$, and from Proposition~\ref{propIWPvalues} for any odd $n>1$ there are exactly two orbits $\A^0_{4}[n]$ and $\A^1_{4}[n]$ distinguished by the spin invariant.
\end{proof}

\begin{remark}
One can use a similar approach to prove the conjecture for $d=3$: construct $\A_9$ using gluing instructions from \S\ref{sectiling} and analyze the action of $R$ and $S$ on $\A_9[n]$. However we will use a more powerful result about the illumination (see \S\ref{secillumination}), which will imply the main result for $d=3$ and $5$.
\end{remark}


\section{Proof for $d, n$ prime and $n > (d^3-d)/4$} \label{secproofgen}
In this section we give a proof of the parity conjecture for prime $d, n$ and $n > (d^3-d)/4$. We will show:

\begin{theorem} \label{thmprime}
For any prime $d$ and any prime $n>(d^3-d)/4$, the set $\A_{d^2}[n]$ consists of two $\SLZ$ orbits $\A^0_{d^2}[n]$ and $\A^1_{d^2}[n]$.
\end{theorem} 

The proof for $d=2$ was given in \S\ref{secproof2}. For any prime $d>2$ and any prime $n>(d^3-d)/4$ the strategy consists of two steps:

\begin{enumerate}
\item Show that every orbit in $\A_{d^2}[n]$ has a representative in every square.

\item Show that all points of $\A_{d^2}[n]$ in the interior of the lighthouse $\LL_1 = (1,1,1,d-1)$ fall into 2 orbits.
\end{enumerate}

These clearly implies the conjecture. Start with any $z \in \A_{d^2}[n]$. The step (1) implies that there exists $A \in \SLZ$ such that $A(x)$ is in $\LL_1$. The step (2) then implies that $A(x)$, and hence $x$, belongs to one of the two orbits. We proceed to give proofs of the two steps.

\noindent
{\bf Step 1.} Consider a point $z \in \A_{d^2}[n]$ inside a cylinder $\CC =  (w_1,s_1,w_2,s_2)$ with the Euclidean coordinates $\displaystyle \left(\frac{a}n , \frac{b}n\right)$ with $\gcd(b,n) =1$. From Theorem~\ref{thmunipotent} it follows that the shortest distance $s$ between the points in the orbit $U \cdot z$ is:
$$
s = \frac1n \cdot \gcd(b + T_{\CC} (w_1+w_2) n, w_1 w_2 (w_1+w_2) n) = \frac1n \cdot \gcd(b + T_{\CC} (w_1+w_2) n, w_1 w_2 (w_1+w_2)),
$$
since $n$ is prime and $\gcd(b,n)=1$. The maximal circumference $w_1 w_2 (w_1+w_2)$ over all cylinders is achieved for an eave with the circumference $k(d-k)d$, where $k = \frac{d-1}2$, and therefore:
$$w_1 w_2 (w_1+w_2) \le \frac{d-1}2\cdot \frac{d+1}2 \cdot d = \frac{d^3 - d}4,$$
which is less than $n$ by the assumption.
It implies that:
$$
n\cdot s =  \gcd(b + T_{\CC} (w_1+w_2) n, w_1 w_2 (w_1+w_2) ) \le w_1 w_2 (w_1+w_2) \le \frac{d^3 - d}4 < n.
$$
Therefore $s < 1$, as long as $\gcd(b,n)=1$.

Now consider an arbitrary point $x \in \A_{d^2}[n]$. Because $\SLZ$ action on $E_0[n]^*$ is transitive it can be sent into the interior of a square in $\A_{d^2}$. Any point that lies in the interior of the squares and has Euclidean coordinates $\left(\frac{a}n , \frac{b}n\right)$ satisfies $\gcd(b,n) =1$, since $n$ is prime. We can now use the above observation: by applying a suitable power $k$ of $S$ it can be sent into the interior of the next square to the right, and similarly, by applying $R \circ S^k \circ R^{-1}$, where $R = \begin{pmatrix} 0 & -1 \\ 1 & 0 \\ \end{pmatrix}$, it can be sent to the next square above. This implies that $x$ can be sent into interior of any square, which finishes the proof of the step 1.

\noindent
{\bf Step 2.} For the lighthouse $\LL_1 = (w_1,s_1,w_2,s_2)= (1,1,1,d-1)$, which consists of two squares, we obtain as in (\ref{unipotent2}):
$$
\displaystyle S :\left(\frac an, \frac bn\right) \mapsto \left(\frac an + \frac bn ,\frac bn\right).
$$
Then the number of $U$-orbits of points $(x, \frac bn) \in \A_{d^2}[n] \cap \LL_1$, where $1 \le b \le n-1$ is:
$$
\gcd(b ,2n ) = \gcd(b,2),
$$
since $n$ is prime. For each $1 \le i \le (n-1)/2$ define the following $U$-orbits:
\begin{align}
N^1(2i-1) = & \left\{  \left(\frac an, \frac bn\right) \in \A_{d^2}[n] \cap \LL_1 \ \middle\vert \   b = 2i-1 \equiv 1 \mod 2  \right\}, \nonumber  \\
N^1(2i) = & \left\{  \left(\frac an, \frac bn\right) \in \A_{d^2}[n]  \cap \LL_1 \ \middle\vert \   b = 2i \equiv 0 \mod 2 \mbox{ and } a  \equiv 0 \mod 2 \right\}, \label{lighthouseorbits} \\
N^0(2i) = & \left\{  \left(\frac an, \frac bn\right) \in \A_{d^2}[n]  \cap \LL_1 \ \middle\vert \   b = 2i \equiv 0 \mod 2 \mbox{ and } a  \equiv 1 \mod 2 \right\}. \nonumber
\end{align}  
It remains to show that:
\begin{enumerate}
\item[(a)] all $N^1(2i-1)$ and $N^1(2i)$ belong to the same $\SLZ$ orbit; and 

\item[(b)] all $N^0(2i)$ belong to another $\SLZ$ orbit.
\end{enumerate}

From Theorem~\ref{thmrotation} we have that $R$ sends the non-cusp pole $O$ with cylinder coordinates $(w_1,s_1,w_2,s_2,t_1,t_2,t_3,h_3)= (1,1,1,d-1,0,0,0,0)$ to the non-cusp pole $R(O)$ with cylinder coordinates $ (1,1,d-1,1,0,0,0,0)$ that lies on the top of the eave $\EE_1$. The lighthouse $\LL_1$ itself is sent to two squares $R(\LL_1)$ of that eave $\EE_1$ that are adjacent to $R(O)$.
For each $1 \le i \le (n-1)/2$ the sets $N^1(2i-1), N^1(2i)$ and $N^0(2i)$ are sent by $R$ to:
\begin{align*}
& \left\{  \left(\pm \frac bn, \frac an\right) \in \A_{d^2}[n] \cap R( \LL) \ \middle\vert \   b = 2i-1 \equiv 1 \mod 2  \right\},   \\
& \left\{  \left(\pm \frac bn, \frac an\right) \in \A_{d^2}[n]  \cap R( \LL) \ \middle\vert \   b = 2i \equiv 0 \mod 2 \mbox{ and } a  \equiv 0 \mod 2 \right\}, \\
& \left\{  \left(\pm \frac bn, \frac an\right) \in \A_{d^2}[n]  \cap R( \LL) \ \middle\vert \   b = 2i \equiv 0 \mod 2 \mbox{ and } a  \equiv 1 \mod 2 \right\}.
\end{align*}

To show (a) and (b) it suffices to prove that there exists:
\begin{enumerate}
\item[(A)] an even $a$ such that all points $\left(x, \frac an\right) \in \A_{d^2}[n] \cap R(\LL)$ form a single $U$-orbit; and

\item[(B)] an odd $a$ such that all points $\left(x, \frac an\right) \in \A_{d^2}[n] \cap R(\LL)$ form two $U$-orbits.
\end{enumerate}

Theorem~\ref{thmunipotent} implies that, for $1 \le a \le n-1$, the number $\nu$ of such $U$-orbits is:
\begin{align*}
\nu = & \gcd(a +  n d \cdot T_1, w_1 w_2 (w_1+w_2) n) = \gcd(a +  n d\cdot T_1, w_1 w_2 (w_1+w_2)) =\\
	& =  \gcd(a +  n d \cdot T_{1}, (d-1)d) =  \gcd(a +  n d \cdot T_{1}, d-1)\cdot  \gcd(a, d),
\end{align*}
where $T_1 \cong -1 \mod (d-1)$. Hence the number of orbits is:
$$
 \nu = \gcd(a - dn, d-1) \cdot  \gcd(a, d) =  \gcd(a - n, d-1) \cdot  \gcd(a, d).
$$
To show (A) let $a = n - d$. It is even and satisfies $1 \le a \le n-1$ and:
$$
 \nu = \gcd(- d , d-1) \cdot  \gcd(n-d, d) =  \gcd(d , d-1) \cdot  \gcd(n, d) = 1.
$$
This completes the proof of (A). 
To show (B) let $a = 2k+1$, where $0 \le k < (n-1)/2$. Then we have:
$$
\nu =  \gcd(2k + 1 -n, d-1)  \cdot  \gcd(2k+1, d) = 2\cdot \gcd(k + (1-n)/2, (d-1)/2)  \cdot  \gcd(2k+1, d).
$$
When $k$ ranges from $0$ to $(d-1)/2< (n-1)/2$, then $\gcd(2k+1, d) =1$ since $d$ is prime, and 
$k + (1-n)/2$ runs through all possible remainders modulo $(d-1)/2$ including $1$, as long as $d\ge 3$. This proves (B) and hence completes the proof of Theorem~\ref{thmprime}. \qed


\section{Everything is illuminated} \label{secillumination}

Illumination problem asks whether all of the translation surface is illuminated by a given point. We say that a point $x \in \A_{d^2}$ is illuminated by a subset $S \subset \A_{d^2}$ if there exists a geodesic segment of $(\A_{d^2},q)$ that starts at some $y\in S$, ends at $x$ and does not pass through singularities. We formulate a conjecture:

\begin{conjecture}[Illumination conjecture] \label{conjillumination}
Light sources at the cusps of the modular curve illuminate all of $X(d)$ except possibly for some of the vertices of the square-tiling.
\end{conjecture} 
We then show that it implies the parity conjecture for prime $d$ and $n>1$:

\begin{theorem}[Illumination conjecture implies parity conjecture] \label{thmillumination}
Assume $d$ is prime and $n>1$. If every $x \in \A_{d^2}[n]$ is illuminated by the set of the cusp poles $P_{ns}(d)$ then $\A_{d^2}[n]$ consists of a single $\SLZ$ orbit when $n$ is even and two $\SLZ$ orbits when $n$ is odd.
\end{theorem}

In this section we will prove Theorem~\ref{thmillumination} and use it to prove the main result for any prime $d$ and all sufficiently large $n$. In \S\ref{secproof35} we will establish the illumination conjecture for $d=3,4$ and $5$ and use it together with Theorem~\ref{thmillumination} to prove the parity conjecture for $d=3$~and~$5$.


%

\noindent
{\bf Background on illumination problem.} The illumination problem goes back to Roger Penrose (1958) and George Tokarsky (1995) who constructed the first examples of rooms with mirror walls (on which light reflects) that are not illuminated by a light source at some point of the room. Penrose's example uses walls in a shape of ellipse and has open subsets that are not illuminated. Tokarsky's example is polygonal, however there is only one point that is not illuminated. Recent works \cite{HST08} and \cite{LMW16} used the $\GL$ action on $\Omega\M_g$ to further investigate this question. In particular, it was shown that for any translation surface and a light source on it the set of points that are not illuminated is finite. Thus we obtain a corollary of Theorem~\ref{thmillumination}:
\begin{corollary}
For any prime $d$ the parity conjecture is true for $n > C_d$, where $C_d$ is some constant that depends on $d$.
\end{corollary} 

\noindent
{\bf Parity implies illumination.} It is also known that the set of points unilluminated by a vertex of a square-tiled surface is a subset of $n$-rational points of the squares for some $n$. If parity conjecture is true, then Proposition~\ref{propIWPvalues} implies that every $\SLZ$ orbit $\A_{d^2}[n]$ has a representative in each square of $\A_{d^2}$. Since illumination is $\SLZ$ invariant property, it implies that every point in $\A_{d^2}[n]$ is illuminated by cusp. Therefore, the parity conjecture (Conjecture~\ref{parity}) implies the illumination conjecture (Conjecture~\ref{conjillumination}) for all $d$ and $n>1$.

The proof of the converse statement is harder and will extensively use the structure of the square-tiling of $\A_{d^2}$ for prime $d$. Note that we only show that the illumination conjecture implies the parity conjecture for $n>1$. It is not known if the same applies for $n=1$ and our methods do not extend to this case.

\noindent
{\bf Idea of proof of Theorem~\ref{thmillumination}.}
We will break the proof into the following three steps. We call a subset $S \subset \A_{d^2}[n]$ an {\em $\SLZ$-subset} if any two points in $S$ are connected by an element of $\SLZ$.
\begin{enumerate}
\item[\Rmnum{1}.\ \ ] We will first show, using illumination conjecture, that any point of $\A_{d^2}[n]$ can be sent via $\SLZ$ to a point in some lighthouse.

\item[\Rmnum{2}.\ ] Then, in any given lighthouse $\LL$, we will describe $U$-orbits of the points in $\A_{d^2}[n] \cap \LL$ and use $U$ action on the eaves to connect these orbits into one or two $\SLZ$-subsets.

\item[\Rmnum{3}.] Finally, we will use $U$ and $R$ action on the eaves to show that for any two lighthouses $\LL_i$ and $\LL_j$ these subsets are connected by $\SLZ$.
\end{enumerate} 

We begin by analyzing unipotent orbits in the eave cylinders and proving a lemma that will be used in the proof of Theorem~\ref{thmillumination}.

\noindent
{\bf Unipotent orbits in eaves.}
Let $\EE_k$ be the eave $(k,  1, d-k, 1)$. Recall form Theorem~\ref{thmunipotent} that for any $\left( \frac an, \frac bn \right)  \in \A_{d^2}[n] \cap \EE_k$:
\begin{equation*}
S : \left( \frac{a}n, \frac{b}n \right) \mapsto \left( \frac{a}n + \frac{b}n + T_{k}\cdot d  , \frac{b}n \right),
\end{equation*}
where $T_k \cdot d \equiv -1 \mod k(d-k)$. For any given $1 \le b \le n-1$ denote the number of $U$-orbits of points $\left( \frac an, \frac bn \right) \in \A_{d^2}[n] \cap \EE_k$ by $\nu_k(b)$. When $\gcd(b,n)=1$:
\begin{multline*} 
\nu_k(b) = \gcd( b + T_{\CC}dn, k(d-k)d n) = \gcd( b + T_{\CC}dn, k(d-k)d) = \\
= \gcd( b + T_{\CC}dn, k(d-k)) \cdot \gcd( b + T_{\CC}dn, d).
\end{multline*}
Note that the last equality used the primality of $d$. We obtain a formula: 
\begin{equation} \label{eaveunip}
\nu_k(b) = \gcd( b - n, k(d-k)) \cdot \gcd( b , d), \mbox{ whenever } \gcd(b,n)=1.
\end{equation}

\begin{lemma} \label{lemmaillumination}
Assume $d>2$ is prime and $n>1$.
\begin{enumerate}
\item[(1)] There exists an integer $1 \le s_1 \le n-1$ such that $\gcd(s_1,n)=1$ and $\nu_k(s_1)=1$.

\item[(2)] If $n \ne d+ 2 $ is odd, then there exists an odd integer $1 \le s_2 \le n-1$ such that $\gcd(s_2,n)=1$ and  $\nu_k(s_2)=2$.

\item[(3)] Assume $n = d+2$. As long as $(d,k) \ne (7,3)$ or $(19,4)$, there exists an odd integer $1 \le s_3 \le n-1$ such that $\gcd(s_3,n)=1$ and $\nu_k(s_3)=2$.
\end{enumerate}

\end{lemma}

\begin{proof}

We will use formula (\ref{eaveunip}) for the proof.

(1) If $n \not\equiv 1 \mod d$ setting $s_1=n-1$ we obtain:
$$
\gcd(s_1,n) = 1 \mbox{ and } \nu_k(s_1) = \gcd( 1, k(d-k)) \cdot \gcd( n-1 , d) = 1,
$$
and if $n \equiv 1 \mod d$ setting $s_1=n-d$ (note that $1 \le s_1 \le n-1$ since $n \ge d+1$) we obtain:
$$
\gcd(s_1,n) = 1 \mbox{ and } \nu_k(s_1) = \gcd( d, k(d-k)) \cdot \gcd( n-d , d) = 1.
$$

(2) If $n \not\equiv 2 \mod d$ setting $s_2=n-2$ we obtain:
$$
\gcd(s_2,n) = 1 \mbox{ and } \nu_k(s_2) = \gcd(2, k(d-k)) \cdot \gcd( n-2 , d) = 2.
$$
If $n \equiv 2 \mod d$ and $n \ne d+2$ then $n \ge 2d+2$. Thus $s_2=n-2d$ satisfies $1 \le s_2 \le n-1$ and we obtain:
$$
\gcd(s_2,n) = 1 \mbox{ and } \nu_k(s_2) = \gcd( 2d, k(d-k)) \cdot \gcd( n-2d , d) = 2.
$$
Note that $s_2$ is odd in both cases.

(3) When $n=d+2$:
$$
\nu_1(1) = \gcd(d+1, d-1) \gcd(1,d) = 2, \nu_2(3) = \gcd(d-1,2(d-2))\gcd(3,d) = 2,
$$
since $k>1$ implies $d>3$. If $k>2$ then we change the unknown variable $s_3$ to $t = \frac{n-s_3}2$. Then if $t$ is integer and if $2 \le t \le \frac{n-1}2$ then $s_3$ is odd and $1 \le s_3 \le n-4 = d -2$. In particular $\gcd(s_3,d)=1$, since $d$ is prime. In addition if:
\begin{equation} \label{find}
\gcd\left( t, \frac{k(d-k)}2\right) = \gcd(t,n) = 1,
\end{equation}
then:
$$
\gcd(s_3, n) = 1 \mbox{ and } \nu_k(s_3) = \gcd( n - s_3, k(d-k)) \gcd( s_3, d) = 2 \gcd\left(t, \frac{k(d-k)}2\right) = 2.
$$
Let $\pi(x)$ be the number of primes less than $x$ and $\mu(x)$ be the number of prime factors of $x$.
We are going to show that there exists a prime $2 \le t \le \frac{n-1}2$ satisfying (\ref{find}) and the corresponding $s_3 = n - 2t$ will satisfy part (3) of the lemma.  We will do so by showing:
$$
\pi\left(\frac{n-1}2\right) \ge \mu(n) + \mu\left(\frac{k(d-k)}2\right),
$$
for large $d$, and running a program for small $d$, which shows that the only exceptions are $(d,k) = (7,3)$ or $(19,4)$.

Note that $\mu(x) \le \log_2(x)$ and $\displaystyle \pi(x) \ge \frac{x}{\log(x) +2}$ for $x \ge 55$ (see \cite{Ros41}). Let $x_0 = \frac{d+1}2$, then we have:
$$
\pi\left(\frac{n-1}2\right) = \pi\left(\frac{d+1}2\right) \ge \frac{x_0}{\log(x_0) +2}.
$$
Because $\frac{k(d-k)}2 \le \frac{(d-1)(d+1)}8$ we obtain:
$$
 \mu(n) + \mu\left(\frac{k(d-k)}2\right) \le \log_2(d+2) + \log_2\left(\frac{(d-1)(d+1)}8 \right)= \log\left( \frac{(d-1)(d+1)(d+2)}8 \right) / \log(2).
$$
Because $(d-1)(d+2) \le (d+1)^2$ for positive $d$ we have:
$$
 \mu(n) + \mu\left(\frac{k(d-k)}2\right) \le \frac{3}{\log(2)} \log(x_0).
$$
If $d > 800$ then $x_0 > 400$ and:
$$\frac{3}{\log(2)} \log(x_0) < \frac{x_0}{\log(x_0) +2}.$$
Therefore there exists a prime $2 \le t \le \frac{n-1}2$ satisfying (\ref{find}). 

For $d<800$ we run a computer program that finds a required $t$ for any prime $d$ and arbitrary $1 \le k \le \frac{d-1}2$, such that $(d,k) \ne (7,3)$ or $(19,4)$ (these cases are analyzed separately below). This finishes the proof of the lemma.
\end{proof}

\noindent
{\bf Step I: into a lighthouse.}
Recall from \S\ref{secabs} and \S\ref{secmodular} that $P_{ns}(d) \subset \A_{d^2}$ is the set of cusp poles, points of $\A_{d^2}$ corresponding to the square-tiled surfaces with a non-separating node. Recall that $\LL_k$ denotes the lighthouse cylinder $(w_1=1, s_1= k, w_2=1, s_2 = d-k)$.

\begin{proposition}[\Rmnum{1}. Into a lighthouse] \label{into}
Assume $d>2$ is prime and $\A_{d^2}[n]$ is illuminated by $P_{ns}(d)$. Then every $\SLZ$ orbit in $\A_{d^2}[n]$ has a representative $\left( 1, \frac bn \right) \in \LL_k \cap \A_{d^2}[n]$ for some $b\in \Z$ with $\gcd(b,n)=1$. \end{proposition} 

\begin{proof}
Take any $x \in \A_{d^2}[n]$ and a straight line segment $s$ that connects it to some $y \in P_{ns}(d)$. Then $\int_s \pm \sqrt q \in \frac1n \Z[i]^*$, the set of primitive $n$-rational points of $\C$. Choose a matrix $A \in \SLZ$ that sends it to a purely real number $\frac bn$ for some $b \in \Z$. Then $A \cdot s$ is a horizontal segment connecting a cusp pole $ A \cdot y \in P_{ns}(d)$ to a point $ A \cdot x \in \A_{d^2}[n]$. Note that such line segments can only be found (1) on the bottom boundary of an eave or (2) on the top boundary of a lighthouse. 

(1) Assume first $A \cdot x$ lies on the bottom boundary of $\EE= (k,1,d-k,1)$ for some $k$. The unipotent subgroup $U$ acts transitively on the cusp poles of $\EE_k$ and hence an appropriate power $i$ of $S$ sends $A \cdot y$ to an cusp pole $(k,1,d-k,1,0,0,0)$ and $(S^i \circ A) \cdot x$ is joined to it by a horizontal line. Rotating this pole by $\pi/2$ we obtain an cusp pole $(1,k,1,d-k,1,1,0)$, which belongs to the lighthouse $\LL_k$. Therefore:
$$
(R \circ S^i \circ A) \cdot x = \left( 1, \frac bn \right) \in \LL_k.
$$
Note that $\left( 1, \frac bn \right) = \left( \frac{n}{n}, \frac bn \right) $ is a primitive $n$-rational point of a square and therefore $\gcd(n,b,n) = \gcd(b,n)=1$. This finishes the proof in case (1).

(2) Now assume $A \cdot x$ lies on the top boundary of some lighthouse $\LL$. Rotating it by $\pi/2$ one obtains a point $(R \circ A)\cdot x = (0, \frac bn)$ on a vertical edge of some eave cylinder $\EE_k=(w_1=k, s_1= 1, w_2=d-k, s_2 = 1)$. Since $\gcd(b,n)=1$ from (\ref{eaveunip}) we obtain:
$$
\nu_k(b) = \gcd( b - n, k(d-k)) \cdot \gcd( b , d).
$$
Assume first that $ \gcd( b , d) = 1$. Then $\nu_k(b) \big\vert k(d-k)$ and for a suitable power $i$ of $S$:
$$
(S^i \circ R \circ A) \cdot x = \left(k(d-k), \frac bn\right) \in \EE_k.
$$
It is a point, which lies directly above the cusp pole with cylinder coordinates:
$$
(w_1=k, s_1= 1, w_2=d-k, s_2 = 1, t_1=0, t_2=0, t_3 = k(d-k) \% d) \in \EE_k,
$$
and Euclidean coordinates:
$$
\left( k(d-k) , 1 \right) \in \EE_k.
$$
The rotation by $\pi/2$ sends this cusp pole to some other cusp pole on the bottom boundary of some eave cylinder $\EE'$ and hence $(R \circ S^i \circ R \circ A) \cdot x$ belongs a horizontal segment that starts at the cusp pole on the bottom boundary of the eave $\EE'$, which brings us back to case (1).

It remains to treat the case (2) when $ \gcd( b , d) \ne 1$ or equivalently $b= r d$ for some $r \in \Z$. In that case $(R \circ A)\cdot x = (0, \frac bn) = (0, \frac {rd}n)$ belongs to a vertical edge of $\EE_k$. Let $s$ be a line segment contained in $\EE_k$ that connects $(R \circ A)\cdot x = (0, \frac {rd}n) \in \EE_k$ and the cusp pole $\left( k(d-k) , 1 \right) \in \EE_k$.
Let $v$ be given by:
$$v = \int_s \pm \sqrt q = \pm \left( \frac {k(d-k)n}n + i \cdot \frac {n-rd}n \right) \in \frac1n \Z[i]^*.$$
Note that $v$ is a $\gcd ( k(d-k) n , n-rd )$ multiple of a primitive element in $\frac1n \Z[i] / \Z[i]$. Also note that $1 = \gcd( b, n) = \gcd ( r d , n ) \implies \gcd(d,n) = 1$ and hence:
$$
\displaystyle d \nmid \gcd ( k(d-k) n , n-rd ).
$$
Therefore for any matrix $A' \in \SLZ$ that sends $z$ to a purely real number the following holds: $A' \cdot z = \frac{b'}n i$, where $d \nmid b' \in \Z$. Therefore $\gcd(b',d) = 1$ and it brings us either to the case (1) or the case (2), which were discussed above.
\end{proof}

\noindent
{\bf Step II: within a lighthouse.}
For the next proposition we introduce notation similar to the one of (\ref{lighthouseorbits}). Let $k$ be any integer such that $1\le k \le (d-1)/2$. For even $n$ and any integer $1\le b \le kn-1$ such that $\gcd(b,n)=1$ (in particular $b$ is odd) define a set:
$$
N_k(b) = \left\{  \left(\frac an, \frac bn\right) \in \A_{d^2}[n] \cap \LL_k \ \middle\vert \ 0 \le a \le 2n \right\}.
$$
For odd $n>1$ and any integer $1\le b \le kn-1$ such that $\gcd(b,n)=1$ define sets:
$$
N^1_k(b)= 
\begin{cases}
 \left\{  \left(\frac an, \frac bn\right) \in \A_{d^2}[n] \cap \LL_k \ \middle\vert \   0 \le a \le 2n  \right\} & \mbox{ when $b$ is odd}  \\
 & \\
  \left\{  \left(\frac an, \frac bn\right) \in \A_{d^2}[n]  \cap \LL_k \ \middle\vert \  0 \le a \le 2n, a \equiv 0 \mod 2 \right\} & \mbox{ when $b$ is even} 
\end{cases}
$$
$$
N^0_k(b) = \left\{  \left(\frac an, \frac bn\right) \in \A_{d^2}[n]  \cap \LL_k \ \middle\vert \   0 \le a \le 2n, a \equiv 1 \mod 2 \right\} \mbox{ \ \ \ when $b$ is even}.
$$

\begin{proposition}[\Rmnum{2}. Within a lighthouse] \label{within} 
Assume $d>2$ is prime and $n>1$. Let $k$ be any integer satisfying $1\le k \le (d-1)/2$. Let us call integer $b$ admissible if $1\le b \le kn-1$ and $\gcd(b,n)=1$.  Then for $N_k(b), N^\epsilon_k(b) \subset \A_{d^2}[n]$ defined above the following holds: 
\begin{enumerate}
\item[(1)] When $n$ is even, the union of $N_k(b)$ over all admissible $b$ belongs to a single $\SLZ$ orbit. 

\item[(2)] When $n$ is odd, the union of all $N^1_k(b)$ over all admissible $b$ belongs to a single $\SLZ$ orbit.

\item[(3)] When $n$ is odd, the union of all $N^0_k(b)$ over all admissible $b$ belongs to a single $\SLZ$ orbit.
\end{enumerate}
\end{proposition}

\begin{proof}
For the lighthouse $\LL_k = (w_1=1,s_1=k,w_2=1,s_2=d-k)$, similarly to (\ref{unipotent2}) we have:
$$
\displaystyle S :\left(\frac an, \frac bn\right) \mapsto \left(\frac an + \frac bn ,\frac bn\right),
$$
and the number of $U$-orbits of points $\left(\frac an, \frac bn\right) \in \A_{d^2}[n] \cap \LL_k$, where $1 \le b \le kn-1$ and $\gcd(b,n)=1$:
$$
\gcd(b,2n) = \gcd(b,2).
$$
It implies that each of the sets $N_k(b), N^1_k(b)$ and $N^0_k(b)$ is itself a single $U$-orbit. Indeed:
(1) when $n$ is even, $b$ has to be odd since $\gcd(b,n)=1$ and then $\gcd(b,2)=1$;
(2) when $n>1$ and $b$ are both odd, $\gcd(b,2)=1$;
(3) when $n>1$ is odd and $b$ is even, $\gcd(b,2)=2$.

It remains to show that the union of $N_k(b)$ or $N^\epsilon_k(b)$ over all $b$ belongs to a single $\SLZ$ orbit. Rotation $R$ sends the lighthouse $\LL_k$ to $2k$ squares $R( \LL_k)$ of the eave cylinder $\LL_k$ that are adjacent to the non-cusp pole $(w_1=k,s_1=1,w_2=d-k,s_2=1,t_1=0,t_2=0,t_3=0)$. 
The images of the above subsets are the following $\SLZ$ subsets. For even $n$:
$$
R(N_k(b)) = \left\{  \left(\pm \frac bn, \frac an\right) \in \A_{d^2}[n] \cap \EE_k \ \middle\vert \ 0 \le a \le n \right\},
$$
and for odd $n>1$:
$$
R(N^1_k(b))= 
\begin{cases}
 \left\{  \left(\pm \frac bn, \frac an\right) \in \A_{d^2}[n] \cap \EE_k \ \middle\vert \   0 \le a \le n  \right\} & \mbox{ when } b\equiv 1 \mod 2  \\
 & \\
  \left\{  \left(\pm \frac bn, \frac an\right) \in \A_{d^2}[n]  \cap \EE_k \ \middle\vert \  0 \le a \le n, a \mbox{ is even} \right\} & \mbox{ when } b \equiv 0 \mod 2
\end{cases}
$$
$$
R(N^0_k(b)) = \left\{  \left(\pm \frac bn, \frac an\right) \in \A_{d^2}[n]  \cap \EE_k \ \middle\vert \   0 \le a \le n, a \mbox{ is odd} \right\} \mbox{ \ \ \ when $b \equiv 0 \mod 2$}.
$$
Lemma~\ref{lemmaillumination} (1) implies that for any $n>1$ there is an (necessarily even) integer $1 \le s_1 \le n-1$, such that $\gcd(s_1,n) = 1$ and all points:
$$
 \left(\pm \frac bn, \frac {s_1}n\right) \in \A_{d^2}[n] \cap \EE_k
$$
belong to the same $U$-orbit. In particular, that proves (a) when $n$ is even and (b) when $n$ is odd. It remains shows that $R(N^0_k(b))$ are all in one orbit when $n$ is odd.

If $n$ is odd and $(d,k) \ne (7,3)$ or $(19,4)$, Lemma~\ref{lemmaillumination} (2) and (3) imply that there exists an odd integer $1 \le s_2 \le n-1$, such that $\gcd(s_2,n) = 1$ and all points:
$$
 \left\{  \left(\pm \frac bn, \frac {s_2}n\right) \in \A_{d^2}[n] \cap \EE_k \ \middle\vert \   b \equiv 1 \mod 2  \right\}
$$
belongs to the same $U$-orbit.

It remains to analyze the cases $(d,k) = (7,3)$ and $(19,4)$. We start with $(d,k) = (7,3)$. Note that $n = 9$ and $\nu_k(b) = \gcd(9-b,12) \gcd(b,7)$. Then we have:
$$
\nu_k(1) = 4 \mbox{ and } \nu_k(7) = 14.
$$
Note that $\nu_k(1) = 4$ implies that all $(x, 1)$ with even $x$ fall into two unipotent orbits generated by $(2,1)$ and $(4,1)$. It suffices to show that $(2,1)  \in R(N^0_k(2))$ and $(4,1)  \in R(N^0_k(4))$ are in the same $\SLZ$ orbit. Note that $(2,1)$ is in the same unipotent orbit with $(18,1) \in R(N^0_k(18))$, which is in the same $\SLZ$ subset $R(N^0_k(18))$ with $(18,7)$. Next $(18,7)$ and $(4,7)$ are in the same unipotent orbit, since $\nu_k(7) = 14$. And finally $(4,7)$ and $(4,1)$ are in the same $\SLZ$ subset $R(N^0_k(4))$, which finishes the proof for $(d,k) = (7,3)$.

For $(d,k) = (19,4)$, note that $n=21$ and $\nu_k(b) = \gcd(21-b,60) \gcd(b,19)$. Then we have:
$$
\nu_k(5) =  4 \mbox{ and } \nu_k(11) = 10.
$$
Similarly all $(x, 5)$ with even $x$ fall into two unipotent orbits generated by $(2,5)$ and $(4,5)$. Note that $(2,5)$ is in the same unipotent orbit with $(14,5)$, since $\nu_k(5) =  4$. Points $(14,5)$ and $(14,11)$ are in the same $\SLZ$ subset $R(N^0_k(14))$. Next $(14,11)$ and $(4,11)$ are in the same unipotent orbit, since $\nu_k(11) = 10$. And finally $(4,11)$ and $(4,1)$ are in the same $\SLZ$ subset $R(N^0_k(4))$, which finishes the proof for $(d,k) = (19,4)$.
\end{proof}

\begin{proposition}[\Rmnum{3}. Between lighthouses] \label{between}
Assume $d>2$ is prime and $n>1$. Let us call an integer $b$ admissible if $1\le b \le kn-1$ and $\gcd(b,n)=1$. Let $N_k(b), N^\epsilon_k(b) \subset \A_{d^2}[n]$ be as above. 
\begin{enumerate}
\item[(a)] When $n$ is even, the union of $N_k(b)$ over all admissible $b$ and $1\le k \le (d-1)/2$ belongs to a single $\SLZ$ orbit $\A_{d^2}[n]$. 

\item[(b)] When $n$ is odd, the union of all $N^1_k(b)$ over all admissible $b$ and $1\le k \le (d-1)/2$ belongs to a single $\SLZ$ orbit $\A^1_{d^2}[n]$.

\item[(c)] When $n$ is odd, the union of all $N^0_k(b)$ over all admissible $b$ and $1\le k \le (d-1)/2$ belongs to a single $\SLZ$ orbit $\A^0_{d^2}[n]$.
\end{enumerate}
\end{proposition} 

\begin{proof}
Note that in the course of the proof of Proposition~\ref{within} we showed that every eave cylinder $\EE_k$ contains a horizontal line $H_k \subset \EE_k$ that satisfies:
\begin{center}
\begin{tabular}{>{\centering\arraybackslash}p{0.7\textwidth}}
all points $H_k \cap \A_{d^2}[n]$ belong to the same $\SLZ$ orbit,
\end{tabular}
\end{center}
and a horizontal line $h_k \subset \EE_k$ that satisfies:
\begin{center}
\begin{tabular}{>{\centering\arraybackslash}p{0.7\textwidth}}
all points $h_k \cap \A_{d^2}[n]$ belong to two $\SLZ$ orbits \\
depending on the parity of their $x$-coordinate.
\end{tabular}
\end{center}
Therefore to show that each of the unions of $N_k(b), N^1_k(b)$ or $N^0_k(b)$ over all $k$ is in a single $\SLZ$ orbit it suffices to show that for any $k$:
\begin{equation} \label{eavesconnect}
R(\EE_k) \cap \EE_1 \ne \varnothing.
\end{equation}

\begin{comment10}
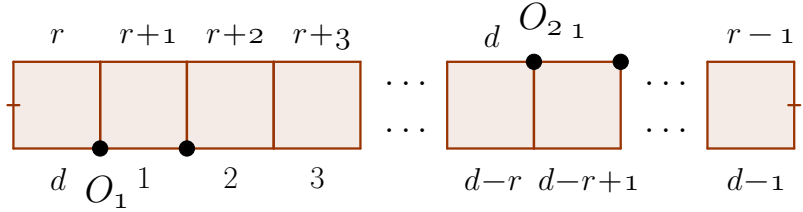
\begin{figure}[H] 
\centering
\definecolor{zzttqq}{rgb}{0.6,0.2,0.}
\begin{tikzpicture}[line cap=round,line join=round,>=triangle 45,x=1.0cm,y=1.0cm]
\clip(-1.,-3.) rectangle (19.,3.);
\fill[line width=1.5pt,color=zzttqq,fill=zzttqq,fill opacity=0.10000000149011612] (0.,-1.) -- (0.,1.) -- (2.,1.) -- (2.,-1.) -- cycle;
\fill[line width=1.5pt,color=zzttqq,fill=zzttqq,fill opacity=0.10000000149011612] (2.,-1.) -- (2.,1.) -- (4.,1.) -- (4.,-1.) -- cycle;
\fill[line width=1.5pt,color=zzttqq,fill=zzttqq,fill opacity=0.10000000149011612] (4.,-1.) -- (4.,1.) -- (6.,1.) -- (6.,-1.) -- cycle;
\fill[line width=1.5pt,color=zzttqq,fill=zzttqq,fill opacity=0.10000000149011612] (6.,-1.) -- (6.,1.) -- (8.,1.) -- (8.,-1.) -- cycle;
\fill[line width=1.5pt,color=zzttqq,fill=zzttqq,fill opacity=0.10000000149011612] (10.,-1.) -- (10.,1.) -- (12.,1.) -- (12.,-1.) -- cycle;
\fill[line width=1.5pt,color=zzttqq,fill=zzttqq,fill opacity=0.10000000149011612] (12.,-1.) -- (12.,1.) -- (14.,1.) -- (14.,-1.) -- cycle;
\fill[line width=1.5pt,color=zzttqq,fill=zzttqq,fill opacity=0.10000000149011612] (16.,1.) -- (18.,1.) -- (18.,-1.) -- (16.,-1.) -- cycle;

\draw (0.7,-1.3) node[anchor=north west] {\LARGE $d$};
\draw (2.7,-1.3) node[anchor=north west] {\LARGE 1};
\draw (4.7,-1.3) node[anchor=north west] {\LARGE 2};
\draw (6.7,-1.3) node[anchor=north west] {\LARGE 3};
\draw (10.25,-1.3) node[anchor=north west] {\LARGE $d\!-\!r\!$};
\draw (11.95,-1.3) node[anchor=north west] {\LARGE $d\!-\!r\!+\!1$};
\draw (16.3,-1.3) node[anchor=north west] {\LARGE $d\!-\!1$};

\draw (0.7,1.9) node[anchor=north west] {\LARGE $r$};
\draw (2.3,2) node[anchor=north west] {\LARGE $r\!+\!1$};
\draw (4.3,2) node[anchor=north west] {\LARGE $r\!+\!2$};
\draw (6.3,2) node[anchor=north west] {\LARGE $r\!+\!3$};
\draw (10.7,2) node[anchor=north west] {\LARGE $d$};
\draw (12.7,2) node[anchor=north west] {\LARGE $1$};
\draw (16.3,2) node[anchor=north west] {\LARGE $r-1$};

\draw (8.4,0.7) node[anchor=north west] {\Huge $\ldots$};
\draw (8.4,-0.4) node[anchor=north west] {\Huge $\ldots$};
\draw (14.4,0.7) node[anchor=north west] {\Huge $\ldots$};
\draw (14.4,-0.4) node[anchor=north west] {\Huge $\ldots$};

\draw (1.5,-1.5) node[anchor=north west] {\Huge $O_1$};
\draw (11.5,2.5) node[anchor=north west] {\Huge $O_2$};

\draw [line width=1.5pt,color=zzttqq] (0.,-1.)-- (0.,1.);
\draw [line width=1.5pt,color=zzttqq] (-0.13742195872908006,0.) -- (0.13742195872908006,0.);
\draw [line width=1.5pt,color=zzttqq] (0.,1.)-- (2.,1.);
\draw [line width=1.5pt,color=zzttqq] (2.,1.)-- (2.,-1.);
\draw [line width=1.5pt,color=zzttqq] (2.,-1.)-- (0.,-1.);
\draw [line width=1.5pt,color=zzttqq] (2.,-1.)-- (2.,1.);
\draw [line width=1.5pt,color=zzttqq] (2.,1.)-- (4.,1.);
\draw [line width=1.5pt,color=zzttqq] (4.,1.)-- (4.,-1.);
\draw [line width=1.5pt,color=zzttqq] (4.,-1.)-- (2.,-1.);
\draw [line width=1.5pt,color=zzttqq] (4.,-1.)-- (4.,1.);
\draw [line width=1.5pt,color=zzttqq] (4.,1.)-- (6.,1.);
\draw [line width=1.5pt,color=zzttqq] (6.,1.)-- (6.,-1.);
\draw [line width=1.5pt,color=zzttqq] (6.,-1.)-- (4.,-1.);
\draw [line width=1.5pt,color=zzttqq] (6.,-1.)-- (6.,1.);
\draw [line width=1.5pt,color=zzttqq] (6.,1.)-- (8.,1.);
\draw [line width=1.5pt,color=zzttqq] (8.,1.)-- (8.,-1.);
\draw [line width=1.5pt,color=zzttqq] (8.,-1.)-- (6.,-1.);
\draw [line width=1.5pt,color=zzttqq] (10.,-1.)-- (10.,1.);
\draw [line width=1.5pt,color=zzttqq] (10.,1.)-- (12.,1.);
\draw [line width=1.5pt,color=zzttqq] (12.,1.)-- (12.,-1.);
\draw [line width=1.5pt,color=zzttqq] (12.,-1.)-- (10.,-1.);
\draw [line width=1.5pt,color=zzttqq] (12.,-1.)-- (12.,1.);
\draw [line width=1.5pt,color=zzttqq] (12.,1.)-- (14.,1.);
\draw [line width=1.5pt,color=zzttqq] (14.,1.)-- (14.,-1.);
\draw [line width=1.5pt,color=zzttqq] (14.,-1.)-- (12.,-1.);
\draw [line width=1.5pt,color=zzttqq] (16.,1.)-- (18.,1.);
\draw [line width=1.5pt,color=zzttqq] (18.,1.)-- (18.,-1.);
\draw [line width=1.5pt,color=zzttqq] (18.137421958729078,0.) -- (17.86257804127092,0.);
\draw [line width=1.5pt,color=zzttqq] (18.,-1.)-- (16.,-1.);
\draw [line width=1.5pt,color=zzttqq] (16.,-1.)-- (16.,1.);
\begin{scriptsize}
\draw [fill=black] (2.,-1.) circle (5pt);
\draw [fill=black] (4.,-1.) circle (5pt);
\draw [fill=black] (14.,1.) circle (5pt);
\draw [fill=black] (12.,1.) circle (5pt);
\end{scriptsize}
\end{tikzpicture}
\vspace{-10pt}
\caption{A square-tiled surface with a non-separating node corresponding to the pole $\left( r(d-1), 1 \right) \in \EE_1$.}
\label{figirredpole}
\end{figure}
\end{comment10}

Fix any integer $1 < k \le (d-1)/2$, we will show that for some $1 < r < d$ the cusp pole $X_r \in \EE_1$ with cylinder coordinates $\left( r(d-1), 1 \right)$ (see Figure~\ref{figirredpole}) satisfies $R(X_r) \in \EE_k$. Note that a vertical line from point $O_1$ ends up in point $O_2$ after passing through exactly $k$ squares if and only if $1 + r k \equiv d-k \mod d$, which is equivalent to $rk \equiv d-k-1 \mod d$ that has a solution $r$ for any $k$ since $d$ is prime. Then $R(X_r) \in \EE_k$, which finishes the proof of (\ref{eavesconnect}).
\end{proof}

Assuming these propositions the proof of Theorem~\ref{thmillumination} is straightforward:

\begin{proof}[Proof of Theorem~\ref{thmillumination}] 
Pick any point $x \in \A_{d^2}[n]$. According to Proposition~\ref{into} it can be sent on a vertical edge in some lighthouse $\LL$, which belongs to one of the subsets $N_k(b), N^\epsilon_k(b)$. By Proposition~\ref{within} the union of all of the $N_k(b), N^\epsilon_k(b)$ in each lighthouse $\LL_k$ forms one or two $\SLZ$ subsets depending on parity of $n$. Finally, by Proposition~\ref{between} the union of all $N_k(b), N^\epsilon_k(b)$ for all lighthouses forms one or two $\SLZ$ subsets depending on parity of $n$. Since any point $x \in \A_{d^2}[n]$ can be sent into them, they generate the whole $\A_{d^2}[n]$, which implies the parity conjecture.

Computing the spin invariant one can show that $N_k(b), N^1_k(b)$ and $N^0_k(b)$ generate $\A_{d^2}[n]$,  $\A^1_{d^2}[n]$ and $\A^0_{d^2}[n]$ respectively.
\end{proof}

\section{Proof for $d=3$ and $5$} \label{secproof35}

In this section we will establish the illumination conjecture for $d=3,4$ and $5$ and use it together with Theorem~\ref{thmillumination} to prove the parity conjecture for $d=3$~and~$5$.

\begin{theorem} \label{thmillum2345}
\begin{enumerate}
\item[(i)] All of $X(2)$, $X(3)$ and $X(4)$ are illuminated by their cusps. 
\item[(ii)] The set of cusps of $X(5)$ illuminates all of $X(5)$ except for the non-cusp pole $P$ corresponding to $2E_0 = \C[i]/2\Z[i]$ joined with $E_0$ at a point.
\end{enumerate}
\end{theorem}

Clearly, all of $X(2) \cong \A_4$ is illuminated by its cusps. We continue with the next cases. The proof of Theorem~\ref{thmillum2345} will be given by analyzing the pictures of the square-tilings of $X(3), X(4)$ and $X(5)$.

\noindent
{\bf Square-tiling of $X(3) \cong \A_{9}$.} 
It is evident from Figure~\ref{figX(3)intro} that all of $\A_9$ is illuminated by the cusp poles (red points). Therefore Theorem~\ref{thmillumination} implies that $\left| \A_{9}[n] \big/ \SLZ \right|$ is 1 when $n$ is even and 2 when $n>1$ is odd. Note that $\A_{9}[1]$ is empty. We illustrate the case $\A_{9}[5]$, as an example of two $\SLZ$ orbits, in Figure~\ref{figX(3)orbits}.

\begin{comment11}
\begin{figure}[H] 
\centering
\includestandalone[width=0.8\textwidth]{tikzX3orbits}
\vspace{-10pt}
\caption{The $\SLZ$ orbits of $\A_9[5]$: $\A^0_9[5]$ (green points) and $\A^1_9[5]$ (blue points).}
\label{figX(3)orbits}
\end{figure}
\end{comment11}

\noindent
{\bf Square-tiling of $X(4) \cong \A_{16}$.} We presented the square-tiling of $\A_{16}$ in Figure~\ref{figX(4)intro}. Here we will establish the illumination conjecture for it. Note that the interiors of the cylinders $\CC_1, \CC_3, \CC_4$ (see Figure~\ref{figX(4)illum}) are illuminated since they contain cusp poles. The closures of these cylinders are also illuminated. 

\begin{comment11}
\begin{figure}[H]
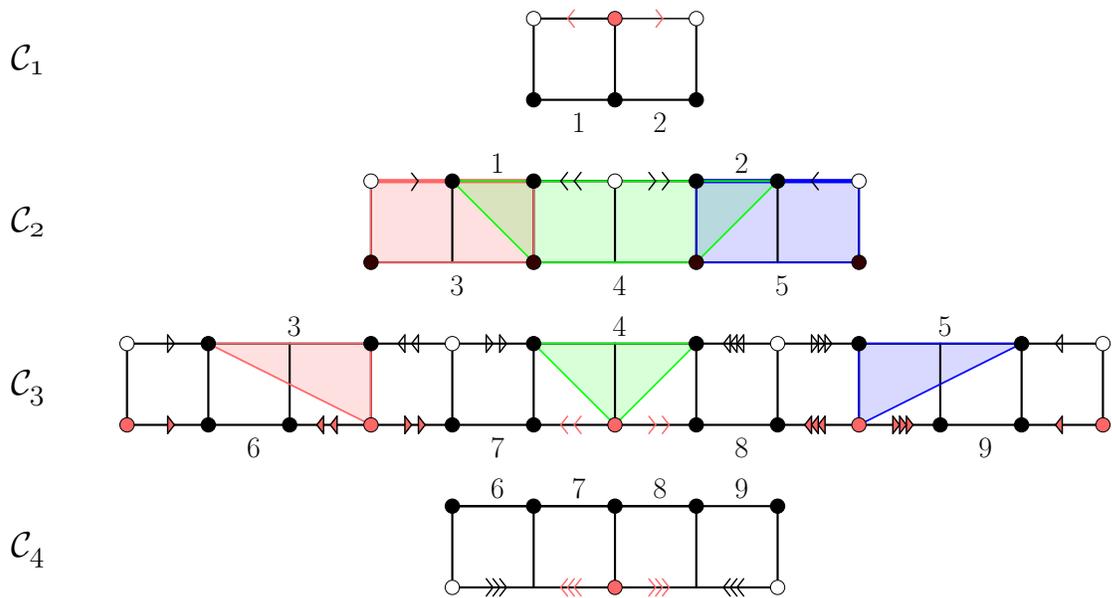
 
\centering
\includestandalone[width=\textwidth]{tikzX4illum}
\vspace{-10pt}
\caption{Illumination of $X(4)\cong \A_{16}$.}
\label{figX(4)illum}
\end{figure}
\end{comment11}

For the cylinder $\CC_2$ note that its red, green and blue regions are illuminated by the cusps of $\CC_3$ through the saddle connections 3, 4 and 5 respectively, which finishes the proof of the illumination conjecture for $\A_{16}$.

\noindent
{\bf Square-tiling of $X(5) \cong \A_{25}$.} Again using gluing instructions one obtains the square-tiling of $\A_{25}$ as in the figure below. Note that the labels are slightly different: we do not put labels on the saddle connections that contain poles (red and white points) as they are always identified to the adjacent ones via rotations by $\pi$; the saddle connection labeled with numbers are identified via parallel translations and the ones labeled with small letters are identified via rotations. 

Note that the interiors of the cylinders $\CC_1, \CC_4, \CC_5, \CC_7$ (see Figure~\ref{figX(5)illum}) are illuminated since they contain cusp poles.

\begin{comment11}
\begin{figure}[H]
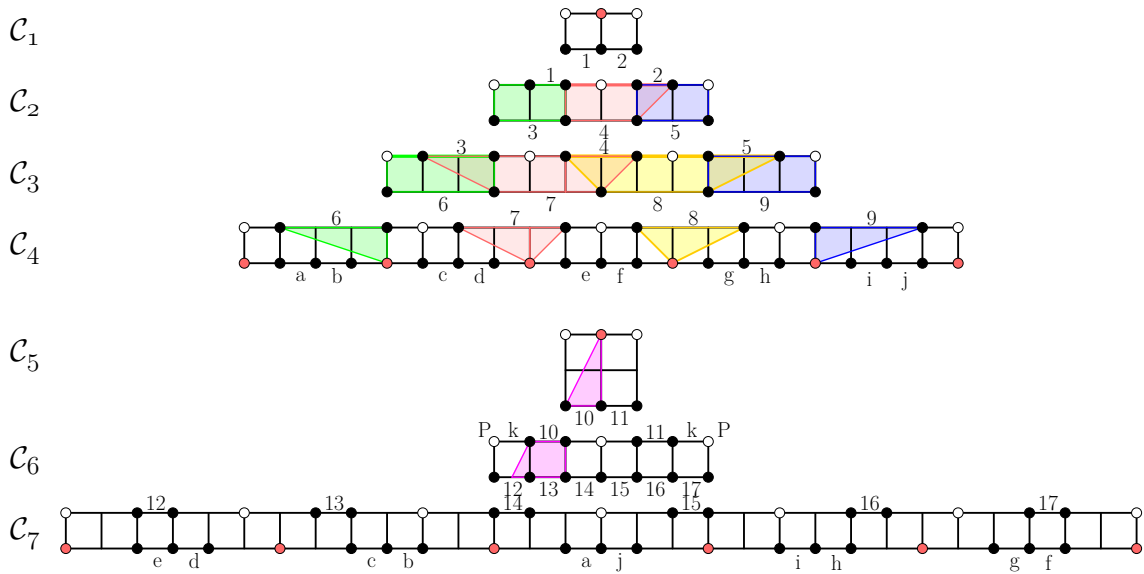
 
\centering
\includestandalone[width=\textwidth]{tikzX5illum}
\vspace{-10pt}
\caption{Illumination of $X(5)\cong \A_{25}$.}
\label{figX(5)illum}
\end{figure}
\end{comment11}

Now let us show that the cylinders $\CC_2$ and $\CC_3$ are illuminated. For $\CC_3$ note that its green, red, yellow and blue regions are illuminated by the cusps of $\CC_4$ through the saddle connections 6, 7, 8 and 9 respectively. For $\CC_2$ note that its green red and blue edges are illuminated by the cusps of $\CC_4$ through the saddle connections 6 and then 3, 7 and then 4, 9 and then 5 respectively. In particular, the saddle connections 1, 2, 3, 4, 5, 6, 7, 8 and 9 are illuminated.

Now we will show that the interior of $\CC_6$ is illuminated except may be for its non-cusp pole $P$. First note that the pink region is illuminated by the cusp of $\CC_5$. Secondly, note that the unipotent subgroup fixes $P$ and thus it fixes the top boundary of $\CC_6$. Therefore, the unipotent orbit of the pink square in $\CC_6$ contains the interior of $\CC_6$. Since illumination is $\SLZ$ invariant property, the interior of $\CC_6$ is illuminated.

Finally, we discuss the illumination of the remaining edges on the boundaries of the cylinders. Edges 10, 11, 12, 13, 14, 15, 16 and 17 are illuminated by the cusps of $\CC_5$ and $\CC_7$. To see the illumination of the edges a, b, c, d, e, f, g, h, i and j one might show that i and j are illuminated by the cusp of $\CC_1$ through the saddle connections 2, then 5 and then 9. Since illumination is an $\SLZ$ invariant property and since the unipotent orbit of the edge i consists of the edges a, c, e and g, and the unipotent orbit of the edge j consists of the edges b, d, f, and h, they are all illuminated.

Now we obtained that all of $\A_{25}$ is illuminated except for may be edge k. Note that P is a wedge sum of to square tori of area $4$ and $1$, thus it is fixed by the whole $\SLZ$. Again, since illumination is $\SLZ$ invariant property, sending k to an edge with a non-zero slope we obtain that all of $\A_{25}$ is illuminated except for may be the non-cusp pole P.

In fact it is possible to show that P is not illuminated by the cusp poles. If there was a line segment connecting P to a cusp pole, one could send it via a matrix from $\SLZ$ to a line segment that connects P to some cusp pole vertically, which is clearly impossible.

\noindent
{\bf Proof for $d=3$ and $5$.} Theorem~\ref{thmillumination} together with the illumination results for $\A_9$ and $\A_{25}$ imply the main result (Theorem~\ref{main}) for $d=3$ and $d=5$ and arbitrary $n$. \qed


\section{Proof for $d=4$} \label{secproof4}

In this section we will give proof of the main result for $d=4$. This proof is quite different in nature from the proofs of the other cases. We will show that there is a degree $3$ branched covering map $f: \A_{16} \to \A_{9}$ that respects the square-tilings. Using this map and the main result for $d=3$ (see \S\ref{secproof35}) we will reduce the problem to the study of the $\SLZ$ orbits of just three points in a single fiber of $f$.

\noindent
{\bf Covers and symmetric groups.} There is a well-known connection between elliptic covers and symmetric groups. We review this connection below.

Let $S_d$ be the symmetric group on $d$ elements. Let $X \in \M_2$ and $\pi: X \to E_0$ be a primitive degree $d$ cover of the square torus with two critical points. One can associate to it a pair of permutations $(s_h, s_v) \in S_d \times S_d$, where $s_h$ is the monodromy of the horizontal loop on $E_0$ and $s_v$ is the monodromy of the vertical loop on $E_0$, satisfying the following properties:
\begin{equation}
\begin{split}
\bullet & \quad s_h, s_v \mbox{ generate a transitive subgroup of } S_d; \mbox{ and} \\
\bullet & \quad [ s_h, s_v] \mbox{ has a cyclic type (2,2)}.
\end{split}
 \label{perm}
\end{equation}

Conversely, any pair $(s_h, s_v) \in S_d \times S_d$ satisfying the above conditions determines a primitive degree $d$ cover $\pi: X \to E_0$ branched over two points.

\noindent
{\bf Construction of the map $f$.} The existence of the covering map $f: \A_{16} \to \A_{9}$ relies on the surjective homomorphism $\phi: S_4 \to S_3$.

Consider Klein four-subgroup $K = \{ (12)(34), (13)(24), (14)(23) \}  \subset S_4$. The quotient $S_4 / K$ is isomorphic to $S_3$. We will denote the quotient homomorphism by $\phi: S_4 \to S_3$. 

Let $(X,\omega) \in \A_{16}$ be any non-vertex of the square-tiling on $\A_{16}$. It defines a unique normalized degree 4 cover $\pi: X \to E_0$. Let $\pi$ be branched over $\pm z$. The corresponding pair of permutations $(s_h, s_v) \in S_4 \times S_4$ that satisfies (\ref{perm}). Then the pair $(\phi(s_h),\phi(s_v)) \in S_3 \times S_3$ also satisfies (\ref{perm}). There is a unique normalized degree 3 cover $\pi': X' \to E_0$ branched over $\pm z$ corresponding to the pair of permutations $(\phi(s_h),\phi(s_v))$. We then define $f: \A_{16}\to  \A_{9}$ to be a unique branched cover such that for any $(X,\omega)$ with two simple zeroes:
$$
f(X,\omega) = (X',\omega'),
$$
where $\omega' = \pi'^*(dz)$.

The map $f$ respects the square-tiling, since both covers $\pi: X \to E_0$ and $\pi': X' \to E_0$ are normalized and branched over $\pm z$. The covering map $f: \A_{16}\to  \A_{9}$ on the level of the square-tilings of $\A_{16}$ and $\A_{9}$ is illustrated in Figure~\ref{figcover43}. The top and bottom cylinders of $\A_{16}$ are respectively a 1-fold and a 2-fold covers of the top cylinder of $\A_9$. Similarly, the other two cylinders of $\A_{16}$ are a 1-fold and a 2-fold covers of the bottom cylinder of $\A_9$. Each square of $\A_{9}$ is labelled by the same letter as its preimages under $f$. The orientation of the letters determines whether the squares of $\A_{16}$ are mapped to $\A_9$ by parallel translations or by $\pi$ rotations. 

\begin{comment12}
\begin{figure}
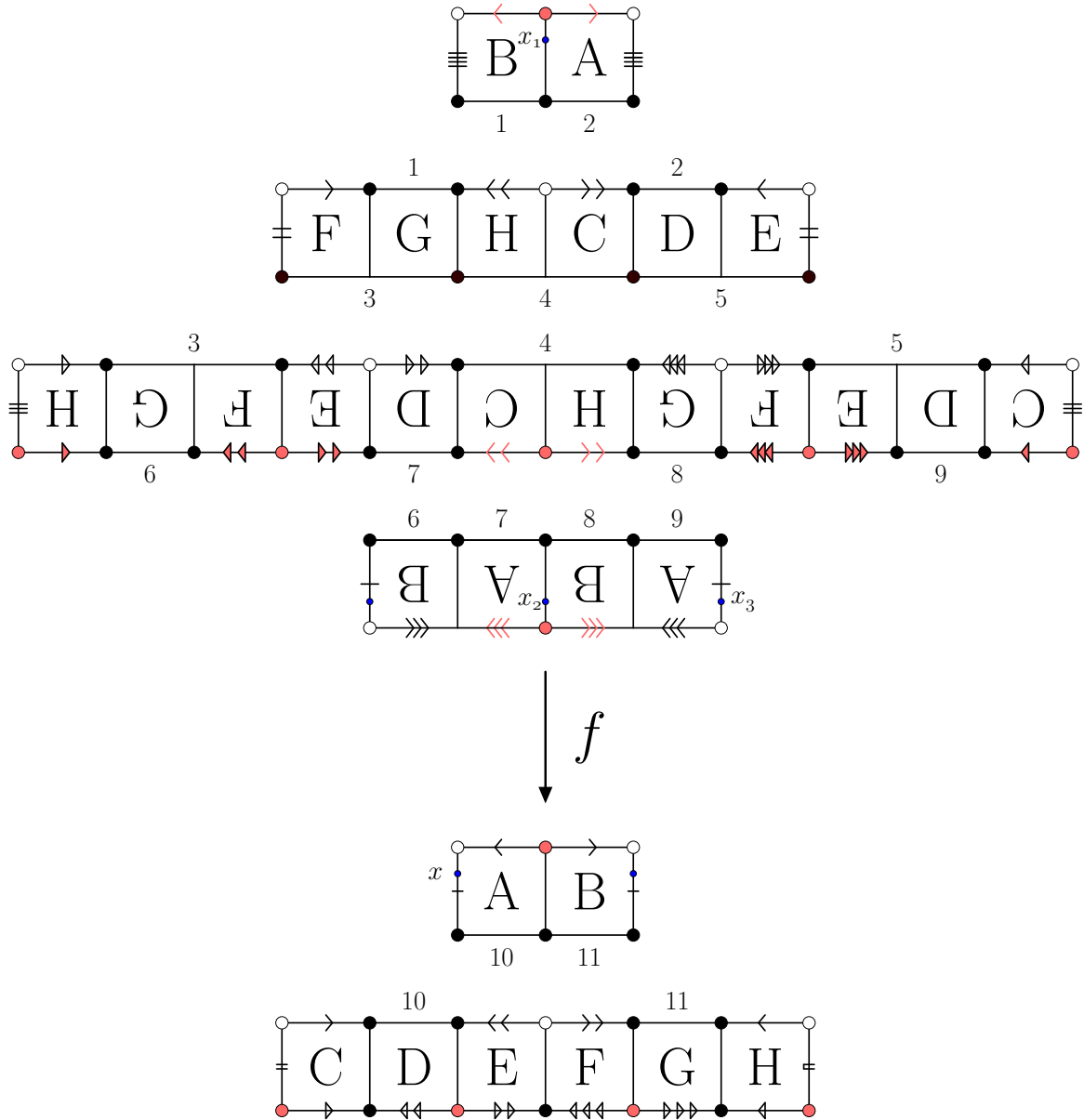

\centering
\includestandalone[width=\textwidth]{tikzcover43}
\vspace{-10pt}
\caption{The degree 3 branched covering map $f: X(4)\to X(3)$ sends squares to squares.}
\label{figcover43}
\end{figure}
\end{comment12}

\noindent
{$\bf 2+2 =_{3} 4$.} Another way of making this construction comes from an elementary fact that the set of 4 elements can be split into two subsets of 2 elements each in 3 different ways. Starting with a degree 4 cover over the torus with two simple ramifications, one can replace each fiber with the three way of splitting it into pairs. That produces a degree 3 cover over the torus with two simple ramifications. One can show that this construction can be completed to give the same covering map $f: \A_{16}\to  \A_{9}$ as above.

\noindent
{\bf Proof for $\bf d=4$.} The set $\A_{16}[1]$ consists of the non-singular vertices of the tiling of $X(4) \cong \A_{16}$. One can verify from the picture of the square-tiling (see Figure~\ref{figcover43}) that all points of $\A_{16}[1]$ lie in a single $\SLZ$ orbit. 

Let $n>1$ be any integer. For any integer $1 \le a \le n-1$, let $x_{a/n} \in \A_9$ be a point on the left edge of the square $A$ that is distance $a/n$ away from the white vertex of $A$ (see Figure~\ref{figcover43}). According to the proof of Theorem~\ref{main} for $d=3$, every $\SLZ$ orbit in $\A_{9}[n]$ contains a representative $x_{1/n}$ when $n$ is even, and a representative $x_{1/n}$ or $x_{2/n}$ when $n$ is odd. Therefore any point in $\A_{16}[n]$ can be sent by a suitable matrix from $\SLZ$ into $f^{-1}(x_{1/n})$ when $n$ is even, and into $f^{-1}(x_{1/n})\cup f^{-1}(x_{2/n})$ when $n$ is odd. 

Let $x = x_{a/n}$ for some $a =1$ or $2$, and $f^{-1}(x) = \{x_1, x_2, x_3\}$ (see Figure~\ref{figcover43}). It suffices to show that for any such $x$ the points $x_1, x_2, x_3$ belong to the same $\SLZ$ orbit.

Recall that $S= \begin{pmatrix} 1& -1 \\ 0&1 \end{pmatrix}$. Then:
$$
S^{2n}: x_2  \mapsto  x_3, \mbox{ when } a=1, \mbox{ and } S^{n}: x_2  \mapsto  x_3, \mbox{ when } a=2.
$$

The image of the top two squares $A$ and $B$ of $\A_{16}$ under the rotation by $\pm \pi/2$ are the two outermost squares $C$ and $H$ in the cylinder of length 12.
The images of their bottoms under $S^2$ are the bottoms of the innermost squares $C$ and $H$.
In its turn, the images of these squares under the rotation by $\pm \pi/2$ are the innermost squares $A$ and $B$ in the cylinder of length 4. Using this one can verify:
$$
R \circ S^2 \circ R: x_1 \mapsto x_2.
$$
Therefore all $x_1, x_2, x_3$ belong to the same $\SLZ$ orbit and $\A_{16}[n]$ consists of two $\SLZ$ orbits when $n>1$ is odd and a single orbit otherwise. \qed

Although this proof hints on the possibility of reducing the case of composite $d$ to prime $d$, unfortunately, the only surjective homomorphisms of symmetric groups $S_d \to S_k$ are given by sign homomorphisms $S_d \to S_2$ and the above homomorphism $S_4 \to S_3$. One can verify that a similar construction of the branched covering map for the sign homomorphism $S_d \to S_2$ gives the discriminant map $\delta: \A_{d^2} \to \A_{2^2} \cong \bfP$.


\newpage 

\appendix 

\section{Counts of elliptic covers} \label{sectioncounts}

In this appendix we review some counts related to the genus $2$ torus covers and give an upper bound on the number of irreducible components of $W_{d^2}[n]$ that does not depend on $n$. 

Recall that congruence subgroup is a subgroup of $\SLZ$ defined as follows:
$$
\displaystyle \Gamma(d) = \left\{  \begin{pmatrix} a&b \\ c&d \end{pmatrix}\in \SLZ \  \middle\vert \ a, d \equiv 1 \mbox{ and } b, c \equiv 0 \mod d
\right\}.
$$
The index of $\Gamma(d) /{\pm Id}$ in $\PSL(2,\Z) = \SLZ /{\pm Id}$ is:
\begin{equation} \label{psl}
\left| \PSL_2(\ZZ{d}) \right|\  =  \begin{cases}
6 \ \ \mbox{for} \ \ d = 2,\\
\\
\displaystyle \frac{d^3}{2} \cdot \prod_{\substack{p | d \\ p \text{ prime}}} (1 - \frac1{p^2}) \ \ \mbox{for} \ \ d \ge 3.
\end{cases}
\end{equation}
Also recall the formula of the index of $\Gamma(d)$ in $\SLZ$:
\begin{equation*}
\left| \SL_2(\ZZ{d}) \right|\ =
\displaystyle d^3 \cdot \prod_{\substack{p | d \\ p \text{ prime}}} (1 - \frac1{p^2}) \mbox{ for all } d.
\end{equation*}

\begin{table}
\centering
\begin{tabular}{ | C{1cm} | C{3cm} | C{8cm} | }
\hline
& & \\
1 & $\deg\delta$ & $\displaystyle \frac{(d-1)}6  \cdot \left| \PSL_2(\ZZ{d}) \right|$  \\ 
& & \\
\hline
& & \\
2 & $g(\A_{d^2})$ & $\displaystyle  \frac{d-6}{12d}  \cdot \left| \PSL_2(\ZZ{d}) \right| + 1$  \\  
& & \\
\hline
& & \\
3 & $\left| P_{ns}(d) \right|$ & $\displaystyle \frac1d  \cdot \left| \PSL_2(\ZZ{d}) \right|$ \\
& & \\
\hline
& & \\
4 & $\left| P_{s}(d) \right|$ & $\displaystyle \frac{5d-6}{12d}  \cdot \left| \PSL_2(\ZZ{d}) \right|$ \\
& & \\
\hline
& & \\
5 & $\left| \A_{d^2}[0] \right|$ & $\displaystyle \frac{3(d-2)}{4d}  \cdot \left| \PSL_2(\ZZ{d}) \right|$  \\
& & \\
\hline
& & \\
6 & \quad $\left| \A^0_{d^2}[0]\right|,$ \newline $d$ is odd  & $\displaystyle \frac{3(d-3)}{8d}  \cdot \left| \PSL_2(\ZZ{d}) \right|$ \\
& & \\
\hline
& & \\
7 & \quad $\left| \A^1_{d^2}[0]\right|,$ \newline $d$ is odd & $\displaystyle \frac{3(d-1)}{8d}  \cdot \left| \PSL_2(\ZZ{d}) \right|$ \\
& & \\
\hline
& & \\
8 & $\left| \A_{d^2}[1]\right|$ & $\displaystyle \frac{(d-2)(d-3)}{3d}  \cdot \left| \PSL_2(\ZZ{d}) \right|$ \\
& & \\
\hline
& & \\ 
9 & $\left| \A_{d^2}[n]\right|$ & $\displaystyle \frac{d-1}{3n} \cdot \left| \PSL_2(\ZZ{d}) \right| \cdot  \left| \SL_2(\ZZ{n}) \right|$ \\
& & \\
\hline
& & \\ 
10 & \quad $\left| \A^0_{d^2}[n]\right|,$ \newline $n$ is odd & $\displaystyle \frac{d-1}{12n} \cdot \left| \PSL_2(\ZZ{d}) \right| \cdot  \left| \SL_2(\ZZ{n}) \right|$ \\
& & \\
\hline
& & \\ 
11 & \quad $\left| \A^1_{d^2}[n]\right|,$ \newline $n$ is odd & $\displaystyle \frac{d-1}{4n} \cdot \left| \PSL_2(\ZZ{d}) \right| \cdot  \left| \SL_2(\ZZ{n}) \right|$ \\
& & \\
\hline
\end{tabular}
\caption{The counts related to elliptic covers of genus 2.}
\label{table}
\end{table}

The following theorem summarizes some results on elliptic covers in genus 2, in particular the ones related to the counts of square-tiled surfaces, which can be found in \cite{Kani06} and \cite{EMS03}:

\begin{theorem} \label{counts} The counts related to elliptic covers given in Table~\ref{table} hold for any integers $d>1$ and $n>1$.
\end{theorem}

Below we give more details on each count and references for proofs.

1. Here $\deg\delta$ is the degree of the discriminant map $\delta: \A_{d^2} \to \bfP$ (see \S\ref{secabs}).
The formula for $\deg\delta$ can be found in \cite{Kani06} (equation (31) and Corollary 30). It also follows from the count of primitive degree $d$ covers of the square torus branched over two given points in \cite{EMS03} (remark after Lemma 4.9). The number of such covers is $\deg \rho = \deg (\sigma \circ \delta) = 4\deg\delta$ (see diagram~(\ref{diag:rhodelta})).
 
2.  Here $g(\A_{d^2})$ denotes the genus of a Riemann surface $\A_{d^2}$. In Theorem~\ref{thmmodular} we established an isomorphism $\A_{d^2} \cong X(d)$. The genus of $X(d)$ is a classical result from the theory of modular curves and can be found in \cite{Shi71} (equation (1.6.4)).

3. In \S\ref{secmodular} we showed that all of the cusps of $Y(d) \subset X(d) \cong \A_{d^2}$ are simple poles of $q$. Recall that $P_{ns}(d)$ denotes the subset of simple poles of the quadratic differential $q$ on $\A_{d^2} \cong X(d)$ that are cusps of the modular curve $Y(d)$. These points correspond to the genus 2 d-square-tiled surfaces with a non-separating node (see \S\ref{secabs}). The formula for $\left| P_{ns}(d) \right|$ can be found in \cite{Kani06} (equation (27)). Note that the formula $\left| P_{ns}(d) \right|$ agrees with the classical formula for the number of cusps of $Y(d)$.

4. Recall that $P_{s}(d)$ denotes the subset of simple poles of the quadratic differential $q$ on $\A_{d^2}$ that are not cusps. These points of $\A_{d^2}$ correpond to the genus 2 d-square-tilied surfaces with a separating node (see \S\ref{secabs}). The formula for $\left| P_{s}(d) \right|$ can be found in \cite{Kani06} (equation (27)).

5. Here $\A_{d^2}[0]$ denotes the set of primitive $d$-square-tiled surfaces in $\Omega\M_2(2)$. This set is in bijection with the set of primitive degree $d$ covers of $E_0$ branched over the origin with a single ramification point of order 3 (see \S\ref{secbackground}). The formula for $\A_{d^2}[0]$ can be found in \cite{EMS03} (remark after Lemma 4.11). 

Recall from Theorem~\ref{thmvertices} that $\A_{d^2}[0] \subset \A_{d^2}$ is also the set of simple zeroes of $q$. It is known that for any Riemann surface $X$, the difference of the number of simple zeroes of a meromorphic quadratic differential on $X$ is equal to $4g-4$, where $g$ is the genus of $X$. Using the results of 2-5 one can verify that:
\begin{align*}
\left|\A_{d^2}[0]\right| - \left|P_{ns}(d)\right| - \left|P_{s}(d)\right|  = 4\cdot g(\A_{d^2})-4.
\end{align*}

6 and 7. In the case of odd $d>3$, $\A_{d^2}[0]$ consists of two $\SLZ$ orbits $\A^0_{d^2}[0]$ and $\A^1_{d^2}[0]$ distinguished by an analogue of the spin invariant. The formulas for $\A^0_{d^2}[0]$ and $\A^1_{d^2}[0]$ are given in \cite{LR06} (Theorem 1).

8. Here $\A_{d^2}[1]$ denotes the set of primitive $d$-square-tiled surfaces in $\Omega\M_2(1,1)$. This set is in bijection with the set of primitive degree $d$ covers of $E_0$ branched over the origin with two ramification points of order 2 (see \S\ref{secbackground}). The formula for $\A_{d^2}[1]$ can be found in \cite{Zm11} (equation (3.11)). 

Recall from Proposition~\ref{cornerfibers} that $\delta$ is has ramifications of order 3 at $\A_{d^2}[0]$ and ramifications of order 2 at $\A_{d^2}[1]$. Using the results of 1, 2, 5 and 8 one can verify that Riemann-Hurwitz formula for the map $\delta: \A_{d^2} \to \bfP \cong \P^1$:
$$2-2g(\A_{d^2}) = 2 \deg \delta - 2 \left| \A_{d^2}[0]\right|  - \left|\A_{d^2}[1]\right|. $$

9. Here $\A_{d^2}[n]$ denotes the set primitive $n$-rational points of the squares in the tiling of $\A_{d^2}$. This set is in bijection with the set of primitive degree $d$ covers of $E_0$ branched over the origin and a torsion point of order $n$. The formula for $\left| \A_{d^2}[n]\right|$ can be found in \cite{Kani06} (Theorem 3).

10 and 11. In the case of odd $n>1$, $\A_{d^2}[n]$ consists of two $\SLZ$ orbits $\A^0_{d^2}[n]$ and $\A^1_{d^2}[n]$ distinguished by the spin invariant $\epsilon$. These formulas are obtained in Theorem~\ref{thmformula}. For even $d$ they can also be found in \cite{KM17} (Theorem 1.1).

\begin{theorem} \label{thmsimplebound}
For any $d$ and $n>1$ we have:
$$\left| {\A_{d^2}[n]}\big/{\SLZ} \right| \le \frac{2(d-1)}3 \ \cdot \left| \PSL_2(\ZZ{d}) \right|.$$
\end{theorem}

\begin{proof}
The group $\SLZ$ acts transitively on the set of primitive $n$-torsion points of $E_0$. Therefore every point $(X,\omega) \in \A_{d^2}[n]$ can taken via $\SLZ$ to an Abelian differential $(X',\omega')$, such that $X$ admits a primitive degree $d$ cover of $E_0$ simply branched over the origin and $z = \frac{i}n \mod \Z[i]$. Using the formula for the number of primitive degree $d$ covers of $E_0$ simply branched over two given points obtained in \cite{EMS03} (remark after Lemma 4.9), we have:
\[
\left| {\A_{d^2}[n]}\big/{\SLZ} \right| \le\  \frac{2(d-1)}3 \ \cdot \left| \PSL_2(\ZZ{d}) \right|.
\pushQED{\qed} 
\qedhere
\popQED
\]
\let\qed\relax
\end{proof}

\section{The pagoda structure of the modular curves} \label{secpagoda}
In this appendix we present the {\em pagoda structure} of the modular curve $X(d)$ that arises for prime $d$.
In particular we will give a simpler and more uniform geometric construction of the square-tiling of $X(d) \cong \A_{d^2}$ for any prime $d$ (compare to \S\ref{sectiling}). We will conclude by presenting pictures of the pagoda structures of $X(d)$ for $d=7,11,13$ and $17$ (see Figure~\ref{figX(7)}, \ref{figX(11)}, \ref{figX(13)} and \ref{figX(17)}).

Throughout this section we will assume $d$ is prime, however analogous results must hold for any $d>1$. We will follow the plan below:

\begin{enumerate}
\item Describe the pagoda structure of $X(d) \cong \A_{d^2}$.

\item Determine the types of singularities on the boundaries of the cylinders of $\A_{d^2}$.

\item Determine the identifications between these boundaries.
\end{enumerate}
The proofs of the main results of this appendix are postponed till the last subsection.

\noindent
{\bf Horizontal cylinders of $\A_{d^2}$.} We begin by reviewing the enumeration of the horizontal cylinders of $(\A_{d^2},q)$ (Theorem~\ref{thmdecomposition}) for prime $d$ (Remark~\ref{remarkordering}):

\begin{proposition} \label{prop:primedecomposition}
For any prime $d$ the absolute period leaf $(\A_{d^2}, q)$ naturally decomposes into a union of horizontal cylinders, whose boundaries are unions of saddle connections, and for which the following conditions hold:

(i) (enumeration) The set of horizontal cylinders $Cyl(\A_{d^2})$ is in bijection with the set of unordered pairs $ \left\{(w_1, s_1), (w_2, s_2)\right\}  \in \Sym^2\N^2,$ satisfying the following conditions:

\begin{itemize}
\item (area) $s_1 w_1 + s_2 w_2 = d$; and

\item (primitivity) $\gcd(s_1, s_2) = 1$.
\end{itemize}

(ii) (dimensions) The height of the cylinder $\CC =  \left\{(w_1, s_1), (w_2, s_2)\right\} $ is
$H_{\CC} = \min(s_1,s_2),$
its circumference is
$W_{\CC} = w_1 w_2 (w_1 + w_2)$.
\end{proposition}

The integers $w_1,s_1,w_2,s_2$ can ordered in such a way that $w_1 < w_2$ if $s_1=s_2=1$, and $s_1<s_2$ otherwise. In this case we will denote the corresponding horizontal cylinder by $(w_1,s_1,w_2,s_2)$.
We distinguish three types of horizontal cylinders $\CC = (w_1,s_1,w_2,s_2)$ of $\A_{d^2}$:

\begin{itemize}
\item (lighthouse) $w_1=w_2=1$;

\item (body) $s_1 < s_2, w_1 \ne w_2$; and

\item (eave) $s_1=s_2=1$.
\end{itemize}

\noindent
{\bf Pagoda structure.} When $d>2$ is prime, the modular curve $X(d)$ has a {\em pagoda structure}, a natural decomposition of $X(d)$ into topological disks, called {\em stories}. Each story is a union of a particular subset of the squares of the tiling of $\A_{d^2} \cong X(d)$. Each story consists of layers of horizontal cylinders that are stacked upon each other in descending order starting from the longest one (eave) at the bottom and with the most narrow one (lighthouse) at the top. The bottoms of the eaves, are linked among each other. All other edges are folded with the adjacent ones. Below we make this description more precise.

When $d>2$ is prime, the set of horizontal cylinders $Cyl(\A_{d^2})$ is a disjoint union of $\frac{d-1}2$ ordered subsets: 
$$\displaystyle S^{(i)} = \left\{ \CC^{(i)}_1, \CC^{(i)}_2, \ldots, \CC^{(i)}_{n_i} \right\}, \mbox{ for every } i = 1, \ldots, \frac{d-1}2,
$$ 
called {\em stories of the pagoda}, satisfying the following properties:
\begin{enumerate}
\item Each story $\mathcal S_i = \bigcup_{\CC \in S^{(i)}} \CC$ is homeomorphic to a disk.

\item The circumferences of the cylinders in each story $S^{(i)}$ are strictly decreasing:
$$W_{\CC^{(i)}_{k+1}} <  W_{ \CC^{(i)}_{k}}, \mbox{ for each } 1 \le k < n_i.$$

\item The heights of the cylinders in each story $S^{(i)}$ are non-decreasing:
$$H_{\CC^{(i)}_{k+1}} \ge H_{ \CC^{(i)}_{k}}, \mbox{ for each } 1 \le k < n_i.$$

\item Each story $S^{(i)}$ starts with the eave $\CC^{(i)}_1 =  (k_i, 1, d-k_i, 1)$, for some $1 \le k_i \le (d-1)/2$, and ends with the lighthouse $\CC^{(i)}_{n_i}=  (1, i, 1, d-i)$. All other cylinders $\CC^{(i)}_{j} \in S^{(i)}$, for $1 < j < n_i$ are body.

\item Every non-eave cylinder $\CC^{(i)}_{k}$ in the story $S^{(i)}$ is determined by the previous one using a simple operation analogous to Euclidean algorithm:
$$
\CC^{(i)}_{k-1} = \{ (w_1, s_1), (w_2, s_2)\}, \mbox{ where } w_1 < w_2\implies 
\CC^{(i)}_{k} = \left\{(w_1+w_2, s_1), (w_2, s_2-s_1)\right\}.
$$
Conversely, every non-lighthouse cylinder $\CC^{(i)}_{k}$ in the story $S^{(i)}$ is determined by the next one:
$$
\CC^{(i)}_{k+1} = \{ (w_1, s_1), (w_2, s_2)\}, \mbox{ where } s_1 < s_2 \implies 
\CC^{(i)}_{k} = \left\{(w_1+w_2, s_1), (w_2, s_2-s_1)\right\}.
$$

\item The cylinders $\CC^{(i)}_j$ and $\CC^{(i)}_{j+1}$, for $1 \le j < n_i$, are adjacent: the edges of the top boundary of $\CC^{(i)}_j \in S^{(i)}$ are only identified via folds with each other or by translation to the edges of the bottom boundary of $\CC^{(i)}_{j+1}$. Informally, each story looks like a pyramid. 

\item The edges of the top boundary of a lighthouse are only identified via a fold with each other.

\item The edges of the bottom boundaries of all eaves are only identified with each other via folds and translations.
\end{enumerate}

This section is dedicated to the proof of the following statement:

\begin{theorem} \label{thm:pagoda}
For every prime $d>2$, the square-tiling of the absolute period leaf $\A_{d^2}$ has a pagoda structure,~i.e. decomposes into $\frac{d-1}2$ stories $\mathcal S_i$, satisfying the properties 1-8 above.
\end{theorem}

\noindent
{\bf Example: pagoda structure of $\A_{25}$.} The square-tiling of the absolute period leaf $\A_{25} \cong X(5)$ was presented in the introduction and reproduced below (see Figure~\ref{figX(5)}). It has two stories $S^{(1)}$ and $S^{(2)}$. The story $S^{(1)}$ consists of four cylinders and $S^{(2)}$ consists of three cylinders:
\begin{align*}
&  (1, 1, 1, 4) \mbox{ with } W = 2 \mbox{ and } H=1  &  & (1, 2, 1, 3)   \mbox{ with } W = 2 \mbox{ and } H=2 \\
&  (2, 1, 1, 3)   \mbox{ with } W = 6 \mbox{ and } H=1 &  &  (1, 1, 2, 2)  \mbox{ with } W = 6 \mbox{ and } H=1\\
& (3, 1, 1, 2)   \mbox{ with } W = 12 \mbox{ and } H=1 & & (2, 1, 3, 1)  \mbox{ with } W = 30 \mbox{ and } H=1 \\
& (1, 1, 4, 1)   \mbox{ with } W = 20 \mbox{ and } H=1 &  &
\end{align*}

To see that $S^{(1)}$ and $S^{(2)}$ are homeomorphic to disks, compare Figure~\ref{figX(5)} to Figure~\ref{figonions}. The idea of this topological presentation of the pagoda structure was communicated to the author by Matt Bainbridge, who called this structure a ``belt of onions''. Both stories in this example are homeomorphic to disks. Two disks $S^{(1)}$ and $S^{(2)}$ are docked along their boundaries formed by the edges $a,b,c,d,e,f,g,h,i$ and $j$. This agrees with the fact that the modular curve $X(5)$ has genus 0.

\begin{commentB}
\begin{figure}[H]
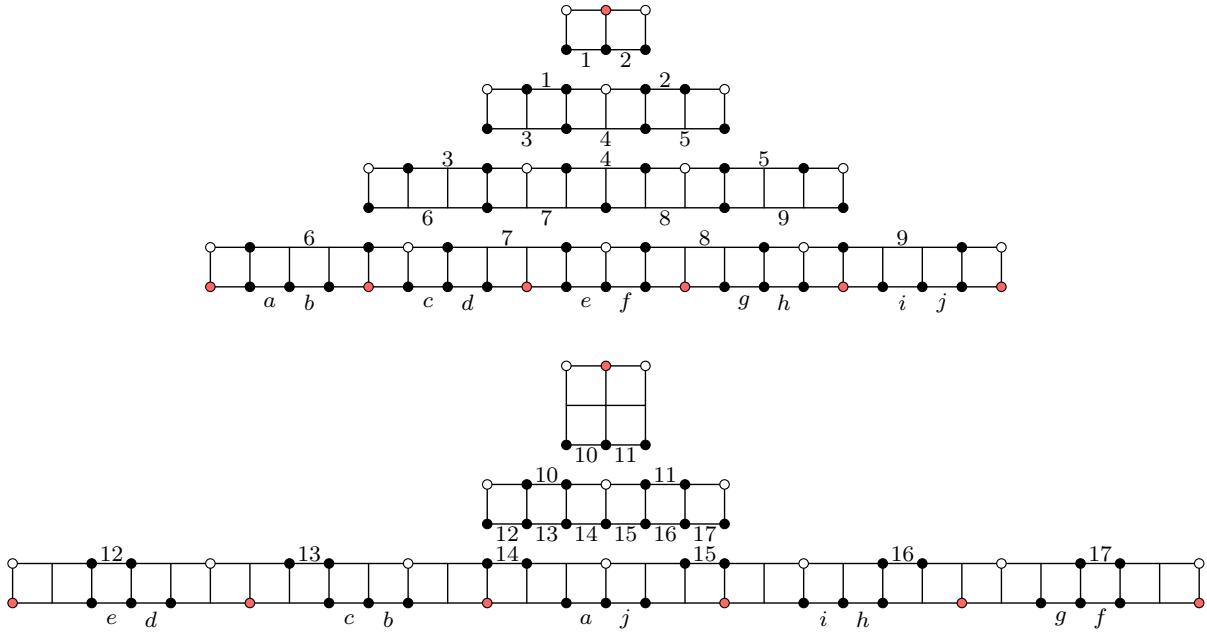
 
  \includestandalone[width=\textwidth]{tikzX5}
  \caption{The pagoda structure of $X(5) \cong \A_{25}$.}
  \label{figX(5)}
\end{figure}

\begin{figure}[H]
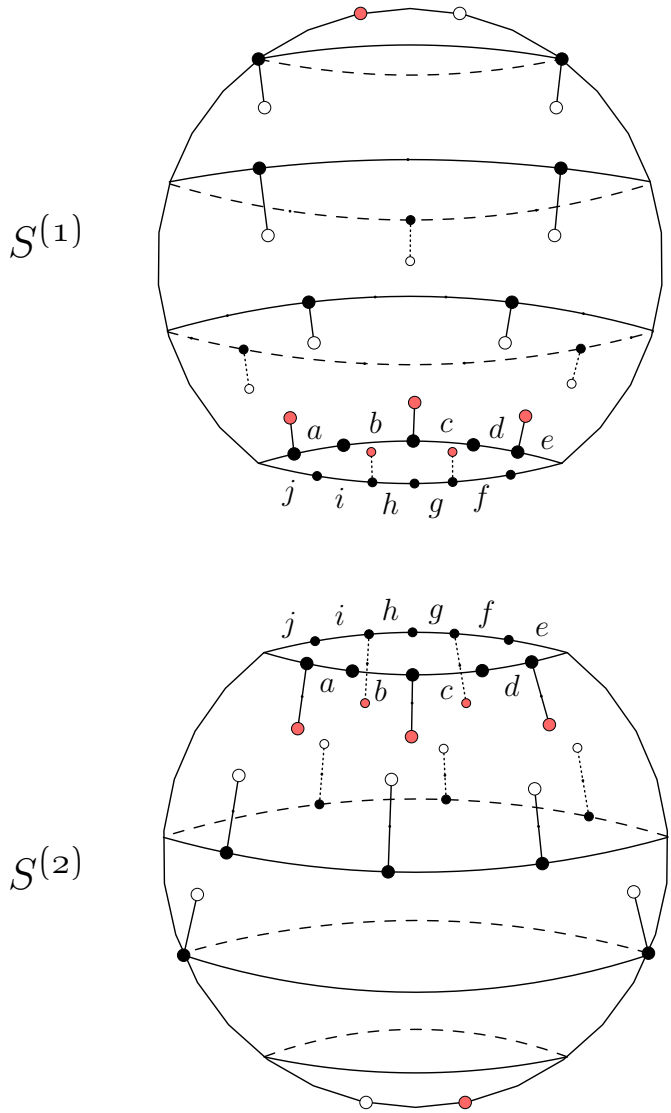
 
\centering
  \includestandalone[width=0.7\textwidth]{tikzonions}
  \caption{The onion structure of $X(5) \cong \A_{25}$.}
  \label{figonions}
\end{figure}
\end{commentB}

The presence of the pagoda structure for any prime $d$ will be evident from the description of the identifications between boundaries of the cylinders that we will carry out below. One can notice that in Figure~\ref{figX(5)intro} the identifications between the boundaries of the adjacent cylinders are given by simply labeling the horizontal segments that start and end in the zeroes of $q$ (black points) with numbers left-to-right. We will generalize this pattern for any prime $d$.

\noindent
{\bf Location of singularities.} We now turn to identifying the types of vertices of the square-tiling of $\A_{d^2}$ at the boundaries of horizontal cylinders, when $d$ is prime.
A cylinder $\mathcal C = (w_1, s_1, w_2, s_2)$ of $\A_{d^2}$ has two boundary components of length $W_{\mathcal C} = w_1 w_2 (w_1 + w_2)$. The {\em origin} of $\CC = (w_1, s_1, w_2, s_2)$ has the Euclidean coordinates $(0,0)$ (see~(\ref{eq:euclcoord})) and the cylinder coordinates $(w_1, s_1, w_2, s_2, t_1=0, t_2 =0, t_3 = 0, h_3=0)$ (see~(\ref{eq:cylcoord})). The set of points with the Euclidean coordinates $(x, 0)$ will be called {\em top} of $\CC$ and the set of points with the Euclidean coordinates $(x,\min(s_1,s_2))$ will be called {\em bottom}.

Recall from Theorem~\ref{thmquaddiff} and Theorem~\ref{thmcusps} that vertices of the square tiling of $\A_{d^2}$ can be zeroes, non-cusp poles, cusp poles or regular points of $q$.

\begin{proposition} \label{cylboundary} 
Let a vertex of the square on the boundary of a cylinder $\CC = (w_1, s_1,w_2, s_2) \subset \A_{d^2}$ have the Euclidean coordinates $(x,y)$, where $y=0$ (top) or $y=s_1$ (bottom). Then:

\begin{itemize}
\item {\bf Lighthouse}:

The boundaries of the lighthouse have length $2$. The top boundary has:
\begin{itemize}
\item a non-cusp pole at $x=0$; and

\item a cusp pole at $x=1$.
\end{itemize}

The bottom boundary has two zeroes at $x=0$ and $x=1$.

\item {\bf Body}:

The top boundary of a body cylinder has:
\begin{itemize}
\item zeroes at $x \equiv w_1 \mod w_1+w_2$ and at $x \equiv w_2 \mod w_1+w_2$;

\item non-cusp poles at $x \equiv 0 \mod w_1+w_2$; and

\item regular points everywhere else.
\end{itemize}

The bottom boundary of a body cylinder has:
\begin{itemize}
\item zeroes at the positions $x \equiv 0 \mod w_1$; and

\item regular points everywhere else.
\end{itemize}

\item {\bf Eave}:

The top boundary of an eave has the same structure as a top of a body cylinder. 

The bottom boundary of an eave cylinder has:
\begin{itemize}
\item cusp poles at  $x \equiv 0 \mod w_1 w_2$;

\item zeroes at $t \equiv 0 \mod w_1$, $t \not\equiv 0 \mod w_1 w_2$ and at $t \equiv 0 \mod w_2$, $t \not\equiv 0 \mod w_1 w_2$; and

\item regular points everywhere else.
\end{itemize}

\end{itemize}

The above cases are illustrated in Figure~\ref{figcylinderboundary}.
\end{proposition}

\begin{commentB}
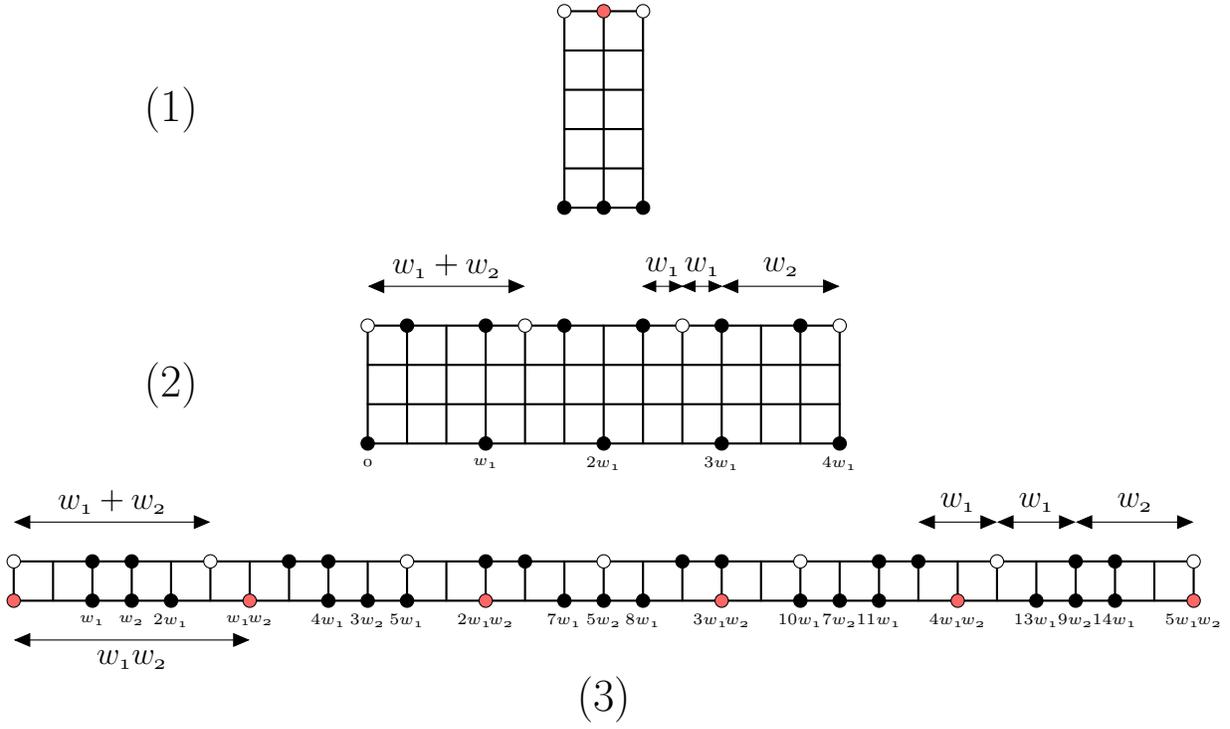
\begin{figure}[H] 
\centering
\definecolor{ffqqqq}{rgb}{1.,0.4,0.4}
\definecolor{ffffff}{rgb}{1.,1.,1.}
\begin{tikzpicture}[line cap=round,line join=round,>=triangle 45,x=0.52cm,y=0.52cm]
\clip(-5.2,-5) rectangle (26,15);

\begin{scope}[shift={(8,6)}]
\draw [line width=0.8pt] (1.,8.)-- (1.,3.);
\draw [line width=0.8pt] (1.,3.)-- (3.,3.);
\draw [line width=0.8pt] (3.,3.)-- (3.,8.);
\draw [line width=0.8pt] (3.,8.)-- (1.,8.);
\draw [line width=0.8pt] (2.,8.)-- (2.,3.);
\draw [line width=0.8pt] (1.,7.)-- (3.,7.);
\draw [line width=0.8pt] (1.,6.)-- (3.,6.);
\draw [line width=0.8pt] (1.,5.)-- (3.,5.);
\draw [line width=0.8pt] (1.,4.)-- (3.,4.);
\begin{tiny}
\draw [fill=ffffff] (1.,8.) circle (2.5pt);
\draw [fill=black] (1.,3.) circle (2.5pt);
\draw [fill=black] (3.,3.) circle (2.5pt);
\draw [fill=ffffff] (3.,8.) circle (2.5pt);
\draw [fill=ffqqqq] (2.,8.) circle (2.5pt);
\draw [fill=black] (2.,3.) circle (2.5pt);
\end{tiny}
\draw (-9,5.5) node  {\Large (1)};
\end{scope}

\begin{scope}[shift={(-3,0)}]

\draw [line width=0.8pt] (7.,6.)-- (7.,3.);
\draw [line width=0.8pt] (7.,3.)-- (19.,3.);
\draw [line width=0.8pt] (19.,3.)-- (19.,6.);
\draw [line width=0.8pt] (19.,6.)-- (7.,6.);
\draw [line width=0.8pt] (7.,5.)-- (19.,5.);
\draw [line width=0.8pt] (7.,4.)-- (19.,4.);
\draw [line width=0.8pt] (8.,6.)-- (8.,3.);
\draw [line width=0.8pt] (9.,6.)-- (9.,3.);
\draw [line width=0.8pt] (10.,6.)-- (10.,3.);
\draw [line width=0.8pt] (11.,6.)-- (11.,3.);
\draw [line width=0.8pt] (12.,6.)-- (12.,3.);
\draw [line width=0.8pt] (13.,6.)-- (13.,3.);
\draw [line width=0.8pt] (14.,6.)-- (14.,3.);
\draw [line width=0.8pt] (15.,6.)-- (15.,3.);
\draw [line width=0.8pt] (16.,6.)-- (16.,3.);
\draw [line width=0.8pt] (17.,6.)-- (17.,3.);
\draw [line width=0.8pt] (18.,6.)-- (18.,3.);

\begin{tiny}
\draw [fill=ffffff] (7.,6.) circle (2.5pt);
\draw [fill=black] (7.,3.) circle (2.5pt);
\draw [fill=black] (19.,3.) circle (2.5pt);
\draw [fill=ffffff] (19.,6.) circle (2.5pt);
\draw [fill=black] (8.,6.) circle (2.5pt);
\draw [fill=black] (10.,6.) circle (2.5pt);
\draw [fill=black] (10.,3.) circle (2.5pt);
\draw [fill=ffffff] (11.,6.) circle (2.5pt);
\draw [fill=black] (12.,6.) circle (2.5pt);
\draw [fill=black] (13.,3.) circle (2.5pt);
\draw [fill=black] (14.,6.) circle (2.5pt);
\draw [fill=ffffff] (15.,6.) circle (2.5pt);
\draw [fill=black] (16.,6.) circle (2.5pt);
\draw [fill=black] (16.,3.) circle (2.5pt);
\draw [fill=black] (18.,6.) circle (2.5pt);
\end{tiny}

\draw (2,4.5) node  {\Large (2)};
\draw (7,2.5) node  {\tiny $0$};
\draw (10,2.5) node  {\tiny$w_1$};
\draw (13,2.5) node  {\tiny2$w_1$};
\draw (16,2.5) node  {\tiny3$w_1$};
\draw (19,2.5) node  {\tiny4$w_1$};

\draw[<->, ] (7,7.) -- (11,7);
\draw (9,7.5) node  {$w_1+w_2$};

\draw[<->, line width = 0.05pt] (14,7.) -- (15,7);
\draw (14.5,7.5) node  {$w_1$};

\draw[<->, line width = 0.05pt] (15,7.) -- (16,7);
\draw (15.5,7.5) node  {$w_1$};

\draw[<->, ] (16,7.) -- (19,7);
\draw (17.5,7.5) node  {$w_2$};

\end{scope}

\draw [line width=0.8pt] (-5.,0.)-- (25.,0.);
\draw [line width=0.8pt] (25.,0.)-- (25.,-1.);
\draw [line width=0.8pt] (25.,-1.)-- (-5.,-1.);
\draw [line width=0.8pt] (-5.,-1.)-- (-5.,0.);
\draw [line width=0.8pt] (-4.,0.)-- (-4.,-1.);
\draw [line width=0.8pt] (-3.,0.)-- (-3.,-1.);
\draw [line width=0.8pt] (-2.,0.)-- (-2.,-1.);
\draw [line width=0.8pt] (-1.,0.)-- (-1.,-1.);
\draw [line width=0.8pt] (0.,0.)-- (0.,-1.);
\draw [line width=0.8pt] (1.,0.)-- (1.,-1.);
\draw [line width=0.8pt] (2.,0.)-- (2.,-1.);
\draw [line width=0.8pt] (3.,0.)-- (3.,-1.);
\draw [line width=0.8pt] (4.,0.)-- (4.,-1.);
\draw [line width=0.8pt] (5.,0.)-- (5.,-1.);
\draw [line width=0.8pt] (6.,0.)-- (6.,-1.);
\draw [line width=0.8pt] (7.,0.)-- (7.,-1.);
\draw [line width=0.8pt] (8.,0.)-- (8.,-1.);
\draw [line width=0.8pt] (9.,0.)-- (9.,-1.);
\draw [line width=0.8pt] (10.,0.)-- (10.,-1.);
\draw [line width=0.8pt] (11.,0.)-- (11.,-1.);
\draw [line width=0.8pt] (12.,0.)-- (12.,-1.);
\draw [line width=0.8pt] (13.,0.)-- (13.,-1.);
\draw [line width=0.8pt] (14.,0.)-- (14.,-1.);
\draw [line width=0.8pt] (15.,0.)-- (15.,-1.);
\draw [line width=0.8pt] (16.,0.)-- (16.,-1.);
\draw [line width=0.8pt] (17.,0.)-- (17.,-1.);
\draw [line width=0.8pt] (18.,0.)-- (18.,-1.);
\draw [line width=0.8pt] (19.,0.)-- (19.,-1.);
\draw [line width=0.8pt] (20.,0.)-- (20.,-1.);
\draw [line width=0.8pt] (21.,0.)-- (21.,-1.);
\draw [line width=0.8pt] (22.,0.)-- (22.,-1.);
\draw [line width=0.8pt] (23.,0.)-- (23.,-1.);
\draw [line width=0.8pt] (24.,0.)-- (24.,-1.);

\draw[<->] (-5,1) -- (0,1);
\draw (-2.5,1.5) node  {$w_1+w_2$};

\draw[<->] (-5,-2.) -- (1,-2);
\draw (-2,-2.5) node  {$w_1 w_2$};

\draw[<->] (18,1.) -- (20,1);
\draw (19,1.5) node  {$w_1$};

\draw[<->] (20,1.) -- (22,1);
\draw (21,1.5) node  {$w_1$};

\draw[<->] (22,1.) -- (25,1);
\draw (23.5,1.5) node  {$w_2$};
%

\draw (10,-3.5) node {\Large (3)};

\draw (-3,-1.5) node  {\tiny $w_1$};
\draw (-1,-1.5) node  {\tiny 2$w_1$};
\draw (3,-1.5) node  {\tiny 4$w_1$};
\draw (5,-1.5) node  {\tiny 5$w_1$};
\draw (9,-1.5) node  {\tiny 7$w_1$};
\draw (11,-1.5) node  {\tiny 8$w_1$};
\draw (15,-1.5) node  {\tiny 10$w_1$};
\draw (17,-1.5) node  {\tiny 11$w_1$};
\draw (21,-1.5) node  {\tiny 13$w_1$};
\draw (23,-1.5) node  {\tiny 14$w_1$};

\draw (-2,-1.5) node  {\tiny $w_2$};
\draw (4,-1.5) node  {\tiny 3$w_2$};
\draw (10,-1.5) node  {\tiny 5$w_2$};
\draw (16,-1.5) node  {\tiny 7$w_2$};
\draw (22,-1.5) node  {\tiny 9$w_2$};

\draw (1,-1.5) node  {\tiny $w_1\!w_2$};
\draw (7,-1.5) node  {\tiny 2$w_1\!w_2$};
\draw (13.0,-1.5) node  {\tiny 3$w_1\!w_2$};
\draw (19.0,-1.5) node  {\tiny 4$w_1\!w_2$};
\draw (25.0,-1.5) node  {\tiny 5$w_1\!w_2$};

\begin{tiny}
\draw [fill=ffffff] (-5.,0.) circle (2.5pt);
\draw [fill=ffffff] (25.,0.) circle (2.5pt);
\draw [fill=ffqqqq] (25.,-1.) circle (2.5pt);
\draw [fill=ffqqqq] (-5.,-1.) circle (2.5pt);
\draw [fill=black] (-3.,0.) circle (2.5pt);
\draw [fill=black] (-3.,-1.) circle (2.5pt);
\draw [fill=black] (-2.,0.) circle (2.5pt);
\draw [fill=black] (-2.,-1.) circle (2.5pt);
\draw [fill=black] (-1.,-1.) circle (2.5pt);
\draw [fill=ffffff] (0.,0.) circle (2.5pt);
\draw [fill=ffqqqq] (1.,-1.) circle (2.5pt);
\draw [fill=black] (2.,0.) circle (2.5pt);
\draw [fill=black] (3.,0.) circle (2.5pt);
\draw [fill=black] (3.,-1.) circle (2.5pt);
\draw [fill=black] (4.,-1.) circle (2.5pt);
\draw [fill=ffffff] (5.,0.) circle (2.5pt);
\draw [fill=black] (5.,-1.) circle (2.5pt);
\draw [fill=black] (7.,0.) circle (2.5pt);
\draw [fill=ffqqqq] (7.,-1.) circle (2.5pt);
\draw [fill=black] (8.,0.) circle (2.5pt);
\draw [fill=black] (9.,-1.) circle (2.5pt);
\draw [fill=ffffff] (10.,0.) circle (2.5pt);
\draw [fill=black] (10.,-1.) circle (2.5pt);
\draw [fill=black] (11.,-1.) circle (2.5pt);
\draw [fill=black] (12.,0.) circle (2.5pt);
\draw [fill=black] (13.,0.) circle (2.5pt);
\draw [fill=ffqqqq] (13.,-1.) circle (2.5pt);
\draw [fill=ffffff] (15.,0.) circle (2.5pt);
\draw [fill=black] (15.,-1.) circle (2.5pt);
\draw [fill=black] (16.,-1.) circle (2.5pt);
\draw [fill=black] (17.,0.) circle (2.5pt);
\draw [fill=black] (17.,-1.) circle (2.5pt);
\draw [fill=black] (18.,0.) circle (2.5pt);
\draw [fill=ffqqqq] (19.,-1.) circle (2.5pt);
\draw [fill=ffffff] (20.,0.) circle (2.5pt);
\draw [fill=black] (21.,-1.) circle (2.5pt);
\draw [fill=black] (22.,0.) circle (2.5pt);
\draw [fill=black] (22.,-1.) circle (2.5pt);
\draw [fill=black] (23.,0.) circle (2.5pt);
\draw [fill=black] (23.,-1.) circle (2.5pt);
\end{tiny}
\end{tikzpicture}
\vspace{-20pt}
\caption{The singularities at the boundaries of
(1) a lighthouse,
(2) a body cylinder,
and (3) an eave.
The cusp are labeled red, the non-cusp poles -- white and zeroes -- black. In these particular examples: (1) $(w_1, s_1, w_2, s_2) = (1, 5, 1, 6)$, (2) $\{ (w_1, s_1), (w_2, s_2) \}= \{(3, 3), (1, 4)\} $, and (3) $(w_1, s_1, w_2, s_2) = (2, 1, 3, 1)$.
}
\label{figcylinderboundary}
\end{figure}
\end{commentB}

\noindent
{\bf Gluing instructions within a story.} We now give instructions of how to identify the edges on the boundaries of  the cylinders of $\A_{d^2}$ within each story,~i.e except for the bottom edges of the eaves. 

Note that every horizontal edge on a saddle connection that starts at a pole of $q$ is identified to another such edge by rotation by $\pi$ around that pole. Therefore it remains to only understand the identifications of the edges that belong to saddle connections between two zeroes of $q$. Proposition~\ref{cylboundary} implies that there are no poles on the bottom boundaries of non-eave cylinders. The following proposition gives gluing instructions between the bottom edges of non-eave cylinders and the top edges of non-lighthouse cylinders of $\A_{d^2}$:

\begin{proposition} \label{prop:gluingwithin}
Let $A$ be a point on the bottom of the cylinder $\CC = \{ (w_1, s_1), (w_2, s_2) \}$ with the Euclidean coordinates $(x, s_2)$, where:
\begin{align*}
& s_1 > s_2, \\
& x = q w_2 + r, \\
& q \in \Z, 0 \le q < w_1 (w_1+w_2), \\
& r \in \R,  0 < r < w_2. 
\end{align*}
Then point $A$ is identified with a point on the top of the cylinder $\CC' = \{ (w'_1, s'_1), (w'_2, s'_2) \}$ with the Euclidean coordinates  $(x', 0)$, where:
\begin{align*}
w'_1 & =w_1, \\
w'_2 & =w_1+w_2, \\ 
s'_1 & =s_1-s_2, \\ 
s'_2 & =s_2,  \\
x' & = q(2w_1+w_2)  + w_1 + r.
\end{align*}

Any other point on the top of a non-lighthouse cylinder belongs to a saddle connection that starts at a pole and is identified with its image under rotation by $\pi$ around that pole.
\end{proposition}

This in particular implies that the bottom edges of every cylinder $\{ (w_1, s_1), (w_2, s_2)\}$ with $s_1 < s_2$ are only identified with the top edges of the cylinder $ \{(w_1+w_2, s_1), (w_2, s_2-s_1)\}$, and the top edges of every cylinder $\{ (w_1, s_1), (w_2, s_2)\}$ with $w_1 < w_2$ are only identified with each other or the top edges of the cylinder
$\{(w_1+w_2, s_1), (w_2, s_2-s_1)\}$. This, together with the observation that the top two edges of every lighthouse are identified with each other (see \S\ref{seclighteave}) and Theorem~\ref{thmdecomposition}, implies the properties 2-8 of the pagoda structure. 

Another implication of Proposition~\ref{prop:gluingwithin} is a simple way to describe the identifications of the horizontal edges on saddle connections between zeroes of $q$ by giving them the following labeling. Label all the edges on the bottom of a non-eave cylinder $\CC = \{ (w_1, s_1), (w_2, s_2) \}$ with $s_1 > s_2$ with numbers from $1$ to $w_1 w_2 (w_1+w_2)$ left-to-right starting from the vertex with the Euclidean coordinates $(0,s_2)$. Now similarly label all the edges on saddle connections that start and end at zeroes of $q$ and lie on the top of a non-lighthouse cylinder $\CC' = \{ (w_1, s_1-s_2), (w_1+w_2, s_2) \}$ with numbers from $1$ to $w_1 w_2 (w_1+w_2)$ left-to-right. The edges with matching labels are identified.

From this perspective it is easy to see that after folding the saddle connections which start at poles of $q$ the boundaries of the cylinders become circles formed by saddles connection between zeroes. The cylinders are then docked to each other along this circles. It follows that each story is homeomorphic to a sphere with boundaries formed by saddle connections on the eave. We next show that the boundary is in fact a single circle.

\noindent
{\bf Gluing instructions between stories.} We now give instructions of how to identify the remaining edges, the ones lying on the bottom boundaries of the eave cylinders of $\A_{d^2}$. As a consequence, we obtain that each story is homeomorphic to a disk.

\begin{proposition} \label{prop:gluingbetween}
Let $A$ be a non-singular point on the bottom of the eave cylinder $\CC = \{ (w_1, 1), (w_2, 1) \}$ with the cylinder coordinates $(w_1, 1, w_2, 1, t_1, t_2, t_3, 1)$, where $w_1 < w_2$ and $t_i \notin \Z$.
Then point $A$ is identified with a point on the bottom of the eave cylinder $\CC' = \{ (w'_1, s'_1), (w'_2, s'_2) \}$ with the cylinder coordinates $(w'_1, 1, w'_2, 1, t'_1, t'_2, t'_3, 1)$, where:
\begin{align*}
w'_1 & =w_1-t_1+t_2, \\
w'_2 & =w_2+t_1-t_2, \\ 
t'_1 & =(w_1-t_1)\%(w_1-t_1+t_2), \\ 
t'_2 & =(w_2-t_2)\%(w_2+t_1-t_2),  \\
t'_3 & = (t_3-t_1-t_2)\%d.
\end{align*}
\end{proposition}

One can see that a point $(w_1, 1, w_2, 1, t_1, t_2, t_3, 1)$ is identified to a point on the same eave if and only if $t_1=t_2=t$ for some $0<t<w_1$. From Proposition~\ref{cylboundary} it follows that such point belongs to the horizontal saddle connection that starts at the pole $(w_1,1,w_2,1,0,0,(t_3-t)\%d)$. Folding the parts of the bottom boundary of an eave that belong to the saddle connections adjacent to simple poles of $q$ one obtains that the boundary of an eave is formed by the edges that are not identified to each other. From the observations above and at the end of the previous subsection we obtain:

\begin{corollary} \label{cordisk}
Every story ${\mathcal S}_i$ of the pagoda is homeomorphic to a disk.
\end{corollary}

\begin{corollary} \label{corgraph}
For every prime $d$, the modular curve $X(d)$ carries an embedded trivalent graph, well-defined up to the action of $\Aut(X(d))$, whose complement is a union of $\frac{d-1}2$ disks.
\end{corollary}

\begin{proof}
Due to the Corollary~\ref{corcanonical} the isomorphism between $\A_{d^2}$ and $X(d)$ is well-defined up to the action of $\Aut(X(d))$. After the choice of isomorphism, the graph is defined by taking a union of all saddle connections on the bottoms of all eaves of $(\A_{d^2},q)$ that connect simple zeroes of $q$. By Corollary~\ref{cordisk} and  the complement is a union of $\frac{d-1}2$ disks.
\end{proof}

\noindent
{\bf Degenerations of 3-cylinder decompostions.} Below we present proofs of Proposition~\ref{cylboundary}, Proposition~\ref{prop:gluingwithin} and Proposition~\ref{prop:gluingbetween}. The proofs are based on analyzing the degenerations of the 3-cylinder decomposition of a generic $(X,\omega) \in \CC \subset \A_{d^2}$ as we approach a boundary of a cylinder $\CC$ in vertical directions. 

Recall from \S\ref{sectiling} that points in the interiors of the horizontal cylinders of $\A_{d^2}$ correspond to Abelian differentials that admit a 3-cylinder decomposition (see Figure~\ref{fig3cyl}). We start by reviewing the 3-cylinder decomposition and defining the 2-cylinder decompositions of Abelian differentials corresponding to the points on the boundaries of the horizontal cylinders of $\A_{d^2}$, except for the bottom boundaries of the eaves.

A generic Abelian differential $(X,\omega)$ in $\A_{d^2}$ admits a decomposition into $3$ horizontal cylinders $C_1$, $C_2$ and $C_3$ of circumferences $w_1$, $w_2$ and $w_1+w_2$ and heights $h_1,h_2$ and $h_3$ respectively. There is a unique way to represent this decomposition as a polygon as follows. Each cylinder is represented by a parallelogram, whose vertices are singularities and edges are saddle connections, with non-horizontal edges identified by parallel translations. First, the parallelogram $C_3$ must be on top of parallelograms $C_1$ and $C_2$. Second, order the cylinders $C_1$ and $C_2$, such that $w_1 \le w_2$ and if $w_1 =w_2$ then $s_1 < s_2$, and put $C_1$ on the left from $C_2$. It remains to define the non-horizontal sides of the parallelograms and identifications of their horizontal sides. 

Let us denote the top boundary of a cylinder $C$ by $C^+$ and the bottom boundary by $C^-$. The boundaries of $C_i$ are closed curves, they are oriented such that their periods are positive real numbers. They also satisfy:
$$
C_3^+ = C_1^- \cup C_2^- \mbox{ and } C_3^- = C_1^+ \cup C_2^+.
$$
For $i=1$ or $2$, there is a unique saddle connection $c_i$ that starts at the left end of $C_i^+$, ends at the left end of $C_i^-$ and satisfies:
$$0 \le \Real\left( \int_{c_i} \omega\right) < w_i.$$  
For $C_3$, there is a unique saddle connection $c_3$ that starts at the left end of $C_1^+ \subset C_3^-$, ends at the left end of $C_i^- \subset C_3^+$ and satisfies:
$$0 \le \Real\left( \int_{c_3} \omega\right) < w_1+w_2.$$  
For $i = 1,2$ or $3$, we define the non-horizontal sides of the parallelogram $C_i$ to be given by the vector $\int_{c_i} \omega \in \C$. Then the twist parameters $t_1, t_2, t_3$, satisfying:
\begin{equation} \label{twists}
0 \le t_1 < w_1, 0 \le t_2 < w_2, 0 \le t_3 < w_1+w_2,
\end{equation}
are simply given by $\Real\left( \int_{c_i} \omega\right)$.

Similarly, any Abelian differential $(X,\omega)$ in $\A_{d^2}$ with two simple zeroes that admits a decomposition into $2$ cylinders $C_1$ and $C_2$ can be uniquely represented as a polygon in the plane as follows. There are two types of 2-cylinder decompositions we have to distinguish: we will call them a {\em 2-cylinder decomposition of type 1} (see Figure~\ref{fig2cyl1}) and a {\em 2-cylinder decomposition of type 2} (see Figure~\ref{fig2cyl2}). For $i=1$ or $2$, denote the circumferences of $C_1, C_2$ by $W^{(i)}_1, W^{(i)}_2$ and the heights by $H^{(i)}_1, H^{(i)}_2$. We order the cylinders such that $W^{(i)}_1 \le W^{(i)}_2$ and if $W^{(i)}_1 = W^{(i)}_2$ then $H^{(i)}_1 < H^{(i)}_2$. The twist parameters $T^{(1)}_1, T^{(1)}_2, T^{(1)}_3$ of the type 1 decomposition satisfy:
\begin{align*} 
0 \le \ &\ T^{(1)}_1 < W^{(1)}_1 \\
0 \le \ &\ T^{(1)}_2 < W^{(1)}_2 \\
0 \le \ &\ T^{(1)}_3 \le W^{(1)}_1, \hspace{34pt}
\end{align*}
and the twist parameters $T^{(2)}_1, T^{(2)}_2, T^{(2)}_3$ of the type 2 decomposition satisfy:
\begin{align*}
0 \le \ &\ T^{(2)}_1 < W^{(2)}_1 \\
0 \le \ &\ T^{(2)}_2 \le W^{(2)}_2-W^{(2)}_1 \\
0 \le \ &\ T^{(2)}_3 < W^{(2)}_2.
\end{align*}
The vector $(W_1, H_1, W_2, H_2, T^{(i)}_1, T^{(i)}_2, T^{(i)}_3)$ is called the {\em cylinder coordinates} of the 2-cylinder decomposition of type $i$, for $i=1$ or $2$. Note that in certain cases of equalities in the inequalities above, the 2-cylinder decomposition is not an Abelian differential anymore, but a square-tiled surface with a node.

\begin{commentB}
\begin{figure}[!htb]
   \begin{minipage}{0.48\textwidth}
     \centering
     \definecolor{qqzzqq}{rgb}{0.,0.6,0.}
\definecolor{qqqqff}{rgb}{0.,0.,1.}
\definecolor{ffffff}{rgb}{1.,1.,1.}
\begin{tikzpicture}[line cap=round,line join=round,>=triangle 45,x=0.4cm,y=0.4cm]
\clip(-35,-52.) rectangle (0.,-4.);
\fill[line width=0.8pt,color=qqqqff,fill=qqqqff,fill opacity=0.10000000149011612] (-10.,-26.) -- (-4.,-40.) -- (-14.,-40.) -- (-20.,-26.) -- cycle;
\fill[line width=0.8pt,color=qqzzqq,fill=qqzzqq,fill opacity=0.10000000149011612] (-20.,-16.) -- (-30.,-26.) -- (-16.,-26.) -- (-6.,-16.) -- cycle;
\draw [line width=2pt,color=qqqqff] (-10.,-26.)-- (-4.,-40.);
\draw [line width=2pt,color=qqqqff] (-14.,-40.)-- (-20.,-26.);
\draw [line width=2pt,color=qqzzqq] (-20.,-16.)-- (-30.,-26.);
\draw [line width=2pt,color=qqzzqq] (-16.,-26.)-- (-6.,-16.);
\draw [->,line width=0.8pt] (-20.,-10.) -- (-16.,-10.);
\draw (-18,-8) node {\Huge $T_3$};
\draw [line width=2pt,color=qqzzqq] (-20.,-16.)-- (-10.,-16.);
\draw [line width=2pt,color=qqzzqq] (-15.174407854833003,-15.563980362917508) -- (-15.174407854833003,-16.43601963708249);
\draw [line width=2pt,color=qqzzqq] (-14.825592145167,-15.563980362917508) -- (-14.825592145167,-16.43601963708249);
\draw [line width=2pt,color=qqqqff] (-10.,-16.)-- (-6.,-16.);
\draw [line width=2pt,color=qqqqff] (-8.,-15.563980362917508) -- (-8.,-16.43601963708249);
\draw [line width=2pt,color=qqzzqq] (-30.,-26.)-- (-20.,-26.);
\draw [line width=2pt,color=qqzzqq] (-25.174407854833007,-25.563980362917512) -- (-25.174407854833007,-26.43601963708249);
\draw [line width=2pt,color=qqzzqq] (-24.825592145167,-25.563980362917512) -- (-24.825592145167,-26.43601963708249);
\draw [line width=2pt,color=qqzzqq] (-20.,-26.)-- (-16.,-26.);
\draw [line width=2pt,color=qqzzqq] (-18.,-25.563980362917512) -- (-18.,-26.43601963708249);
\draw [line width=2pt,color=qqqqff] (-16.,-26.)-- (-10.,-26.);
\draw [line width=2pt,color=qqqqff] (-13.174407854833001,-25.563980362917512) -- (-13.174407854833001,-26.43601963708249);
\draw [line width=2pt,color=qqqqff] (-12.825592145166999,-25.563980362917512) -- (-12.825592145166999,-26.43601963708249);
\draw [line width=2pt,color=qqzzqq] (-20.,-26.)-- (-16.,-26.);
\draw [line width=2pt,color=qqzzqq] (-18.,-25.563980362917512) -- (-18.,-26.43601963708249);
\draw [->,line width=0.8pt] (-20.,-46.) -- (-14.,-46.);
\draw (-17,-48) node {\Huge $T_1$};
\draw [line width=2pt,color=qqqqff] (-14.,-40.)-- (-10.,-40.);
\draw [line width=2pt,color=qqqqff] (-12.,-39.563980362917505) -- (-12.,-40.43601963708249);
\draw [line width=0.1pt,dotted] (-20.,-26.)-- (-20.,-46.);
\draw [line width=0.1pt,dotted] (-14.,-40.)-- (-14.,-46.);
\draw [line width=0.1pt,dotted] (-20.,-26.)-- (-20.,-10.);
\draw [line width=0.1pt,dotted] (-16.,-26.)-- (-16.,-10.);
\draw [line width=2pt,color=qqqqff] (-10.,-40.)-- (-4.,-40.);
\draw [line width=2pt,color=qqqqff] (-7.174407854833005,-39.563980362917505) -- (-7.174407854833005,-40.43601963708249);
\draw [line width=2pt,color=qqqqff] (-6.825592145167001,-39.563980362917505) -- (-6.825592145167001,-40.43601963708249);
\draw [line width=0.1pt,dotted] (-30.,-26.)-- (-30.,-10.);
\draw [->,line width=0.8pt] (-30.,-10.) -- (-20.,-10.);
\draw (-25,-8) node {\Huge $T_2$};
\draw [line width=0.1pt,dotted] (-20.,-16.)-- (-20.,-10.);
\begin{scriptsize}
\draw [fill=ffffff] (-10.,-26.) circle (5pt);
\draw [fill=ffffff] (-14.,-40.) circle (5pt);
\draw [fill=ffffff] (-4.,-40.) circle (5pt);
\draw [fill=ffffff] (-20.,-26.) circle (5pt);
\draw [fill=black] (-20.,-16.) circle (5pt);
\draw [fill=ffffff] (-10.,-16.) circle (5pt);
\draw [fill=black] (-6.,-16.) circle (5pt);
\draw [fill=black] (-30.,-26.) circle (5pt);
\draw [fill=black] (-16.,-26.) circle (5pt);
\draw [fill=ffffff] (-20.,-26.) circle (5pt);
\draw [fill=black] (-16.,-26.) circle (5pt);
\draw [fill=black] (-10.,-40.) circle (5pt);
\end{scriptsize}
\end{tikzpicture}
      \vspace{-23pt}
     \caption{Twist parameters of the 2-cylinder decomposition of type 1.}
     \label{fig2cyl1}
   \end{minipage}\hfill
   \begin {minipage}{0.48\textwidth}
     \centering
     \definecolor{qqzzqq}{rgb}{0.,0.6,0.}
\definecolor{qqqqff}{rgb}{0.,0.,1.}
\definecolor{zzttqq}{rgb}{0.6,0.2,0.}
\definecolor{ffffff}{rgb}{1.,1.,1.}
\definecolor{ttqqqq}{rgb}{0.2,0.,0.}
\begin{tikzpicture}[line cap=round,line join=round,>=triangle 45,x=0.4cm,y=0.4cm]
\clip(35,-66) rectangle (70,-16);
\fill[line width=2pt,color=zzttqq,fill=zzttqq,fill opacity=0.10000000149011612] (42.,-30.) -- (40.,-40.) -- (62.,-40.) -- (64.,-30.) -- cycle;
\fill[line width=2pt,color=qqqqff,fill=qqqqff,fill opacity=0.10000000149011612] (48.,-40.) -- (52.,-44.) -- (44.,-44.) -- (40.,-40.) -- cycle;
\fill[line width=2pt,color=qqzzqq,fill=qqzzqq,fill opacity=0.10000000149011612] (48.,-40.) -- (54.,-40.) -- (62.,-40.) -- (62.,-40.) -- cycle;
\draw [->,line width=0.8pt] (48.,-50.) -- (54.,-50.);
\draw (51,-52) node {\Huge $T_2$};
\draw [line width=2pt,color=zzttqq] (42.,-30.)-- (40.,-40.);
\draw [line width=2pt,color=zzttqq] (62.,-40.)-- (64.,-30.);
\draw [->,line width=0.8pt] (40.,-50.) -- (44.,-50.);
\draw (42,-52) node {\Huge $T_1$};
\draw [line width=2pt,color=qqqqff] (48.,-40.)-- (52.,-44.);
\draw [line width=2pt,color=qqqqff] (52.,-44.)-- (44.,-44.);
\draw [line width=2pt,color=qqqqff] (48.,-44.507311210422706) -- (48.,-43.4926887895773);
\draw [line width=2pt,color=qqqqff] (44.,-44.)-- (40.,-40.);
\draw [line width=2pt,color=qqqqff] (40.,-40.)-- (48.,-40.);
\draw [line width=2pt,color=qqzzqq] (54.,-40.)-- (62.,-40.);
\draw [line width=2pt,color=qqzzqq] (57.797075515830905,-39.4926887895773) -- (57.797075515830905,-40.507311210422706);
\draw [line width=2pt,color=qqzzqq] (58.202924484169074,-39.4926887895773) -- (58.202924484169074,-40.507311210422706);
\draw [line width=2pt,color=qqzzqq] (62.,-40.)-- (62.,-40.);
\draw [line width=2pt,color=qqqqff] (42.,-30.)-- (50.,-30.);
\draw [line width=2pt,color=qqqqff] (46.,-29.4926887895773) -- (46.,-30.507311210422706);
\draw [line width=0.1pt,dotted] (48.,-40.)-- (48.,-50.);
\draw [line width=0.1pt,dotted] (40.,-40.)-- (40.,-24.);
\draw [line width=0.1pt,dotted] (42.,-30.)-- (42.,-24.);
\draw [->,line width=0.8pt] (40.,-24.) -- (42.,-24.);
\draw (41,-22) node {\Huge $T_3$};
\draw [line width=0.1pt,dotted] (40.,-40.)-- (40.,-50.);
\draw [line width=0.1pt,dotted] (44.,-44.)-- (44.,-50.);
\draw [line width=2pt,color=qqzzqq] (48.,-40.)-- (54.,-40.);
\draw [line width=2pt,color=qqzzqq] (51.,-39.4926887895773) -- (51.,-40.507311210422706);
\draw [line width=2pt,color=qqzzqq] (58.,-30.)-- (64.,-30.);
\draw [line width=2pt,color=qqzzqq] (61.,-29.4926887895773) -- (61.,-30.507311210422706);
\draw [line width=2pt,color=qqzzqq] (50.,-30.)-- (58.,-30.);
\draw [line width=2pt,color=qqzzqq] (53.79707551583091,-29.4926887895773) -- (53.79707551583091,-30.507311210422706);
\draw [line width=2pt,color=qqzzqq] (54.20292448416909,-29.4926887895773) -- (54.20292448416909,-30.507311210422706);
\draw [line width=0.1pt,dotted] (54.,-40.)-- (54.,-50.);
\begin{scriptsize}
\draw [fill=ttqqqq] (42.,-30.) circle (5.5pt);
\draw [fill=ffffff] (62.,-40.) circle (5.5pt);
\draw [fill=black] (64.,-30.) circle (5.5pt);
\draw [fill=ffffff] (40.,-40.) circle (5.5pt);
\draw [fill=ffffff] (58.,-30.) circle (5.5pt);
\draw [fill=ffffff] (48.,-40.) circle (5.5pt);
\draw [fill=black] (52.,-44.) circle (5.5pt);
\draw [fill=black] (50.,-30.) circle (5.5pt);
\draw [fill=black] (44.,-44.) circle (5.5pt);
\draw [fill=black] (54.,-40.) circle (5.5pt);
\draw [fill=black] (62.,-40.) circle (5.5pt);
\end{scriptsize}
\end{tikzpicture}
     \vspace{-30pt}
     \caption{Twist parameters of the 2-cylinder decomposition of type 2.}
     \label{fig2cyl2}
   \end{minipage}
\end{figure}
\end{commentB}

Now take any generic $(X,\omega) \in \CC = \{(w_1, s_1), (w_2, s_2)\} \subset \A_{d^2}$, where $w_1 < w_2$ with non-integer twist parameters $t_1, t_2, t_3 \notin \Z$, satisfying (\ref{twists}). We will describe how coordinates change after we zip up or zip down a singularity. For zipping down we will assume $s_1\ne s_2$, the case $s_1=s_2=1$ will be treated later. There are three cases to consider:
(1) $0 < t_3 < w_1$ (see Figure~\ref{figtobdry1}.1); (2) $w_1 < t_3 < w_2$ (see Figure~\ref{figtobdry2}.1); and (3) $w_2 < t_3 < w_1+w_2$ (see Figure~\ref{figtobdry3}.1). For the Abelian differentials obtained by moving the white singularity upwards see Figure~\ref{figtobdry1}.2, Figure~\ref{figtobdry2}.2 and Figure~\ref{figtobdry3}.2. For the Abelian differentials obtained by moving it downwards see Figure~\ref{figtobdry1}.3, Figure~\ref{figtobdry2}.3 and Figure~\ref{figtobdry3}.3.

\begin{commentB}
\begin{figure}[H]
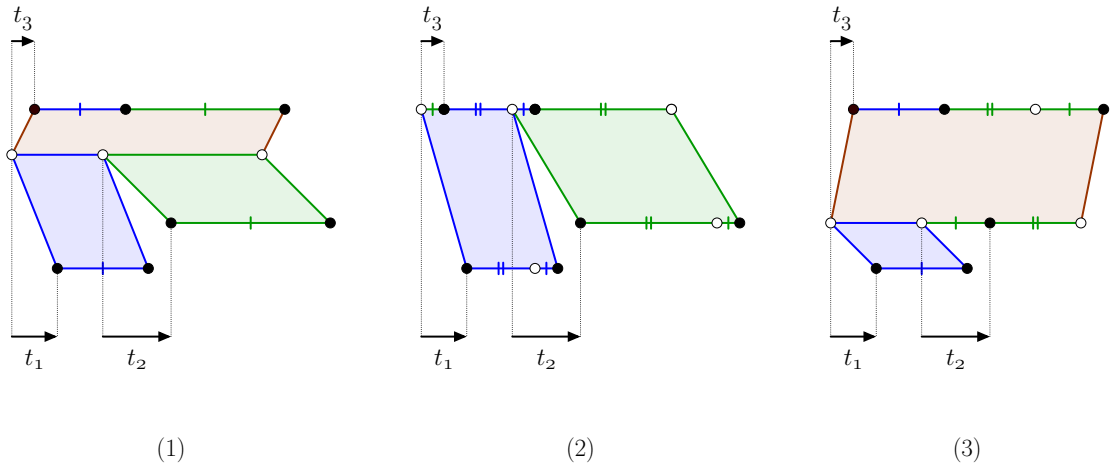
 
\centering
  \includestandalone[width=\textwidth]{tikztobdry1}
  \caption{(1) An Abelian differential $(X,\omega)\in \A_{d^2}$ with $0 < t_3 < w_1$. (2) Zipping up and (3) zipping down of its white singularity.}
  \label{figtobdry1}
\end{figure}

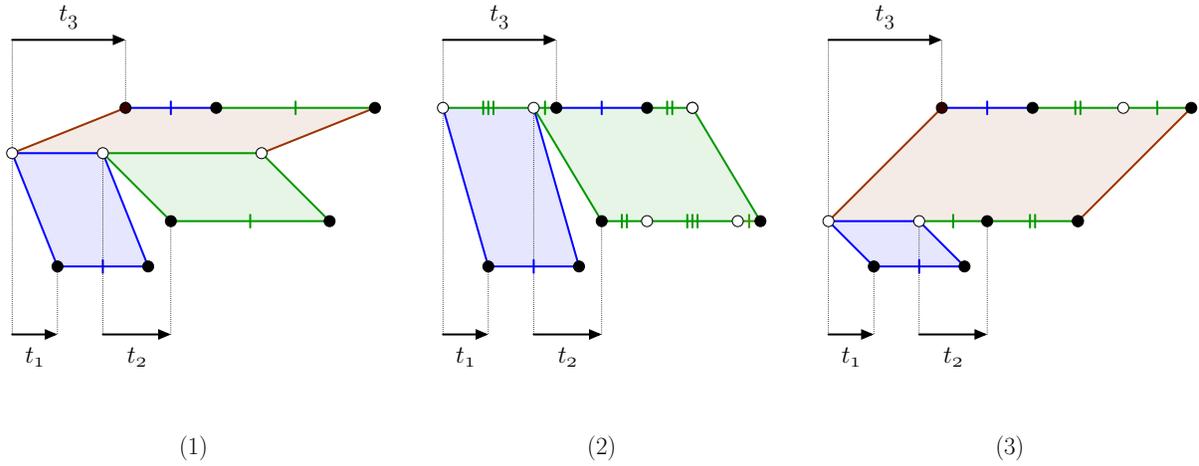
\begin{figure}[H] 
\centering
  \definecolor{ttzzqq}{rgb}{0.2,0.6,0.}
\definecolor{qqzzqq}{rgb}{0.,0.6,0.}
\definecolor{qqqqff}{rgb}{0.,0.,1.}
\definecolor{zzttqq}{rgb}{0.6,0.2,0.}
\definecolor{ffffff}{rgb}{1.,1.,1.}
\definecolor{ttqqqq}{rgb}{0.2,0.,0.}
\begin{tikzpicture}[line cap=round,line join=round,>=triangle 45,x=0.4cm,y=0.4cm]
\clip(-34,-64) rectangle (74,-18);
\fill[line width=2pt,color=zzttqq,fill=zzttqq,fill opacity=0.10000000149011612] (50.,-30.) -- (40.,-40.) -- (62.,-40.) -- (72.,-30.) -- cycle;
\fill[line width=2pt,color=qqqqff,fill=qqqqff,fill opacity=0.10000000149011612] (-24.,-34.) -- (-20.,-44.) -- (-28.,-44.) -- (-32.,-34.) -- cycle;
\fill[line width=2pt,color=qqzzqq,fill=qqzzqq,fill opacity=0.10000000149011612] (-24.,-34.) -- (-18.,-40.) -- (-4.,-40.) -- (-10.,-34.) -- cycle;
\fill[line width=2pt,color=zzttqq,fill=zzttqq,fill opacity=0.10000000149011612] (-22.,-30.) -- (-32.,-34.) -- (-10.,-34.) -- (0.,-30.) -- cycle;
\fill[line width=2pt,color=qqqqff,fill=qqqqff,fill opacity=0.10000000149011612] (14.,-30.) -- (18.,-44.) -- (10.,-44.) -- (6.,-30.) -- cycle;
\fill[line width=2pt,color=qqzzqq,fill=qqzzqq,fill opacity=0.10000000149011612] (14.,-30.) -- (20.,-40.) -- (34.,-40.) -- (28.,-30.) -- cycle;
\fill[line width=2pt,color=qqqqff,fill=qqqqff,fill opacity=0.10000000149011612] (48.,-40.) -- (52.,-44.) -- (44.,-44.) -- (40.,-40.) -- cycle;
\fill[line width=2pt,color=qqzzqq,fill=qqzzqq,fill opacity=0.10000000149011612] (48.,-40.) -- (54.,-40.) -- (62.,-40.) -- (62.,-40.) -- cycle;
\draw [->,line width=2pt] (48.,-50.) -- (54.,-50.);
\draw (51,-52) node {\Huge $t_2$};
\draw [line width=2pt,color=zzttqq] (50.,-30.)-- (40.,-40.);
\draw [line width=2pt,color=zzttqq] (62.,-40.)-- (72.,-30.);
\draw [->,line width=2pt] (-32.,-50.) -- (-28.,-50.);
\draw (-30,-52) node {\Huge $t_1$};
\draw [->,line width=2pt] (-24.,-50.) -- (-18.,-50.);
\draw (-21,-52) node {\Huge $t_2$};

\draw [line width=2pt,color=qqqqff] (-24.,-34.)-- (-20.,-44.);
\draw [line width=2pt,color=qqqqff] (-20.,-44.)-- (-28.,-44.);
\draw [line width=2pt,color=qqqqff] (-24.,-44.5639568912708) -- (-24.,-43.43604310872921);
\draw [line width=2pt,color=qqqqff] (-28.,-44.)-- (-32.,-34.);
\draw [line width=2pt,color=qqqqff] (-32.,-34.)-- (-24.,-34.);
\draw [line width=2pt,color=qqzzqq] (-24.,-34.)-- (-18.,-40.);
\draw [line width=2pt,color=qqzzqq] (-18.,-40.)-- (-4.,-40.);
\draw [line width=2pt,color=qqzzqq] (-11.,-39.4360431087292) -- (-11.,-40.56395689127079);
\draw [line width=2pt,color=qqzzqq] (-4.,-40.)-- (-10.,-34.);
\draw [line width=2pt,color=qqzzqq] (-10.,-34.)-- (-24.,-34.);
\draw [line width=2pt,color=zzttqq] (-22.,-30.)-- (-32.,-34.);
\draw [line width=2pt,color=zzttqq] (-10.,-34.)-- (0.,-30.);
\draw [line width=2pt,color=qqqqff] (-22.,-30.)-- (-14.,-30.);
\draw [line width=2pt,color=qqqqff] (-18.,-29.436043108729205) -- (-18.,-30.563956891270795);
\draw [line width=2pt,color=qqzzqq] (-14.,-30.)-- (0.,-30.);
\draw [line width=2pt,color=qqzzqq] (-7.,-29.436043108729205) -- (-7.,-30.563956891270795);
\draw [line width=0.1pt,dotted] (-24.,-34.)-- (-24.,-50.);
\draw [line width=0.1pt,dotted] (-32.,-34.)-- (-32.,-24.);
\draw [line width=0.1pt,dotted] (-22.,-30.)-- (-22.,-24.);
\draw [->,line width=2pt] (-32.,-24.) -- (-22.,-24.);
\draw (-27,-22) node {\Huge $t_3$};
\draw [line width=0.1pt,dotted] (-32.,-34.)-- (-32.,-50.);
\draw [line width=0.1pt,dotted] (-28.,-44.)-- (-28.,-50.);

\draw [->,line width=2pt] (6.,-50.) -- (10.,-50.);
\draw (8,-52) node {\Huge $t_1$};
\draw [->,line width=2pt] (14.,-50.) -- (20.,-50.);
\draw (17,-52) node {\Huge $t_2$};

\draw [line width=2pt,color=qqqqff] (14.,-30.)-- (18.,-44.);
\draw [line width=2pt,color=qqqqff] (10.,-44.)-- (6.,-30.);
\draw [line width=2pt,color=qqzzqq] (14.,-30.)-- (20.,-40.);
\draw [line width=2pt,color=qqzzqq] (34.,-40.)-- (28.,-30.);
\draw [line width=0.1pt,dotted] (14.,-30.)-- (14.,-50.);
\draw [line width=0.1pt,dotted] (6.,-30.)-- (6.,-24.);
\draw [->,line width=2pt] (6.,-24.) -- (16.,-24.);
\draw (11,-22) node {\Huge $t_3$};
\draw [line width=0.1pt,dotted] (6.,-30.)-- (6.,-50.);
\draw [line width=0.1pt,dotted] (10.,-44.)-- (10.,-50.);
\draw [->,line width=2pt] (40.,-50.) -- (44.,-50.);
\draw (42,-52) node {\Huge $t_1$};
\draw [line width=2pt,color=qqqqff] (48.,-40.)-- (52.,-44.);
\draw [line width=2pt,color=qqqqff] (52.,-44.)-- (44.,-44.);
\draw [line width=2pt,color=qqqqff] (48.,-44.5639568912708) -- (48.,-43.43604310872921);
\draw [line width=2pt,color=qqqqff] (44.,-44.)-- (40.,-40.);
\draw [line width=2pt,color=qqqqff] (40.,-40.)-- (48.,-40.);
\draw [line width=2pt,color=qqzzqq] (54.,-40.)-- (62.,-40.);
\draw [line width=2pt,color=qqzzqq] (57.77441724349167,-39.4360431087292) -- (57.77441724349167,-40.56395689127079);
\draw [line width=2pt,color=qqzzqq] (58.22558275650832,-39.4360431087292) -- (58.22558275650832,-40.56395689127079);
\draw [line width=2pt,color=qqzzqq] (62.,-40.)-- (62.,-40.);
\draw [line width=2pt,color=qqqqff] (50.,-30.)-- (58.,-30.);
\draw [line width=2pt,color=qqqqff] (54.,-29.436043108729205) -- (54.,-30.563956891270795);
\draw [line width=0.1pt,dotted] (48.,-40.)-- (48.,-50.);
\draw [line width=0.1pt,dotted] (40.,-40.)-- (40.,-24.);
\draw [line width=0.1pt,dotted] (50.,-30.)-- (50.,-24.);
\draw [->,line width=2pt] (40.,-24.) -- (50.,-24.);
\draw (45,-22) node {\Huge $t_3$};
\draw [line width=0.1pt,dotted] (40.,-40.)-- (40.,-50.);
\draw [line width=0.1pt,dotted] (44.,-44.)-- (44.,-50.);
\draw [line width=2pt,color=qqzzqq] (20.,-40.)-- (24.,-40.);
\draw [line width=2pt,color=qqzzqq] (21.774417243491676,-39.4360431087292) -- (21.774417243491676,-40.56395689127079);
\draw [line width=2pt,color=qqzzqq] (22.225582756508324,-39.4360431087292) -- (22.225582756508324,-40.56395689127079);
\draw [line width=2pt,color=qqzzqq] (48.,-40.)-- (54.,-40.);
\draw [line width=2pt,color=qqzzqq] (51.,-39.4360431087292) -- (51.,-40.56395689127079);
\draw [line width=2pt,color=qqzzqq] (66.,-30.)-- (72.,-30.);
\draw [line width=2pt,color=qqzzqq] (69.,-29.436043108729205) -- (69.,-30.563956891270795);
\draw [line width=2pt,color=qqzzqq] (58.,-30.)-- (66.,-30.);
\draw [line width=2pt,color=qqzzqq] (61.774417243491676,-29.436043108729205) -- (61.774417243491676,-30.563956891270795);
\draw [line width=2pt,color=qqzzqq] (62.225582756508324,-29.436043108729205) -- (62.225582756508324,-30.563956891270795);
\draw [line width=0.1pt,dotted] (54.,-40.)-- (54.,-50.);
\draw [line width=2pt,color=qqqqff] (16.,-30.)-- (24.,-30.);
\draw [line width=2pt,color=qqqqff] (20.,-29.436043108729205) -- (20.,-30.563956891270795);
\draw [line width=2pt,color=qqqqff] (10.,-44.)-- (18.,-44.);
\draw [line width=2pt,color=qqqqff] (14.,-43.43604310872921) -- (14.,-44.5639568912708);
\draw [line width=2pt,color=qqzzqq] (24.,-40.)-- (32.,-40.);
\draw [line width=2pt,color=qqzzqq] (27.54883448698336,-39.4360431087292) -- (27.54883448698336,-40.56395689127079);
\draw [line width=2pt,color=qqzzqq] (28.,-39.4360431087292) -- (28.,-40.56395689127079);
\draw [line width=2pt,color=qqzzqq] (28.451165513016655,-39.4360431087292) -- (28.451165513016655,-40.56395689127079);
\draw [line width=2pt,color=ttzzqq] (32.,-40.)-- (34.,-40.);
\draw [line width=2pt,color=ttzzqq] (33.,-39.4360431087292) -- (33.,-40.56395689127079);
\draw [line width=2pt,color=qqzzqq] (24.,-30.)-- (28.,-30.);
\draw [line width=2pt,color=qqzzqq] (25.77441724349168,-29.436043108729205) -- (25.77441724349168,-30.563956891270795);
\draw [line width=2pt,color=qqzzqq] (26.225582756508327,-29.436043108729205) -- (26.225582756508327,-30.563956891270795);
\draw [line width=2pt,color=qqzzqq] (6.,-30.)-- (14.,-30.);
\draw [line width=2pt,color=qqzzqq] (9.548834486983353,-29.436043108729205) -- (9.548834486983353,-30.563956891270795);
\draw [line width=2pt,color=qqzzqq] (10.,-29.436043108729205) -- (10.,-30.563956891270795);
\draw [line width=2pt,color=qqzzqq] (10.45116551301665,-29.436043108729205) -- (10.45116551301665,-30.563956891270795);
\draw [line width=2pt,color=qqzzqq] (14.,-30.)-- (16.,-30.);
\draw [line width=2pt,color=qqzzqq] (15.,-29.436043108729205) -- (15.,-30.563956891270795);
\draw [line width=0.1pt,dotted] (20.,-40.)-- (20.,-50.);
\draw [line width=0.1pt,dotted] (-18.,-40.)-- (-18.,-50.);
\draw [line width=0.1pt,dotted] (16.,-30.)-- (16.,-24.);

\draw (-16.,-60.) node {\Huge (1)};
\draw (20.,-60.) node {\Huge (2)};
\draw (56.,-60.) node {\Huge (3)};

\begin{scriptsize}
\draw [fill=ttqqqq] (50.,-30.) circle (5.5pt);
\draw [fill=ffffff] (62.,-40.) circle (5.5pt);
\draw [fill=black] (72.,-30.) circle (5.5pt);
\draw [fill=ffffff] (40.,-40.) circle (5.5pt);
\draw [fill=ffffff] (66.,-30.) circle (5.5pt);
\draw [fill=ffffff] (-32.,-34.) circle (5.5pt);
\draw [fill=ffffff] (-24.,-34.) circle (5.5pt);
\draw [fill=black] (-20.,-44.) circle (5.5pt);
\draw [fill=black] (-14.,-30.) circle (5.5pt);
\draw [fill=black] (0.,-30.) circle (5.5pt);
\draw [fill=black] (-28.,-44.) circle (5.5pt);
\draw [fill=ffffff] (-10.,-34.) circle (5.5pt);
\draw [fill=black] (-18.,-40.) circle (5.5pt);
\draw [fill=black] (-4.,-40.) circle (5.5pt);
\draw [fill=ttqqqq] (-22.,-30.) circle (5.5pt);
\draw [fill=ffffff] (6.,-30.) circle (5.5pt);
\draw [fill=ffffff] (14.,-30.) circle (5.5pt);
\draw [fill=black] (18.,-44.) circle (5.5pt);
\draw [fill=black] (24.,-30.) circle (5.5pt);
\draw [fill=black] (28.,-30.) circle (5.5pt);
\draw [fill=black] (10.,-44.) circle (5.5pt);
\draw [fill=ffffff] (28.,-30.) circle (5.5pt);
\draw [fill=black] (20.,-40.) circle (5.5pt);
\draw [fill=black] (34.,-40.) circle (5.5pt);
\draw [fill=ffffff] (48.,-40.) circle (5.5pt);
\draw [fill=black] (52.,-44.) circle (5.5pt);
\draw [fill=black] (58.,-30.) circle (5.5pt);
\draw [fill=black] (44.,-44.) circle (5.5pt);
\draw [fill=black] (54.,-40.) circle (5.5pt);
\draw [fill=black] (62.,-40.) circle (5.5pt);
\draw [fill=ffffff] (24.,-40.) circle (5.5pt);
\draw [fill=black] (16.,-30.) circle (5.5pt);
\draw [fill=ffffff] (32.,-40.) circle (5.5pt);
\end{scriptsize}
\end{tikzpicture}
  \caption{(1) An Abelian differential $(X,\omega)\in \A_{d^2}$ with $w_1 < t_3 < w_2$. (2) Zipping up and (3) zipping down of its white singularity.}
  \label{figtobdry2}
\end{figure}

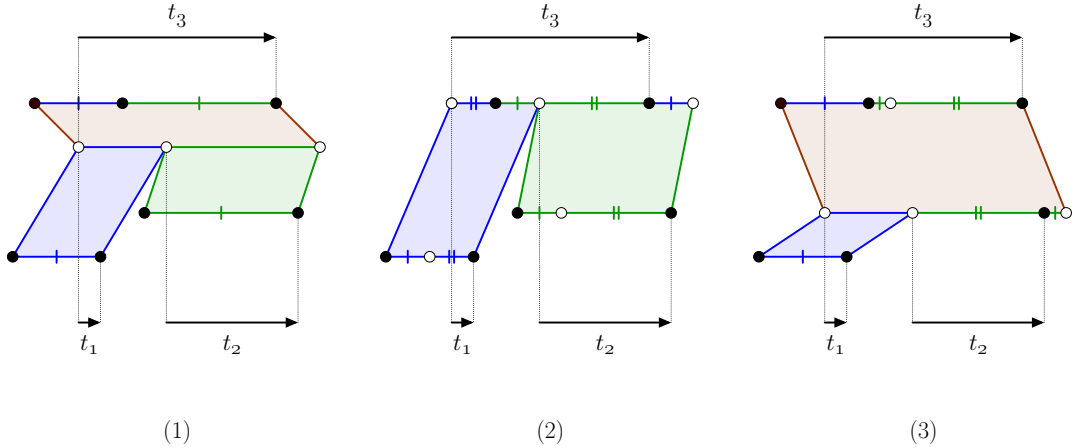
\begin{figure}[H] 
\centering
  \definecolor{zzttqq}{rgb}{0.6,0.2,0.}
\definecolor{ttqqqq}{rgb}{0.2,0.,0.}
\definecolor{qqzzqq}{rgb}{0.,0.6,0.}
\definecolor{qqqqff}{rgb}{0.,0.,1.}
\definecolor{ffffff}{rgb}{1.,1.,1.}
\begin{tikzpicture}[line cap=round,line join=round,>=triangle 45,x=0.4cm,y=0.4cm]
\clip(-44.,-64.) rectangle (68.,-18.);

\begin{scope}[shift = {(-34,0)}]

\fill[line width=2pt,color=qqqqff,fill=qqqqff,fill opacity=0.10000000149011612] (12.,-34.) -- (6.,-44.) -- (-2.,-44.) -- (4.,-34.) -- cycle;
\fill[line width=2pt,color=qqzzqq,fill=qqzzqq,fill opacity=0.10000000149011612] (12.,-34.) -- (10.,-40.) -- (24.,-40.) -- (26.,-34.) -- cycle;
\fill[line width=2pt,color=zzttqq,fill=zzttqq,fill opacity=0.10000000149011612] (0.,-30.) -- (4.,-34.) -- (26.,-34.) -- (22.,-30.) -- cycle;

\draw [line width=2pt,color=qqqqff] (12.,-34.)-- (6.,-44.);
\draw [line width=2pt,color=qqqqff] (6.,-44.)-- (-2.,-44.);
\draw [line width=2pt,color=qqqqff] (2.,-44.55578771891686) -- (2.,-43.444212281083146);
\draw [line width=2pt,color=qqqqff] (-2.,-44.)-- (4.,-34.);
\draw [line width=2pt,color=qqqqff] (4.,-34.)-- (12.,-34.);
\draw [line width=2pt,color=qqzzqq] (12.,-34.)-- (10.,-40.);
\draw [line width=2pt,color=qqzzqq] (10.,-40.)-- (24.,-40.);
\draw [line width=2pt,color=qqzzqq] (17.,-39.44421228108314) -- (17.,-40.55578771891686);
\draw [line width=2pt,color=qqzzqq] (24.,-40.)-- (26.,-34.);
\draw [line width=2pt,color=qqzzqq] (26.,-34.)-- (12.,-34.);
\draw [line width=2pt,color=zzttqq] (0.,-30.)-- (4.,-34.);
\draw [line width=2pt,color=zzttqq] (26.,-34.)-- (22.,-30.);
\draw [line width=2pt,color=qqqqff] (0.,-30.)-- (8.,-30.);
\draw [line width=2pt,color=qqqqff] (4.,-29.444212281083143) -- (4.,-30.55578771891686);
\draw [line width=2pt,color=qqzzqq] (8.,-30.)-- (22.,-30.);
\draw [line width=2pt,color=qqzzqq] (15.,-29.444212281083143) -- (15.,-30.55578771891686);
\draw [line width=0.1pt,dotted] (12.,-34.)-- (12.,-50.);
\draw [line width=0.1pt,dotted] (4.,-34.)-- (4.,-24.);
\draw [line width=0.1pt,dotted] (22.,-30.)-- (22.,-24.);
\draw [line width=0.1pt,dotted] (24.,-40.)-- (24.,-50.);
\draw [line width=0.1pt,dotted] (4.,-34.)-- (4.,-50.);
\draw [line width=0.1pt,dotted] (6.,-44.)-- (6.,-50.);

\draw (13.,-22.) node {\Huge $t_3$};
\draw [->,line width=2pt] (4.,-24.) -- (22.,-24.);
\draw (5.,-52.) node {\Huge $t_1$};
\draw [->,line width=2pt] (4.,-50.) -- (6.,-50.);
\draw (18.,-52.) node {\Huge  $t_2$};
\draw [->,line width=2pt] (12.,-50.) -- (24.,-50.);

\draw [fill=ffffff] (4.,-34.) circle (5.5pt);
\draw [fill=ffffff] (12.,-34.) circle (5.5pt);
\draw [fill=black] (6.,-44.) circle (5.5pt);
\draw [fill=black] (8.,-30.) circle (5.5pt);
\draw [fill=black] (22.,-30.) circle (5.5pt);
\draw [fill=black] (-2.,-44.) circle (5.5pt);
\draw [fill=ffffff] (26.,-34.) circle (5.5pt);
\draw [fill=black] (10.,-40.) circle (5.5pt);
\draw [fill=black] (24.,-40.) circle (5.5pt);
\draw [fill=ttqqqq] (0.,-30.) circle (5.5pt);

\end{scope}

\begin{scope}[shift = {(34,0)}]
\fill[line width=2pt,color=qqqqff,fill=qqqqff,fill opacity=0.10000000149011612] (-22.,-30.) -- (-28.,-44.) -- (-36.,-44.) -- (-30.,-30.) -- cycle;
\fill[line width=2pt,color=qqzzqq,fill=qqzzqq,fill opacity=0.10000000149011612] (-22.,-30.) -- (-24.,-40.) -- (-10.,-40.) -- (-8.,-30.) -- cycle;
\fill[line width=2pt,color=zzttqq,fill=zzttqq,fill opacity=0.10000000149011612] (-30.,-30.) -- (-30.,-30.) -- (-8.,-30.) -- (-12.,-30.) -- cycle;

\draw [line width=2pt,color=qqqqff] (-22.,-30.)-- (-28.,-44.);
\draw [line width=2pt,color=qqqqff] (-36.,-44.)-- (-30.,-30.);
\draw [line width=2pt,color=qqzzqq] (-22.,-30.)-- (-24.,-40.);
\draw [line width=2pt,color=qqzzqq] (-10.,-40.)-- (-8.,-30.);
\draw [line width=2pt,color=zzttqq] (-30.,-30.)-- (-30.,-30.);
\draw [line width=2pt,color=qqqqff] (-30.,-30.)-- (-26.,-30.);
\draw [line width=2pt,color=qqqqff] (-28.222315087566752,-29.444212281083143) -- (-28.222315087566752,-30.55578771891686);
\draw [line width=2pt,color=qqqqff] (-27.77768491243325,-29.444212281083143) -- (-27.77768491243325,-30.55578771891686);
\draw [line width=0.1pt,dotted] (-22.,-30.)-- (-22.,-50.);
\draw [line width=0.1pt,dotted] (-30.,-30.)-- (-30.,-50.);
\draw [line width=0.1pt,dotted] (-10.,-40.)-- (-10.,-50.);
\draw [line width=2pt,color=qqqqff] (-36.,-44.)-- (-32.,-44.);
\draw [line width=2pt,color=qqqqff] (-34.,-43.444212281083146) -- (-34.,-44.55578771891686);
\draw [line width=2pt,color=qqqqff] (-32.,-44.)-- (-28.,-44.);
\draw [line width=2pt,color=qqqqff] (-30.22231508756675,-43.444212281083146) -- (-30.22231508756675,-44.55578771891686);
\draw [line width=2pt,color=qqqqff] (-29.777684912433244,-43.444212281083146) -- (-29.777684912433244,-44.55578771891686);
\draw [line width=2pt,color=qqqqff] (-12.,-30.)-- (-8.,-30.);
\draw [line width=2pt,color=qqqqff] (-10.,-29.444212281083143) -- (-10.,-30.55578771891686);
\draw [line width=2pt,color=qqzzqq] (-24.,-40.)-- (-20.,-40.);
\draw [line width=2pt,color=qqzzqq] (-22.,-39.44421228108314) -- (-22.,-40.55578771891686);
\draw [line width=2pt,color=qqzzqq] (-20.,-40.)-- (-10.,-40.);
\draw [line width=2pt,color=qqzzqq] (-15.222315087566756,-39.44421228108314) -- (-15.222315087566756,-40.55578771891686);
\draw [line width=2pt,color=qqzzqq] (-14.777684912433253,-39.44421228108314) -- (-14.777684912433253,-40.55578771891686);
\draw [line width=2pt,color=qqzzqq] (-22.,-30.)-- (-12.,-30.);
\draw [line width=2pt,color=qqzzqq] (-17.22231508756675,-29.444212281083143) -- (-17.22231508756675,-30.55578771891686);
\draw [line width=2pt,color=qqzzqq] (-16.777684912433244,-29.444212281083143) -- (-16.777684912433244,-30.55578771891686);
\draw [line width=2pt,color=qqzzqq] (-26.,-30.)-- (-22.,-30.);
\draw [line width=2pt,color=qqzzqq] (-24.,-29.444212281083143) -- (-24.,-30.55578771891686);
\draw [line width=0.1pt,dotted] (-28.,-44.)-- (-28.,-50.);
\draw [line width=0.1pt,dotted] (-30.,-30.)-- (-30.,-24.);
\draw [line width=0.1pt,dotted] (-12.,-30.)-- (-12.,-24.);

\draw (-21,-22) node {\Huge $t_3$};
\draw [->,line width=2pt] (-30.,-24.) -- (-12.,-24.);
\draw (-29.,-52.) node {\Huge $t_1$};
\draw [->,line width=2pt] (-30.,-50.) -- (-28.,-50.);
\draw (-16.,-52.) node {\Huge $t_2$};
\draw [->,line width=2pt] (-22.,-50.) -- (-10.,-50.);

\draw [fill=ffffff] (-30.,-30.) circle (5.5pt);
\draw [fill=ffffff] (-22.,-30.) circle (5.5pt);
\draw [fill=black] (-28.,-44.) circle (5.5pt);
\draw [fill=black] (-26.,-30.) circle (5.5pt);
\draw [fill=black] (-12.,-30.) circle (5.5pt);
\draw [fill=black] (-36.,-44.) circle (5.5pt);
\draw [fill=ffffff] (-8.,-30.) circle (5.5pt);
\draw [fill=black] (-24.,-40.) circle (5.5pt);
\draw [fill=black] (-10.,-40.) circle (5.5pt);
\draw [fill=ffffff] (-30.,-30.) circle (5.5pt);
\draw [fill=ffffff] (-32.,-44.) circle (5.5pt);
\draw [fill=ffffff] (-20.,-40.) circle (5.5pt);

\end{scope}

\fill[line width=2pt,color=qqqqff,fill=qqqqff,fill opacity=0.10000000149011612] (46.,-40.) -- (40.,-44.) -- (32.,-44.) -- (38.,-40.) -- cycle;
\fill[line width=2pt,color=zzttqq,fill=zzttqq,fill opacity=0.10000000149011612] (34.,-30.) -- (38.,-40.) -- (60.,-40.) -- (56.,-30.) -- cycle;

\draw [line width=2pt,color=qqqqff] (46.,-40.)-- (40.,-44.);
\draw [line width=2pt,color=qqqqff] (40.,-44.)-- (32.,-44.);
\draw [line width=2pt,color=qqqqff] (36.,-44.55578771891686) -- (36.,-43.444212281083146);
\draw [line width=2pt,color=qqqqff] (32.,-44.)-- (38.,-40.);
\draw [line width=2pt,color=qqqqff] (38.,-40.)-- (46.,-40.);
\draw [line width=2pt,color=zzttqq] (34.,-30.)-- (38.,-40.);
\draw [line width=2pt,color=zzttqq] (60.,-40.)-- (56.,-30.);
\draw [line width=2pt,color=qqqqff] (34.,-30.)-- (42.,-30.);
\draw [line width=2pt,color=qqqqff] (38.,-29.444212281083143) -- (38.,-30.55578771891686);
\draw [line width=0.1pt,dotted] (46.,-40.)-- (46.,-50.);
\draw [line width=0.1pt,dotted] (38.,-40.)-- (38.,-24.);
\draw [line width=0.1pt,dotted] (56.,-30.)-- (56.,-24.);
\draw [line width=0.1pt,dotted] (58.,-40.)-- (58.,-50.);
\draw [line width=0.1pt,dotted] (38.,-40.)-- (38.,-50.);
\draw [line width=0.1pt,dotted] (40.,-44.)-- (40.,-50.);
\draw [line width=2pt,color=qqzzqq] (42.,-30.)-- (44.,-30.);
\draw [line width=2pt,color=qqzzqq] (43.,-29.444212281083143) -- (43.,-30.55578771891686);
\draw [line width=2pt,color=qqzzqq] (44.,-30.)-- (56.,-30.);
\draw [line width=2pt,color=qqzzqq] (49.777684912433244,-29.444212281083143) -- (49.777684912433244,-30.55578771891686);
\draw [line width=2pt,color=qqzzqq] (50.22231508756675,-29.444212281083143) -- (50.22231508756675,-30.55578771891686);
\draw [line width=2pt,color=qqzzqq] (46.,-40.)-- (58.,-40.);
\draw [line width=2pt,color=qqzzqq] (51.77768491243325,-39.44421228108314) -- (51.77768491243325,-40.55578771891686);
\draw [line width=2pt,color=qqzzqq] (52.222315087566756,-39.44421228108314) -- (52.222315087566756,-40.55578771891686);
\draw [line width=2pt,color=qqzzqq] (58.,-40.)-- (60.,-40.);
\draw [line width=2pt,color=qqzzqq] (59.,-39.44421228108314) -- (59.,-40.55578771891686);

\draw (47,-22) node {\Huge $t_3$};
\draw [->,line width=2pt] (38.,-24.) -- (56.,-24.);

\draw (39.,-52.) node {\Huge $t_1$};
\draw [->,line width=2pt] (38.,-50.) -- (40.,-50.);

\draw (52.,-52.) node {\Huge $t_2$};
\draw [->,line width=2pt] (46.,-50.) -- (58.,-50.);

\draw (-21.,-60.) node {\Huge (1)};
\draw (13.,-60.) node {\Huge (2)};
\draw (47.,-60.) node {\Huge (3)};

\begin{scriptsize}
\draw [fill=ffffff] (38.,-40.) circle (5.5pt);
\draw [fill=ffffff] (46.,-40.) circle (5.5pt);
\draw [fill=black] (40.,-44.) circle (5.5pt);
\draw [fill=black] (42.,-30.) circle (5.5pt);
\draw [fill=black] (56.,-30.) circle (5.5pt);
\draw [fill=black] (32.,-44.) circle (5.5pt);
\draw [fill=ffffff] (60.,-40.) circle (5.5pt);
\draw [fill=black] (58.,-40.) circle (5.5pt);
\draw [fill=ttqqqq] (34.,-30.) circle (5.5pt);
\draw [fill=ffffff] (44.,-30.) circle (5.5pt);
\end{scriptsize}
\end{tikzpicture}
  \caption{(1) An Abelian differential $(X,\omega)\in \A_{d^2}$ with $w_2 < t_3 < w_1 + w_2$. (2) Zipping up and (3) zipping down of its white singularity.}
  \label{figtobdry3}
\end{figure}
\end{commentB}

We first find expression of the cylinder coordinates of zipping down. Let $(X,\omega)$ be a generic  $\in \CC = \{(w_1, s_1), (w_2, s_2)\} \subset \A_{d^2}$, where $s_1 < s_2$ with half-integer twist parameters $t_1, t_2, t_3 \in \R \setminus \Z$, satisfying (\ref{twists}). Moving the singularity downwards in all three cases we obtain 2-cylinder decompositions with cylinder coordinates satisfying (see Figures~\ref{figtobdry1result}.2, \ref{figtobdry2result}.2 and \ref{figtobdry3result}.2):
\begin{align*}
& W^{(2)}_1  = w_1 & & H^{(2)}_1  = s_2-s_1 & & 0 < T^{(2)}_1  = t_1 < w_1 & \\ 
& W^{(2)}_2  = w_1+w_2 & & H^{(2)}_2  = s_2 & & 0 < T^{(2)}_2  = t_2 < w_2 &  \\
& & & & &  0 < T^{(2)}_3  = t_3 < w_1+w_2.
\end{align*}

Now we find expression of the cylinder coordinates of zipping up. Let $(X,\omega)$ be a generic  $\in \CC = \{(w_1, s_1), (w_2, s_2)\} \subset \A_{d^2}$, where $w_1 < w_2$ with half-integer twist parameters $t_1, t_2, t_3 \in \R \setminus \Z$, satisfying (\ref{twists}). Moving the singularity upwards we obtain a 2-cylinder decomposition with cylinder coordinates satisfying, in the case (1) $0 < t_3 < w_1$ (see Figure~\ref{figtobdry1result}.1):
\begin{align*}
& W^{(1)}_1 = w_1 & & H^{(1)}_1 = s_1 & & 0 < T^{(1)}_1  = t_1 < w_1 & \\ 
& W^{(1)}_2 = w_2 & & H^{(1)}_2 = s_2 & & 0 < T^{(1)}_2  = (w_2 - t_2 + t_3) \% w_2 < w_2 \hspace{27pt} &  \\
& & & & & 0 < T^{(1)}_3  = t_3 < w_1,
\end{align*}
in the case (2) $w_1 < t_3 < w_2$ (see Figure~\ref{figtobdry2result}.1):
\begin{align*}
& W^{(2)}_1 = w_1 & & H^{(2)}_1 = s_1 & & 0 < T^{(2)}_1  = t_1 < w_1 & \\
& W^{(2)}_2 = w_2 & & H^{(2)}_2 = s_2 & & 0 < T^{(2)}_2  = t_3 - w_1 < w_2-w_1 & \\
& & & & & 0 < T^{(2)}_3  = (2t_3-t_2-w_1) \% w_2 < w_2,  \hspace{15pt} \
\end{align*}
and in the case (3) $w_2 < t_3 < w_1+w_2$ (see Figure~\ref{figtobdry3result}.1):
\begin{align*}
& W^{(1)}_1 = w_1 & & H^{(1)}_1 = s_1 & & 0 < T^{(1)}_1  = (w_1 + w_2 + t_1 - t_3) \% w_1 < w_1 & \\
& W^{(1)}_2 = w_2 & & H^{(1)}_2 = s_2 & & 0 < T^{(1)}_2  =  (t_3 - t_1 - w_1) \% w_2 < w_2 &  \\
& & & & & 0 < T^{(1)}_3  = w_1 + w_2  - t_3 < w_1.  
\end{align*}

\begin{commentB}
\begin{figure}[H] 
\centering
  \definecolor{qqzzqq}{rgb}{0.,0.6,0.}
\definecolor{qqqqff}{rgb}{0.,0.,1.}
\definecolor{zzttqq}{rgb}{0.6,0.2,0.}
\definecolor{ttqqqq}{rgb}{0.2,0.,0.}
\definecolor{ffffff}{rgb}{1.,1.,1.}
\begin{tikzpicture}[line cap=round,line join=round,>=triangle 45,x=0.4cm,y=0.4cm]
\clip(-54,-60) rectangle (54,0);
\fill[line width=0.8pt,color=zzttqq,fill=zzttqq,fill opacity=0.10000000149011612] (6.,-20.) -- (4.,-30.) -- (26.,-30.) -- (28.,-20.) -- cycle;
\fill[line width=0.8pt,color=qqqqff,fill=qqqqff,fill opacity=0.10000000149011612] (-10.,-26.) -- (-6.,-40.) -- (-14.,-40.) -- (-18.,-26.) -- cycle;
\fill[line width=0.8pt,color=qqzzqq,fill=qqzzqq,fill opacity=0.10000000149011612] (-34.,-16.) -- (-30.,-26.) -- (-16.,-26.) -- (-20.,-16.) -- cycle;
\fill[line width=0.8pt,color=qqqqff,fill=qqqqff,fill opacity=0.10000000149011612] (12.,-30.) -- (16.,-34.) -- (8.,-34.) -- (4.,-30.) -- cycle;
\fill[line width=0.8pt,color=qqzzqq,fill=qqzzqq,fill opacity=0.10000000149011612] (12.,-30.) -- (18.,-30.) -- (26.,-30.) -- (26.,-30.) -- cycle;
\draw [->,line width=0.8pt] (12.,-46.) -- (18.,-46.);
\draw (15, -48) node {\Huge $t_2=T_2$};
\draw [line width=2pt,color=zzttqq] (6.,-20.)-- (4.,-30.);
\draw [line width=2pt,color=zzttqq] (26.,-30.)-- (28.,-20.);
\draw [line width=2pt,color=qqqqff] (-10.,-26.)-- (-6.,-40.);
\draw [line width=2pt,color=qqqqff] (-14.,-40.)-- (-18.,-26.);
\draw [->,line width=0.8pt] (-22.,-4.) -- (-16.,-4.);
\draw (-19, -2) node {\Huge $t_2$};
\draw [line width=2pt,color=qqzzqq] (-34.,-16.)-- (-30.,-26.);
\draw [line width=2pt,color=qqzzqq] (-16.,-26.)-- (-20.,-16.);
\draw [line width=0.1pt,dotted] (-22.,-16.)-- (-22.,-4.);
\draw [->,line width=0.8pt] (-18.,-10.) -- (-16.,-10.);
\draw (-17, -8) node {\Huge $t_3=T_3$};
\draw [->,line width=0.8pt] (4.,-46.) -- (8.,-46.);
\draw (6, -48) node {\Huge $t_1=T_1$};
\draw [line width=2pt,color=qqqqff] (12.,-30.)-- (16.,-34.);
\draw [line width=2pt,color=qqqqff] (16.,-34.)-- (8.,-34.);
\draw [line width=2pt,color=qqqqff] (12.,-34.43601963708249) -- (12.,-33.56398036291751);
\draw [line width=2pt,color=qqqqff] (8.,-34.)-- (4.,-30.);
\draw [line width=2pt,color=qqqqff] (4.,-30.)-- (12.,-30.);
\draw [line width=2pt,color=qqzzqq] (18.,-30.)-- (26.,-30.);
\draw [line width=2pt,color=qqzzqq] (21.825592145166993,-29.56398036291751) -- (21.825592145166993,-30.43601963708249);
\draw [line width=2pt,color=qqzzqq] (22.174407854832996,-29.56398036291751) -- (22.174407854832996,-30.43601963708249);
\draw [line width=2pt,color=qqzzqq] (26.,-30.)-- (26.,-30.);
\draw [line width=2pt,color=qqqqff] (6.,-20.)-- (14.,-20.);
\draw [line width=2pt,color=qqqqff] (10.,-19.56398036291751) -- (10.,-20.43601963708249);
\draw [line width=0.1pt,dotted] (12.,-30.)-- (12.,-46.);
\draw [line width=0.1pt,dotted] (4.,-30.)-- (4.,-10.);
\draw [line width=0.1pt,dotted] (6.,-20.)-- (6.,-10.);
\draw [->,line width=0.8pt] (4.,-10.) -- (6.,-10.);
\draw (5, -8) node {\Huge $t_3=T_3$};
\draw [line width=0.1pt,dotted] (4.,-30.)-- (4.,-46.);
\draw [line width=0.1pt,dotted] (8.,-34.)-- (8.,-46.);
\draw [line width=2pt,color=qqzzqq] (-34.,-16.)-- (-22.,-16.);
\draw [line width=2pt,color=qqzzqq] (-28.174407854833003,-15.563980362917508) -- (-28.174407854833003,-16.43601963708249);
\draw [line width=2pt,color=qqzzqq] (-27.825592145167,-15.563980362917508) -- (-27.825592145167,-16.43601963708249);
\draw [line width=2pt,color=qqqqff] (-22.,-16.)-- (-20.,-16.);
\draw [line width=2pt,color=qqqqff] (-21.,-15.563980362917508) -- (-21.,-16.43601963708249);
\draw [line width=2pt,color=qqzzqq] (-30.,-26.)-- (-18.,-26.);
\draw [line width=2pt,color=qqzzqq] (-24.174407854833,-25.56398036291751) -- (-24.174407854833,-26.43601963708249);
\draw [line width=2pt,color=qqzzqq] (-23.825592145166997,-25.56398036291751) -- (-23.825592145166997,-26.43601963708249);
\draw [line width=2pt,color=qqzzqq] (-18.,-26.)-- (-16.,-26.);
\draw [line width=2pt,color=qqzzqq] (-17.,-25.56398036291751) -- (-17.,-26.43601963708249);
\draw [line width=2pt,color=qqzzqq] (12.,-30.)-- (18.,-30.);
\draw [line width=2pt,color=qqzzqq] (15.,-29.56398036291751) -- (15.,-30.43601963708249);
\draw [line width=2pt,color=qqzzqq] (22.,-20.)-- (28.,-20.);
\draw [line width=2pt,color=qqzzqq] (25.,-19.56398036291751) -- (25.,-20.43601963708249);
\draw [line width=2pt,color=qqzzqq] (14.,-20.)-- (22.,-20.);
\draw [line width=2pt,color=qqzzqq] (17.825592145166997,-19.56398036291751) -- (17.825592145166997,-20.43601963708249);
\draw [line width=2pt,color=qqzzqq] (18.174407854833003,-19.56398036291751) -- (18.174407854833003,-20.43601963708249);
\draw [line width=0.1pt,dotted] (18.,-30.)-- (18.,-46.);
\draw [line width=2pt,color=qqqqff] (-16.,-26.)-- (-10.,-26.);
\draw [line width=2pt,color=qqqqff] (-13.174407854833001,-25.56398036291751) -- (-13.174407854833001,-26.43601963708249);
\draw [line width=2pt,color=qqqqff] (-12.825592145166999,-25.56398036291751) -- (-12.825592145166999,-26.43601963708249);
\draw [line width=2pt,color=qqzzqq] (-18.,-26.)-- (-16.,-26.);
\draw [line width=2pt,color=qqzzqq] (-17.,-25.56398036291751) -- (-17.,-26.43601963708249);
\draw [->,line width=0.8pt] (-18.,-46.) -- (-14.,-46.);
\draw (-16, -48) node {\Huge $t_1=T_1$};
\draw [line width=2pt,color=qqqqff] (-14.,-40.)-- (-12.,-40.);
\draw [line width=2pt,color=qqqqff] (-13.,-39.563980362917505) -- (-13.,-40.43601963708249);
\draw [line width=0.1pt,dotted] (-18.,-26.)-- (-18.,-46.);
\draw [line width=0.1pt,dotted] (-14.,-40.)-- (-14.,-46.);
\draw [line width=0.1pt,dotted] (-16.,-26.)-- (-16.,-4.);
\draw [line width=0.1pt,dotted] (-18.,-26.)-- (-18.,-10.);
\draw [line width=0.1pt,dotted] (-16.,-26.)-- (-16.,-10.);
\draw [line width=2pt,color=qqqqff] (-12.,-40.)-- (-6.,-40.);
\draw [line width=2pt,color=qqqqff] (-9.174407854833007,-39.563980362917505) -- (-9.174407854833007,-40.43601963708249);
\draw [line width=2pt,color=qqqqff] (-8.825592145167004,-39.563980362917505) -- (-8.825592145167004,-40.43601963708249);
\draw [line width=0.1pt,dotted] (-30.,-26.)-- (-30.,-10.);
\draw [line width=0.1pt,dotted] (-20.,-16.)-- (-20.,-10.);
\draw [->,line width=0.8pt] (-30.,-10.) -- (-20.,-10.);
\draw (-25, -8) node {\Huge $T_2$};

\draw (-20, -58) node {\Huge (1)};
\draw (16, -58) node {\Huge (2)};

\begin{scriptsize}
\draw [fill=ffffff] (-10.,-26.) circle (5.5pt);
\draw [fill=ffffff] (-14.,-40.) circle (5.5pt);
\draw [fill=ffffff] (-6.,-40.) circle (5.5pt);
\draw [fill=ffffff] (-18.,-26.) circle (5.5pt);
\draw [fill=ttqqqq] (6.,-20.) circle (5.5pt);
\draw [fill=ffffff] (26.,-30.) circle (5.5pt);
\draw [fill=black] (28.,-20.) circle (5.5pt);
\draw [fill=ffffff] (4.,-30.) circle (5.5pt);
\draw [fill=ffffff] (22.,-20.) circle (5.5pt);
\draw [fill=black] (-34.,-16.) circle (5.5pt);
\draw [fill=ffffff] (-22.,-16.) circle (5.5pt);
\draw [fill=black] (-20.,-16.) circle (5.5pt);
\draw [fill=black] (-30.,-26.) circle (5.5pt);
\draw [fill=black] (-16.,-26.) circle (5.5pt);
\draw [fill=ffffff] (12.,-30.) circle (5.5pt);
\draw [fill=black] (16.,-34.) circle (5.5pt);
\draw [fill=black] (14.,-20.) circle (5.5pt);
\draw [fill=black] (8.,-34.) circle (5.5pt);
\draw [fill=black] (18.,-30.) circle (5.5pt);
\draw [fill=black] (26.,-30.) circle (5.5pt);
\draw [fill=ffffff] (-18.,-26.) circle (5.5pt);
\draw [fill=black] (-16.,-26.) circle (5.5pt);
\draw [fill=black] (-12.,-40.) circle (5.5pt);
\end{scriptsize}
\end{tikzpicture}
  \caption{The twist coordinates of the 2-cylinder decompositions of Abelian differentials in (1) Figure~\ref{figtobdry1}.2 and (2) Figure~\ref{figtobdry1}.3.}
  \label{figtobdry1result}
\end{figure}
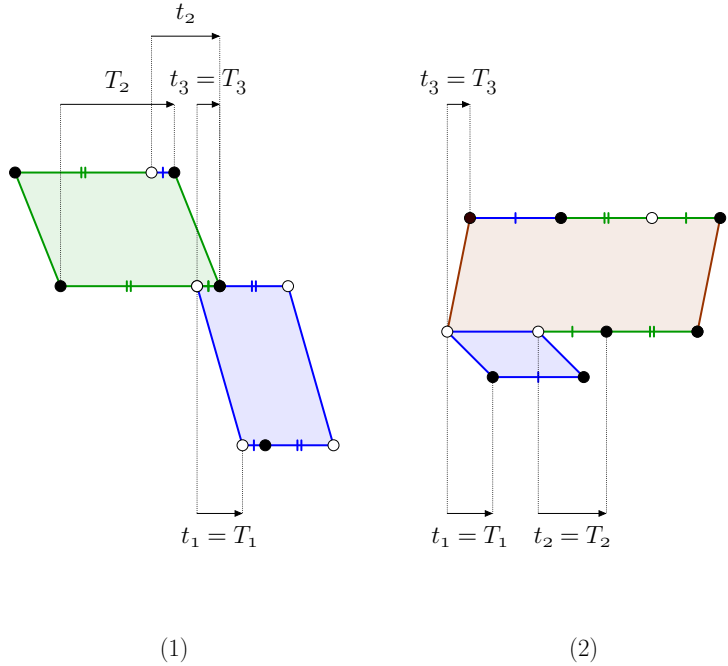

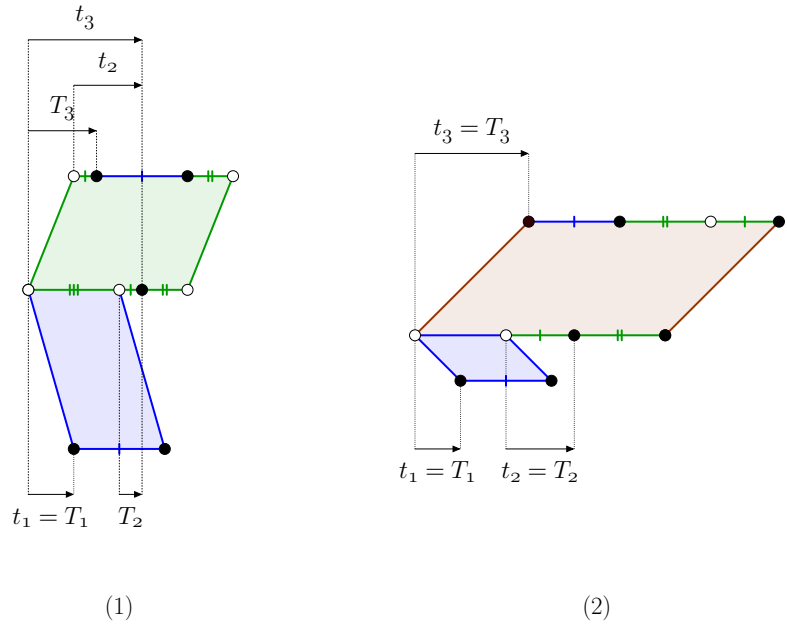
\begin{figure}[H] 
\centering
  \definecolor{qqzzqq}{rgb}{0.,0.6,0.}
\definecolor{qqqqff}{rgb}{0.,0.,1.}
\definecolor{zzttqq}{rgb}{0.6,0.2,0.}
\definecolor{ttqqqq}{rgb}{0.2,0.,0.}
\definecolor{ffffff}{rgb}{1.,1.,1.}
\begin{tikzpicture}[line cap=round,line join=round,>=triangle 45,x=0.4cm,y=0.4cm]
\clip(-20,-66) rectangle (88,-6);

\draw (14,-64) node {\Huge (1)};
\draw (56, -64) node {\Huge (2)};
\fill[line width=0.8pt,color=zzttqq,fill=zzttqq,fill opacity=0.10000000149011612] (50.,-30.) -- (40.,-40.) -- (62.,-40.) -- (72.,-30.) -- cycle;
\fill[line width=0.8pt,color=qqqqff,fill=qqqqff,fill opacity=0.10000000149011612] (14.,-36.) -- (18.,-50.) -- (10.,-50.) -- (6.,-36.) -- cycle;
\fill[line width=0.8pt,color=qqzzqq,fill=qqzzqq,fill opacity=0.10000000149011612] (10.,-26.) -- (6.,-36.) -- (20.,-36.) -- (24.,-26.) -- cycle;
\fill[line width=0.8pt,color=qqqqff,fill=qqqqff,fill opacity=0.10000000149011612] (48.,-40.) -- (52.,-44.) -- (44.,-44.) -- (40.,-40.) -- cycle;
\fill[line width=0.8pt,color=qqzzqq,fill=qqzzqq,fill opacity=0.10000000149011612] (48.,-40.) -- (54.,-40.) -- (62.,-40.) -- (62.,-40.) -- cycle;
\draw [->,line width=0.8pt] (48.,-50.) -- (54.,-50.);
\draw (51,-52) node {\Huge $t_2=T_2$};
\draw [line width=2pt,color=zzttqq] (50.,-30.)-- (40.,-40.);
\draw [line width=2pt,color=zzttqq] (62.,-40.)-- (72.,-30.);
\draw [->,line width=0.8pt] (6.,-54.) -- (10.,-54.);
\draw (8,-56) node {\Huge $t_1=T_1$};
\draw [line width=2pt,color=qqqqff] (14.,-36.)-- (18.,-50.);
\draw [line width=2pt,color=qqqqff] (10.,-50.)-- (6.,-36.);
\draw [line width=2pt,color=qqzzqq] (10.,-26.)-- (6.,-36.);
\draw [line width=2pt,color=qqzzqq] (20.,-36.)-- (24.,-26.);
\draw [line width=0.1pt,dotted] (10.,-50.)-- (10.,-54.);
\draw [->,line width=0.8pt] (40.,-50.) -- (44.,-50.);
\draw (42,-52) node {\Huge $t_1=T_1$};
\draw [line width=2pt,color=qqqqff] (48.,-40.)-- (52.,-44.);
\draw [line width=2pt,color=qqqqff] (52.,-44.)-- (44.,-44.);
\draw [line width=2pt,color=qqqqff] (48.,-44.437971511076874) -- (48.,-43.56202848892313);
\draw [line width=2pt,color=qqqqff] (44.,-44.)-- (40.,-40.);
\draw [line width=2pt,color=qqqqff] (40.,-40.)-- (48.,-40.);
\draw [line width=2pt,color=qqzzqq] (54.,-40.)-- (62.,-40.);
\draw [line width=2pt,color=qqzzqq] (57.824811395569235,-39.56202848892313) -- (57.824811395569235,-40.437971511076874);
\draw [line width=2pt,color=qqzzqq] (58.175188604430744,-39.56202848892313) -- (58.175188604430744,-40.437971511076874);
\draw [line width=2pt,color=qqzzqq] (62.,-40.)-- (62.,-40.);
\draw [line width=2pt,color=qqqqff] (50.,-30.)-- (58.,-30.);
\draw [line width=2pt,color=qqqqff] (54.,-29.56202848892313) -- (54.,-30.43797151107687);
\draw [line width=0.1pt,dotted] (48.,-40.)-- (48.,-50.);
\draw [line width=0.1pt,dotted] (40.,-40.)-- (40.,-24.);
\draw [line width=0.1pt,dotted] (50.,-30.)-- (50.,-24.);
\draw [->,line width=0.8pt] (40.,-24.) -- (50.,-24.);
\draw (45,-22) node {\Huge $t_3=T_3$};
\draw [line width=0.1pt,dotted] (40.,-40.)-- (40.,-50.);
\draw [line width=0.1pt,dotted] (44.,-44.)-- (44.,-50.);
\draw [line width=2pt,color=qqzzqq] (48.,-40.)-- (54.,-40.);
\draw [line width=2pt,color=qqzzqq] (51.,-39.56202848892313) -- (51.,-40.437971511076874);
\draw [line width=2pt,color=qqzzqq] (66.,-30.)-- (72.,-30.);
\draw [line width=2pt,color=qqzzqq] (69.,-29.56202848892313) -- (69.,-30.43797151107687);
\draw [line width=2pt,color=qqzzqq] (58.,-30.)-- (66.,-30.);
\draw [line width=2pt,color=qqzzqq] (61.824811395569235,-29.56202848892313) -- (61.824811395569235,-30.43797151107687);
\draw [line width=2pt,color=qqzzqq] (62.175188604430744,-29.56202848892313) -- (62.175188604430744,-30.43797151107687);
\draw [line width=0.1pt,dotted] (54.,-40.)-- (54.,-50.);
\draw [line width=2pt,color=qqqqff] (10.,-50.)-- (18.,-50.);
\draw [line width=2pt,color=qqqqff] (14.,-49.56202848892313) -- (14.,-50.43797151107688);
\draw [line width=2pt,color=qqzzqq] (6.,-36.)-- (14.,-36.);
\draw [line width=2pt,color=qqzzqq] (9.649622791138494,-35.56202848892313) -- (9.649622791138494,-36.437971511076874);
\draw [line width=2pt,color=qqzzqq] (10.,-35.56202848892313) -- (10.,-36.437971511076874);
\draw [line width=2pt,color=qqzzqq] (10.35037720886151,-35.56202848892313) -- (10.35037720886151,-36.437971511076874);
\draw [line width=2pt,color=qqzzqq] (10.,-26.)-- (12.,-26.);
\draw [line width=2pt,color=qqzzqq] (11.,-25.56202848892313) -- (11.,-26.437971511076874);
\draw [line width=2pt,color=qqqqff] (12.,-26.)-- (20.,-26.);
\draw [line width=2pt,color=qqqqff] (16.,-25.56202848892313) -- (16.,-26.437971511076874);
\draw [line width=2pt,color=qqzzqq] (20.,-26.)-- (24.,-26.);
\draw [line width=2pt,color=qqzzqq] (21.82481139556924,-25.56202848892313) -- (21.82481139556924,-26.437971511076874);
\draw [line width=2pt,color=qqzzqq] (22.17518860443075,-25.56202848892313) -- (22.17518860443075,-26.437971511076874);
\draw [line width=2pt,color=qqzzqq] (14.,-36.)-- (16.,-36.);
\draw [line width=2pt,color=qqzzqq] (15.,-35.56202848892313) -- (15.,-36.437971511076874);
\draw [line width=2pt,color=qqzzqq] (16.,-36.)-- (20.,-36.);
\draw [line width=2pt,color=qqzzqq] (17.824811395569245,-35.56202848892313) -- (17.824811395569245,-36.437971511076874);
\draw [line width=2pt,color=qqzzqq] (18.17518860443075,-35.56202848892313) -- (18.17518860443075,-36.437971511076874);
\draw [line width=0.8pt,dotted] (6.,-36.)-- (6.,-14.);
\draw [line width=0.8pt,dotted] (16.,-36.)-- (16.,-14.);
\draw [->,line width=0.8pt] (6.,-14.) -- (16.,-14.);
\draw (11,-12) node {\Huge $t_3$};
\draw [line width=0.8pt,dotted] (6.,-36.)-- (6.,-54.);
\draw [line width=0.8pt,dotted] (10.,-26.)-- (10.,-18.);
\draw [->,line width=0.8pt] (10.,-18.) -- (16.,-18.);
\draw (13,-16) node {\Huge $t_2$};
\draw [line width=0.8pt,dotted] (12.,-26.)-- (12.,-22.);
\draw [->,line width=0.8pt] (6.,-22.) -- (12.,-22.);
\draw (9,-20) node {\Huge $T_3$};
\draw [->,line width=0.8pt] (14.,-54.) -- (16.,-54.);
\draw (15,-56) node {\Huge $T_2$};
\draw [line width=0.8pt,dotted] (14.,-36.)-- (14.,-54.);
\draw [line width=0.8pt,dotted] (16.,-36.)-- (16.,-54.);
\begin{scriptsize}
\draw [fill=ffffff] (10.,-26.) circle (5.5pt);
\draw [fill=ffffff] (20.,-36.) circle (5.5pt);
\draw [fill=black] (6.,-36.) circle (5.5pt);
\draw [fill=ffffff] (24.,-26.) circle (5.5pt);
\draw [fill=ttqqqq] (50.,-30.) circle (5.5pt);
\draw [fill=ffffff] (62.,-40.) circle (5.5pt);
\draw [fill=black] (72.,-30.) circle (5.5pt);
\draw [fill=ffffff] (40.,-40.) circle (5.5pt);
\draw [fill=ffffff] (66.,-30.) circle (5.5pt);
\draw [fill=ffffff] (6.,-36.) circle (5.5pt);
\draw [fill=ffffff] (14.,-36.) circle (5.5pt);
\draw [fill=black] (18.,-50.) circle (5.5pt);
\draw [fill=black] (10.,-50.) circle (5.5pt);
\draw [fill=ffffff] (48.,-40.) circle (5.5pt);
\draw [fill=black] (52.,-44.) circle (5.5pt);
\draw [fill=black] (58.,-30.) circle (5.5pt);
\draw [fill=black] (44.,-44.) circle (5.5pt);
\draw [fill=black] (54.,-40.) circle (5.5pt);
\draw [fill=black] (62.,-40.) circle (5.5pt);
\draw [fill=black] (12.,-26.) circle (5.5pt);
\draw [fill=black] (20.,-26.) circle (5.5pt);
\draw [fill=black] (16.,-36.) circle (5.5pt);
\end{scriptsize}
\end{tikzpicture}
  \caption{The twist coordinates of the 2-cylinder decompositions of Abelian differentials in (1) Figure~\ref{figtobdry2}.2 and (2) Figure~\ref{figtobdry2}.3.}
  \label{figtobdry2result}
\end{figure}

\begin{figure}[H] 
\centering
  \definecolor{zzttqq}{rgb}{0.6,0.2,0.}
\definecolor{ttqqqq}{rgb}{0.2,0.,0.}
\definecolor{qqzzqq}{rgb}{0.,0.6,0.}
\definecolor{ffffff}{rgb}{1.,1.,1.}
\definecolor{qqqqff}{rgb}{0.,0.,1.}
\begin{tikzpicture}[line cap=round,line join=round,>=triangle 45,x=0.4cm,y=0.4cm]
\clip(-48,-60) rectangle (60,1);
\fill[line width=0.8pt,color=qqqqff,fill=qqqqff,fill opacity=0.10000000149011612] (-14.,-26.) -- (-8.,-40.) -- (-16.,-40.) -- (-22.,-26.) -- cycle;
\fill[line width=0.8pt,color=qqzzqq,fill=qqzzqq,fill opacity=0.10000000149011612] (-20.,-16.) -- (-18.,-26.) -- (-4.,-26.) -- (-6.,-16.) -- cycle;
\fill[line width=0.8pt,color=qqqqff,fill=qqqqff,fill opacity=0.10000000149011612] (22.,-32.) -- (16.,-36.) -- (8.,-36.) -- (14.,-32.) -- cycle;
\fill[line width=0.8pt,color=zzttqq,fill=zzttqq,fill opacity=0.10000000149011612] (10.,-22.) -- (14.,-32.) -- (36.,-32.) -- (32.,-22.) -- cycle;
\draw [->,line width=0.8pt] (-22.,-48.) -- (-4.,-48.);
\draw (-13, -50) node {\Huge $t_3$};
\draw [->,line width=0.8pt] (-16.,-12.) -- (-4.,-12.);
\draw (-10, -10) node {\Huge $t_2$};
\draw [line width=2pt,color=qqqqff] (-14.,-26.)-- (-8.,-40.);
\draw [line width=2pt,color=qqqqff] (-16.,-40.)-- (-22.,-26.);
\draw [line width=2pt,color=qqzzqq] (-20.,-16.)-- (-18.,-26.);
\draw [line width=2pt,color=qqzzqq] (-4.,-26.)-- (-6.,-16.);
\draw [line width=2pt,color=qqqqff] (-22.,-26.)-- (-18.,-26.);
\draw [line width=2pt,color=qqqqff] (-20.179157475084462,-25.55210631228886) -- (-20.179157475084462,-26.44789368771114);
\draw [line width=2pt,color=qqqqff] (-19.820842524915538,-25.55210631228886) -- (-19.820842524915538,-26.44789368771114);
\draw [line width=0.1pt,dotted] (-16.,-16.)-- (-16.,-12.);
\draw [line width=0.1pt,dotted] (-4.,-26.)-- (-4.,-12.);
\draw [->,line width=0.8pt] (-14.,-44.) -- (-12.,-44.);
\draw (-13, -46) node {\Huge $t_1$};
\draw [line width=2pt,color=qqzzqq] (-16.,-16.)-- (-6.,-16.);
\draw [line width=2pt,color=qqzzqq] (-11.179157475084462,-15.55210631228886) -- (-11.179157475084462,-16.44789368771114);
\draw [line width=2pt,color=qqzzqq] (-10.820842524915536,-15.55210631228886) -- (-10.820842524915536,-16.44789368771114);
\draw [line width=2pt,color=qqzzqq] (-18.,-26.)-- (-14.,-26.);
\draw [line width=2pt,color=qqzzqq] (-16.,-25.55210631228886) -- (-16.,-26.44789368771114);
\draw [line width=2pt,color=qqzzqq] (-14.,-26.)-- (-4.,-26.);
\draw [line width=2pt,color=qqzzqq] (-9.179157475084462,-25.55210631228886) -- (-9.179157475084462,-26.44789368771114);
\draw [line width=2pt,color=qqzzqq] (-8.820842524915538,-25.55210631228886) -- (-8.820842524915538,-26.44789368771114);
\draw [line width=2pt,color=qqqqff] (-20.,-16.)-- (-16.,-16.);
\draw [line width=2pt,color=qqqqff] (-18.,-15.55210631228886) -- (-18.,-16.44789368771114);
\draw [->,line width=0.8pt] (14.,-44.) -- (16.,-44.);
\draw (15, -46) node {\Huge $t_1=T_1$};
\draw [->,line width=0.8pt] (22.,-44.) -- (34.,-44.);
\draw (28, -46) node {\Huge $t_2=T_2$};
\draw [line width=2pt,color=qqqqff] (22.,-32.)-- (16.,-36.);
\draw [line width=2pt,color=qqqqff] (16.,-36.)-- (8.,-36.);
\draw [line width=2pt,color=qqqqff] (12.,-36.447893687711144) -- (12.,-35.552106312288856);
\draw [line width=2pt,color=qqqqff] (8.,-36.)-- (14.,-32.);
\draw [line width=2pt,color=qqqqff] (14.,-32.)-- (22.,-32.);
\draw [line width=2pt,color=zzttqq] (10.,-22.)-- (14.,-32.);
\draw [line width=2pt,color=zzttqq] (36.,-32.)-- (32.,-22.);
\draw [line width=2pt,color=qqqqff] (10.,-22.)-- (18.,-22.);
\draw [line width=2pt,color=qqqqff] (14.,-21.55210631228886) -- (14.,-22.44789368771114);
\draw [line width=0.1pt,dotted] (22.,-32.)-- (22.,-44.);
\draw [line width=0.1pt,dotted] (14.,-32.)-- (14.,-12.);
\draw [line width=0.1pt,dotted] (34.,-32.)-- (34.,-44.);
\draw [->,line width=0.8pt] (14.,-12.) -- (32.,-12.);
\draw (23, -10) node {\Huge $t_3=T_3$};
\draw [line width=0.1pt,dotted] (14.,-32.)-- (14.,-44.);
\draw [line width=2pt,color=qqzzqq] (18.,-22.)-- (20.,-22.);
\draw [line width=2pt,color=qqzzqq] (19.,-21.55210631228886) -- (19.,-22.44789368771114);
\draw [line width=2pt,color=qqzzqq] (20.,-22.)-- (32.,-22.);
\draw [line width=2pt,color=qqzzqq] (25.82084252491553,-21.55210631228886) -- (25.82084252491553,-22.44789368771114);
\draw [line width=2pt,color=qqzzqq] (26.179157475084455,-21.55210631228886) -- (26.179157475084455,-22.44789368771114);
\draw [line width=2pt,color=qqzzqq] (22.,-32.)-- (34.,-32.);
\draw [line width=2pt,color=qqzzqq] (27.820842524915545,-31.55210631228886) -- (27.820842524915545,-32.44789368771114);
\draw [line width=2pt,color=qqzzqq] (28.17915747508447,-31.55210631228886) -- (28.17915747508447,-32.44789368771114);
\draw [line width=2pt,color=qqzzqq] (34.,-32.)-- (36.,-32.);
\draw [line width=2pt,color=qqzzqq] (35.,-31.55210631228886) -- (35.,-32.44789368771114);
\draw [line width=2pt,color=qqzzqq] (-18.,-26.)-- (-14.,-26.);
\draw [line width=2pt,color=qqzzqq] (-16.,-25.55210631228886) -- (-16.,-26.44789368771114);
\draw [line width=2pt,color=qqqqff] (-12.,-40.)-- (-8.,-40.);
\draw [line width=2pt,color=qqqqff] (-10.,-39.55210631228886) -- (-10.,-40.447893687711144);
\draw [line width=2pt,color=qqqqff] (-12.,-40.)-- (-16.,-40.);
\draw [line width=2pt,color=qqqqff] (-13.820842524915538,-40.447893687711144) -- (-13.820842524915538,-39.55210631228886);
\draw [line width=2pt,color=qqqqff] (-14.179157475084464,-40.447893687711144) -- (-14.179157475084464,-39.55210631228886);
\draw [line width=0.1pt,dotted] (-22.,-26.)-- (-22.,-48.);
\draw [line width=0.1pt,dotted] (16.,-36.)-- (16.,-44.);
\draw [line width=0.1pt,dotted] (32.,-22.)-- (32.,-12.);
\draw [line width=0.1pt,dotted] (-12.,-40.)-- (-12.,-44.);
\draw [line width=0.1pt,dotted] (-4.,-26.)-- (-4.,-48.);
\draw [line width=0.1pt,dotted] (-14.,-26.)-- (-14.,-44.);
\draw [line width=0.1pt,dotted] (-22.,-26.)-- (-22.,-46.);
\draw [line width=0.1pt,dotted] (-18.,-26.)-- (-18.,-8.);
\draw [line width=0.1pt,dotted] (-6.,-16.)-- (-6.,-8.);
\draw [->,line width=0.8pt] (-18.,-8.) -- (-6.,-8.);
\draw (-12, -6) node {\Huge $T_2$};
\draw [->,line width=0.8pt] (-18.,-4.) -- (-14.,-4.);
\draw (-16, -2) node {\Huge $T_3$};
\draw [line width=0.1pt,dotted] (-16.,-40.)-- (-16.,-44.);
\draw [->,line width=0.8pt] (-22.,-44.) -- (-16.,-44.);
\draw (-19, -46) node {\Huge $T_1$};
\draw [line width=0.1pt,dotted] (-18.,-26.)-- (-18.,-4.);
\draw [line width=0.1pt,dotted] (-14.,-26.)-- (-14.,-4.);

\draw (-12, -58) node {\Huge (1)};
\draw (22, -58) node {\Huge (2)};

\begin{scriptsize}
\draw [fill=qqqqff] (-14.,-26.) circle (5.5pt);
\draw [fill=ffffff] (-16.,-40.) circle (5.5pt);
\draw [fill=ffffff] (-8.,-40.) circle (5.5pt);
\draw [fill=ffffff] (-22.,-26.) circle (5.5pt);
\draw [fill=black] (-20.,-16.) circle (5.5pt);
\draw [fill=black] (-18.,-26.) circle (5.5pt);
\draw [fill=ffffff] (-16.,-16.) circle (5.5pt);
\draw [fill=black] (-6.,-16.) circle (5.5pt);
\draw [fill=black] (-18.,-26.) circle (5.5pt);
\draw [fill=black] (-4.,-26.) circle (5.5pt);
\draw [fill=ffffff] (-22.,-26.) circle (5.5pt);
\draw [fill=ffffff] (-14.,-26.) circle (5.5pt);
\draw [fill=ffffff] (14.,-32.) circle (5.5pt);
\draw [fill=ffffff] (22.,-32.) circle (5.5pt);
\draw [fill=black] (16.,-36.) circle (5.5pt);
\draw [fill=black] (18.,-22.) circle (5.5pt);
\draw [fill=black] (32.,-22.) circle (5.5pt);
\draw [fill=black] (8.,-36.) circle (5.5pt);
\draw [fill=ffffff] (36.,-32.) circle (5.5pt);
\draw [fill=black] (34.,-32.) circle (5.5pt);
\draw [fill=ttqqqq] (10.,-22.) circle (5.5pt);
\draw [fill=ffffff] (20.,-22.) circle (5.5pt);
\draw [fill=black] (-12.,-40.) circle (5.5pt);
\end{scriptsize}
\end{tikzpicture}
  \caption{The twist coordinates of the 2-cylinder decompositions of Abelian differentials in (1) Figure~\ref{figtobdry3}.2 and (2) Figure~\ref{figtobdry3}.3.}
  \label{figtobdry3result}
\end{figure}
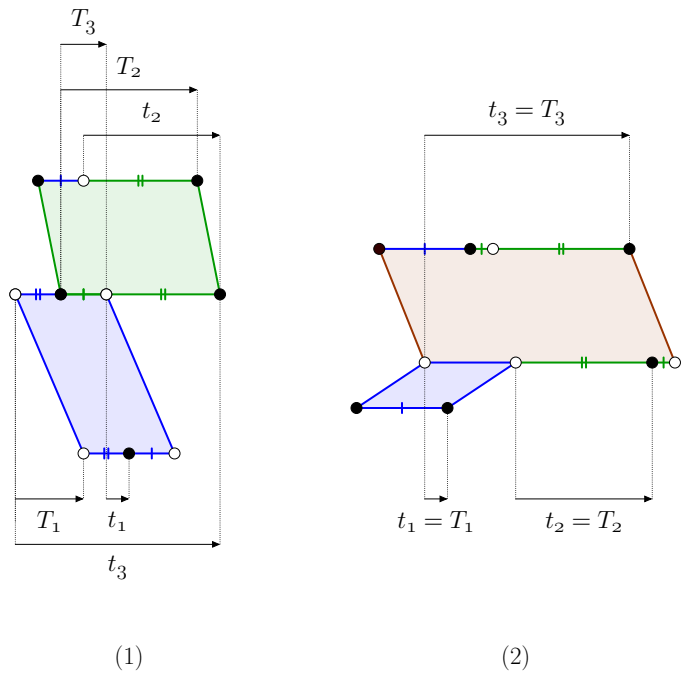
\end{commentB}

From this one can easily verify Proposition~\ref{cylboundary}, Proposition~\ref{prop:gluingwithin}:
\begin{proof}[Proof of Proposition~\ref{cylboundary}]
A vertex of the square-tiling of $\A_{d^2}$ that corresponds to the 2-cylinder decomposition of type 1 is:
\begin{itemize}
\item a non-cusp pole of $q$, whenever $T^{(1)}_3 = 0$;

\item a zero of $q$, whenever $T^{(1)}_3 = W^{(1)}_1$; and

\item a regular point of $q$ otherwise.
\end{itemize}

Similarly, a vertex of the square-tiling that corresponds to the 2-cylinder decomposition of type 2 is:
\begin{itemize}
\item a zero of $q$, whenever $T^{(2)}_2 = 0$ or $W^{(2)}_2-W^{(2)}_1$; and

\item a regular point of $q$ otherwise.
\end{itemize}

For zipping up we have $W^{(1)}_1 = w_1$, $W^{(2)}_2-W^{(2)}_1=w_2$, $T^{(1)}_3 = t_3$ or $w_1 + w_2 - t_3$ and $T^{(2)}_2 = t_3 - w_1$. It follows that: a vertex with Euclidean coordinates $(x,0)$ is a non-cusp pole of $q$, whenever $x \equiv t_3 \equiv 0 \mod w_1+w_2$, a zero of $q$, whenever $x \equiv t_3 \equiv w_1 \mbox{ or } w_2 \mod w_1+w_2$, and a regular point of $q$ otherwise; and a vertex with Euclidean coordinates $(x,0)$ is a non-cusp pole of $q$.
 
For zipping down we have $T^{(2)}_2 = t_2$. Recall that $s_2 < s_1$ in this case. It follows that a vertex with the Euclidean coordinates $(x, s_2)$ is a zero of $q$, whenever$x \equiv t_2 \equiv 0 \mod w_2$, and a regular point of $q$ otherwise.
\end{proof}

\begin{proof}[Proof of Proposition~\ref{prop:gluingwithin}]
Assume that points $(x, s_2)$ and $(x', 0)$ are obtained by vertical degenerations of 3-cylinder decompositions with twists $t_1,t_2,t_3$ and $t'_1,t'_2,t'_3$ respectively. 
Note that $$t'_3 = x' \% (w'_1 + w'_2) = (q(2w_1+w_2)  + w_1 + r) \% (2w_1 +w_2) = w_1 + r,$$
and hence it satisfies $w'_1 < t'_3 < w'_2$. Therefore both $(x, s_2)$ and $(x', 0)$ have 2-cylinder decompositions of type 2.

According to the formulas above the twist coordinates $T_1, T_2, T_3$ of the 2-cylinder decomposition of $(x, s_2)$ are:
\begin{align*}
T_1 & = t_1 = (q w_2 + r) \% w_1\\
T_2 & = t_2 = (q w_2 + r) \% w_2 = r \% w_2 = r \\
T_3 & = t_3 = (q w_2 + r) \% (w_1+w_2).
\end{align*}
Similarly, for the twist coordinates $T'_1, T'_2, T'_3$ of the 2-cylinder decomposition of $(x', s_2)$, we obtain:
\begin{align*}
T'_1 & = t'_1 = (q(2w_1+w_2)  + w_1 + r) \% w'_1 = (q w_2 + r) \% w_1 = T_1 \\
T'_2 & = t'_3 - w'_1 = w_1 + r - w_1 = r = T_2 \\
T'_3 & = (2t'_3-t'_2-w'_1) \% w'_2 = (2w_1 + 2r - (q(2w_1+w_2)  + w_1 + r)  - w_1) \% (w_1+w_2) = \\ 
& = (q w_2 + r) \% (w_1+w_2) = T_3.
\end{align*}
Therefore the points $(x, s_2)$ and $(x', 0)$ are identified.
\end{proof}

Now we prove Proposition~\ref{prop:gluingbetween} by investigating the zipping down process for the case $s_1=s_2=1$.

\begin{proof}[Proof of Proposition~\ref{prop:gluingbetween}]
Consider a generic Abelian differnetial $(X,\omega)\in \A_{d^2}$ with the cylinder coordinates $(w_1,1,w_2,1,t_1,t_2,t_3,1)$ with $t_i \notin \Z$ (see Figure~\ref{figtobdryeave}.1). Zipping down the white singularity produces a 1-cylinder Abelian differential (see Figure~\ref{figtobdryeave}.2).

\begin{commentB}
\begin{figure}[H] 
\centering
  \definecolor{qqzzqq}{rgb}{0.,0.6,0.}
\definecolor{qqqqff}{rgb}{0.,0.,1.}
\definecolor{zzttqq}{rgb}{0.6,0.2,0.}
\definecolor{ffffff}{rgb}{1.,1.,1.}
\definecolor{ttqqqq}{rgb}{0.2,0.,0.}
\begin{tikzpicture}[line cap=round,line join=round,>=triangle 45,x=0.4cm,y=0.4cm]
\clip(-37.,-50.) rectangle (68.,-20.);

\begin{scope}[shift={(-36,0)}]
\fill[line width=2pt,color=qqqqff,fill=qqqqff,fill opacity=0.10000000149011612] (12.,-32.) -- (14.,-34.) -- (6.,-34.) -- (4.,-32.) -- cycle;
\fill[line width=2pt,color=qqzzqq,fill=qqzzqq,fill opacity=0.10000000149011612] (12.,-32.) -- (18.,-34.) -- (32.,-34.) -- (26.,-32.) -- cycle;
\fill[line width=2pt,color=zzttqq,fill=zzttqq,fill opacity=0.10000000149011612] (8.,-30.) -- (4.,-32.) -- (26.,-32.) -- (30.,-30.) -- cycle;

\draw [->,line width=2pt] (4.,-36.) -- (6.,-36.);
\draw [->,line width=2pt] (12.,-36.) -- (18.,-36.);
\draw [line width=2pt,color=qqqqff] (12.,-32.)-- (14.,-34.);
\draw [line width=2pt,color=qqqqff] (14.,-34.)-- (6.,-34.);
\draw [line width=2pt,color=qqqqff] (6.,-34.)-- (4.,-32.);
\draw [line width=2pt,color=qqqqff] (4.,-32.)-- (12.,-32.);
\draw [line width=2pt,color=qqzzqq] (12.,-32.)-- (18.,-34.);
\draw [line width=2pt,color=qqzzqq] (18.,-34.)-- (32.,-34.);
\draw [line width=2pt,color=qqzzqq] (32.,-34.)-- (26.,-32.);
\draw [line width=2pt,color=qqzzqq] (26.,-32.)-- (12.,-32.);
\draw [line width=2pt,color=zzttqq] (8.,-30.)-- (4.,-32.);
\draw [line width=2pt,color=zzttqq] (26.,-32.)-- (30.,-30.);
\draw [line width=2pt,color=qqqqff] (8.,-30.)-- (16.,-30.);
\draw [line width=2pt,color=qqqqff] (12.,-29.4926887895773) -- (12.,-30.507311210422706);
\draw [line width=2pt,color=qqqqff] (10.,-29.4926887895773-4) -- (10.,-30.507311210422706-4);
\draw [line width=2pt,color=qqzzqq] (16.,-30.)-- (30.,-30.);
\draw [line width=2pt,color=qqzzqq] (22.5,-29.4926887895773) -- (22.5,-30.507311210422706);
\draw [line width=2pt,color=qqzzqq] (24.5,-29.4926887895773-4) -- (24.5,-30.507311210422706-4);
\draw [line width=0.1pt,dotted] (12.,-32.)-- (12.,-36.);
\draw [line width=0.1pt,dotted] (4.,-32.)-- (4.,-28.);
\draw [line width=0.1pt,dotted] (8.,-30.)-- (8.,-28.);
\draw [line width=0.1pt,dotted] (18.,-34.)-- (18.,-36.);
\draw [->,line width=2pt] (4.,-28.) -- (8.,-28.);
\draw [line width=0.1pt,dotted] (4.,-32.)-- (4.,-36.);
\draw [line width=0.1pt,dotted] (6.,-34.)-- (6.,-36.);

\draw (5,-38) node {\Huge $t_1$};
\draw (15,-38) node {\Huge $t_2$};
\draw (6,-26) node {\Huge $t_3$};

\draw [fill=ffffff] (4.,-32.) circle (5pt);
\draw [fill=ffffff] (12.,-32.) circle (5pt);
\draw [fill=black] (14.,-34.) circle (5pt);
\draw [fill=black] (16.,-30.) circle (5pt);
\draw [fill=black] (30.,-30.) circle (5pt);
\draw [fill=black] (6.,-34.) circle (5pt);
\draw [fill=ffffff] (26.,-32.) circle (5pt);
\draw [fill=black] (18.,-34.) circle (5pt);
\draw [fill=black] (32.,-34.) circle (5pt);
\draw [fill=ttqqqq] (8.,-30.) circle (5pt);

\end{scope}

\begin{scope}[shift={(34,0)}]

\coordinate (B2) at (6,-34);
\coordinate (C2) at (12,-34);
\coordinate (D2) at (18,-34);
\coordinate (E2) at (26,-34);
\coordinate (A2) at (28,-34);

\coordinate (G2) at (24,-30);
\coordinate (H2) at (16,-30);
\coordinate (I2) at (14,-30);
\coordinate (J2) at (8,-30);
\coordinate (F2) at (2,-30);

\fill[line width=2pt,color=zzttqq,fill=zzttqq,fill opacity=0.10000000149011612] (B2) -- (A2) -- (G2) -- (F2) -- cycle;

\draw [line width=2pt,color=qqqqff] (B2)-- (C2);
\draw [line width=2pt,color=qqqqff] (E2)-- (A2);
\draw [line width=2pt,color=qqqqff] (H2)-- (J2);

\draw [line width=2pt,color=qqzzqq] (C2)-- (E2);
\draw [line width=2pt,color=qqzzqq] (G2)-- (H2);
\draw [line width=2pt,color=qqzzqq] (J2)-- (F2);

\draw [line width=2pt,color=black] (A2)-- (G2);
\draw [line width=2pt,color=black] (B2)-- (F2);
\draw [line width=2pt,color=black] (D2)-- (I2);

\draw [line width=2pt,color=qqqqff] (15.,-29.5) -- (15.,-30.5);
\draw [line width=2pt,color=qqqqff] (11.25,-29.5) -- (11.25,-30.5);
\draw [line width=2pt,color=qqqqff] (10.75,-29.5) -- (10.75,-30.5);

\draw [line width=2pt,color=qqqqff] (27.,-29.5-4) -- (27.,-30.5-4);
\draw [line width=2pt,color=qqqqff] (11.25-2,-29.5-4) -- (11.25-2,-30.5-4);
\draw [line width=2pt,color=qqqqff] (10.75-2,-29.5-4) -- (10.75-2,-30.5-4);

\draw [line width=2pt,color=qqzzqq] (19.75,-29.5) -- (19.75,-30.5);
\draw [line width=2pt,color=qqzzqq] (20.25,-29.5) -- (20.25,-30.5);
\draw [line width=2pt,color=qqzzqq] (5,-29.5) -- (5,-30.5);

\draw [line width=2pt,color=qqzzqq] (19.75+2,-29.5-4) -- (19.75+2,-30.5-4);
\draw [line width=2pt,color=qqzzqq] (20.25+2,-29.5-4) -- (20.25+2,-30.5-4);
\draw [line width=2pt,color=qqzzqq] (27-12,-29.5-4) -- (27-12,-30.5-4);

\draw [<->,line width=0.5pt] (6.,-36.) -- (12.,-36.);
\draw [<->,line width=0.5pt] (12.,-36.) -- (18.,-36.);
\draw [<->,line width=0.5pt] (18.,-36.) -- (26.,-36.);
\draw [<->,line width=0.5pt] (26.,-36.) -- (28.,-36.);

\draw [->,line width=0.5pt] (6.,-28.) -- (2.,-28.);
\draw [line width=0.1pt,dotted] (2.,-30.)-- (2.,-28.);
\draw [line width=0.1pt,dotted] (6.,-34.)-- (6.,-28.);

\draw (9,-38) node {\LARGE $w_1-t_1$};
\draw (15,-38) node {\LARGE $t_2$};
\draw (22,-38) node {\LARGE $w_2-t_2$};
\draw (27,-38) node {\LARGE $t_1$};
\draw (4,-26) node {\LARGE $t'_3$};

\draw [fill=ffffff] (A2) circle (5pt);
\draw [fill=ffffff] (C2) circle (5pt);
\draw [fill=ffffff] (E2) circle (5pt);
\draw [fill=ffffff] (G2) circle (5pt);
\draw [fill=ffffff] (I2) circle (5pt);
\draw [fill=black] (B2) circle (5pt);
\draw [fill=black] (D2) circle (5pt);
\draw [fill=black] (F2) circle (5pt);
\draw [fill=black] (H2) circle (5pt);
\draw [fill=black] (J2) circle (5pt);
\end{scope}

\begin{scope}[shift={(0,0)}]

\coordinate (A) at (4,-34);
\coordinate (B) at (6,-34);
\coordinate (C) at (12,-34);
\coordinate (D) at (18,-34);
\coordinate (E) at (26,-34);
\coordinate (F) at (30,-30);
\coordinate (G) at (24,-30);
\coordinate (H) at (16,-30);
\coordinate (I) at (14,-30);
\coordinate (J) at (8,-30);

\fill[line width=2pt,color=zzttqq,fill=zzttqq,fill opacity=0.10000000149011612] (A) -- (E) -- (F) -- (J) -- cycle;

\draw [line width=2pt,color=qqqqff] (A)-- (C);
\draw [line width=2pt,color=qqqqff] (C)-- (H);
\draw [line width=2pt,color=qqqqff] (H)-- (J);
\draw [line width=2pt,color=qqqqff] (J)-- (A);
\draw [line width=2pt,color=qqzzqq] (C)-- (E);
\draw [line width=2pt,color=qqzzqq] (E)-- (F);
\draw [line width=2pt,color=qqzzqq] (F)-- (H);
\draw [line width=2pt,color=qqzzqq] (H)-- (C);

\draw [line width=2pt,color=qqqqff] (15.,-29.5) -- (15.,-30.5);
\draw [line width=2pt,color=qqqqff] (11.25,-29.5) -- (11.25,-30.5);
\draw [line width=2pt,color=qqqqff] (10.75,-29.5) -- (10.75,-30.5);

\draw [line width=2pt,color=qqqqff] (15.-10,-29.5-4) -- (15.-10,-30.5-4);
\draw [line width=2pt,color=qqqqff] (11.25-2,-29.5-4) -- (11.25-2,-30.5-4);
\draw [line width=2pt,color=qqqqff] (10.75-2,-29.5-4) -- (10.75-2,-30.5-4);

\draw [line width=2pt,color=qqzzqq] (19.75,-29.5) -- (19.75,-30.5);
\draw [line width=2pt,color=qqzzqq] (20.25,-29.5) -- (20.25,-30.5);
\draw [line width=2pt,color=qqzzqq] (27,-29.5) -- (27,-30.5);

\draw [line width=2pt,color=qqzzqq] (19.75+2,-29.5-4) -- (19.75+2,-30.5-4);
\draw [line width=2pt,color=qqzzqq] (20.25+2,-29.5-4) -- (20.25+2,-30.5-4);
\draw [line width=2pt,color=qqzzqq] (27-12,-29.5-4) -- (27-12,-30.5-4);

\draw [<->,line width=0.5pt] (4.,-36.) -- (6.,-36.);
\draw [<->,line width=0.5pt] (6.,-36.) -- (12.,-36.);
\draw [<->,line width=0.5pt] (12.,-36.) -- (18.,-36.);
\draw [<->,line width=0.5pt] (18.,-36.) -- (26.,-36.);

\draw [->,line width=0.5pt] (4.,-28.) -- (8.,-28.);
\draw [line width=0.1pt,dotted] (4.,-34.)-- (4.,-28.);
\draw [line width=0.1pt,dotted] (8.,-30.)-- (8.,-28.);

\draw (5,-38) node {\LARGE $t_1$};
\draw (9,-38) node {\LARGE $w_1-t_1$};
\draw (15,-38) node {\LARGE $t_2$};
\draw (22,-38) node {\LARGE $w_2-t_2$};
\draw (6,-26) node {\LARGE $t_3$};

\draw [fill=ffffff] (A) circle (5pt);
\draw [fill=ffffff] (C) circle (5pt);
\draw [fill=ffffff] (E) circle (5pt);
\draw [fill=ffffff] (G) circle (5pt);
\draw [fill=ffffff] (I) circle (5pt);
\draw [fill=black] (B) circle (5pt);
\draw [fill=black] (D) circle (5pt);
\draw [fill=black] (F) circle (5pt);
\draw [fill=black] (H) circle (5pt);
\draw [fill=black] (J) circle (5pt);

\end{scope}

\draw (-18,-44) node {\Huge (1)};
\draw (16,-44) node {\Huge (2)};
\draw (52,-44) node {\Huge (3)};

\end{tikzpicture}
  \caption{(1) An Abelian differential $(X,\omega)\in \A_{d^2}$ with $s_1=s_2=1$, (2) zipping down of its white singularity, and (3) its second splitting into two tori.}
  \label{figtobdryeave}
\end{figure}
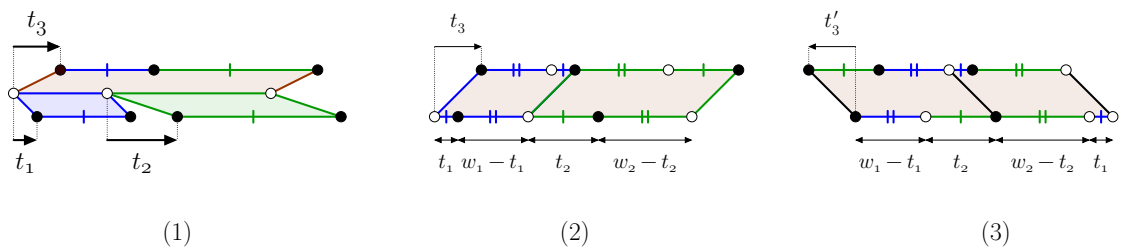
\end{commentB}

An Abelian differential obtained in this way comes with a choice of a {\em splitting} into two tori $\C/\Lambda_1 \underset{I}{\#} \C/\Lambda_2$, where $\Lambda_1 = w_1\Z \oplus (t_3-t_1+i)\Z$, $\Lambda_2=w_2\Z \oplus (t_3-t_2+i)\Z$ and $I=t_3+i$. For background on splitting, or connected sum of 1-forms see, for example, Section 7 of \cite{McM07a}. There is another splitting for the same 1-cylinder Abelian differential (see Figure~\ref{figtobdryeave}.3) into two tori $\C/\Lambda'_1 \underset{I'}{\#} \C/\Lambda'_2$, where $\Lambda'_1 = (w_1-t_1+t_2)\Z \oplus (t_3-t_1+i)\Z$, $\Lambda'_2=(w_2+t_1-t_2)\Z \oplus (t_3-t_2+i)\Z$ and $I'=t_3-t_1-t_2+i$. This splitting can be obtained from zipping down a generic $(X',\omega')\in \A_{d^2}$ given by the cylinder coordinates $(w'_1,1,w'_2,1,t'_1,t'_2,t'_3,1)$, where:
\begin{align*}
w'_1 &=w_1-t_1+t_2 \\
w'_2 &=w_2+t_1-t_2 \\
t'_1 &= (-t_2) \% (w_1-t_1+t_2) = (w_1-t_1)\% (w_1-t_1+t_2) \\
t'_2 &= (-t_1) \% (w_2+t_1-t_2) = (w_2-t_2)\% (w_2+t_1-t_2) \\
t'_3 &=t_3-t_1-t_2.
\end{align*}
Therefore the two points $(w_1,1,w_2,1,t_1,t_2,t_3,1)$ and $(w'_1,1,w'_2,1,t'_1,t'_2,t'_3,1)$ are identified since they represent the same 1-cylinder Abelian differential.
\end{proof}

\noindent
{\bf Pictures of the pagoda structures of $X(d)$.} We conclude by presenting pictures of the pagoda structures of $X(d)$ for $d=7,11,13$ and $17$ (see Figure~\ref{figX(7)}, \ref{figX(11)}, \ref{figX(13)} and \ref{figX(17)}). When viewed on a computer, the pictures can be zoomed in to see the structures of the boundaries: the red points are the cusps of $X(d)$, the white points are the remaining simple poles of $q$ and the black points are the simple zeroes of $q$. The identifications of the boundaries of the strips within each story of the pagoda should be clear from the pictures and the gluing instructions (see Propositions~\ref{prop:gluingwithin} and \ref{prop:gluingbetween}) described in this section.

\begin{commentB}
\begin{figure}[H]
    \centering
    \includegraphics[width=\textwidth]{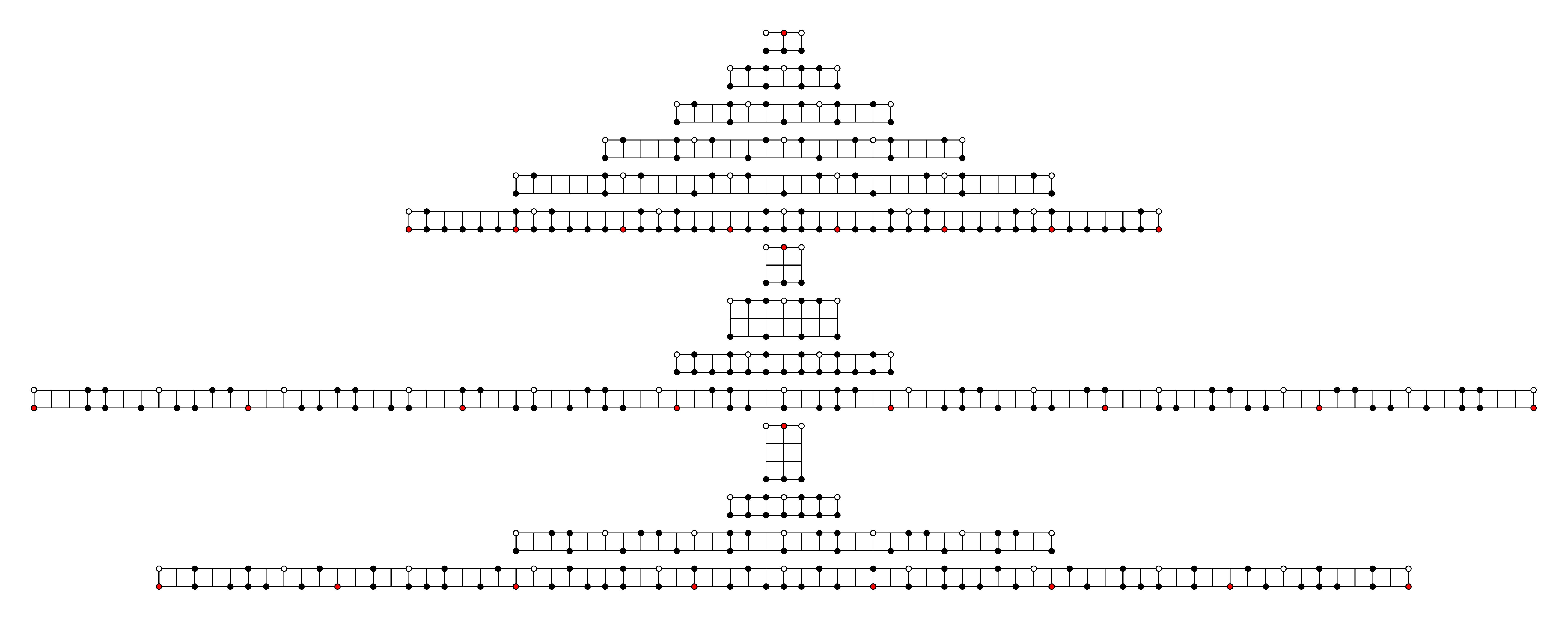}
    \caption{The pagoda structure of X(7).}
     \label{figX(7)}
\end{figure}

\begin{figure}[H]
    \centering
    \includegraphics[width=\textwidth]{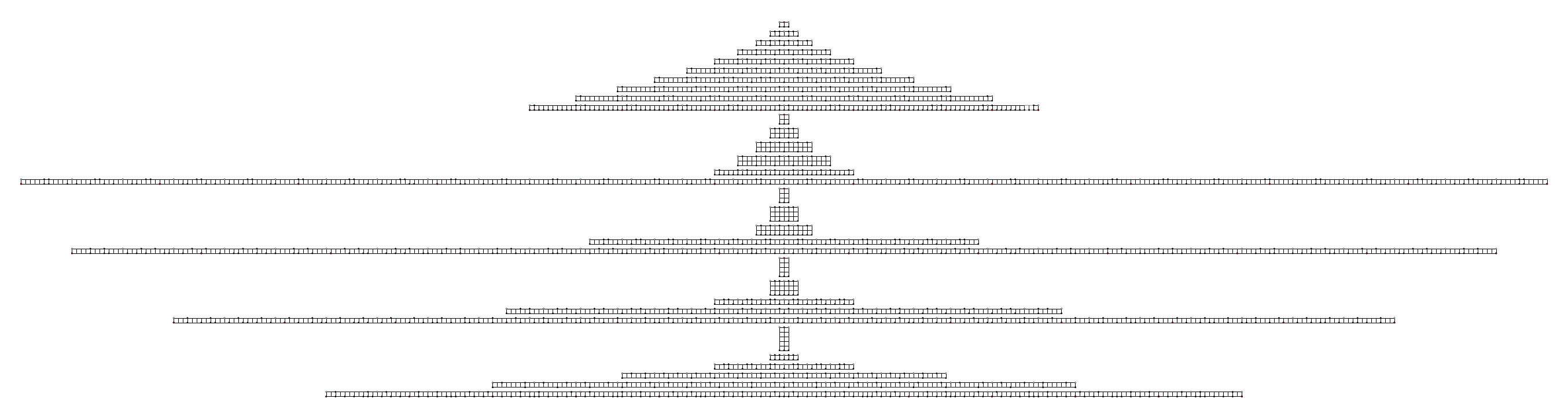}
    \caption{The pagoda structure of X(11).}
     \label{figX(11)}
\end{figure}

\begin{figure}[H]
    \centering
    \includegraphics[width=\textwidth]{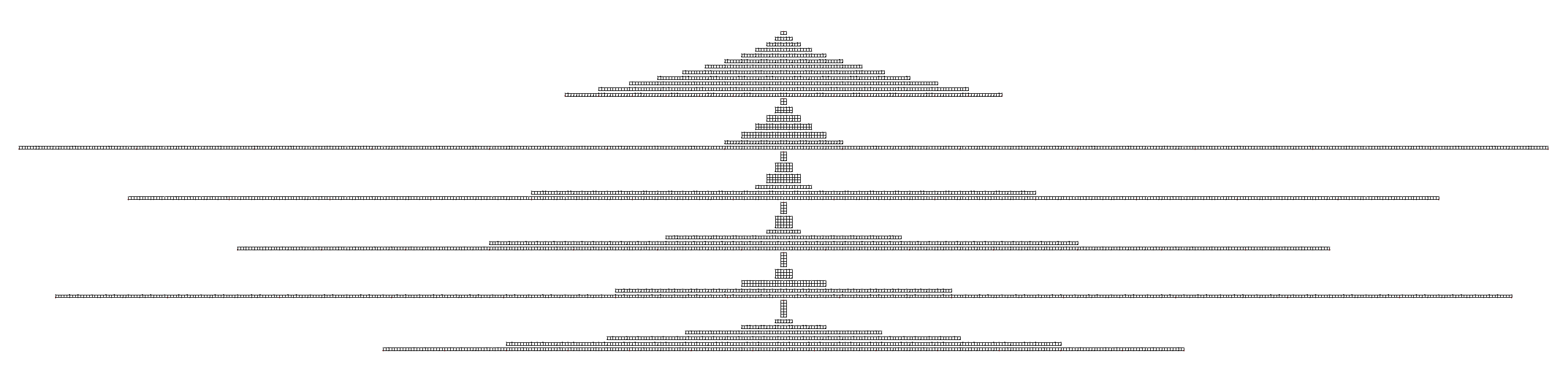}
    \caption{The pagoda structure of X(13).}
     \label{figX(13)}
\end{figure}

\begin{figure}[H]
    \centering
    \includegraphics[width=\textwidth]{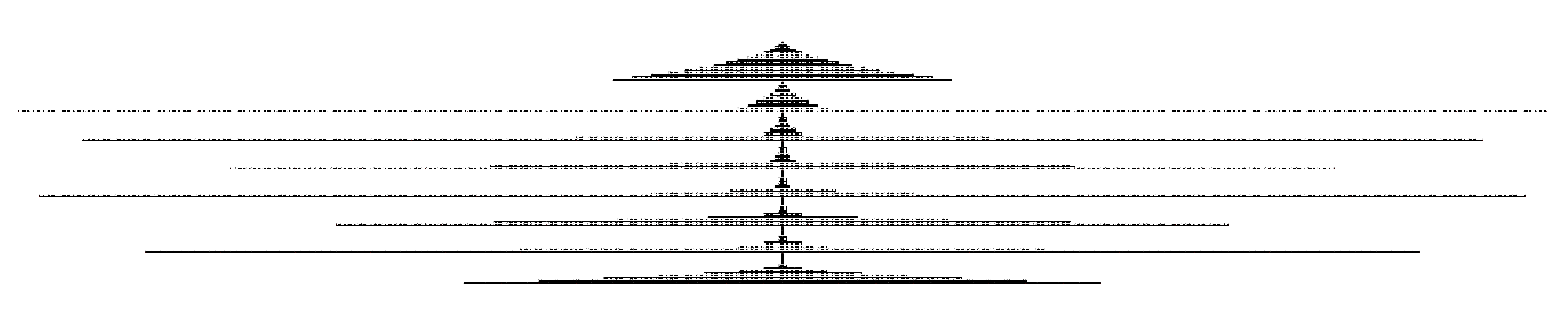}
    \caption{The pagoda structure of X(17).}
     \label{figX(17)}
\end{figure}
\end{commentB}

\end{document}